\pdfoutput=1

\documentclass[
	bigMargin = false,
	font = palatino,
	final,
	]{MyClass}

\usepackage{filecontents}

\usepackage{myMathHeader}
\RenewDocumentCommand{\setSymbol}{o}{
	\nonscript \, #1 \mid 
	\allowbreak
	\nonscript \,
	\mathopen{ }
}
\renewcommand{\symplMatrix}{\mathbb{J}}

\let\ConnSpace\relax
\let\GauGroup\relax
\let\DiffGroup\relax
\let\AutGroup\relax
\MyNewMathOperator{\ConnSpace}		{command={\mathrm{C}}}
\MyNewMathOperator{\GauGroup}		{command={\mathrm{Gau}}}
\MyNewMathOperator{\DiffGroup}		{command={\mathrm{Diff}}}
\MyNewMathOperator{\AutGroup}		{command={\mathrm{Aut}}}

\crefname{section}{Chapter}{Chapters}
\crefname{subsection}{Section}{Sections}
\crefname{subsubsection}{Subsection}{Subsections}

\Title{Group-valued momentum maps for actions of automorphism groups}

\Abstract{
The space of smooth sections of a symplectic fiber bundle carries a 
natural symplectic structure. We provide a general framework to determine 
the momentum map for the action of the group of bundle automorphism on 
this space. Since, in general, this action does not admit a classical 
momentum map, we introduce the more general class of group-valued momentum 
maps which is inspired by the Poisson Lie setting. In this approach, the 
group-valued momentum map assigns to every section of the symplectic fiber 
bundle a principal circle-bundle with connection. The power of this 
general framework is illustrated in many examples: we construct generalized 
Clebsch variables for fluids with integral helicity; the anti-canonical 
bundle turns out to be the momentum map for the action of the group of  
symplectomorphisms on the space of compatible complex structures; the 
Teichm\"uller moduli space is realized as a symplectic orbit reduced space 
associated to a coadjoint orbit of $\SLGroup(2,\R)$ and spaces 
related to the other coadjoint orbits are identified and studied.
Moreover, we show that the momentum map for the group of bundle automorphisms on the space of connections over a Riemann  surface encodes, besides the curvature, also topological information of the bundle.
}



\Institution{mpi}{
	Max Planck Institute for Mathematics in the Sciences,
	04103 Leipzig, Germany
	and
	Institut für Theoretische Physik, 
	Universität Leipzig, 
	04009 Leipzig, Germany 
	and Institute of Applied Mathematics, Delft University of Technology, 2628 XE Delft, Netherlands}
\Institution{sjtu}{
	School of Mathematical Sciences and Ministry of Education 
	Laboratory of Scientific Computing (MOE-LSC), 
	Shanghai Jiao Tong University, Minhang District,
	800 Dongchuan Road, 200240 China,  
	Section de Math\'ematiques, Universit\'e de Gen\`eve,
	2--4 rue du Li\`evre, case postale 64,
	1211 Gen\`eve 4, and Ecole Polytechnique F\'ed\'erale
	de Lausanne, 1015 Lausanne,  
	Switzerland	}

\Author[
	affiliation = mpi,
	email = tobias.diez@mis.mpg.de,
	funding = {Partially supported by the Max Planck Institute for Mathematics in the Sciences (Leipzig), by the German National Academic Foundation (DAAD) and by the above NWO grant 639.032.734.}
	]{diez}{Tobias Diez}
\Author[
	affiliation = sjtu,
	email = {ratiu@sjtu.edu.cn, tudor.ratiu@epfl.ch},
	funding = {Partially supported by by the National Natural Science Foundation of China (No. 11871334) and by NCCR SwissMAP grant of the Swiss National Science Foundation.}
	]{ratiu}{Tudor S. Ratiu}

\begin{document}

\date{\today}
\MakeTitle
\vskip 1.8em
\tableofcontents


\section{Introduction}

Noether's theorem states that every symmetry of a given system has a 
corresponding conservation law.
In terms of symplectic geometry, these conserved quantities are encoded 
in the momentum map.
It was, however, quickly realized that the momentum map geometry not only 
plays an important role in dynamical systems but is also a valuable tool 
in the study of differential geometric questions.
In their seminal work, \textcite{AtiyahBott1983} showed that the curvature 
of a connection on a principal bundle over a Riemannian surface is the 
momentum map for the action of the group of gauge transformations.
They applied Morse theory to the norm-squared of the momentum map (the 
Yang--Mills functional) in order to obtain the cohomology of the moduli 
space of Yang--Mills solutions which, by the Narasimhan--Seshadri theorem, 
can be identified with the moduli space of stable holomorphic structures.
Within the same circle of ideas, \textcite{Fujiki1992} and 
\citeauthor{Donaldson1997} 
\parencite{Donaldson1997,Donaldson1999,Donaldson2003} provided a momentum 
map picture for the relationship between the existence of constant scalar 
curvature K\"ahler metrics and stability in the sense of geometric 
invariant theory.
\medskip

\emph{The first main aim of this paper is to provide a framework which 
encompasses the gauge theory setting of \citeauthor{AtiyahBott1983} and, 
at the same time, the action of diffeomorphism groups of 
\citeauthor{Fujiki1992} and \citeauthor{Donaldson1997}.}
Our starting point is a symplectic fiber bundle of the form 
\( F = P \times_G \FibreBundleModel{F} \) for a principal \( G \)-bundle 
\( P \to M \), where the typical fiber \( \FibreBundleModel{F} \) is 
endowed with a \( G \)-invariant symplectic form. 
The fiberwise symplectic structure, combined with a volume form on the 
base \( M \), induces a symplectic form \( \Omega \) on the space 
\( \SectionSpaceAbb{F} \) of sections of \( F \to  M \).
The gauge group of \(P\) acts in a natural way on \( \SectionSpaceAbb{F} \), 
leaving the induced symplectic form \( \Omega \) invariant.
As we will see, this action possesses a momentum map which is completely 
determined by the momentum map of the \( G \)-action on the fiber 
\( \FibreBundleModel{F} \). Suppose  the bundle \( P \) is natural, \ie, 
it comes with a lift of \( \DiffGroup(M) \) to bundle automorphisms  (for 
example, this is the case when \( P \) is the frame bundle of \(M\)).
In this case, every subgroup of the  group of volume-preserving 
diffeomorphisms acts naturally on the space \( \SectionSpaceAbb{F} \) 
of sections and leaves \( \Omega \) invariant. Our first result is 
summarized in \Cref{thm::symplecticFibreBundle:momentumMapVolume,thm::symplecticFibreBundle:momentumMapSymplectic} which determine the momentum map for the group of 
volume-preserving diffeomorphisms and symplectomorphisms, respectively.
There are, essentially, two contributions to the momentum map.
The first term is the pull-back of the fiberwise symplectic structure. 
The second term involves the fiber momentum map and, morally speaking, 
captures how much the lift of diffeomorphisms to bundle automorphisms 
shifts in the vertical direction. The interesting point is that the 
momentum map for the automorphism group on the \emph{infinite-dimensional} 
space of sections is canonically constructed from the 
\emph{finite-dimensional} symplectic \( G \)-manifold 
\( \FibreBundleModel{F} \).

In contrast to the case of the action of the gauge group, the momentum map 
for the symplectic action of the diffeomorphism group on the space of 
sections does not exist in full generality. 
This was already pointed out in \parencite{Donaldson1999,Donaldson2003}.
The obstruction has a topological character, \ie, certain cohomology groups 
have to vanish. To remedy this situation, one restricts attention  to a 
certain \enquote{exact} subgroup, \eg, the subgroup of Hamiltonian 
diffeomorphisms in the group of all symplectomorphisms. The action of 
this subgroup then usually admits  a classical momentum map.
Working from a completely different point of view, similar observations 
were made by \textcite{Gay-BalmazVizman2012} in their study of the 
classical dual pair in hydrodynamics.
In this case, the symplectic action of volume-preserving diffeomorphisms 
on a symplectic manifold of mappings only has a momentum map under certain 
topological conditions and one is forced to work with suitable central 
extensions of the group of exact volume-preserving diffeomorphisms.

In this paper, we take the viewpoint that the above mentioned topological 
obstructions are not a bug but a feature of the theory.
The action of the diffeomorphism group interacts with, and is largely 
determined by, the topological structure of the bundle.
Thus, one would expect to capture certain topological data (like 
characteristic classes) that are \enquote{conserved} by the action.
Such \enquote{conservation laws} should be encoded in the momentum map. 
Since the classical momentum map takes values in a continuous vector space, 
there is no space to \enquote{store} discrete topological information.
Hence, whenever those classes do not vanish, a classical  momentum map 
does not exists. Nonetheless, one could hope that a generalized momentum map exists and captures the conserved topological data. 
\medskip

\emph{The second main aim of the paper is to translate these philosophical 
remarks into explicit mathematical statements. In order to do this, 
we generalize the notion of momentum maps allowing them to take values 
in groups.} Our concept of a group-valued momentum map is inspired by the 
Lu momentum map \parencite{Lu1990,LuWeinstein1990} in Poisson geometry.
We emphasize right away that the group-valued momentum map we introduce 
in this paper is a vast generalization of many notions of momentum maps 
appearing in the literature including circle-valued, cylinder-valued, and 
Lie algebra-valued momentum maps. We show that our generalized group-valued 
momentum map always exists for the action of the diffeomorphism group, 
without any topological assumptions on the base but some integrability 
conditions on the fiber model. The resulting momentum map captures 
topological invariants of the geometry, exactly in (the dual of) those 
cohomology classes which prevented the existence of a classical momentum 
map. This approach of extending the definition of the momentum map, besides 
the situation described above in Poisson geometry, in order to capture 
conservation laws not available using the classical definition, has been 
used successfully before in the theory of the cylinder-valued and optimal 
momentum maps; see \parencite{OrtegaRatiu2003} for a detailed presentation.

To see how our approach works, consider the following setting.
Let \( (M, \mu) \) be a closed (\ie, compact and boundaryless) 
\(n\)-manifold with volume form \( \mu \) and \( (F, \omega) \) a 
symplectic manifold. The space \( \sFunctionSpace(M, F) \) of smooth 
maps from \( M \) to \( F \) carries the weak symplectic form
\begin{equation}
\Omega_\phi(X,Y) = \int_M \omega_{\phi(m)} \bigl(X(m), Y(m)\bigr)\, \mu(m),
\end{equation}
where \( \phi \in \sFunctionSpace(M, F) \) and \( X, Y \in \TBundle_\phi 
\sFunctionSpace(M, F) \), \ie, the maps \( X,Y: M \to \TBundle F \) satisfy 
\( X(m), Y(m) \in \TBundle_{\phi(m)}F \) for all \( m \in M \).
The natural action by precomposition of \( \DiffGroup_\mu(M) \), the 
group of diffeomorphisms of \( M \) preserving the volume form \( \mu \), 
leaves \( \Omega \) invariant. If \( \omega \) is exact, say with 
primitive \( \vartheta \), then the momentum map is given by
\begin{equation}
{}\mkern1mu\overline{\mkern-1mu \SectionMapAbb{J}}: \sFunctionSpace(M, F) 
\to \DiffFormSpace^1(M) \slash \dif \DiffFormSpace^0(M), \qquad \phi 
\mapsto \equivClass{\phi^* \vartheta},
\end{equation}
where the space of volume-preserving vector fields 
\(\VectorFieldSpace_\mu(M)\) (the vector fields whose \( \mu \)-divergence 
vanishes) is identified with the space of  closed \( (n-1) \)-forms, so 
that \( \VectorFieldSpace_\mu(M)^* \isomorph 
\DiffFormSpace^1(M) \slash \dif \DiffFormSpace^0(M) \) with respect to 
the natural integration pairing. More generally, 
\textcite{Gay-BalmazVizman2012} showed that a (non-equivariant) momentum 
map also exists when the pull-back of 
\( \omega \) by all maps \( \phi \in \sFunctionSpace(M, F) \) is exact; 
for example, this happens when \( \sCohomology^2(M, \R) \) is trivial.
Our generalized group-valued momentum map takes no longer values in 
\( \VectorFieldSpace_\mu(M)^* \), but instead in the Abelian group 
\( \csCohomology^2(M, \UGroup(1)) \) that parametrizes principal circle 
bundles with connections modulo gauge equivalence. If \( (F, \omega) \) 
has a prequantum bundle \( (L, \vartheta) \), then the map
\begin{equation}
\SectionMapAbb{J}: \sFunctionSpace(M, F) \to \csCohomology^2(M, \UGroup(1)), 
\qquad \phi \mapsto \equivClass{\phi^* (L, \vartheta)},
\end{equation}
which sends \( \phi \) to the pull-back bundle with connection 
\( \phi^* (L, \vartheta) \), is a group-valued momentum map. 
We see that no (topological) restrictions have to be imposed on \( M \) 
and only the integrability condition of the symplectic form \( \omega \) 
on \( F \) is needed for the existence of a group-valued momentum map.
In contrast to the classical momentum map, a 
\( \csCohomology^2(M, \UGroup(1)) \)-valued momentum map contains 
topological information. First, the Chern class of the bundle, as a 
class in \( \sCohomology^2(M, \Z) \), is available from the generalized 
momentum map. In our simple example, this is the pull-back of the Chern 
class of \( L \). A second class in \( \sCohomology^1(M, \UGroup(1)) \) 
is related to the equivariance of the momentum map; see 
\cref{rem:momentumMapDiffVol:topologicalData} for details. 
\medskip 

\begin{table}[tbp]
   	\makebox[\textwidth]{%
	\centering
	\small
	\renewcommand{\arraystretch}{2}
	\begin{tabular}{l l l l l}
		\toprule
			&
			Space &
			Action of &
			Chern class &
			\makecell[c]{Secondary \\ topological class}
			\\
		\midrule
			Hydrodynamics &
			\( \sFunctionSpace(M, F) \) & 
			\( \DiffGroup_\mu(M) \) &
			\makecell[l]{\( 0 \) (total vorticity) \\ 
			in \( \sCohomology^2(M, \Z) \)} &
			\makecell[l]{Circulations \\ in \( \sCohomology^1(M, \R) \)}
			\\
			\makecell[l]{Lagrangian \\ embeddings} &
			\( \sFunctionSpace(L, M) \) &
			\( \DiffGroup_\mu(L) \) &
			\makecell[l]{Torsion class \\ in \( \sCohomology^2(M, \Z) \)} &
			\makecell[l]{Liouville class \\ in \( \sCohomology^1(M, \R) \)}
			\\
			Kähler geometry &
			\( \sSectionSpace(\FrameBundle M \times_{\SpGroup} \SpGroup 
			\slash \UGroup) \) &
			\( \DiffGroup_\omega(M) \) &
			\makecell[l]{\( c_1(M) \cup \equivClass{\omega}^{n-1} \) \\ 
			in \( \sCohomology^{2n}(M, \Z) \)}
			\\
			Quantomorphism &
			\( \sSectionSpace(P \times_{\UGroup(1)} \C) \) &
			\( \AutGroup_\Gamma(P) \) &
			trivial &
			\\
			Gauge theory &
			\( \ConnSpace(P) \) &
			\( \AutGroup(P) \) &
			\makecell[l]{Torsion class \\ in \( \sCohomology^{2n}(M, \Z) \)}
			\\
		\bottomrule
	\end{tabular}
	}
\caption{Overview of the examples discussed in \cref{sec::applications}.
Here, \( \mu \) is a volume form and \( \omega \) a symplectic form.
Moreover, \( Q \to M \) denotes a prequantum circle bundle with connection 
\( \Gamma \) and \( P \to M \) is an arbitrary principal \( G \)-bundle.
The frame bundle is denoted by \( \FrameBundle M \).
We also abbreviated the homogeneous space \( \SpGroup(2n, \R) \slash 
\UGroup(n) \) by \( \SpGroup \slash \UGroup \).
	}
\label{table:introduction:overviewExamples}
\end{table}

In \cref{table:introduction:overviewExamples}, we illustrate the conclusions 
implied by our general framework in several important examples.
Next, we comment separately on each one of them.

\Textcite{MarsdenWeinstein1983} construct Clebsch variables for ideal 
fluids starting from a similar infinite-dimensional symplectic system as 
discussed above. It turns out, that every vector field represented in those 
Clebsch variables has vanishing helicity, \ie, such a fluid configuration 
has trivial topology and no links or knots.
Thus, topological interesting configurations such as the ones constructed in \parencite{EncisoPeraltaSalas2012} cannot be written in terms of classical Clebsch variables. 
Our more general framework allows to construct generalized Clebsch variables for vector fields with integral helicity; see \cref{sec:applications:hydrodynamics}.

When applied to the space of Lagrangian immersions, the group-valued 
momentum map recovers the Liouville class as the conserved topological 
data. Moreover, we realize moduli spaces of Lagrangian immersions (and modifications thereof) as symplectic quotients (reduced spaces).

A wide range of interesting examples with geometric significance are 
obtained when the typical fiber \( \FibreBundleModel{F} \) is a 
symplectic homogeneous space \( G \slash H \).
In this case, sections of \( \FrameBundle M \times_G \FibreBundleModel{F} \) 
correspond to reductions of the \( G \)-frame bundle \( \FrameBundle M \) to 
\( H \). Special focus is put on the space of almost complex structures 
compatible with a given symplectic structure, \ie, 
\( \FibreBundleModel{F} = \SpGroup(2n, \R) \slash \UGroup(n) \).
In this case, the group-valued momentum map for the group of 
symplectomorphisms assigns to an almost complex structure the 
anti-canonical bundle. It was already observed by 
\textcite{Fujiki1992,Donaldson1997} that the Hermitian scalar curvature 
furnishes a classical momentum map for the action of the group of 
Hamiltonian symplectomorphisms. Of course, the Hermitian scalar curvature 
is the curvature of the anti-canonical bundle.
Thus, the group-valued momentum map combines the geometric curvature 
structure with the topological data of the anti-canonical bundle.
For the case of a \( 2 \)-dimensional base manifold, we realize the 
Teichm\"uller moduli space with the symplectic Weil--Petersson form as a 
symplectic orbit reduced space.

Finally, we revisit the classical gauge theoretic setting of \textcite{AtiyahBott1983} and 
extend their work in two ways.
First, we generalize from \( 2 \)-dimensional surfaces to arbitrary symplectic manifolds \( M \) as the base (a similar extension was already discussed by \textcite{Donaldson1987a}).
Secondly, we investigate the group-valued momentum map for the action of the \emph{full automorphism group} on the space of connections.
The existence of the momentum map for this action is tightly connected to the triviality of a certain class in \( \sCohomology^3(M, \Z) \) that is canonically constructed from the symplectic structure and the principal \( G \)-bundle.
\medskip

We emphasize that, in contrast to most papers discussing 
infinite-dimensional symplectic geometry, we do not work formally, but 
really address the functional analytical problems arising from the 
transition to the infinite-dimensional setting. In particular, smoothness of 
maps between infinite dimensional manifolds is understood in the sense of 
locally convex spaces as, for example, presented in \parencite{Neeb2006}.

Throughout the paper, integrality of certain symplectic forms plays a 
central role. We remark that, to a large extent, this assumption was made 
for convenience. Most results carry over directly to symplectic forms 
with discrete period groups \( \per \omega \subseteq \R \), without much
technical effort. In spirit, our results also hold in the general setting 
without any assumptions on the period group; however, then one is forced 
to work in the diffeological category because the quotient 
\( \R \slash \per \omega \) may no longer be a Lie group.

Finally, we note that most of our symplectic reduced spaces are obtained 
as (sometimes singular) \emph{orbit} reduced spaces, a theory that is not 
yet present in the literature for infinite dimensional systems, even though 
we state theorems using it. However, the techniques in \parencite{DiezThesis} 
which completely treats infinite dimensional singular symplectic 
\emph{point} reduction, combined with the strategy in 
\parencite{OrtegaRatiu2003} for finite dimensional singular symplectic 
orbit reduction, yields a general theory of infinite dimensional singular 
symplectic orbit reduction, which is precisely what is needed here.
In the interest of the focus of this paper, as well as its length, this 
theory will be presented in a future article.

\paragraph*{Structure of the paper}
Starting from the Poisson setting, \cref{sec::groupMomentumMap} introduces 
the notion of a group-valued momentum map. It is shown that this 
generalized momentum map still satisfies Noether's theorem. 
Moreover, existence and uniqueness questions are addressed.
\Cref{sec::groupMomentumMap:forExtension} prepares the calculation of the 
momentum map for the automorphism group by giving a construction of the 
momentum map for a group extension (to the best of our knowledge, this 
construction is original, even for classical momentum maps).
In \cref{sec::globalAnalysisSymplecticFibreBundles}, the symplectic 
geometry of the space of sections of a symplectic fiber bundle is 
discussed. First, the momentum map for the gauge group is determined in 
\cref{sec:momentum_map_for_the_gauge_group} and then for the group of 
volume-preserving diffeomorphisms, culminating in 
\cref{thm::symplecticFibreBundle:momentumMapVolume}.
The momentum map for the group of symplectomorphisms, given in 
\cref{thm::symplecticFibreBundle:momentumMapSymplectic}, is obtained by 
viewing it as a subgroup of the volume-preserving diffeomorphisms.
Finally, \cref{sec::applications} discusses the applications outlined above.

For the convenience of the reader, in \cref{sec::conventions} we summarize 
the notations and conventions from tensor calculus, symplectic and Poisson 
geometry, \etc, in force throughout this article.
In \cref{sec::fibreIntegration} and \cref{sec::differentialCharacter}, 
we recall the main properties of fiber integration and differential 
characters, respectively. The only original work in the appendices is 
\cref{prop::differentialCharacter:pullback}, 
\cref{sec::hatProductFibreBundles}, and 
\cref{sec::hatProduct:differentialCharacters}, where we calculate the 
derivative of the pull-back map and generalize the hat product of 
\parencite{Vizman2011} to fiber bundles and differential characters, 
respectively.

\paragraph*{Acknowledgments}
We are indebted to F.~Gay-Balmaz, J.~Huebschmann, B.~Janssens, J.-H.~Lu, K.-H.~Neeb, 
G.~Rudolph, and C.~Vizman for helpful discussions.
We gratefully acknowledge support of the German National Academic 
Foundation (DAAD), the National Natural Science Foundation of China (NNSFC),
the Swiss National Science Foundation (SNF), the Max Planck Institute for 
Mathematics in the Sciences in Leipzig, the Swiss Federal Institute of 
Technology in Lausanne, and the Shanghai Jiao Tong University which
made our collaboration possible.

\section{Group-Valued Momentum Maps}
\label{sec::groupMomentumMap}
In order to handle the full groups of volume-preserving diffeomorphisms or 
symplectomorphisms, we need a more general concept of a momentum map. 
Our starting point is the notion of a momentum map in Poisson geometry 
as introduced by \citeauthor{LuWeinstein1990} 
\parencite{LuWeinstein1990,Lu1990}. We hasten to add that this 
group-valued momentum map is not built on the pattern from the theory of 
quasi-Hamiltonian actions and shall comment on the relation between the 
two later on. Standard definitions and terminology, with the specification 
of sign and coefficient conventions, are recalled in 
\cref{sec::conventions}.

\subsection[Poisson momentum maps]{Poisson Lie group momentum maps}
 
All manifolds and Lie groups occurring in this subsection are assumed 
to be finite-dimensional. The proofs of the statements below can be found 
in \parencite{LuWeinstein1990, Lu1990}.

A Lie group \( G \) is a \emphDef{Poisson Lie group} if it is simultaneously 
a Poisson manifold (relative to the underlying manifold structure) such that 
group multiplication and inversion are Poisson maps.
Let \( \varpi_G \in \VectorFieldSpace^2(G) \) denote the Poisson tensor 
of \( G \). Let \( (M, \varpi_M) \) be a Poisson manifold.
A \emphDef{Poisson action} of the Poisson Lie group \( G \) on 
\( (M, \varpi_M) \) is a smooth (left) action \( G \times M \to M \) 
which is, in addition, a Poisson map (with \( G \times M \) endowed with 
the product Poisson structure \( \varpi_G \times \varpi_M \), \ie, the 
Poisson bracket of two functions depending only on \( G \) is given by 
the Poisson bracket on \( G \), the Poisson bracket of two functions 
depending only on \( M \) is given by the Poisson bracket on \( M \), 
and functions on \( G \) Poisson commute with functions on \( M \)).

The Poisson tensor \( \varpi_G \) of a Poisson Lie group \( G \), with Lie 
algebra \( \LieA{g} \), necessarily vanishes at the identity element 
\(e \in G\), which then allows for the definition of the \emphDef{intrinsic 
derivative} \( \epsilon: \LieA{g} \to \LieA{g}\wedge \LieA{g} \) by 
\( \epsilon(A)\defeq (\difLie_X\varpi_G)_e \), where \( X \in 
\VectorFieldSpace(G) \) is an arbitrary vector field satisfying 
\( X_e = A \) and \( \difLie_X \) denotes the Lie derivative in the 
direction \( X \). The dual map \( \epsilon^\ast: \LieA{g}^\ast \wedge 
\LieA{g}^\ast \to \LieA{g}^\ast \) satisfies the Jacobi identity and thus 
endows \( \LieA{g}^\ast \) with a Lie algebra structure.
The unique connected and simply connected Lie group \( G^* \) whose Lie 
algebra is \( \LieA{g}^\ast \) is called the \emphDef{dual group} of 
\( G \). The Lie group \( G^* \) has a unique Poisson structure 
\( \varpi_{G^*} \) relative to which \( G^* \) is a Poisson Lie group 
such that the intrinsic derivative of \( \varpi_{G^*} \) coincides with
the Lie bracket on \( \LieA{g} \). If \( G \) is connected and simply 
connected, the intrinsic derivative \( \epsilon \) is a cocycle which 
uniquely determines both Poisson Lie tensors \( \varpi_G \) 
and \( \varpi_{G^*} \).

Let \( G \times M \to M \) be a left Poisson action of the Poisson Lie 
group \( (G, \varpi_G) \) on the Poisson manifold \( (M, \varpi_M) \), 
\ie, the left action map \( G \times M \to M \) is 
Poisson. A smooth map \( J: M \to G^* \), if it exists, is called a 
\emphDef{momentum map} of this action if
\begin{equation}
\label{eq::luMomentumMap:defining}
	A^* + \varpi_M\left(\cdot, J^* A^l\right) = 0, \quad 
	\text{for all} \quad A \in \LieA{g}.
\end{equation}
Here, \( A^* \) denotes the fundamental (or infinitesimal generator) 
vector field on \( M \) induced by the infinitesimal action of 
\(  A \in \LieA{g} \), \ie,
\begin{equation}
	A^*(m) \defeq \left.\frac{d}{dt}\right|_{t=0}\exp(tA) \cdot m, \quad 
	\text{for all}\quad  m \in M,
\end{equation}
where \(g \cdot  m\) denotes the action of \(g \in  G\) on \(m \in  M\).
The second term in~\eqref{eq::luMomentumMap:defining} is interpreted in 
the following way. Since \( \LieA{g} \) is the dual of \( \LieA{g}^\ast \) 
(which is the Lie algebra of \( G^* \)), we may think of \( A \) as a 
linear functional on \( \LieA{g}^\ast \) and extend it to a left invariant 
one-form \( A^l\in \DiffFormSpace^1(G^*) \), \ie, \( (A^l)_a (v) = 
\scalarProd{A}{\LeftTrans_{a^{-1}} v} \) for every \( a \in G^* \) and 
\( v \in \TBundle_a G^* \), where \(\LeftTrans_{a^{-1}}\) denotes both 
the left translation by \(a^{-1} \in  G^*\) in \(G^*\) and its tangent 
map (derivative) on \(\TBundle G^*\).

Assume now that the Poisson manifold \( (M, \varpi_M) \) is symplectic with 
symplectic form \( \omega \) and let us unwind the definition in this case.
For any \( X_m \in \TBundle_m M \) we have 
\begin{equation}
\label{eq::luMomentumMap:calculationCaseSymplectic}
\begin{split}
	\omega_m(A^*_m, X_m) 
		&= (\varpi_M)_m \left(\bigl(\varpi_M^\sharp \bigr)^{-1}A^*_m, 
		\bigl(\varpi_M^\sharp \bigr)^{-1} X_m \right)
		\stackrel{\eqref{eq::luMomentumMap:defining}}
		= - \bigl(J^* A^l\bigr)_m (X_m) \\ 
		& = - A^l_{J(m)} \bigl(\tangent_m J (X_m)\bigr)
		= - \scalarProd{A}{\LeftTrans_{J(m)^{-1}} \tangent_m J (X_m)} \\
		&= - \scalarProd{A}{(\difLog J)_m (X_m)},
\end{split}
\end{equation}
where \( \tangent_m J: \TBundle_m M \rightarrow  \TBundle_{J(m)} G^* \) 
is the derivative (tangent map) of \( J: M \to G^* \) and 
\( \difLog J \in \DiffFormSpace^1(M, \LieA{g}^*) \), defined by the last 
equality, is its \emph{left logarithmic derivative}. 

Note that the identity \( \omega_m(A^*_m, X_m) + 
\scalarProd{A}{(\difLog J)_m (X_m) } = 0 \) proved 
in~\eqref{eq::luMomentumMap:calculationCaseSymplectic} does not use 
the fact that \( G \) is a Poisson Lie group.
Indeed, this identity still  makes sense if the momentum map is 
replaced by a smooth map \( J: M \to H \) with values in an arbitrary 
Lie group \( H \), as long as there is a duality between the Lie algebras 
of \( G \) and \( H \). This observation leads to our generalization of 
Lu's momentum map. In order to define this generalization, we need a 
few preliminary concepts that are inspired by their counterparts 
in the theory of Poisson Lie groups.

\subsection{Dual pairs of Lie algebras}
A \emphDef{dual pair of Lie algebras} (not necessarily finite-dimensional) 
consists of two Lie algebras \( \LieA{g} \) and \( \LieA{h} \), which are 
in duality through a given (weakly) non-degenerate bilinear map 
\( \kappa: \LieA{g} \times \LieA{h} \to \R \). 
Using notation stemming from functional analysis, we often write the 
dual pair as \( \kappa(\LieA{g}, \LieA{h}) \). Intuitively, we think 
of \( \LieA{h} \) as the dual vector space of \( \LieA{g} \), endowed 
with its own Lie bracket operation. For this reason, we often denote 
\( \LieA{g}^\ast \defeq \LieA{h} \), even though \( \LieA{g}^\ast \) 
is not necessarily the functional analytic dual of \( \LieA{g} \).
Two Lie groups \( G \) and \( H \) are said to be \emphDef{dual} to each 
other if there exists a (weakly) non-degenerate  bilinear form 
\( \kappa: \LieA{g} \times \LieA{h} \to \R \) relative to which the 
associated Lie algebras are in duality. We use the notation 
\( \kappa(G, H) \) in this case. As for Lie algebras, we often write 
\( G^* \defeq H \), intuitively thinking of \( G^* \) as the dual Lie group, 
as in the theory of Poisson Lie groups (\parencite{LuWeinstein1990}).

Note that the notion of a dual pair of Lie algebras involves only the 
underlying vector spaces, while the Lie brackets play no role.
We introduce a more rigid concept of duality, which takes all structures 
into account. For a given dual pair \( \kappa(\LieA{g}, \LieA{h}) \) of 
Lie algebras, define a bilinear skew-symmetric bracket on the \emph{double} 
\( \LieA{d} \defeq \LieA{g} \times \LieA{h} \) by
\begin{equation}
	\label{eq::lieAlgebraDualPair:bracketDouble}
	\LieBracket{(A, \mu)}{(B, \nu)} = (\LieBracket{A}{B}_\LieA{g} - 
	\CoadAction_\mu B + \CoadAction_\nu A, \LieBracket{\mu}{\nu}_\LieA{h} - 
	\CoadAction_A \nu + \CoadAction_B \mu),
\end{equation}
for \( A,B \in \LieA{g} \), \( \mu, \nu\in \LieA{h} \), where the 
infinitesimal coadjoint actions are defined with respect to \( \kappa \) 
by\footnote{In infinite dimensions, \cref{eq::definitonCoadjontAction} 
ensures only uniqueness of the adjoints but not their existence. In the 
sequel, we always assume that the coadjoint actions exist.}
\begin{equation}
\label{eq::definitonCoadjontAction}
\begin{split}
\kappa(B, \CoadAction_A \mu) &= \kappa(\LieBracket{A}{B}_\LieA{g}, \mu), \\
\kappa(\CoadAction_\mu A, \nu) &= \kappa(A,\LieBracket{\mu}{\nu}_\LieA{h}).
\end{split} 
\end{equation}
However, the bracket operation defined 
in~\eqref{eq::lieAlgebraDualPair:bracketDouble} does not satisfy the 
Jacobi identity, in general. A dual pair \( \kappa(\LieA{g}, \LieA{h}) \) 
of Lie algebras is called a \emphDef{Lie bialgebra}, if the 
bracket~\eqref{eq::lieAlgebraDualPair:bracketDouble} on 
\( \LieA{d} = \LieA{g} \times \LieA{h} \) is a Lie bracket.
In this case, we denote the double by \( \LieA{g} \bowtie \LieA{h} \).
The notation displays the similarity to the semidirect product but also 
emphasizes that the bracket on \( \LieA{g} \times \LieA{h} \) uses both the 
adjoint actions of \( \LieA{g} \) and \( \LieA{h} \). 

\begin{example}
\label{ex::dualPairLieAlgebra:AbelianDualIsLieBialgebra}
Let \( \kappa(\LieA{g}, \LieA{h}) \) be a dual pair.
Assume that \( \LieA{h} \) is an Abelian Lie algebra.
Then~\eqref{eq::definitonCoadjontAction} implies that the coadjoint 
action \( \CoadAction_\mu: \LieA{g} \to \LieA{g} \) is trivial for 
every \( \mu \in \LieA{h} \).
Hence the bracket on the double \( \LieA{d} \) simplifies to 
\begin{equation}
	\LieBracket{(A, \mu)}{(B, \nu)} 
	= (\LieBracket{A}{B}_\LieA{g}, - \CoadAction_A \nu + \CoadAction_B \mu).
\end{equation}
Thus, \( \LieA{d} \) is the semidirect product 
\( \LieA{g} \rSemiProduct_{\CoadAction} \LieA{h} \) of Lie algebras, 
where \( \LieA{g} \) acts on \( \LieA{h} \) by the \( \kappa \)-coadjoint 
action~\eqref{eq::definitonCoadjontAction}.
Furthermore, the Jacobi identity always holds and thus 
\( \kappa(\LieA{g}, \LieA{h}) \) is a Lie bialgebra.
\end{example}

\begin{example}[Group of volume-preserving diffeomorphisms]
\label{ex::dualPairLieGroups:volumePreservingDiffeos}
Let \( M \) be a compact manifold endowed with a volume form \( \mu \).
Then the group \( G = \DiffGroup_\mu(M) \) of volume-preserving 
diffeomorphisms is a Fr\'echet Lie group, see 
\parencite[Theorem~2.5.3]{Hamilton1982}.
Its Lie algebra consists of \( \mu \)-divergence-free vector fields 
\( X \) on \( M \). Equivalently,
\begin{equation}
	\LieA{g} = \VectorFieldSpace_\mu(M) = 
	\set{X \in \VectorFieldSpace(M) \given \dif (X \contr \mu) = 0},
\end{equation}
where \( \dif \) denotes the exterior derivative on forms and 
\(X \contr \mu \defeq {\rm i}_X \mu \) is the interior product 
(contraction on the first index) of \(X\) with \(\mu \).
Hence, we also identify \( \VectorFieldSpace_\mu(M) \) with 
\( \clDiffFormSpace^{\dim M-1}(M) \) via \( X \mapsto X \contr \mu \), 
where \( \clDiffFormSpace^{k}(M) \) denotes the space of closed 
\( k \)-forms on \( M \).
Thus \( \DiffFormSpace^1(M) \slash \dif \DiffFormSpace^0(M) \) is the 
regular dual with respect to the weakly non-degenerate integration paring
\begin{equation}
\label{eq::diffAction:dualPairOfVolPresVectorFields}
\kappa(X,\alpha) \defeq (-1)^{\dim M - 1}\int_M (X \contr \mu) \wedge \alpha 
= \int_M (X \contr \alpha) \, \mu.
\end{equation}
We now note that a \( 1 \)-form \( \alpha \) can be interpreted as a 
trivial principal circle bundle with curvature \( \dif \alpha \). 
From this point of view, \(\DiffFormSpace^1(M) \slash \dif 
\DiffFormSpace^0(M) \) parametrizes gauge equivalence classes of 
connections on a trivial principal circle bundle.
Thus, it is natural to think of it as the Lie algebra of the Abelian 
group \( H \defeq \csCohomology^2(M, \UGroup(1)) \) of all principal 
circle bundles with connections, modulo gauge equivalence.
This heuristic argument can be made rigorous using the theory of 
Cheeger--Simons differential characters; see 
\cref{sec::differentialCharacter} for a quick review.
In summary, we get a dual pair \( \kappa(\DiffGroup_\mu(M), 
\csCohomology^2(M, \UGroup(1))) \) of Lie groups. For later use, it is 
convenient to introduce the notation \( \csAlgebra^2(M, \UGroup(1)) \) for 
the Lie algebra \( \DiffFormSpace^1(M) \slash \dif \DiffFormSpace^0(M) \) 
of \( \csCohomology^2(M, \UGroup(1)) \).
\end{example}

\begin{example}[Group of symplectomorphisms]
\label{ex::dualPairLieGroups:symplectomorphisms}
Consider a compact symplectic manifold \( (M, \omega) \).
The group \( G = \DiffGroup_\omega(M) \) of symplectomorphisms is a 
Fr\'echet Lie group \parencite[Theorem~43.12.]{KrieglMichor1997} with 
Lie algebra
\begin{equation}
\LieA{g} = \VectorFieldSpace_\omega(M)
 = \set{X \in \VectorFieldSpace(M) \given \dif (X \contr \omega) = 0}.
\end{equation}
The map \( X \mapsto X \contr \omega \) identifies \( \LieA{g} \) with the 
space \( \clDiffFormSpace^1(M) \) of closed \( 1 \)-forms on \( M \). 
Thus the regular dual of \( \LieA{g} \) with respect to the natural 
integration paring 
\begin{equation}
	\label{eq::diffAction:dualPairOSymplPresVectorFields}
	\kappa(X, \alpha) \defeq \frac{(-1)^{\dim M - 1}}{\left(\frac{1}{2} 
	\dim M -1\right)!} \int_M (X \contr \omega) \wedge \alpha
\end{equation}
is \( \csAlgebra^{\dim M}(M, \UGroup(1)) \defeq \DiffFormSpace^{\dim M-1}(M) 
\slash \dif\DiffFormSpace^{\dim M-2}(M)\). The prefactor in front of the 
integral turns out to be a convenient choice in 
\cref{sec:globalAnalysisSymplecticFibreBundles:momentumMapForSymplectomorphisms}.
Similarly to the \cref{ex::dualPairLieGroups:volumePreservingDiffeos} of 
volume-preserving diffeomorphisms, the Abelian Lie algebra 
\( \csAlgebra^{\dim M}(M, \UGroup(1)) \) is integrated by the group 
\(\csCohomology^{\dim M}(M, \UGroup(1)) \) of Cheeger--Simons differential 
characters with degree \( \dim M \).
These can be thought of as equivalence classes of circle \( n \)-bundles 
with connections in the sense of higher differential geometry.
\end{example}

\begin{remark}
The group of volume-preserving diffeomorphism and the group of 
symplectomorphisms both have an Abelian dual group.
Thus they fall in the realm of 
\cref{ex::dualPairLieAlgebra:AbelianDualIsLieBialgebra}.
In particular, both dual pairs \( (\VectorFieldSpace_\mu(M), 
\csAlgebra^2(M, \UGroup(1))) \) and \( (\VectorFieldSpace_\omega(M), 
\csAlgebra^{\dim M}(M, \UGroup(1))) \) are actually Lie bialgebras.
Ignoring the particularities of the infinite-dimensional setting for a 
moment, Drinfeld's theorem \parencite[Theorem~2.2.2]{EtingofSchiffmann2002} 
states that there are essentially unique connected and simply connected 
Poisson Lie groups, whose Lie algebras are \( \VectorFieldSpace_\mu(M) \) 
and \( \VectorFieldSpace_\omega(M) \), respectively.
We do not know if the groups \(\DiffGroup_\mu(M)\) or \(\DiffGroup_\omega(M)\) 
carry a non-trivial Poisson Lie structure integrating the above Lie bialgebras 
(this would require to find a non-trivial integration of the adjoint action, 
see \parencite{Drinfeld1993} and \parencite[Theorem~2.31]{Lu1990}).  Moreover, 
we are not aware of any Poisson Lie structure on these groups, such that 
the actions discussed in \cref{sec::globalAnalysisSymplecticFibreBundles} 
are Poisson maps.
\end{remark}

\begin{example}[Gauge group]
\label{ex::dualPairLieGroups:gaugeGroup}
Let \( P \to M \) be a right principal \( G \)-bundle over the compact 
manifold \( M \). The group \( \GauGroup(P) \) of gauge transformations is 
identified with the space of sections of \( P \times_G G \defeq 
(P \times G)/G \) and thus is a Fr\'echet Lie group with Lie algebra 
\( \GauAlgebra(P) = \sSectionSpace( \AdBundle P ) \), the space of 
sections of the adjoint bundle \( \AdBundle P \defeq (P\times \LieA{g})/G \), 
see \parencite{CirelliMania1985}. Denote the dual of the adjoint bundle 
by \( \CoAdBundle P \defeq (P \times \LieA{g}^\ast)/G\), the action of 
\( G \) on \( \LieA{g}^\ast \) being the left coadjoint action.
The natural pairing
\begin{equation}
	\label{eq::diffAction:dualPairOfGaugeGroup}
	\kappa(\phi, \alpha) = \int_M \dualPair{\phi}{\alpha}, 
	\qquad \phi \in \GauAlgebra(P),
	\alpha \in \DiffFormSpace^{\dim M}(M, \CoAdBundle P),
\end{equation}
identifies \( \DiffFormSpace^{\dim M}(M, \CoAdBundle P) \) as the regular 
dual to \( \GauAlgebra(P) \). In particular, if \( M \) is endowed with 
a volume form \( \mu \), then  \( \CoGauAlgebra(P) = 
\sSectionSpace(\CoAdBundle P) \) is the dual by integration against 
\( \mu \):
\begin{equation}
	\scalarProd{\cdot}{\cdot}_{\AdAction}: \sSectionSpace(\AdBundle P) \times 
	\sSectionSpace(\CoAdBundle P) \to \R, \qquad (\phi, \varrho) \mapsto 
	\int_M \dualPair{\phi}{\varrho} \, \mu.
\end{equation}
Moreover, an \( \AdAction_G \)-invariant non-degenerate bilinear form 
on \( \LieA{g} \) identifies \( \CoGauAlgebra(P) \) with \( \GauAlgebra(P) \) 
so that \( \GauAlgebra(P) \) is self-dual in this case. 
\end{example}

\subsection{Group-valued momentum maps}

With these preliminaries out of the way, we can introduce our 
generalization of a group-valued momentum map.
\begin{defn}
Let \( M \) be \(  G \)-manifold endowed with a symplectic form\footnote{We 
do not require the symplectic form \( \omega \) to be invariant under 
the group action. In fact, we will see below in 
\cref{prop::groupMomentumMap:existenceMaurerCartan} that this will not be the 
case if the dual group \( H \) is non-Abelian.} \( \omega \).
A\emphDef{group-valued momentum map} is a pair \( (J, \kappa) \), where 
\( \kappa(G,G^*) \) is a dual pair of Lie groups and \( J: M \to G^*\) 
is a smooth map satisfying
	\begin{equation}
		\label{eq::momentumMap:DefEq}
		A^* \contr \omega + \kappa(A, \difLog J) = 0, \quad A \in \LieA{g}.
	\end{equation}
In this formula, \( A^* \) denotes the fundamental vector field on \( M \) 
induced by \( A \in \LieA{g} \), \( \difLog J \in 
\DiffFormSpace^1(M, \LieA{\LieA{g}^*}) \) is the left logarithmic 
derivative of \( J \), \( \LieA{g} \) is the Lie algebra of \( G \), 
and \( \LieA{g}^\ast \) is the Lie algebra of \( G^* \).
\end{defn}

\begin{example}[As a generalization of the usual momentum map]
	\label{ex_standard_momentum_map}
Let \( (M, \omega) \) be a finite-dimensional, symplectic \( G \)-manifold.
We view \(G^*= \LieA{g}^* \) as an Abelian group and consider the natural 
pairing \( \kappa: \LieA{g} \times \LieA{g}^* \to \R \). Thus a 
\( \LieA{g}^* \)-valued momentum map is a smooth map \( J: M \to \LieA{g}^* \) 
satisfying the usual relation
	\begin{equation}
		A^* \contr \omega + \dif J_A = 0, \quad A \in \LieA{g},
	\end{equation}
where \( J_A: M \to \R \) is defined by \( J_A (m) = \kappa\bigl(A, J(m)\bigr) \) 
for \(m \in  M\), and \( \dif J_A = \kappa(A, \difLog J) \) holds as a 
consequence of the Abelian character of \( \LieA{g}^* \).
\end{example}

\begin{example}[As a generalization of the Lie algebra-valued momentum map]
If \( \kappa: \LieA{g} \times \LieA{g} \to \R \) is a continuous, weakly 
non-degenerate, \( \AdAction_G \)-invariant symmetric bilinear form, 
one identifies the dual \( \LieA{g}^* \) with \( \LieA{g} \) and the 
coadjoint orbits with adjoint orbits. In view of \cref{ex_standard_momentum_map}, 
we are led to the concept of a Lie algebra-valued momentum map. Although 
this notion has been in the literature since the mid 70s, it was recently 
formalized by \textcite[Definition~4.3]{NeebSahlmannEtAl2014}: a Lie 
algebra-valued momentum map is a smooth map \( J: M \to \LieA{g} \) such 
that, for all \( A \in \LieA{g} \), the component functions \( J_A = 
\kappa(A, J): M \to \R \) satisfy
	\begin{equation}
		A^* \contr \omega + \dif J_A = 0.
	\end{equation}
It is immediate from the definition, that such a Lie algebra-valued 
momentum map can be regarded as a group-valued momentum map with 
respect to the dual pair \( \kappa(G, \LieA{g}) \), where \( \LieA{g} \) 
is viewed as an Abelian Lie group.
\end{example}

\begin{example}[As a generalization of the Poisson momentum map]
\label{ex::momentumMap:generalizationPoisson}
We now describe how our motivational example fits in the new framework. 
Let \( (M, \omega) \) be a finite-dimensional symplectic manifold. 
Suppose a finite-dimensional Poisson Lie group \( G \) acts on \( M \) 
such that the action is Poisson. The Poisson structure of \( G \) 
induces a Lie bracket on the dual \( \LieA{g}^* \) of the Lie algebra 
\parencite[Theorem~2.18]{Lu1990}. Let \(G^* \) denote the unique connected 
and simply connected Lie group integrating the Lie algebra \( \LieA{g}^* \).
As above, let \( \kappa \) be the natural pairing 
\( \LieA{g} \times \LieA{g}^* \to \R \). Hence a \( G^* \)-valued 
momentum map is a smooth map \( J: M \to G^* \) satisfying 
	\begin{equation}
		A^* \contr \omega + \kappa(A, \difLog J) = 0.
	\end{equation}
As we have explained above, this equation is a reformulation of the usual Lu 
momentum map relation in the context of symplectic geometry, 
see~\eqref{eq::luMomentumMap:calculationCaseSymplectic}.
In other words, our group-valued momentum map is the natural generalization 
of the Poisson momentum map if the Lie group \( G \) is not necessarily 
a Poisson Lie group.
\end{example}

\begin{example}[As a generalization of the circle-valued momentum map]
Consider a symplectic manifold \( (M, \omega) \) with a symplectic 
action of \( G = \UGroup(1) \). We let \(G^*= \UGroup(1) \) and take 
(minus) the Killing form \( \kappa: \UAlgebra(1) \times \UAlgebra(1) 
\to \R \) as the pairing between the Lie algebras of \( G \) and \(G^*\),
\ie, \( \kappa(x, y) = xy \) is the usual multiplication of real 
numbers under the identification \( \UAlgebra(1) \isomorph \R \). 
Thus, a map \( J: M \to \UGroup(1) \) is a group-valued momentum map 
if and only if
	\begin{equation}
		1^* \contr \omega + \difLog J = 0.
	\end{equation}
In this way, we recover the usual definition of a circle-valued momentum 
map (see, for example, \parencite[Definition~1]{PelayoRatiu2012}).
\end{example}

\begin{example}[Symplectic torus] 
\label{ex:symplecticTorus}
Consider a symplectic vector space \( (V, \omega) \). Let \( \Lambda \) 
be a lattice in \( V \), \ie, a discrete subgroup of the additive group 
\(V\). The symplectic structure is invariant under the natural action of 
\( \Lambda \) on \( V \) by translations and hence it descends to a 
symplectic form \(\omega_T\) on the torus \( T = V \slash \Lambda \). 
Moreover, the translation action of \( V \) on itself commutes with the 
lattice action and thus induces a symplectic \( V \)-action on \( T \). 
The action of \( V \) on itself has the momentum map
	\begin{equation}
	J: V \to V^*, \qquad v \mapsto \omega(v, \cdot).
	\end{equation}
However, \( J \) is not invariant under the lattice action and so does 
\emph{not} descend to a momentum map for the induced action of \( V \) 
on the torus \( T \). Indeed, it is well-known that, for cohomological 
reasons, the symplectic action on the torus does not admit a standard 
momentum map. Rather, \( J \) transforms as
	\begin{equation}
	J(v + \lambda) = J(v) + \omega(\lambda, \cdot), \quad \lambda \in \Lambda.
	\end{equation}
Thus, if \( \omega(\lambda_1, \lambda_2) \in \Z \) holds for all 
\( \lambda_1, \lambda_2 \in \Lambda \), then \( J \) is equivariant 
with respect to the dual lattice action
	\begin{equation}
	\Lambda^* = \set{\alpha \in V^* \given \alpha(\lambda) \in \Z\;\; 
	\text{for all}\;\; \lambda \in \Lambda}.
	\end{equation}
In this case, \( J \) induces a \( V^* \slash \Lambda^* \)-valued momentum 
map \( J_T \) on the torus \(T\). It is interesting to note that the 
integrality condition \( \omega(\lambda_1, \lambda_2) \in \Z \) is 
equivalent to \( \omega_T \) being prequantizable. 
\end{example}

\begin{example}[As a generalization of the cylinder-valued momentum map]
In a finite-dimensional context, \textcite{CondevauxDazordMolino1988} 
introduced a momentum map with values in the cylinder \( C \defeq 
\LieA{g}^* \slash H \), where \( H \) is the holonomy group of a flat 
connection on some bundle constructed in terms of the symplectic form 
and the action (in our language, \( \alpha \) defined 
in~\eqref{eq::groupMomentumMap:defPrimitive} below plays the role of the 
connection form). If the holonomy group \( H \) is discrete, then \( C \) 
is a Lie group with Lie algebra \( \LieA{g}^* \).  Thus \( C \) is a dual group.
Under the identification of the Lie algebra \( \LieA{c} = \LieA{g}^* \), 
the cylinder-valued momentum map also satisfies~\eqref{eq::momentumMap:DefEq}, 
see \parencite[Theorem~5.2.8]{OrtegaRatiu2003}, and hence is a group-valued 
momentum map. The group-valued momentum map for the symplectic torus 
discussed in \cref{ex:symplecticTorus} is also the cylinder-valued 
momentum map, see \parencite[Example~5.2.5]{OrtegaRatiu2003}.

The case when the holonomy group \( H \subseteq \LieA{g}^* \) has 
accumulation points is pathological both in the framework of cylinder- 
as well as group-valued momentum maps, see 
\cref{ex:groupValuedMomentumMap:exampleNonExistence} and 
\cf \parencite[Example~5.2.6]{OrtegaRatiu2003}. 
\end{example}

As we have seen, the notion of a group-valued momentum map unifies 
many different concepts of momentum maps. Despite its general nature, 
a group-valued momentum map still captures conserved quantities of 
the dynamical system, \ie, it has the \emph{Noether property} (see 
\parencite[Definition~4.3.1]{OrtegaRatiu2003}).
\begin{prop}[Noether's theorem]
Let \( (M, \omega) \) be symplectic \( G \)-manifold.
Assume that the action has a \(G^*\)-valued momentum map \( J: M \to G^* \).
Let \( h \in \sFunctionSpace(M) \) be a smooth function for which the 
Hamiltonian vector field \( X^h \) exists and has a unique local 
flow\footnote{Recall that vector fields on Fr\'echet manifolds do not 
need to have flows. Having a flow is more or less equivalent to local 
in time solutions of the corresponding partial differential equation.}.
If \( h \) is \( G \)-invariant, then \( J \) is constant along the 
integral curves of \( X^h \). 
\end{prop}
\begin{proof}
	Let \( A \in \LieA{g} \) and \( m \in M \).
	Using the defining equation for the momentum map, we have
	\begin{equation}\begin{split}
	\kappa(A, (\difLog J)_m X^h_m)
		&= - \omega_m (A^*_m, X^h_m) 
		= - (\dif h)_m A^*_m \\
		&= - \difFracAt{}{\varepsilon}{0} h(\exp(\varepsilon A) \cdot m)
		= 0
	\end{split}
	\end{equation}
	by \(G\)-invariance of \(h\).
	Since \( A \in \LieA{g} \) is arbitrary and the pairing \( \kappa \) is weakly non-degenerate, we conclude 
	\( (\difLog J) X^h = 0 \).
	Hence, \( J \) is constant along integral curves of \( X^h \).
\end{proof}

\subsection{Existence and uniqueness}
Let \( (M, \omega) \) be a symplectic manifold.
For a given \( G \)-action on \( M \) and a fixed dual pair 
\( \kappa(G, G^*) \) of Lie groups, we will now search for conditions 
which ensure the existence of a \( G^* \)-valued momentum map.

Define a \( \LieA{g}^\ast \)-valued \( 1 \)-form \( \alpha \in 
\DiffFormSpace^1(M, \LieA{g}^\ast) \) by 
\( A^* \contr \omega + \kappa(A, \alpha)=0 \), that is,
\begin{equation}
\label{eq::groupMomentumMap:defPrimitive}
\kappa(A, \alpha_m (X_m)) = \omega_m (X_m, A_m^*) \quad \text{for all } X_m 
\in \TBundle_m M, A \in \LieA{g}.
\end{equation}
In infinite dimensions, the dual pairing \( \kappa \) is often not 
strongly non-degenerate. In such cases, there might not exists 
\( \alpha \in \DiffFormSpace^1(M, \LieA{g}^\ast) \) 
satisfying~\eqref{eq::groupMomentumMap:defPrimitive}, although 
\( \alpha \) is unique, if it exists.
We will assume for the rest of the subsection, that we have 
such an \( \alpha \).

By~\eqref{eq::momentumMap:DefEq}, the \( G \)-action on \( M \) 
admits a \(G^* \)-valued momentum map if and only if 
\( \alpha \in \DiffFormSpace^1(M,\LieA{g}^\ast) \) defined 
by~\eqref{eq::groupMomentumMap:defPrimitive} is log-exact, \ie, 
if \( \alpha = \delta J \) for some smooth function \(J: M \to G^*\).
A necessary condition is that \( \alpha \) satisfies the Maurer--Cartan 
equation, as the next lemma shows. Below, \( \wedgeLie{\alpha}{\beta} \) 
means the wedge product of the \( \LieA{g}^\ast \)-valued forms 
\( \alpha \) and \( \beta \) on \( M \) associated to the bracket 
operation on \( \LieA{g}^\ast \) (as the Lie algebra of \( G^* \)).

\begin{lemma}[{Corresponds to \parencite[Theorem~3.7]{Lu1990}}]
\label{prop::groupMomentumMap:existenceMaurerCartan}
\begin{thmenumerate}*
\item The \( \LieA{g}^\ast \)-valued one-form \( \alpha \) on \( M \) 
defined by~\eqref{eq::groupMomentumMap:defPrimitive} satisfies 
the Maurer-Cartan equation
\begin{equation}
\label{eq::momentumMap:existence:maurerCartan}
\dif \alpha + \frac{1}{2} \wedgeLie{\alpha}{\alpha} = 0
\end{equation}
if and only if \( \difLie_{A^*} \omega = 
\frac{1}{2} \kappa(A, \wedgeLie{\alpha}{\alpha}) \) holds for 
all \( A \in \LieA{g} \). 
\item If the \(G\)-action on \(M\) admits a \(G^*\)-valued momentum 
map \(J: M \to G^*\), then \(\delta J \in \DiffFormSpace^1(M,\LieA{g}^\ast)\) 
satisfies the Maurer-Cartan 
equation~\eqref{eq::momentumMap:existence:maurerCartan}. \qedhere
\end{thmenumerate}
\end{lemma}

If the form \( \alpha \in \DiffFormSpace^1(M, \LieA{g}^*) \) defined 
by~\eqref{eq::groupMomentumMap:defPrimitive} exists and satisfies 
the Maurer-Cartan equation~\eqref{eq::momentumMap:existence:maurerCartan}, 
then we say that the \( G \)-action on \( (M, \omega) \) is 
\emphDef{\(G^*\)-symplectic}. For Abelian \(G^*\), the notions of 
symplectic and \(G^*\)-symplectic group actions coincide. 

\begin{proof}
(i) For \( A \in \LieA{g} \) and \( X,Y \in \VectorFieldSpace(M) \), we calculate
\begin{equation}
\begin{split}
(\difLie_{A^*} \omega) (X,Y) 	&= \dif(A^* \contr \omega)(X,Y) \\
&= X(\omega(A^*,Y)) - Y (\omega(A^*,X)) - \omega(A^*,\LieBracket{X}{Y}) \\
&\oversetSmall{\eqref{eq::groupMomentumMap:defPrimitive}}{=}
-\kappa\big(A, X(\alpha(Y))-Y(\alpha(X))-\alpha(\LieBracket{X}{Y})\big) \\
&= - \kappa(A, (\dif \alpha) (X,Y)),
\end{split}
\end{equation}
which is equivalent to \(\difLie_{A^*}\omega+\kappa(A, \dif \alpha) = 0\), from which the claim follows.

(ii) If the \( G \)-action admits a \( G^* \)-valued momentum map 
$J:M \to G^*$, then $\delta J$ satisfies the Maurer-Cartan
equation~\eqref{eq::momentumMap:existence:maurerCartan} if and only 
if \(\difLie_{A^*}\omega + \kappa(A, \dif \delta J) = 0\), as we just 
showed in (i). This identity is obtained by taking the exterior 
derivative of the defining relation~\eqref{eq::momentumMap:DefEq}.
\end{proof}

Although, at first glance, it may seem unnatural to consider momentum 
maps for actions which do not leave the symplectic form invariant, 
we remind the reader that this is also the case in the Poisson category.
There, the action \( \Upsilon: G \times M \to M \) is a Poisson map, 
while for fixed \( g \in G \) the map \( \Upsilon_g: M \to M \) might 
not preserve the Poisson structure. If one restricts attention to 
objects invariant under the action, then non-closed almost symplectic 
structures have to be considered. In other words, this path leads to quasi-Hamiltonian systems \parencite{HuebschmannJeffrey1994, AlekseevMalkinEtAl1998}.

\begin{example}[Poisson Lie Momentum Map \parencite{FlaschkaRatiu1996}]
A Poisson action of a non-Abelian Poisson Lie group on a symplectic 
manifold, even if it admits a momentum map (see 
\cref{ex::momentumMap:generalizationPoisson}), never preserves the 
symplectic form. For example, consider the Poisson Lie group 
$\SUGroup(n)$ endowed with the Lu-Weinstein Poisson structure.
Its dual Poisson Lie group is $B \defeq AN$, where $A$ is the subgroup 
of positive diagonal matrices, $N$ is the subgroup of complex 
triangular matrices with 1 on the diagonal, both viewed as real subgroups 
of $\SLGroup(n,\C)$, also viewed as a real Lie group.
Thus, $B$ has a natural Poisson Lie group structure whose symplectic 
leaves are the sets $\Sigma_{\boldsymbol{\lambda}}$ 
of matrices with fixed singular values $\boldsymbol{\lambda} \defeq 
\{\lambda_1 \geq \cdots \geq \lambda_n>0\}$
(the dressing orbits). The Lie group $\SUGroup(n)$ acts on the symplectic manifold $\Sigma_{\boldsymbol{\lambda}} \times \Sigma_{\boldsymbol{\mu}}$ 
in a Poisson fashion and admits the momentum map 
$J: \Sigma_{\boldsymbol{\lambda}} \times \Sigma_{\boldsymbol{\mu}} \ni
(b_1, b_2) \mapsto b_1b_2 \in B$.
The Poisson action is given in the following way.
First, one shows that if $b_1 \in \Sigma_{\boldsymbol{\lambda}}$, 
$b_2\in\Sigma_{\boldsymbol{\mu}}$, and $b_1b_2\in \Sigma_{\boldsymbol{\nu}}$, 
then there exists $b_3\in \Sigma_{\boldsymbol{\lambda}}$, 
$b_4 \in\Sigma_{\boldsymbol{\mu}}$ such that $b_3b_4 =
\operatorname{diag}(\nu_1, \ldots, \nu_n)$.
Second, since $b_1b_2\in \Sigma_{\boldsymbol{\nu}}$, there exist  
$k_1, k_2 \in \SUGroup(n)$ such that $k_1b_1b_2k_2 = 
\operatorname{diag}(\nu_1, \ldots, \nu_n)$.
With these notations, the aforementioned Poisson action of 
$\SUGroup(n)$ on $\Sigma_{\boldsymbol{\lambda}} \times 
\Sigma_{\boldsymbol{\mu}}$ is given by $k_2^{-1}\cdot 
(b_1, b_2)  \defeq (b_3, b_4)$.
See \parencite{FlaschkaRatiu1996} for the proofs of these statements 
and their link to a convexity theorem for compact Poisson Lie group 
actions on compact connected symplectic manifolds.
In this case, the theorem states that $J(\Sigma_{\boldsymbol{\lambda}} 
\times \Sigma_{\boldsymbol{\mu}})\cap A_+$ is the exponential of a 
convex polytope, where $A_+$ denotes the set of elements of $A$ 
with decreasingly ordered entries. 
\end{example}

Returning to the existence and uniqueness question, recall that 
the Maurer--Cartan theory involves also a topological obstruction 
for the existence of a primitive for \( \alpha \) in the form of 
the so-called period homomorphism \parencite[Theorem~III.1.2]{Neeb2006}.
To strengthen \cref{prop::groupMomentumMap:existenceMaurerCartan} 
in terms of the period map, we need to recall the notion of a 
regular Lie group (see \parencite[Definition II.5.5]{Neeb2006}).
A Lie group \( G \) modeled on a locally convex space is 
\emphDef{regular} if for each curve $c \in \sFunctionSpace([0,1], \LieA{g})$, 
the initial value problem $\difLog \eta(t) \defeq \LeftTrans_{\eta(t)^{-1}} 
\dot{\eta}(t) = c(t)$,  $\eta(0) = e$, formulated in terms of the left 
logarithmic derivative \( \difLog \eta \) of $\eta$, has a solution 
$\eta_c \in \sFunctionSpace([0,1], G)$ and the endpoint evaluation map 
$\sFunctionSpace([0,1], \LieA{g}) \ni c \mapsto \eta_c(1) \in G$ is 
smooth. If \( G \) is regular, then it has a smooth exponential 
function. All Banach (so, in particular, all finite dimensional) 
Lie groups are regular.

Fix a point \( m_0 \in M \) and consider a piece-wise smooth loop 
\( \gamma: I \to M \) based at \( m_0 \).
The pull-back of \( \alpha \in \DiffFormSpace^1(M,\LieA{g}^*) \) 
by $\gamma$ is a \( \LieA{g}^* \)-valued $1$-form \( \gamma^* \alpha \) 
on \( I \). Denote by \(\eta_\gamma \in \sFunctionSpace(I, G^*) \) 
the solution of the initial value problem
\begin{equation}
\label{eq:groupMomentumMap:initialValueProblemEta}
\difLog \eta = \gamma^* \alpha, \qquad \eta(0) = e,
\end{equation}
which exists if \( G^* \) is regular. Evaluating \( \eta_\gamma \) 
at the endpoint \( 1 \), we obtain the period homomorphism 
\( \per_\alpha: \fundamentalGroup_1(M, m_0)\ni [\gamma] \mapsto 
\eta_\gamma(1) \in G^* \), where $[\gamma]$ is the homotopy class 
of the loop $\gamma$.

\begin{prop}
Let \( (M, \omega) \) be a connected symplectic manifold and 
\( \kappa(G,G^*) \) a dual pair of Lie groups. In infinite dimensions, 
we additionally assume that \( G^* \) is a regular Lie group.
Let \( G \) act on \( M \) in a \(G^*\)-symplectic way.
Then there exists a \(G^*\)-valued momentum map if and only if 
the period homomorphism \( \per_\alpha: \fundamentalGroup_1(M, m_0) 
\to G^*\) is trivial. Moreover, the momentum map is unique up to 
translation by a constant element \( h \in G^*\). 
\end{prop}
\begin{proof}
This proposition follows directly from \cref{prop::groupMomentumMap:existenceMaurerCartan} and the general 
fundamental theorem for Lie group-valued functions 
\parencite[Theorem III.1.2.]{Neeb2006}. 

For the convince of the reader, we quickly sketch how the period 
homomorphism enters the game. Recall that an ordinary closed 
\( 1 \)-form has a local primitive by the Poincar\'e lemma. The 
Maurer--Cartan equation~\eqref{eq::momentumMap:existence:maurerCartan} 
replaces the role of closedness for Lie algebra-valued \( 1 \)-forms. 
Since the action is \( G^* \)-symplectic, the solution \( \alpha \) 
of~\eqref{eq::groupMomentumMap:defPrimitive} 
satisfies~\eqref{eq::momentumMap:existence:maurerCartan} by 
\cref{prop::groupMomentumMap:existenceMaurerCartan}. Hence, on 
connected and simply connected domains, a \( G^* \)-valued momentum 
map exists. 
If \( M \) is not simply connected, then we pull back everything to 
the universal covering \( \tilde{M} \to M \) and obtain a lifted 
primitive \( \tilde{J}: \tilde{M} \to G^* \). A simple calculation 
reveals that \( \tilde{J} \) is equivariant with respect to 
\( \per_\alpha \), \ie,
\begin{equation}
\tilde{J}([\gamma] \cdot \tilde{m}) = 
\per_\alpha ([\gamma]) \cdot \tilde{J}(\tilde{m}), 
\quad \tilde{m} \in \tilde{M}, \quad 
[\gamma] \in \fundamentalGroup_1(M, m_0).
\end{equation}
Thus \( \tilde{J} \) drops to \( M \) if and only if the period 
homomorphism is trivial.
\end{proof}
\begin{example} 
\label{ex:groupValuedMomentumMap:exampleNonExistence}
A slight modification of \cref{ex:symplecticTorus} yields a 
symplectic group action without a group-valued momentum map.
Let \( M = T^2 = (\R \slash \Z)^4 \) be the product of two tori 
and denote its elements by \( 4 \)-tuples 
\( (\varphi_1, \varphi_2, \psi_1, \psi_2 ) \).
Endow \( M \) with the symplectic form
\begin{equation}
\omega = \dif \varphi_1 \wedge \dif \varphi_2 + 
\sqrt{2} \dif \psi_1 \wedge \dif \psi_2.
\end{equation}
The circle action given by \( \lambda \cdot 
(\varphi_1, \varphi_2, \psi_1, \psi_2 ) = 
(\varphi_1 - \lambda, \varphi_2, \psi_1 - \lambda, \psi_2 ) \) 
is clearly symplectic. The \( 1 \)-form \( \alpha \) defined by~\cref{eq::groupMomentumMap:defPrimitive} becomes 
\begin{equation}
\alpha = \dif \varphi_2 + \sqrt{2} \dif \psi_2.
\end{equation}
As the generators of \( \fundamentalGroup_1(M) \) we take the 
four natural loops \( \gamma_i \), where, for \( 1 \leq i \leq 4 \), 
the loop \( \gamma_i: I \to M \) winds once around the \( i \)-th 
circle in \( M = (\R \slash \Z)^4 \). The pull-back of \( \alpha \) 
by \( \gamma_1 \) and \( \gamma_3 \) vanishes and we find
\begin{equation}
\gamma_2^* \alpha = \dif t \quad \text{and} \quad \gamma_4^* \alpha 
= \sqrt{2}\dif t,
\end{equation}
where \( t \) denotes the canonical variable in \( I = [0,1] \).
Since there are only two connected one-dimensional Lie groups, 
the only possible choices for the dual group are \( G^* = \R \) 
and \( G^* = \R \slash \Z \). In both cases, the initial value problem~\eqref{eq:groupMomentumMap:initialValueProblemEta} has 
the unique solutions \( \eta_{\gamma_2}(t) = t \) and 
\( \eta_{\gamma_4}(t) = \sqrt{2} t \).
Thus neither for \( G^* = \R \) nor for \( G^* = \R \slash \Z \) 
the period homomorphism is trivial and hence no group-valued 
momentum map exists for this action.

This example exhibits another phenomenon that is particular for 
group-valued momentum maps: the action of a subgroup may not posses 
a group-valued momentum map even if the bigger group has a group-valued 
momentum map.  In fact, the action 
\( (\lambda_1, \lambda_2) \cdot (\varphi_1, \varphi_2, \psi_1, \psi_2 ) 
= (\varphi_1 - \lambda_1, \varphi_2, \psi_1 - \lambda_2, \psi_2 ) \) 
by \( G = S^1 \times S^1 \) has a group-valued momentum map (which 
is the product of two copies of the one discussed in 
\cref{ex:symplecticTorus}) but the action of the diagonally 
embedded circle has no group-valued momentum map as we have just seen.
\end{example}

\subsection{Equivariance and Poisson property}
 
Fix a dual pair \( \kappa(G, G^*) \) of Lie groups.
Let \(  G \) act on the symplectic manifold \( (M, \omega) \) 
such that the action has a group-valued momentum map \( J: M \to G^*\).
A natural question to ask is in which sense \( J \) is equivariant.
For classical momentum maps with values in \( \LieA{g}^* \), the 
equivariance of interest is with respect to the coadjoint action.
Before we give the definition of the analogue of the coadjoint action 
in the group-valued setting, let us recall a few basic notations 
and conventions; see \cref{sec::conventions} for details.
The coadjoint action of \( G \) on \( \LieA{g}^* \) is defined with respect to the duality pairing 
\( \kappa \) by\footnote{In infinite dimensions, \( \kappa \) is, 
in general, only weakly non-degenerate and thus this relation only 
ensures uniqueness of \( \CoAdAction \) but not its existence. In 
the sequel, we assume that \( \CoAdAction \) exists.} 
\( \kappa(A, \CoAdAction_g \mu) = \kappa(\AdAction_g A, \mu) \) 
for \( g \in G \), \( A \in \LieA{g} \) and \( \mu \in \LieA{g}^* \).
The coadjoint action of \( G^* \) on \( \LieA{g} \) is similarly defined 
and will be denoted by \( \CoAdAction \) as well.
For an action \( \Upsilon: G \times G^* \to G^* \), we introduce 
the partial maps \( \Upsilon_g: G^* \to  G^* \) and 
\( \Upsilon_\eta: G \to G^* \) by \( \Upsilon_g (\eta) \equiv 
\Upsilon_\eta (g) \equiv \Upsilon(g, \eta ) \) for \( g \in G \) 
and \( \eta \in G^* \).

\begin{defn}
A \emphDef{coconjugation action} on the dual pair \( \kappa(G, G^*) \) 
is a pair \( (\Upsilon, \Upsilon^*) \) consisting of a left action 
\( \Upsilon: G \times G^* \to G^* \) of \( G \) on \( G^* \) and a 
left action \( \Upsilon^*: G^* \times G \to G \) of \( G^* \) on 
\( G \) which integrate the coadjoint actions in the sense that 
\( \Upsilon_g (e) = e \), \( \Upsilon^*_\eta(e) = e \) and
\begin{equation}
	\label{eq:coconjugation:defn}
	\tangent_e \Upsilon_g (\mu) = \CoAdAction_{g^{-1}} \mu,
	\qquad
	\tangent_e \Upsilon^*_\eta (A) = \CoAdAction_{\eta^{-1}} A
\end{equation}
holds for all \( g \in G \), \( \eta \in G^* \), \(\mu \in \LieA{g}^* \), 
and \( A \in \LieA{g} \). If, moreover, 
\begin{equation}
\label{eq:coconjugation:matched}
\Upsilon_g (\eta_1 \eta_2) = \Upsilon_g (\eta_1) \, 
\Upsilon_{\Upsilon^*_{\eta_1^{-1}}(g)} (\eta_2),
\qquad
\Upsilon^*_\eta(g_1 g_2) = \Upsilon^*_{\Upsilon_{g_2}(\eta^{-1})^{-1}}(g_1) 
\, \Upsilon^*_\eta(g_2)
\end{equation}
for all \( g_1, g_2 \in G \) and \( \eta_1, \eta_2 \in G^* \), then we 
say that the tuple \( (G, G^*, \Upsilon, \Upsilon^*) \) (or shortly 
\( (\Upsilon, \Upsilon^*) \)) is a \emph{coconjugation matched pair}.
\end{defn}
A pair of groups endowed with two actions on each other that are 
compatible in the sense of~\eqref{eq:coconjugation:matched} is called 
a \emph{matched pair of groups}. The notion of a matched pair of 
groups appears in \parencite[Section~2]{Takeuchi1981} and has been 
connected to the theory of Poisson Lie groups in 
\parencite{Majid1990,LuWeinstein1990}.
The compatibility condition~\eqref{eq:coconjugation:matched} allows to 
define a (non-trivial) Lie group structure on the product 
\( G \times G^* \) by
\begin{equation}
\label{eq:coconjugation:crossedbiproduct}
(g_1, \eta_1) \cdot (g_2, \eta_2) = 
\bigl( g_1 \Upsilon^*_{\eta_1}(g_2^{-1})^{-1}, 
\Upsilon_{g_2^{-1}} (\eta_1^{-1})^{-1} \eta_2 \bigr),
\end{equation}
see \parencite[Proposition~2.2]{Takeuchi1981}.
Conversely, every factorization \( D = G G^* \) of a Lie group 
\( D \) into two Lie subgroups \( G \) and \( G^* \) with 
\( G \intersect G^* = \set{e} \) yields actions of \( G \) 
and \( G^* \) on each other that satisfy~\eqref{eq:coconjugation:matched} 
by declaring
\begin{equation}
	\label{eq:coconjugation:factorizationToMatchedPair}
	\eta g^{-1} = \bigl(\Upsilon^*_\eta(g)\bigr)^{-1} 
	\bigl(\Upsilon_g(\eta^{-1})\bigr)^{-1} \in G G^*.
\end{equation}

The class of coconjugation matched pairs is rather rigid, especially 
if one of the coconjugation actions is predetermined.
\begin{lemma}
\label{prop:coconjugation:uniquness}
Let \( \kappa(G, G^*) \) be a dual pair of Lie groups with a matched 
pair \( (\Upsilon, \Upsilon^*) \). If 
\( \tilde{\Upsilon}: G \times G^* \to G^* \) is another action 
of \( G \) on \( G^* \) such that \( (\tilde{\Upsilon}, \Upsilon^*) \) 
forms a coconjugation matched pair, then there exists a map 
\( c: G \times \pi_0 (G^*) \to G^* \) with \( c(g, \equivClass{e}) = e \) 
such that \( \tilde{\Upsilon}(g, \eta) = c(g, \equivClass{\eta}) \, 
\Upsilon(g, \eta) \). In particular, if \( G^* \) is connected, then 
coconjugation actions \( \Upsilon \) compatible with \( \Upsilon^* \) 
are unique.
\end{lemma}
\begin{proof}
Using the defining properties~\eqref{eq:coconjugation:defn} and~\eqref{eq:coconjugation:matched}, for \( g \in G \), 
\( \eta \in G^* \) and \( \mu \in \LieA{g}^* \), we calculate
\begin{equation}
\label{eq:coconjugation:tangentActionEveryPoint}
\tangent_\eta \Upsilon_g (\eta \ldot \mu)
= \difFracAt{}{\varepsilon}{0} \Upsilon_g (\eta \mu_\varepsilon)
= \difFracAt{}{\varepsilon}{0} \Upsilon_g (\eta) \, 
\Upsilon_{\Upsilon^*_{\eta^{-1}}(g)}(\mu_\varepsilon)
= \Upsilon_g (\eta) \ldot 
\bigl(\CoAdAction_{\Upsilon^*_{\eta^{-1}}(g)^{-1}} \mu\bigr),
\end{equation}
where \( \varepsilon \mapsto \mu_\varepsilon \) is a curve in \( G^* \) 
starting at the identity with \( \difFracAt{}{\varepsilon}{0} 
\mu_\varepsilon = \mu \). A similar result holds, of course, also 
for \( \Upsilon \) replaced by \( \tilde{\Upsilon} \).
In particular, the left-logarithmic derivatives of 
\( \Upsilon_g: G^* \to G^* \) and \( \tilde{\Upsilon}_g: G^* \to G^* \) 
coincide for every \( g \in G \). The uniqueness lemma for Lie 
group-valued functions \parencite[Remark II.3.6.(c)]{Neeb2006} hence 
implies that on every connected component of \( G^* \) the maps 
\( \Upsilon_g \) and \( \tilde{\Upsilon}_g \) differ by a constant.
Thus, there exists a map \( c: G \times \pi_0 (G^*) \to G^* \) such 
that \( \tilde{\Upsilon}(g, \eta) = c(g, \equivClass{\eta}) \, 
\Upsilon(g, \eta) \). Finally, \( \tilde{\Upsilon}(g, e)  = e = 
\Upsilon(g, e) \) implies \( c(g, \equivClass{e}) = e \).
\end{proof}

\begin{example}
\label{ex:coconjugation:dualLieAlgebra}
Let \( G \) be a Lie group with Lie algebra \( \LieA{g} \) and 
let \( \kappa(\LieA{g}, \LieA{g}^*) \) be a dual pair of (locally convex) 
vector spaces. Considering \( \LieA{g}^* \) as an Abelian Lie group, we 
obtain the dual pair \( \kappa(G, \LieA{g}^*) \) of Lie groups. Then the 
coadjoint representation \( \Upsilon(g, \mu) = \CoAdAction_{g^{-1}} \mu \) 
and the trivial action \( \Upsilon^*(\mu, g) = g \) form a coconjugation 
matched pair on \( \kappa(G, \LieA{g}^*) \).
\end{example}
Since \( \LieA{g}^* \) is connected, the previous proposition implies 
that the coadjoint representation is the only coconjugation action 
compatible with the trivial action of \( \LieA{g}^* \) on \( G \).
However, there are interesting examples where the dual group does not 
act trivially.

\begin{example}
	\label{ex:coconjugation:iwasawaDecomp}
	Let \( D \) be a finite-dimensional connected complex semisimple Lie group and let \( K \) be its compact real form.
Consider the Iwasawa decomposition \( D = K A N \) and set \( B = A N \).
Then \( B \) is a (solvable) Lie subgroup of \( D \) with Lie algebra 
\( \LieA{b} \). Let \( \pr_K: D \to K \) and \( \pr_B: D \to B \) be 
the natural projections, and \( \pr_{\LieA{k}}: \LieA{d} \to \LieA{k} \) 
and \( \pr_{\LieA{b}}: \LieA{d} \to \LieA{b} \) be the associated 
projections on the level of Lie algebras.
By~\eqref{eq:coconjugation:factorizationToMatchedPair}, the 
decomposition \( D = K B \) induces the following actions:
\begin{equation}\label{eq:coconjugation:iwasawaDecomp:actions}
	\Upsilon_k(b) = \pr_B (b^{-1}k^{-1})^{-1},
	\qquad
	\Upsilon^*_b(k) = \pr_K (bk^{-1})^{-1}.
\end{equation}
Going back to the construction of the Iwasawa decomposition, one 
can show that \( \LieA{k} \) and \( \LieA{b} \) are isotropic 
subspaces of \( \LieA{d} \) relative to \( \Im \kappa_{\LieA{d}} \), 
see \parencite[p.~520]{LuWeinstein1990}.
Thus, we obtain a dual pair \( \kappa(K, B) \) of Lie groups.
We have \( \tangent_e \Upsilon_k (\mu) = \pr_{\LieA{b}}(\AdAction_k \mu) \) 
and \( \tangent_e \Upsilon^*_b(A) = \pr_{\LieA{k}}(\AdAction_b A) \) 
for \( \mu \in \LieA{b} \) and \( A \in \LieA{k} \).
These identities show that the actions integrate the coadjoint actions 
and thus \( (\Upsilon, \Upsilon^*) \) is a coconjugation action 
on \( \kappa(K, B) \). As \( D = K B \) is a decomposition in 
subgroups, \( (\Upsilon, \Upsilon^*) \) is a coconjugation matched pair.
\end{example}

Recall that, in the classical setting \( G^* \equiv \LieA{g}^* \), 
the coadjoint action is closely connected to the natural Poisson 
structure of Kostant--Kirillov--Souriau.
Moreover, the Cartan and Iwasawa decompositions are well-known 
constructions in the theory of Poisson Lie groups. Thus, one would 
expect a close interplay between coconjugation actions and Poisson 
structures. As we will see, this is indeed the case.

Poisson structures in infinite dimensions need to be treated with 
caution.
For example, according to \parencite{BeltitaGolinskiEtAl2018} not every Poisson bracket is induced by a Poisson bivector field.
We refer the reader to \parencite{NeebSahlmannEtAl2014} for 
a recent discussion of the problems concerning infinite-dimensional 
Poisson geometry. In the present case, we can use the group structure 
to circumvent most of these issues. Let us start by recalling the 
necessary terminology and notation from the finite-dimensional 
theory of Poisson Lie groups. Let \( H \) be a finite-dimensional Lie 
group, \( \LieA{h} \) be its Lie algebra with dual \( \LieA{h}^* \).
We recall that for a smooth map \( F: M \to H \) the 
\emph{left logarithmic derivative} is the \( \LieA{h} \)-valued 
\( 1 \)-form \( \difLog F \) on \( M \) defined by 
\( \difLog_m F (X) = F(m)^{-1} \ldot \tangent_m F (X) \) for 
\( X \in \TBundle_m M \). In the following, we also need the concept 
of the \emph{left derivative}. This time, the domain of the map is 
a Lie group, \ie, we are considering a smooth map \( f: H \to N \). 
The left derivative \( \tangentLeft_h f: \LieA{h} \to \TBundle_{f(h)}N \) 
of \( f \) at \( h \in H \) in the direction \( A \in \LieA{h} \) is 
defined by
\begin{equation}
	\tangentLeft_h f (A) \defeq \tangent_h f (h \ldot A), 
\end{equation}
that is, as the usual derivative \( \tangent f: \TBundle H \to \TBundle N \) 
precomposed with the left trivialization \( H \times \LieA{h} \isomorph 
\TBundle H \). In particular, for a real-valued function \( f: H \to \R \), 
the left derivative \( \tangentLeft_h f: \LieA{h} \to \R \) can be viewed 
as an element of \( \LieA{h}^* \).

In finite dimensions, a Poisson Lie group is a Lie group \( H \) endowed 
with a Poisson structure \( \varpi_H \in \VectorFieldSpace^2(H) \) such 
that the multiplication \( H \times H \to H \) is a Poisson map (with 
respect to the product Poisson structure on \( H \times H \)).
By \parencite[Theorem~2.18]{Lu1990}, the linearization of \( \varpi_H \) 
at the identity \( e \in H \) induces a Lie bracket on \( \LieA{h}^* \) 
by
\begin{equation}
	\LieBracket{A}{B}_{\varpi_H} = \dif_e (\varpi_H (\bar{A}, \bar{B})), 
\end{equation}
where \( A, B \in \LieA{h}^* \) and \( \bar{A}, \bar{B} \) are arbitrary 
\( 1 \)-forms on \( H \) with \( \bar{A}_e = A \) and \( \bar{B}_e = B \).
In the left-trivialization \( \TBundle H \isomorph H \times \LieA{h} \), 
the bivector field \( \varpi_H \) is identified with a smooth map 
\( \Lambda_H: H \to \ExtBundle^2 \LieA{h} \) and we have 
\( \LieBracket{A}{B}_{\varpi_H} = 
\tangent_e \bigl(\Lambda_H (\cdot) (A, B)\bigr) \).

Let \( \kappa(\LieA{g}, \LieA{h}) \) be a dual pair of finite-dimensional 
Lie algebras. For every \( \LieA{g} \)-module \( M \), consider the 
Chevalley--Eilenberg complex \( \Hom(\ExtBundle^\bullet \LieA{g}, M) \) 
consisting of continuous alternating multilinear maps \( \lambda: 
\ExtBundle^k \LieA{g} \to M \) with differential
\begin{equation}
\begin{split}
\dif^{\textrm{CE}} \lambda (A_1, \dotsc A_{k+1})
&=\sum_i(-1)^{i+1}A_i\ldot\lambda(A_1,\dotsc\hat{\imath}\dotsc A_{k+1})\\
&\quad + \sum_{i<j} (-1)^{i+j} \lambda(\commutator{A_i}{A_j}, A_1, 
\dotsc \hat{\imath} \dotsc \hat{\jmath} \dotsc A_{k+1}).
\end{split}
\end{equation}
In particular, for \( M = \LieA{h} \) with (minus) the coadjoint action, 
we have
\begin{equation}\begin{split}
\dif^{\textrm{CE}}_{\LieA{h}} \lambda (A_1, \dotsc A_{k+1})
&= \sum_i (-1)^{i} \CoadAction_{A_i} \lambda(A_1, \dotsc 
\hat{\imath} \dotsc A_{k+1})  \\
&\quad + \sum_{i<j} (-1)^{i+j} \lambda(\commutator{A_i}{A_j}, A_1, 
\dotsc \hat{\imath} \dotsc \hat{\jmath} \dotsc A_{k+1}).
\end{split}
\end{equation}
Moreover, using the identification \(\Hom(\ExtBundle^\bullet \LieA{g},\R) 
\isomorph \ExtBundle^\bullet \LieA{g}^* \isomorph \ExtBundle^\bullet 
\LieA{h} \), we can transfer the Chevalley–Eilenberg differential 
to \( \ExtBundle^\bullet \LieA{h} \). Explicitly, it is given by
\begin{equation}
\begin{split}
\kappa \bigl(A_1 \wedge \dotsb &\wedge A_{k+1}, 
\dif^{\textrm{CE}}_{\LieA{g}} \nu \bigr)    \\
&= \sum_{i < j} (-1)^{i+j} \kappa \bigl(\LieBracket{A_i}{A_j}_\LieA{g} 
\wedge A_1 \wedge \dotsb \hat{\imath} \dotsb \hat{\jmath} \dotsb 
\wedge A_{k+1}, \nu \bigr)
\end{split}
\end{equation}
for every \( \nu \in \ExtBundle^k \LieA{h} \) and \( A_i \in \LieA{g} \).
Moreover, given \( T \in \Hom(\LieA{h}, \ExtBundle^2 \LieA{h}) \isomorph 
\LieA{g} \tensorProd \ExtBundle^2 \LieA{h} \) and 
\( \beta \in \ExtBundle^2 \LieA{h} \), we define the 
contraction \( T \contr \beta \in \ExtBundle^3 \LieA{h} \) by
\begin{equation}
\label{eq:poissonLie:contractionDef}
(A \otimes \alpha) \contr \beta \defeq \alpha \wedge (A \contr \beta) 
= (A \contr \beta) \wedge \alpha
\end{equation}
for \( A \in \LieA{g} \) and \( \alpha \in \ExtBundle^2 \LieA{h} \).
Finally, recall that the (algebraic) Schouten--Nijenhuis bracket 
\( \LieBracketDot: \ExtBundle^k \LieA{h} \times \ExtBundle^l \LieA{h} 
\to \ExtBundle^{k+l-1} \LieA{h} \) is defined by
\begin{equation}
\LieBracket{\nu_1 \wedge \dotsb \wedge \nu_k}{\mu_1 \wedge \dotsb 
\wedge \mu_l}
= \sum_{i, j} (-1)^{i+j} \LieBracket{\nu_i}{\mu_j}_{\LieA{h}} 
\wedge \nu_1 \wedge \dotsb \hat{\imath} \dotsb \wedge \nu_k 
\wedge \mu_1 \wedge \dotsb \hat{\jmath} \dotsb \wedge \mu_l \, .
\end{equation}
Let us introduce a related bracket as follows.
For linear maps \( \phi: \ExtBundle^k \LieA{g} \to \LieA{h} \) 
and \( \varphi: \ExtBundle^l \LieA{g} \to \LieA{h} \), define 
\( \LieBracket{\phi}{\varphi}: \ExtBundle^{k+l} \LieA{g} \to 
\LieA{h} \) by
\begin{equation}
\begin{split}
\LieBracket{\phi}{\varphi}(A) 
&= \sum_{\sigma} (-1)^{\sgn \sigma} 
\Bigl( \phi\bigl( \CoadAction_{\varphi(\sigma_l(A))} \sigma_k(A)\bigr) 
- \varphi\bigl( \CoadAction_{\phi(\sigma_k(A))} \sigma_l(A)\bigr) 
		\\
&\qquad\qquad\qquad\qquad+ \LieBracket[\big]{\phi(\sigma_k(A))}
{\varphi(\sigma_l(A))}_{\LieA{h}} \Bigr),
\end{split}\end{equation}
where the sum is over (unordered) partitions \( \sigma \) of 
\( A \in \ExtBundle^{k+l} \LieA{g} \) in pairs \( \sigma_k(A) 
\in \ExtBundle^k \LieA{g} \) and \( \sigma_l(A) \in 
\ExtBundle^{l} \LieA{g} \).
In particular, for a linear map \( \phi: \LieA{g} \to \LieA{h} \), 
we get
\begin{equation}
\label{eq:poissonLie:strangeSchouten}
\frac{1}{2} \LieBracket{\phi}{\phi}(A, B)
= \phi\bigl(\CoadAction_{\phi(A)} B\bigr) - 
\phi\bigl(\CoadAction_{\phi(B)} A\bigr) + 
\LieBracket[\big]{\phi(A)}{\phi(B)}_{\LieA{h}}
\end{equation}
for all \( A, B \in \LieA{g} \).

The following result is the first step towards a reformation of the 
notion of a Poisson Lie structure that is suitable for the 
infinite-dimensional setting. 

\begin{prop}
	\label{prop:poissonLie:inLeftTriv}
	Let \( \kappa(G, H) \) be a dual pair of finite-dimensional Lie groups.
	Let \( \Lambda_H: H \to \ExtBundle^2 \LieA{h} \) be a smooth map and let \( \varpi_H \in \VectorFieldSpace^2(H) \) be the associated bivector field.
	Then \( \varpi_H \) is a Poisson Lie structure on \( H \) whose linearization \( \LieBracketDot_{\varpi_H} \) coincides with the Lie bracket on \( \LieA{g} \isomorph \LieA{h}^* \) if and only if \( \Lambda_H \) possesses the following properties:
	\begin{thmenumerate}
		\item (multiplicativity)
			For all \( \eta, \zeta \in H \),
			\begin{equation}
				\label{eq:poissonLie:multiplicative}
				\Lambda_H(\eta \, \zeta) = \AdAction_{\zeta^{-1}} \Lambda_H(\eta) + \Lambda_H(\zeta),
			\end{equation}
			where \( \AdAction \) denotes the natural extension to \( \ExtBundle^2 \LieA{h} \) of the adjoint action of \( H \) on \( \LieA{h} \).
		\item (Poisson property)
			\begin{equation}
				\label{eq:poissonLie:poisson}
				\dif^{\textrm{CE}}_{\LieA{g}} \, \Lambda_H - \frac{1}{2} \LieBracket{\Lambda_H}{\Lambda_H} = 0.
			\end{equation}
		\item (compatibility)
			For all \( \mu \in \LieA{h} \) and \( A, B \in \LieA{g} \),
			\begin{equation}
				\label{eq:poissonLie:compatible}
				\kappa \bigl( A \wedge B, \tangent_e \Lambda_H (\mu) \bigr) = \kappa \bigl(\LieBracket{A}{B}_{\LieA{g}}, \mu \bigr).
				\qedhere
			\end{equation}
	\end{thmenumerate}
\end{prop}
Note that the Poisson property is reformulated in terms of the purely algebraic condition~\eqref{eq:poissonLie:poisson}.
For the particular case where \( H \) is Abelian,~\eqref{eq:poissonLie:poisson} holds if and only if \( \Lambda_H (\eta) \) is a \( 2 \)-cocycle on \( \LieA{g} \) for all \( \eta \in H \).
\begin{proof}
	By \parencite[Lemma~4.1]{AlekseevskyGrabowskiMarmoEtAl1998}, multiplication \( H \times H \to H \) is a Poisson map if and only if \( \Lambda_H \) is multiplicative in the sense that~\eqref{eq:poissonLie:multiplicative} holds.

	If \( \bar{A}, \bar{B} \in \DiffFormSpace^1(H) \) are left-invariant extensions of \( A, B \in \LieA{g} \isomorph \LieA{h}^* \), then by definition \( (\varpi_H)_\eta (\bar{A}, \bar{B}) = \Lambda_H(\eta)(A, B) \) for every \( \eta \in H \).
	Thus, the Lie bracket on \( \LieA{g} \) induced by \( \varpi_H \) is given by
	\begin{equation}
		\kappa \bigl(\LieBracket{A}{B}_{\varpi_H}, \mu \bigr)
			= \kappa \bigl( \dif_e \bigl(\varpi_H (\bar{A}, \bar{B})\bigr), \mu \bigr)
			= \kappa \bigl( A \wedge B, \tangent_e \Lambda_H(\mu) \bigr). 
	\end{equation}
	Thus, \( \LieBracketDot_{\varpi_H} \) coincides with the original Lie bracket on \( \LieA{g} \) if and only if~\eqref{eq:poissonLie:compatible} holds.

	By \parencite[Proposition~3.4]{AlekseevskyGrabowskiMarmoEtAl1998}, the Schouten--Nijenhuis bracket of \( \varpi_H \) with itself vanishes if and only if 
	\begin{equation}
		\label{eq:poissonLie:poissonLeft}
		\bigl(\tangentLeft \Lambda_H\bigr) \contr \Lambda_H = \frac{1}{2} \LieBracket{\Lambda_H}{\Lambda_H},
	\end{equation}
	where, for every \( \eta \in H \), the left derivative \( \tangentLeft_\eta \Lambda_H \) is an element of \( \Hom(\LieA{h}, \ExtBundle^2 \LieA{h}) \). 
	By differentiating~\eqref{eq:poissonLie:multiplicative}, we obtain
	\begin{equation}
		\tangentLeft_\eta \Lambda_H (\mu)
			= \tangent_\eta \Lambda_H (\eta \ldot \mu)
			= - \adAction_\mu \Lambda_H(\eta) + \tangent_e \Lambda_H (\mu).
	\end{equation}
	Using the compatibility relation~\eqref{eq:poissonLie:compatible}, \cref{prop:poissonLie:Tcontr,prop:poissonLie:omegaContr} below imply that
	\begin{equation}
		\bigl(\tangentLeft \Lambda_H\bigr) \contr \Lambda_H = \LieBracket{\Lambda_H}{\Lambda_H} - \dif^{\textrm{CE}}_{\LieA{g}} \, \Lambda_H.
	\end{equation}
	Thus,~\eqref{eq:poissonLie:poissonLeft} is equivalent to~\eqref{eq:poissonLie:poisson}.
\end{proof}
\begin{lemma}
	\label{prop:poissonLie:Tcontr}
	Let \( \kappa(\LieA{g}, \LieA{h}) \) be a dual pair of finite-dimensional Lie algebras.
	For every \( \Lambda \in \ExtBundle^2 \LieA{h} \), the operator \( T: \LieA{h} \to \ExtBundle^2 \LieA{h} \) defined by \( T(\mu) = \adAction_\mu \Lambda \) satisfies
	\begin{equation}
		T \contr \Lambda = - \LieBracket{\Lambda}{\Lambda}.
		\qedhere	
	\end{equation}
\end{lemma}
\begin{proof}
	Choose a basis \( \set{e_i} \) in \( \LieA{g} \) and the dual basis \( \set{\varepsilon^i} \) in \( \LieA{h} \).
	Write
	\begin{equation}
		T = e_i \tensorProd \adAction_{\varepsilon^i} \Lambda \qquad \text{and} \qquad \Lambda = \frac{1}{2} \Lambda_{kl} \, \varepsilon^k \wedge \varepsilon^l \, .
	\end{equation}
	Then, we have
	\begin{equation}
		e_i \contr \Lambda = \frac{1}{2} \Lambda_{kl} (\delta_i^k \varepsilon^l - \delta_i^l \varepsilon^k) = \Lambda_{il} \, \varepsilon^l
	\end{equation}
	and
	\begin{equation}
		\adAction_{\varepsilon^i} \Lambda 
			= \frac{1}{2} \Lambda_{kl} \, \bigl(\commutator{\varepsilon^i}{\varepsilon^k}_{\LieA{h}} \wedge \varepsilon^l + \varepsilon^k \wedge \commutator{\varepsilon^i}{\varepsilon^l}_{\LieA{h}}\bigr)
			= \Lambda_{kl} \, \commutator{\varepsilon^i}{\varepsilon^k}_{\LieA{h}} \wedge \varepsilon^l \, .
	\end{equation}
	Thus, using~\eqref{eq:poissonLie:contractionDef}, we obtain
	\begin{equation}\begin{split}
		T \contr \Lambda
			&= \adAction_{\varepsilon^i} \Lambda \wedge (e_i \contr \Lambda)
			= \Lambda_{kl} \Lambda_{ij} \, \commutator{\varepsilon^i}{\varepsilon^k}_{\LieA{h}} \wedge \varepsilon^l \wedge \varepsilon^j \, .
	\end{split}\end{equation}
	On the other hand, we find
	\begin{equation}\label{eq:poissonLie:Tcontr:shoutenLambda}\begin{split}
		\LieBracket{\Lambda}{\Lambda} 
			&= \frac{1}{4} \Lambda_{kl} \Lambda_{ij} \, \LieBracket{\varepsilon^k \wedge \varepsilon^l}{\varepsilon^i \wedge \varepsilon^j}
			\\
			&= \frac{1}{4} \Lambda_{kl} \Lambda_{ij} \, \bigl( 
				\LieBracket{\varepsilon^k}{\varepsilon^i}_\LieA{h} \wedge \varepsilon^l \wedge \varepsilon^j
				- \LieBracket{\varepsilon^k}{\varepsilon^j}_\LieA{h} \wedge \varepsilon^l \wedge \varepsilon^i
			\\
			& \qquad\qquad\quad
				- \LieBracket{\varepsilon^l}{\varepsilon^i}_\LieA{h} \wedge \varepsilon^k \wedge \varepsilon^j
				+ \LieBracket{\varepsilon^l}{\varepsilon^j}_\LieA{h} \wedge \varepsilon^k \wedge \varepsilon^i
				\bigr)
			\\
			&= \Lambda_{kl} \Lambda_{ij} \, \LieBracket{\varepsilon^k}{\varepsilon^i}_\LieA{h} \wedge \varepsilon^l \wedge \varepsilon^j \, .
	\end{split}\end{equation}
	Comparing the last two identities yields the claim.
\end{proof}

\begin{lemma}
	\label{prop:poissonLie:omegaContr}
	Let \( \kappa(\LieA{g}, \LieA{h}) \) be a dual pair of finite-dimensional Lie algebras.
	For every \( \Lambda \in \ExtBundle^2 \LieA{h} \),
	\begin{equation}
		\omega \contr \Lambda = - \dif^{\textrm{CE}}_{\LieA{g}} \, \Lambda,
	\end{equation}
	where the operator \( \omega: \LieA{h} \to \ExtBundle^2 \LieA{h} \) is defined by \( \kappa \bigl( A \wedge B, \omega (\mu) \bigr) = \kappa \bigl(\LieBracket{A}{B}_{\LieA{g}}, \mu \bigr) \) for all \( A, B \in \LieA{g} \).
\end{lemma}
\begin{proof}
	Fix a basis \( \set{ e_k } \) of \( \LieA{g} \).
	Denote by \( \commutatorDot^k \in \ExtBundle^2 \LieA{h}\) the \( k \)-th component of the commutator, \ie, \( \commutator{A}{B}^k e_k = \commutator{A}{B}_\LieA{g} \) for all \( A, B \in \LieA{g} \).
	With this notation, we have \( \omega = e_k \otimes \commutatorDot^k \), because
	\begin{equation}
		\kappa \bigl(A \wedge B, \omega(\mu) \bigr)
			= \kappa(\commutator{A}{B}_\LieA{g}, \mu)
			= \kappa(e_k, \mu)\, \commutator{A}{B}^k.
	\end{equation}
	Thus we find
	\begin{equation}\begin{split}
		(\omega \contr \Lambda) (A, B, C)
			&= \bigl((e_k \otimes \commutatorDot^k) \contr \Lambda\bigr) (A, B, C)
			\\
			&= \bigl((e_k \contr \Lambda) \wedge \commutatorDot^k\bigr) (A, B, C)
			\\
			&= \Lambda(e_k, A) \, \commutator{B}{C}^k - \Lambda(e_k, B) \, \commutator{A}{C}^k + \Lambda(e_k, C) \, \commutator{A}{B}^k
			\\
			&= \Lambda(\commutator{B}{C},A) - \Lambda(\commutator{A}{C},B) + \Lambda(\commutator{A}{B}, C)
			\\
			&= - \Lambda(A,\commutator{B}{C}) + \Lambda(B, \commutator{A}{C}) - \Lambda(C, \commutator{A}{B})
			\\
			&= - \dif^{\textrm{CE}}_{\LieA{g}} \, \Lambda (A, B, C).
			\qedhere
	\end{split}\end{equation}
\end{proof}

\Cref{prop:poissonLie:inLeftTriv} yields a convenient reformulation of what it means for a bivector field \( \varpi_H \in \VectorFieldSpace^2(H) \) to be a Poisson Lie structure on \( H \) in terms of the associated map \( \Lambda_H: H \to \ExtBundle^2 \LieA{h} \).
In the sequel, it will be convenient to view the latter as a map
\begin{equation}
	\pi_H: H \times \LieA{g} \to \LieA{h},
\end{equation}
that is, \( \Lambda_H(\eta)(A, B) = \kappa\bigl(B, \pi_H(\eta, A)\bigr) \).
From this viewpoint, we have the following equivalent description of a Poisson Lie structure on \( H \).
\begin{prop}
	\label{prop:poissonLie:inLeftTrivDualized}
	Let \( \kappa(G, H) \) be a dual pair of finite-dimensional Lie groups.
	Let \( \Lambda_H: H \to \ExtBundle^2 \LieA{h} \) be a smooth map satisfying the properties of \cref{prop:poissonLie:inLeftTriv} (so that its extension \( \varpi_H \in \VectorFieldSpace^2(H) \) is a Poisson Lie structure on \( H \)).
	Then the map \( \pi_H: H \times \LieA{g} \to \LieA{h} \) defined by \( \pi_H(\eta, A) = A \contr \Lambda_H \) satisfies the following conditions:
	\begin{thmenumerate}
		\item (linearity and skew-symmetry)
			\( \pi_H \) is linear in the second argument and satisfies
			\begin{equation}
				\label{eq:poissonLie:inLeftTrivDualized:skewsymmetric}
				\kappa\bigl(A, \pi_H(\eta, B)\bigr) = - \kappa\bigl(B, \pi_H(\eta, A)\bigr)
			\end{equation}
			for all \( A, B \in \LieA{g} \) and \( \eta \in H \).
		\item (multiplicativity)
			For all \( \eta, \zeta \in H \) and \( A \in \LieA{g} \),
			\begin{equation}
				\label{eq:poissonLieDef:multiplicative}
				\pi_H(\eta \, \zeta, A) = \AdAction_{\zeta^{-1}} \pi_H(\eta, \CoAdAction_{\zeta^{-1}} A) + \pi_H(\zeta, A).
			\end{equation}
		\item (Poisson property)
			\begin{equation}
				\label{eq:poissonLieDef:poisson}
				\dif^{\textrm{CE}}_{\LieA{h}} \, \pi_H - \frac{1}{2} \LieBracket{\pi_H}{\pi_H} = 0.
			\end{equation}
		\item (compatibility)
			For all \( \mu \in \LieA{h} \) and \( A \in \LieA{g} \),
			\begin{equation}
				\label{eq:poissonLieDef:compatible}
				\tangent_e \bigl(\pi_H (\cdot, A)\bigr) (\mu) = \CoadAction_A \mu,
			\end{equation}
			where \( \pi_H (\cdot, A): H \to \LieA{h} \).
	\end{thmenumerate}
	Conversely, to every map \( \pi_H: H \times \LieA{g} \to \LieA{h} \) satisfying these conditions, the map \( \Lambda_H: H \to \ExtBundle^2 \LieA{h} \) defined by \( \pi_H(\eta, A) = A \contr \Lambda_H \) posses the properties of \cref{prop:poissonLie:inLeftTriv}.
\end{prop}
\begin{proof}
	Linearity in the second argument and skew-symmetry of \( \pi_H \) is clear as \( \Lambda_H \) takes values in \( \ExtBundle^2 \LieA{h} \).
	Moreover, the multiplicativity and compatibility conditions are simple reformulations of the corresponding properties in \cref{prop:poissonLie:inLeftTriv}.

	Thus, it is left to show that~\eqref{eq:poissonLie:poisson} and~\eqref{eq:poissonLieDef:poisson} are equivalent.
	Since this a completely algebraic question (for fixed \( \eta \in H \)), it suffices to consider a linear map \( \pi: \LieA{g} \to \LieA{h} \) with associated \( \Lambda \in \ExtBundle^2 \LieA{h} \).
	First note that we have
	\begin{equation}\begin{split}
		\kappa\bigl(C, \dif^{\textrm{CE}}_{\LieA{h}} \, \pi (A, B)\bigr)
			&= \kappa\bigl(C, -\CoadAction_A \pi(B) + \CoadAction_B \pi(A) - \pi(\commutator{A}{B}) \bigr)
			\\
			&= - \kappa\bigl(\commutator{A}{C}, \pi(B)\bigr) + \kappa\bigl(\commutator{B}{C}, \pi(A)\bigr) - \kappa\bigl(C, \pi(\commutator{A}{B}) \bigr)
			\\
			&= \kappa\bigl(\commutator{A}{C} \wedge B - \commutator{B}{C} \wedge A - \commutator{A}{B} \wedge C, \Lambda \bigr)
			\\
			&= \kappa\bigl(A \wedge B \wedge C, \dif^{\textrm{CE}}_{\LieA{g}} \Lambda \bigr).
	\end{split}\end{equation}
	Choose a basis \( \set{e_i} \) in \( \LieA{g} \) and the dual basis \( \set{\varepsilon^i} \) in \( \LieA{h} \).
	According to~\eqref{eq:poissonLie:Tcontr:shoutenLambda}, we have \( \LieBracket{\Lambda}{\Lambda} = \Lambda_{kl} \Lambda_{pq} \, \LieBracket{\varepsilon^k}{\varepsilon^p}_r \varepsilon^r \wedge \varepsilon^l \wedge \varepsilon^q \), where \( \Lambda = \frac{1}{2} \Lambda_{kl} \varepsilon^k \wedge \varepsilon^l \) and \( \LieBracket{\varepsilon^k}{\varepsilon^p}_r \varepsilon^r = \LieBracket{\varepsilon^k}{\varepsilon^p}_{\LieA{h}} \).
	Thus, 
	\begin{equation}\begin{split}
		e_j \contr e_i \contr \LieBracket{\Lambda}{\Lambda}
			&= \Lambda_{kl} \Lambda_{pq} \, \LieBracket{\varepsilon^k}{\varepsilon^p}_r 
				\bigl(\delta_i^r \delta_j^l \varepsilon^q + \delta_i^l \delta_j^q \varepsilon^r + \delta_i^q \delta_j^r \varepsilon^l
				\\
				&\qquad\qquad\qquad\qquad\quad- \delta_i^r \delta_j^q \varepsilon^l - \delta_i^q \delta_j^l \varepsilon^r - \delta_i^l \delta_j^r \varepsilon^q \bigr)
			\\
			&= 2 \Lambda_{kl} \Lambda_{pq} \, \LieBracket{\varepsilon^k}{\varepsilon^p}_r \bigl(\delta_i^r \delta_j^l \varepsilon^q + \delta_i^l \delta_j^q \varepsilon^r + \delta_i^q \delta_j^r \varepsilon^l \bigr)
			\\
			&= 2 \bigl( \Lambda_{kj} \Lambda_{lr} \, \LieBracket{\varepsilon^k}{\varepsilon^l}_i + \Lambda_{ki} \Lambda_{lj} \, \LieBracket{\varepsilon^k}{\varepsilon^l}_r + \Lambda_{kr} \Lambda_{li} \, \LieBracket{\varepsilon^k}{\varepsilon^l}_j\bigr) \varepsilon^r ,
	\end{split}\end{equation}
	because \( \Lambda_{kl} \Lambda_{pq} \, \LieBracket{\varepsilon^k}{\varepsilon^p}_r \) is skew-symmetric in the indices \( l \) and \( q \).
	On the other hand, using~\eqref{eq:poissonLie:strangeSchouten}, the identities \( \pi(e_i) = e_i \contr \Lambda = \Lambda_{il} \varepsilon^l \) and \( \CoadAction_{\varepsilon^l} e_i = \LieBracket{\varepsilon^l}{\varepsilon^k}_i e_k \) imply
	\begin{equation}\begin{split}
		\frac{1}{2} \LieBracket{\pi}{\pi} (e_i, e_j)
			&= \pi\bigl(\CoadAction_{\pi(e_i)} e_j\bigr) - \pi\bigl(\CoadAction_{\pi(e_j)} e_i\bigr) + \LieBracket[\big]{\pi(e_i)}{\pi(e_j)}_{\LieA{h}}
			\\
			&= \Lambda_{il} \Lambda_{kr} \LieBracket{\varepsilon^l}{\varepsilon^k}_j \varepsilon^r
				- \Lambda_{jl} \Lambda_{kr} \LieBracket{\varepsilon^l}{\varepsilon^k}_i \varepsilon^r
				+ \Lambda_{il} \Lambda_{jk} \LieBracket{\varepsilon^l}{\varepsilon^k}_r \varepsilon^r .
	\end{split}\end{equation}
	In summary, we obtain
	\begin{equation}
		\kappa\Bigl(A \wedge B \wedge C, \dif^{\textrm{CE}}_{\LieA{g}} \Lambda - \frac{1}{2}\LieBracket{\Lambda}{\Lambda}\Bigr)
			= \kappa\Bigl(C, \dif^{\textrm{CE}}_{\LieA{h}} \, \pi (A, B) - \frac{1}{2} \LieBracket{\pi}{\pi}(A, B)\Bigr),
	\end{equation}
	so that~\eqref{eq:poissonLie:poisson} and~\eqref{eq:poissonLieDef:poisson} are equivalent indeed.
\end{proof}
It is this form that we adopt as the definition of a Poisson Lie structure in the infinite-dimensional setting.
\begin{defn}
	Let \( \kappa(G, H) \) be a dual pair of (possibly infinite-dimensional) Lie groups.
	A smooth map \( \pi_H: H \times \LieA{g} \to \LieA{h} \) is said to be a \emphDef{Poisson Lie structure on \( H \) compatible with the Lie bracket on \( \LieA{g} \)} if it satisfies the properties stated in \cref{prop:poissonLie:inLeftTrivDualized}.
\end{defn}
In \parencite{NeebSahlmannEtAl2014} it has been observed that a Poisson tensor on an infinite-dimensional manifold does \emph{not} induce a Poisson bracket on the space of all smooth functions.
This continues to be true in our particular setting of Poisson Lie structures.
Indeed, a Poisson Lie structure \( \pi_H \) on \( H \) defines a Lie bracket
\begin{equation}
	\label{eq:poissonLie:poissonBracketFromTensor}
	\poisson{f}{g}(\eta) = \kappa\bigl(\tangentLeft_\eta f, \pi_H(\eta, \tangentLeft_\eta g) \bigr)
\end{equation}
only on the subalgebra of \( \sFunctionSpace(H) \) consisting of functions \( f \) for which the functional \( \tangentLeft f: \LieA{h} \to \R \) can be represented by an element of \( \LieA{g} \).
Conversely, we note that, in infinite dimensions, the Poisson bracket on the space of smooth functions may depend on higher-order derivatives of functions and thus may not be induced by a Poisson tensor, see \parencite{BeltitaGolinskiEtAl2018}.  

After this excursion into the theory of infinite-dimensional Poisson Lie groups, let us return to our original quest of establishing a relation between coconjugation actions and induced Poisson Lie structures.
In finite dimensions, every Poisson Lie group with complete dressing vector fields can be integrated to matched pair of Lie groups, see \parencite[Theorem~3.12]{LuWeinstein1990}.
The following result establishes the corresponding inverse \textquote{derivation} operation. 
\begin{thm}
	\label{prop:poissonLie:fromCoconjugationToPoisson}
	Let \( \kappa(G, G^*) \) be a dual pair of Lie groups.
	For every coconjugation matched pair \( (\Upsilon, \Upsilon^*) \) on \( \kappa(G, G^*) \), define \( \pi_{G}: G \times \LieA{g}^* \to \LieA{g} \) and \( \pi_{G^*}: G^* \times \LieA{g} \to \LieA{g}^* \) by 
	\begin{equation}
		\pi_{G}(g, \mu) \defeq \tangent_e \Upsilon^*_{g^{-1}} \bigl(\CoAdAction_{g^{-1}} \mu\bigr) \ldot g,
		\qquad
		\pi_{G^*}(\eta, A) \defeq - \eta^{-1} \ldot \tangent_e \Upsilon_\eta \bigl(\CoAdAction_{\eta^{-1}} A\bigr).
	\end{equation}
	If \( \pi_G \) and \( \pi_{G^*} \) are skew-symmetric in the sense of~\eqref{eq:poissonLie:inLeftTrivDualized:skewsymmetric}, then they define Poisson Lie structures on \( G \) and on \( G^* \), respectively.
\end{thm}
\begin{proof}
	We will give the proof only for \( \pi_{G^*} \) as similar arguments show that \( \pi_G \) is a Poisson Lie structure as well.
	We have to verify that \( \pi_{G^*} \) satisfies the conditions of \cref{prop:poissonLie:inLeftTrivDualized}.
	
	First, note that \( \pi_{G^*} \) is clearly linear in the second argument and it is skew-symmetric by assumption.

	The compatibility condition~\eqref{eq:coconjugation:matched} implies
	\begin{equation}\begin{split}
		\pi_{G^*}(\eta \, \zeta, A) 
			&= - (\zeta^{-1} \eta^{-1}) \ldot \tangent_e \Upsilon_{\eta\zeta} (\CoAdAction_{\eta^{-1}}\CoAdAction_{\zeta^{-1}}A) \\
			&= - (\zeta^{-1} \eta^{-1}) \ldot \tangent_e \bigl(\Upsilon_{\eta} \cdot (\Upsilon_{\zeta} \circ \Upsilon^*_{\eta^{-1}}) \bigr) (\CoAdAction_{\eta^{-1}}\CoAdAction_{\zeta^{-1}}A) \\
			&= - \AdAction_{\zeta^{-1}} \bigl(\eta^{-1} \ldot \tangent_e \Upsilon_{\eta} (\CoAdAction_{\eta^{-1}}\CoAdAction_{\zeta^{-1}}A)\bigr) 
			\\&\qquad- \zeta^{-1} \ldot \tangent_e \Upsilon_{\zeta} \bigl(\tangent_e \Upsilon^*_{\eta^{-1}}(\CoAdAction_{\eta^{-1}}\CoAdAction_{\zeta^{-1}}A) \bigr) \\
			&= \AdAction_{\zeta^{-1}} \bigl(\pi_{G^*}(\eta, \CoAdAction_{\zeta^{-1}}A)\bigr) + \pi_{G^*}(\zeta, A).
	\end{split}\end{equation}
	This verifies the multiplicativity of \( \pi_{G^*} \), \cf~\eqref{eq:poissonLieDef:multiplicative}.

	Let us now turn to the compatibility relation~\eqref{eq:poissonLieDef:compatible}.
	Note that we have \( (\difLog \Upsilon)_{(e, \eta)} (A, 0) = -\pi_{G^*}(\eta, \CoAdAction_{\eta} A) \) and \( (\difLog \Upsilon)_{(g, \eta)}(0, \eta \ldot \mu) = \CoAdAction_{\Upsilon^*_{\eta^{-1}}(g)^{-1}} \mu \), see~\eqref{eq:coconjugation:tangentActionEveryPoint}.
	Since the \( 1 \)-form \( \difLog \Upsilon \in \DiffFormSpace^1(G \times G^*, \LieA{g}^*) \) satisfies the Maurer--Cartan equation, we thus obtain
	\begin{equation}\label{eq:poissonLie:fromCoconjugationToPoisson:tangentPoisson}\begin{split}
		0
		&= \Bigl(\dif\difLog \Upsilon + \frac{1}{2} \LieBracket{\difLog \Upsilon}{\difLog \Upsilon}\Bigr)_{(e,e)} \bigl((0,\mu), (A,0)\bigr) \\
		&= (0, \bar{\mu})_{(e,e)} \bigl(\difLog \Upsilon (\bar{A}, 0)\bigr) - (\bar{A}, 0)_{(e,e)} \bigl(\difLog \Upsilon (0, \bar{\mu})\bigr)\\
			&\qquad\qquad\qquad+ \LieBracket[\big]{(\difLog \Upsilon)_{(e,e)} (0,\mu)}{(\difLog \Upsilon)_{(e,e)} (A,0)} \\
		&= - \tangent_e \bigl(\eta \mapsto \pi_{G^*}(\eta, \CoAdAction_{\eta} A)\bigr) (\mu) - \tangent_e \bigl(g \mapsto \CoAdAction_{g^{-1}} \mu\bigr)(A) + 0 \\
		&= - \tangent_e \bigl(\pi_{G^*}(\cdot, A)\bigr) (\mu) + \CoadAction_{A} \mu,
	\end{split}\end{equation}
	where \( \bar{A} \) and \( \bar{\mu} \) denote the left-invariant extensions of \( A \) and \( \mu \), respectively.
	Thus, \( \tangent_e \bigl(\pi_{G^*} (\cdot, A)\bigr) (\mu) = \CoadAction_A \mu \), which is the compatibility relation~\eqref{eq:poissonLieDef:compatible}.

	For the Poisson property~\eqref{eq:poissonLieDef:poisson}, we need a bit of preparation.
	First, the fact that \( \Upsilon \) is a left \( G \)-action implies
	\begin{equation}\begin{split}
		\difLog_g \Upsilon_\eta (g \ldot A)
			&= \Upsilon(g, \eta)^{-1} \ldot \tangent_{\eta} \Upsilon_g \bigl(\tangent_e \Upsilon_\eta (A)\bigr)
			\\
			&= - \Upsilon(g, \eta)^{-1} \ldot \tangent_{\eta} \Upsilon_g \bigl(\eta \ldot \pi_{G^*}(\eta, \CoAdAction_{\eta} A) \bigr)
			\\
			&= - \CoAdAction_{\Upsilon^*_{\eta^{-1}}(g)^{-1}} \pi_{G^*}(\eta, \CoAdAction_{\eta} A),
	\end{split}\end{equation}
	where, for the last equality, we used~\eqref{eq:coconjugation:tangentActionEveryPoint}.
	Since \( \difLog \Upsilon_\eta \in \DiffFormSpace^1(G, \LieA{g}^*) \) satisfies the Maurer--Cartan equation for every \( \eta \in G^* \), we thus find
	\begin{equation}\label{eq:poissonLie:fromCoconjugationToPoisson:tangentAction}\begin{split}
		0
		&= \Bigl(\dif\difLog \Upsilon_\eta + \frac{1}{2} \LieBracket{\difLog \Upsilon_\eta}{\difLog \Upsilon_\eta}\Bigr)_e (A, B) \\
		&= \bar{A}_e \bigl(\difLog \Upsilon_\eta (\bar{B})\bigr) - \bar{B}_e \bigl(\difLog \Upsilon_\eta (\bar{A})\bigr) - \difLog_e \Upsilon_\eta (\LieBracket{\bar{A}}{\bar{B}})\\
			&\qquad+ \LieBracket{\difLog_e \Upsilon_\eta (A)}{\difLog_e \Upsilon_\eta (B)}
			\\
		&= \CoadAction_{\CoAdAction_\eta A} \pi_{G^*}(\eta, \CoAdAction_{\eta} B) - \CoadAction_{\CoAdAction_\eta B} \pi_{G^*}(\eta, \CoAdAction_{\eta} A)\\
			&\qquad + \pi_{G^*}(\eta, \CoAdAction_\eta \LieBracket{A}{B})
			+ \LieBracket{\pi_{G^*}(\eta, \CoAdAction_\eta A)}{\pi_{G^*}(\eta, \CoAdAction_\eta B)}.
	\end{split}\end{equation}
	Similarly, the twisted multiplicativity~\eqref{eq:coconjugation:matched} of \( \Upsilon^*_\eta \) implies
	\begin{equation}\begin{split}
		\difLog_g \Upsilon^*_\eta (g \ldot A)
			&= \Upsilon^*(g, \eta)^{-1} \ldot \tangent_{g} \Upsilon^*_\eta ( g \ldot A)
			\\
			&= \Upsilon^*(g, \eta)^{-1} \ldot \bigl(\tangent_\eta \Upsilon^*_g (\pi_{G^*}(\eta^{-1}, \CoAdAction_{\eta^{-1}} A) \ldot \eta) + \Upsilon^*_\eta(g) \ldot \tangent_e \Upsilon^*_\eta(A)\bigr)
			\\
			&= \AdAction_{\Upsilon^*_\eta(g)^{-1}} \pi_{G}\bigl(\Upsilon^*_\eta(g)^{-1}, \CoAdAction_{\Upsilon^*_\eta(g)^{-1}} \pi_{G^*}(\eta^{-1}, \CoAdAction_{\eta^{-1}} A)\bigr) + \CoAdAction_{\eta^{-1}} A \,,
	\end{split}\end{equation}
	where we used \( \tangent_e \bigl(g \mapsto \Upsilon_g(\eta^{-1})^{-1}\bigr)(A) = \pi_{G^*}(\eta^{-1}, \CoAdAction_{\eta^{-1}} A) \ldot \eta \) and \( \tangent_\eta \Upsilon^*_g(\mu \ldot \eta) = \tangent_e \Upsilon^*_{\Upsilon^*_\eta(g)} (\mu) = \pi_{G}(\Upsilon^*_\eta(g)^{-1}, \CoAdAction_{\Upsilon^*_\eta(g)^{-1}} \mu) \ldot \Upsilon^*_\eta(g) \).
	Taking the derivative with respect to \( g \) at the identity in the direction \( B \in \LieA{g} \) yields
	\begin{equation}\begin{split}
		\bar{B}_e \bigl(\difLog \Upsilon^*_\eta (\bar{A})\bigr)
			&= - \tangent_e \bigl(\pi_{G}(\cdot, \pi_{G^*}(\eta^{-1}, \CoAdAction_{\eta^{-1}} A))\bigr)\bigr(\tangent_e \Upsilon^*_\eta (B)\bigl)
			\\
			&= - \CoadAction_{\pi_{G^*}(\eta^{-1}, \CoAdAction_{\eta^{-1}} A)} (\CoAdAction_{\eta^{-1}} B)
	\end{split}\end{equation}
	Since \( \difLog \Upsilon^*_\eta \in \DiffFormSpace^1(G, \LieA{g}) \) satisfies the Maurer--Cartan equation for every \( \eta \in G^* \), we thus find
	\begin{equation}\label{eq:poissonLie:fromCoconjugationToPoisson:tangentCoAction}\begin{split}
		0
		&= \Bigl(\dif\difLog \Upsilon^*_\eta + \frac{1}{2} \LieBracket{\difLog \Upsilon^*_\eta}{\difLog \Upsilon^*_\eta}\Bigr)_e (A, B) \\
		&= \bar{A}_e \bigl(\difLog \Upsilon^*_\eta (\bar{B})\bigr) - \bar{B}_e \bigl(\difLog \Upsilon^*_\eta (\bar{A})\bigr) - \difLog_e \Upsilon^*_\eta (\LieBracket{\bar{A}}{\bar{B}})\\
			&\qquad+ \LieBracket{\difLog_e \Upsilon^*_\eta (A)}{\difLog_e \Upsilon^*_\eta (B)}
			\\
		&= -\CoadAction_{\pi_{G^*}(\eta^{-1}, \CoAdAction_{\eta^{-1}} B)} (\CoAdAction_{\eta^{-1}} A) + \CoadAction_{\pi_{G^*}(\eta^{-1}, \CoAdAction_{\eta^{-1}} A)} (\CoAdAction_{\eta^{-1}} B)
		    - \CoAdAction_{\eta^{-1}} \LieBracket{A}{B}
			\\
			&\qquad+ \LieBracket{\CoAdAction_{\eta^{-1}} A}{\CoAdAction_{\eta^{-1}} B}
	\end{split}\end{equation}
	Using~\eqref{eq:poissonLie:fromCoconjugationToPoisson:tangentAction} and~\eqref{eq:poissonLie:fromCoconjugationToPoisson:tangentCoAction}, we now obtain
	\begin{equation}\label{eq:poissonLie:fromCoconjugationToPoisson:poisson}\begin{split}
		\dif^{\textrm{CE}}_{\LieA{h}} \, \pi_{G^*} &(\eta, A, B) - \frac{1}{2} \LieBracket{\pi_{G^*}}{\pi_{G^*}}(\eta, A, B)\\
			&= -\CoadAction_A \pi_{G^*}(\eta, B) + \CoadAction_B \pi_{G^*}(\eta, A) - \pi_{G^*}(\eta, \LieBracket{A}{B}) \\
			&\qquad - \pi_{G^*}(\eta, \CoadAction_{\pi_{G^*}(\eta, A)} B) + \pi_{G^*}(\eta, \CoadAction_{\pi_{G^*}(\eta, B)} A)\\
			&\qquad - \LieBracket[\big]{\pi_{G^*}(\eta, A)}{\pi_{G^*}(\eta, B)}
			\\
			&= \pi_{G^*}(\eta, \CoAdAction_{\eta} \LieBracket{\CoAdAction_{\eta^{-1}} A}{\CoAdAction_{\eta^{-1}} B}) 
			 - \pi_{G^*}(\eta, \LieBracket{A}{B}) \\
			&\qquad - \pi_{G^*}(\eta, \CoadAction_{\pi_{G^*}(\eta, A)} B) + \pi_{G^*}(\eta, \CoadAction_{\pi_{G^*}(\eta, B)} A)
			\\
			&= 0.
	\end{split}\end{equation}
	This verifies the Poisson property~\eqref{eq:poissonLieDef:poisson}.
\end{proof}

Note that, by construction, the coconjugation action \( \Upsilon \) and the Poisson bivector field \( \pi_{G^*} \) are connected by
\begin{equation}
- \eta \ldot \pi_{G^*} (\eta, \CoAdAction_{\eta} A) = \tangent_e \Upsilon_\eta (A) \in \TBundle_\eta G^*, \qquad \eta \in G^*, A \in \LieA{g}.
\end{equation}
In Poisson geometry, actions satisfying this relation with respect to a Poisson structure on a Lie group are called \emphDef{dressing transformations} (see, \eg, \parencite[Definition~2.38]{Lu1990} or \parencite{FlaschkaRatiu1996}).
For questions concerning the equivariance of momentum maps, it feels however more natural to put the action first and then derive the Poisson structure from it, than the other way around.
In summary, we obtain the following hierarchy:
\begin{equationcd}[tikz={row sep=large}]
	\text{Coconjugation actions} \; (\Upsilon, \Upsilon^*)
	\arrow[d, bend left=50, "\;\text{\cref{prop:poissonLie:fromCoconjugationToPoisson}}"]
	\\
	\text{Dual pair of Poisson Lie groups} \; \kappa(G, G^*)
	\arrow[d, bend left=50, "\;\text{\parencite[Theorem~1.8]{LuWeinstein1990}}"]
	\arrow[u, bend left=50, dotted, "\text{dressing transformations}\;"]
	\\
	\text{Lie bialgebra} \; \kappa(\LieA{g}, \LieA{g}^*)
	\arrow[u, bend left=50, dotted, "\text{\parencite[Theorem~1.8]{LuWeinstein1990}}\;"]
\end{equationcd}
The dotted arrows indicate that the integration of structure is not always possible.
The integration of a Lie bialgebra to a dual pair of Poisson Lie groups, \parencite[Theorem~1.8.]{LuWeinstein1990}, requires the groups to be connected and simply connected.
Moreover, to obtain dressing transformations one needs that the dressing vector fields are complete, see \parencite[Theorem~3.12]{LuWeinstein1990}.

\begin{example}
	For the coadjoint action dual pair \( \kappa(G, \LieA{g}^*) \) discussed in \cref{ex:coconjugation:dualLieAlgebra}, the induced Poisson Lie structure on \( \LieA{g}^* \) is given by
	\begin{equation}
		\pi_{\LieA{g}^*}(\mu, A) = \CoadAction_A \mu \,.
	\end{equation}
	This is the usual Kostant--Kirillov--Souriau Poisson structure.
\end{example}
\begin{example}
	Continuing with \cref{ex:coconjugation:iwasawaDecomp}, let \( D \) be a finite-dimensional connected complex semisimple Lie group and let \( D = K B \) be its Iwasawa decomposition.
	The action~\eqref{eq:coconjugation:iwasawaDecomp:actions} of \( K \) on \( B \) induces the following Poisson Lie structure on \( B \):
	\begin{equation}
		\pi_B (b, A) = - \pr_{\LieA{b}} (\AdAction_{b^{-1}} \CoAdAction_{b^{-1}} A).
	\end{equation}
	The corresponding bivector field is given by
	\begin{equation}
		\Lambda_b (A, B) = \Im \kappa_{\LieA{d}} \bigl(\pr_{\LieA{b}} (\AdAction_b A), \pr_{\LieA{k}}(\AdAction_{b} B)\bigr),
	\end{equation}
	which coincides with the one given in \parencite[Definition~4.2]{LuRatiu1991}.
	Similarly, the action~\eqref{eq:coconjugation:iwasawaDecomp:actions} of \( B \) on \( K \) induces the Poisson Lie structure on \( K \)
	\begin{equation}
		\pi_K (k, \mu) = - \pr_{\LieA{k}} (\AdAction_{k^{-1}} \CoAdAction_{k^{-1}} \mu)
	\end{equation}
	with bivector field
	\begin{equation}
		\Lambda_k (\mu, \nu) = \Im \kappa_{\LieA{d}} \bigl(\pr_{\LieA{k}} (\AdAction_k \mu), \pr_{\LieA{b}}(\AdAction_{k} \nu)\bigr).
	\end{equation}
	This coincides with the Poisson structure of \parencite[Theorem~4.3]{LuWeinstein1990}.
\end{example}

\begin{example}
\label{ex:dualPairLieGroups:volumePreservingDiffeos:coconjugation}
Consider the dual pair of Lie groups \( \kappa(\DiffGroup_\mu(M), \csCohomology^2(M, \UGroup(1))) \) from \cref{ex::dualPairLieGroups:volumePreservingDiffeos}.
The coadjoint action of a diffeomorphism is given by pull-back.
Thus a natural coconjugation action is given by
\begin{equation}
\DiffGroup_\mu(M) \times \csCohomology^2(M, \UGroup(1)) \to 
\csCohomology^2(M, \UGroup(1)), \quad (\phi, h) \mapsto (\phi^{-1})^* h.
\end{equation}
This indeed endows \( \kappa(\DiffGroup_\mu(M), \csCohomology^2(M, \UGroup(1))) \) with the structure of a coconjugation matched pair (relative to the trivial action of \( \csCohomology^2(M, \UGroup(1)) \) on \( \DiffGroup_\mu(M) \)).
Using \cref{prop::differentialCharacter:pullback}, we see that the induced Poisson Lie structure on \( \csCohomology^2(M, \UGroup(1)) \) is defined by
\begin{equation}
\csCohomology^2(M, \UGroup(1)) \times \VectorFieldSpace_\mu(M) \to 
\DiffFormSpace^1(M) \slash \dif \DiffFormSpace^0(M), \quad 
(h, X) \mapsto \equivClass{X \contr \curv h}. 
\end{equation}
Under the integration pairing, this corresponds to the map
	\begin{equation}
		\csCohomology^2(M, \UGroup(1)) \times \VectorFieldSpace_\mu(M) \times \VectorFieldSpace_\mu(M) \to \R \quad
		(h, X, Y) \mapsto \int_M \curv h (X, Y) \, \mu. 
	\end{equation}
Note that, for fixed \( h \in \csCohomology^2(M, \UGroup(1)) \), this is precisely the Lichnerowicz cocycle on \( \VectorFieldSpace_\mu(M) \) defined by the \( 2 \)-form \( \curv h \).
In other words, the Lichnerowicz cocycle is derived from the pull-back action.
\end{example}

For the rest of the section, we consider the case where \( G^* \) is Abelian in more detail.
To emphasize the Abelian nature we write group operation in \( G^* \) as addition.
In this case, the coadjoint action of \( G^* \) on \( \LieA{g} \) is trivial so that we can choose \( \Upsilon^*: G^* \times G \to G \) to be the trivial action.
Accordingly, we refer to the \( G \)-action \( \Upsilon \) on \( G^* \) as the conjugation action.
Moreover, the matched pair condition~\eqref{eq:coconjugation:matched} simplifies to
\begin{equation}
	\Upsilon_g(\eta_1 + \eta_2) = \Upsilon_g(\eta_1) + \Upsilon_g(\eta_2),
\end{equation}
and we will call \( \Upsilon \) standard if this relation is satisfies. 

The following is a refinement of \cref{prop:coconjugation:uniquness} in the Abelian setting.
\begin{prop}
\label{prop:coconjugationUniqueness}
Let \( \Upsilon: G \times G^* \to G^* \) be a standard 
coconjugation action.
A map \( c: G \to G^* \) defines a coconjugation 
action by \( \tilde{\Upsilon} \defeq \Upsilon + c \) if and only 
if \( c \) satisfies \( c(gh) = c(g) + \Upsilon_g (c(h)) \) for 
all \( g, h \in G \). That is, 
\( \tilde{\Upsilon} \) is a coconjugation action if and only 
if \( c \) is a \( 1 \)-cocycle with respect to \( \Upsilon \).
\end{prop}
\begin{proof}
Suppose that \( c: G \to G^* \) is a map satisfying \( c(gh) = c(g) + \Upsilon_g (c(h)) \).
Define the map \( \tilde{\Upsilon}: G \times G^* \to G^* \) by \( \tilde{\Upsilon}_g (\eta) = \Upsilon_g (\eta) + c(g) \).
It is clear that \( \tilde{\Upsilon} \) still integrates the coadjoint action.
Thus it is left to show that \( \tilde{\Upsilon} \) is a left action indeed.
This follows from the following straightforward calculation:
\begin{equation}\begin{split}
	\left( \tilde{\Upsilon}_g \circ \tilde{\Upsilon}_h\right) (\eta)
		&= \Upsilon_g \bigl(\Upsilon_h (\eta) + c(h)\bigr) + c(g) \\
		&=  \Upsilon_g \bigl(\Upsilon_h (\eta)\bigr) + \Upsilon_g \bigl(c(h)\bigr) + c(g) \\
		&=  \Upsilon_{gh} (\eta) + c(gh) \\
		&= \tilde{\Upsilon}_{gh}(\eta).
\end{split}\end{equation}
Reading the argument backwards establishes the converse claim in the 
\enquote{if and only if} statement.
\end{proof}

\begin{example}
The coadjoint representation is a standard coconjugation action of \( G \) on \( G^* \defeq \LieA{g}^* \). Moreover, since \( \LieA{g}^* \) is connected, the previous proposition in combination with \cref{prop:coconjugation:uniquness} establishes a bijection between coconjugation actions on \( \LieA{g}^* \) and \( 1 \)-cocycles \( c: G \to \LieA{g}^* \). The coconjugation action corresponding to a \(1\)-cocycle \( c \) is the affine action \( (g, \mu) \mapsto \CoAdAction_{g^{-1}} \mu + c(g) \), which plays an important role for classical non-equivariant momentum maps (see \parencite[Definiton~4.5.23]{OrtegaRatiu2003}).
\end{example}

According to \cref{prop:poissonLie:fromCoconjugationToPoisson}, every standard coconjugation action yields an associated Lie Poisson structure \( \pi_{G^*} \) on \( G^* \) (the induced Lie Poisson structure \( \pi_G \) on \( G \) vanishes, because \( \Upsilon^* \) is trivial).
For the Poisson structure induced on \( G^* \) we have the following affine version.
\begin{prop}
	\label{prop:poissonLie:fromCoconjugationToAffinePoisson}
	Let \( \kappa(G, G^*) \) be a dual pair of Lie groups with Abelian \( G^* \).
	For every standard coconjugation action \( \Upsilon \) and for every cocycle \( c: G \to G^* \), let \( \tilde{\Upsilon} = \Upsilon + c \) be the associated coconjugation action and define \( \tilde{\pi}_{G^*}: G^* \times \LieA{g} \to \LieA{g}^* \) by 
	\begin{equation}
		\tilde{\pi}_{G^*}(\eta, A) \defeq - \eta^{-1} \ldot \tangent_e \tilde{\Upsilon}_\eta (A).
	\end{equation}
	If \( \pi_{G^*} \) is skew-symmetric in the sense of~\eqref{eq:poissonLie:inLeftTrivDualized:skewsymmetric}, then it defines an affine Poisson structures on \( G^* \).
\end{prop}
\begin{proof}
	Since we have \( \tilde{\Upsilon}_{\zeta + \eta} = \tilde{\Upsilon}_{\zeta} + \tilde{\Upsilon}_{\eta} - c \), a similar calculation as in the proof of \cref{prop:poissonLie:fromCoconjugationToPoisson} shows that
	\begin{equation}\begin{split}
	\tilde{\pi}_{G^*}(\zeta + \eta, A) 
	&=\tilde{\pi}_{G^*}(\zeta, A) + \tilde{\pi}_{G^*}(\eta, A) 
	- \tangent_e c (A)  \\
	&=\tilde{\pi}_{G^*}(\zeta, A) + \tilde{\pi}_{G^*}(\eta, A) 
	- \tilde{\pi}_{G^*}(0, A).
	\end{split}\end{equation}
	Thus \( \tilde{\pi}_{G^*} \) is an affine bivector field.
	The Poisson property follows as in \cref{prop:poissonLie:fromCoconjugationToPoisson}.
\end{proof}

For a given coconjugation action \( \Upsilon \) of \( G \) on \( G^* \), we say that the group-valued momentum map is \emphDef{equivariant} if it is \( G \)-equivariant as a map \( J: M \to G^* \).
It is well-known that a classical momentum map is equivariant if and only if it is a Poisson map.
This relation carries over to our more general setting.

As above, we first have to reformulate what it means for \( J \) to be a Poisson map in terms of \( \pi_{G^*} \).
For a smooth function \( f:G^* \to \R \), the left derivative \( \tangentLeft_\eta f: \LieA{g}^* \to \R \) at \( \eta \in G^* \) is an element of the double dual \( \LieA{g}^{**} \).
In finite dimensions, we can view \( \tangentLeft_\eta f \) thus as an element of \( \LieA{g} \).
The Poisson bracket of \( f, g \in \sFunctionSpace(G^*) \) is given by \( \poisson{f}{g}_{G^*}(\eta) = \kappa\bigl(\tangentLeft_\eta f, \pi_{G^*}(\eta, \tangentLeft_\eta g)\bigr) \) according to~\eqref{eq:poissonLie:poissonBracketFromTensor}.
Given a group-valued momentum map \( J: M \to G^* \), we calculate for \( X_m \in \TBundle_m M \)
\begin{equation}\begin{split}
 	\tangent_m (f \circ J)(X) 
 		&= \tangent_{J(m)} f \bigl(J(m) \ldot \difLog_m J(X)\bigr)
 		= \kappa\bigl(\tangentLeft_{J(m)} f, \difLog_m J(X)\bigr)
 		\\
 		&= - \omega_m \bigl((\tangentLeft_{J(m)} f) \ldot m, X\bigr)
\end{split}\end{equation} 
by~\eqref{eq::momentumMap:DefEq}.
Hence, we obtain (see~\eqref{eq::luMomentumMap:defining})
\begin{equation}\begin{split}
	\poisson{f \circ J}{g \circ J}_M(m)
		&= \varpi_m \bigl(\tangent_m (f \circ J), \tangent_m (g \circ J)\bigr)
		\\
		&= \omega_m \bigl((\tangentLeft_{J(m)} f) \ldot m, (\tangentLeft_{J(m)}g) \ldot m\bigr).
\end{split}\end{equation}
Thus, in summary, \( J \) is a Poisson map, \ie, \( \poisson{f \circ J}{g \circ J}_M = \poisson{f}{g}_{G^*} \circ J \), if and only if 
\begin{equation}\label{eq:momentumMap:Poisson}
	\omega_m (A \ldot m, B \ldot m) = \kappa\bigl(A, \pi_{G^*}(J(m), B)\bigr)
\end{equation}
holds for all \( A, B \in \LieA{g} \) and \( m \in M \).
This equation no longer relies on reflexivity to make sense and so we adopt it as the definition for \( J \) to be a Poisson map in infinite dimensions.
Recall that the left-hand side of~\eqref{eq:momentumMap:Poisson} defines the so called \emphDef{non-equivariance} cocycle \( \sigma_m (A, B) = \omega_m (A \ldot m, B \ldot m) \). 
Thus, \( J \) is a Poisson map if and only if the Lie algebra cocycles \( \sigma_m \) and \( \kappa\bigl(\cdot, \pi_{G^*} (J(m), \cdot)\bigr) \) coincide for all \( m \in M \).

\begin{prop}
	Let \( \kappa(G, G^*) \) be a dual pair of Lie groups where \( G^* \) is Abelian.
	Assume that \(  G \) acts on the symplectic manifold \( (M, \omega) \) with a group-valued momentum map \( J: M \to G^*\).
	If \( J \) is equivariant with respect to a given coconjugation action \( \tilde{\Upsilon} = \Upsilon + c \), then it is a Poisson map relative to the induced Poisson structure \( \tilde{\pi}_{G^*} \) of \cref{prop:poissonLie:fromCoconjugationToAffinePoisson}.
\end{prop}
\begin{proof}
	If \( J: M \to G^* \) is equivariant, then \( J(g \cdot m) = \tilde{\Upsilon}_g J(m) \) relative to the coconjugation action \( \tilde{\Upsilon}: G \times G^* \to G^* \).
	Thus,
	\begin{equation}
		\difLog_m J (B \ldot m)
		= J(m)^{-1} \ldot \tangent_e \tilde{\Upsilon}_{J(m)} (B)
		= - \tilde{\pi}_{G^*}\bigl(J(m), B \bigr).
	\end{equation}
	On the other hand, \( \kappa\bigl(A, \difLog_m J (B \ldot m)\bigr) = - \omega_m (A \ldot m, B \ldot m) \) by the definition of the momentum map, and the claim follows.
\end{proof}

\subsection{Momentum maps for group extensions}
\label{sec::groupMomentumMap:forExtension}

Consider the following set-up. Let \( H, K, G \) be Lie groups which fit in the exact sequence
\begin{equationcd}
\label{eq::momentumMap:forExtension:extension}
e \to[r] 	& H \to[r, "\iota"] 	& K \to[r, "\pi"] 	& G \to[r] 	& e,
\end{equationcd}
where \( \iota \) and \( \pi \) are Lie group morphisms. 
Suppose that \( K \) acts on a symplectic manifold \( (M, \omega) \) such that the diffeomorphism defined by each group element is symplectic. 
We are interested in the momentum map for this action.
Starting from the extension sequence~\eqref{eq::momentumMap:forExtension:extension}, it is natural to seek an expression of the momentum map for the \( K \)-action in terms of the momentum maps for the groups \( H \) and \( G \), assuming they exist. 
Similar questions occur in the context of symplectic reduction by stages (see \parencite{MarsdenMisiolekEtAl2007}).

The corresponding sequence on the infinitesimal level
\begin{equationcd}[label=eq::groupMomentumMap:forExtensionInfinitisimal]
0 \to[r] 	& \LieA{h} \to[r, "\iota"] 	& \LieA{k} \to[r, "\pi"] 	
& \LieA{g} \to[r] 	& 0
\end{equationcd}
always splits as vector spaces but not necessarily as Lie algebras. Let us fix a splitting \( \sigma: \LieA{g} \to \LieA{k} \) in the category of locally convex vector spaces and write $\LieA{h}\oplus_\sigma \LieA{g} = \LieA{k}$ for the corresponding direct sum. 
Thus, every \( A \in \LieA{k} \) can be uniquely written as the sum \( A = \iota(A_H) + \sigma(A_G) \) with \( A_H \in \LieA{h} \) and \( A_G = \pi(A) \in \LieA{g} \).

\begin{prop}
\label{prop::groupMomentumMap:forExtension}
Let \( \kappa(G, G^*) \) be a dual pair of Lie groups and let \( \dualPair{\LieA{h}}{\LieA{h}} \) be a pairing of \( \LieA{h} \) with itself.
Assume that the induced action of \( H \) on \( M \) has a standard momentum map \( J_H: M \to \LieA{h} \) with respect to the pairing \(\dualPair{\LieA{h}}{\LieA{h}} \) and that there exists a map \( J_\sigma: M \to G^*\) satisfying
\begin{equation}
\label{eq:groupMomentumMap:forExtension:momentumMapSplitting}
\sigma(A_G)^* \contr \omega + \kappa(A_G, \difLog J_\sigma) = 0, 
\qquad A_G \in \LieA{g}.
\end{equation}		
Then, \( J_K = (J_H, J_\sigma): M \to \LieA{h} \times G^*\) is a group-valued momentum map for the \( K \)-action with respect to the pairing 
\( \scalarProdR{\LieA{h} \oplus_\sigma \LieA{g}}{\LieA{h} \oplus 
\LieA{g}^*} 
= \dualPair{\LieA{h}}{\LieA{h}} + \kappa(\LieA{g}, \LieA{g}^*) \).
\end{prop}

\begin{proof}
The momentum map equation
\begin{equation}
A^* \contr \omega + \scalarProdR{A}{\difLog J_K} = 0, \qquad A \in \LieA{k}
\end{equation}	
is linear in \( A = \iota(A_H) + \sigma(A_G) \) and thus decomposes in
	\begin{equation}
		A_H^* \contr \omega + \dualPair{A_H}{\dif J_H} = 0
	\end{equation}
and in 
	\begin{equation}
		\sigma(A_G)^* \contr \omega + \kappa(A_G, \difLog J_\sigma) = 0.
	\end{equation}
Both equations hold by assumption.
\end{proof}

Formally,~\eqref{eq:groupMomentumMap:forExtension:momentumMapSplitting} looks like a momentum map relation.
However, we do not assume that \( \sigma \) is a splitting on the level of Lie algebras and, hence \( G \), or its Lie algebra \( \LieA{g} \), does not act on \( M \) via \( \sigma \). 
However, if \( G \) happens to act on \( M \) through a different splitting \( \chi: G \to K \) which is a group section of $\pi$
(see~\eqref{eq::momentumMap:forExtension:extension}), then \( J_\sigma \) is the momentum map up to some twisting by \( J_H \).

\begin{prop}\label{prop::groupMomentumMap:forLift}
	In the setting above, let \( \chi: \LieA{g} \to \LieA{k} \) be a Lie algebra homomorphism splitting the exact sequence~\eqref{eq::groupMomentumMap:forExtensionInfinitisimal} and hence there is an induced (infinitesimal) \( \LieA{g} \)-action on \( M \).
	Define \( \tau_{\sigma\chi} \defeq \chi - \sigma: \LieA{g} \to \LieA{h} \) and denote  its dual map with respect to the chosen pairings by \( \tau_{\sigma\chi}^*: \LieA{h}^* \to \LieA{g}^\ast \).
	Moreover, assume that the dual group \(G^*\) of \( G \) is Abelian.
	Then, 
	\begin{equation}
		J_\chi: M \to G^*, \qquad m \mapsto J_\sigma(m) \cdot \exp( \tau_{\sigma\chi}^* J_H (m) )
	\end{equation}
	is a group-valued momentum map for the \( G \)-action on \( M \).
	Moreover, \( J_\chi \) does not depend on the splitting \( \sigma \).
\end{prop}
\begin{proof}
For \( B \in \LieA{g} \), the infinitesimal generator vector field is \( B^* = (\chi(B))^* = \sigma(B)^* + \tau_{\sigma\chi}(B)^* \). Since \(G^*\) is by assumption Abelian, we have the ordinary product rule \( \difLog J_\chi = \difLog J_\sigma + \tau_{\sigma\chi}^* \dif J_H \). Calculating, using the defining equations for \( J_H \) and \( J_\sigma \), yields
	\begin{equation}\begin{split}
		B^* \contr \omega &+ \kappa(B, \difLog J_\chi)
			\\
			&= \sigma(B)^* \contr \omega + \kappa(B, \difLog J_\sigma)
			 + \tau_{\sigma\chi}(B)^* \contr \omega  + \kappa(B, \tau_{\sigma\chi}^* \dif J_H) \\
			&= \sigma(B)^* \contr \omega + \kappa(B, \difLog J_\sigma)
			 - \dualPair{\tau_{\sigma\chi}(B)}{\dif J_H} + \kappa(B, \tau_{\sigma\chi}^* \dif J_H) \\
			&= 0,
	\end{split}\end{equation}
	which shows that \( J_\chi \) is a  group-valued momentum map.

It is left  to show that the momentum map \( J_\chi \) does not depend on the splitting that  was used in the construction.
	Let \( \sigma_1 \) and \( \sigma_2 \) be different splittings.
	First, we start with the observation that if \( J_{\sigma_1} \) satisfies~\eqref{eq:groupMomentumMap:forExtension:momentumMapSplitting} for \( \sigma_1 \), then the map defined by \( J_{\sigma_2} \defeq J_{\sigma_1} \cdot \exp ( - (\sigma_1 - \sigma_2)^* J_H) \) satisfies the same relation with \( \sigma \) replaced by \( \sigma_2 \).
	Indeed, since \( G^* \) is assumed to be Abelian, the chain rule for the logarithmic derivative simplifies and is not twisted by the adjoint representation.
	We thus obtain 
	\begin{equation}
	 	\difLog J_{\sigma_2} 
	 		= \difLog J_{\sigma_1} - (\sigma_1 - \sigma_2)^* \dif J_H .
	\end{equation} 
Using~\eqref{eq:groupMomentumMap:forExtension:momentumMapSplitting} and the momentum relation for \( J_H \), we get for all \( A_G \in \LieA{g} \)
\begin{equation}\begin{split}
	\kappa(A_G, \difLog J_{\sigma_2})
		&= \kappa(A_G, \difLog J_{\sigma_1}) - \kappa(A_G, (\sigma_1 - \sigma_2)^* \dif J_H)
		\\
		&= \kappa(A_G, \difLog J_{\sigma_1}) - \dualPair{(\sigma_1 - \sigma_2) A_G}{\dif J_H}
		\\
		&= - (\sigma_1 A_G)^* \contr \omega + ((\sigma_1 - \sigma_2) A_G)^* \contr \omega
		\\
		&= - (\sigma_2 A_G)^* \contr \omega.
\end{split}\end{equation}
On the other hand, we clearly have \( \tau_{\sigma_2 \chi}^* = \tau_{\sigma_1 \chi}^* + (\sigma_1 - \sigma_2)^* \).
Hence, in summary, 
\begin{equation}\begin{split}
	J_{\sigma_2} \cdot \exp(\tau_{\sigma_2 \chi}^* J_H)
	 	&= J_{\sigma_1} \cdot \exp ( - (\sigma_1 - \sigma_2)^* J_H) \exp (\tau_{\sigma_1 \chi}^* J_H) \exp ((\sigma_1 - \sigma_2)^* J_H)
	 	\\
	 	&= J_{\sigma_1} \cdot \exp(\tau_{\sigma_1 \chi}^* J_H)
\end{split}\end{equation}
and thus \( J_\chi \) is independent of the chosen splitting \( \sigma \).
\end{proof}

In a similar spirit, we can determine the momentum map for the action of a subgroup.
\begin{prop}\label{prop::momentumMap:forSubgroups}
Let \( G \) act on the symplectic manifold \( (M, \omega) \) with 
group-valued momentum map \( J: M \to G^* \). Let \( \iota: H \to G \) be 
a Lie group homomorphism, so that \( H \) acts through \( G \) on \( M \). 
Fix a dual group \( H^* \) of \( H \).
Assume\footnotemark{} that there is a Lie group homomorphism \( \rho: G^* \to H^* \) whose associated Lie algebra homomorphism \( \rho: \LieA{g}^* \to \LieA{h}^*\) 
is the dual of \( \iota: \LieA{h} \to \LieA{g} \) with respect to 
\( \kappa(\LieA{g}, \LieA{g}^*) \) and \( \dualPair{\LieA{h}}
{\LieA{h}^*} \). Then \( J^H \defeq \rho \circ J: M \to H^* \) is 
a group-valued momentum map for the induced \( H \)-action.
\footnotetext{Note that these assumptions are automatically satisfied in 
finite dimensions if \( G \) is \( 1 \)-connected, \ie, connected and
$\pi_1(G, e) = 1$. In infinite dimensions, however, the adjoint 
\( \rho: \LieA{g}^* \to \LieA{h}^*\) of the linear map \( \iota \) does 
not need to exist, and even if it exists, it does not necessarily 
integrate to a Lie group homomorphism (for this, we would need some 
regularity assumptions on the dual group \( H^*\), \cf 
\parencite[Theorem~III.1.5]{Neeb2006}).}
\end{prop}
\begin{proof}
	Let \( B \in \LieA{h} \). Then using the fact that taking the 
logarithmic derivative commutes with Lie group homomorphisms 
\parencite[Proposition~II.4.1.1]{Neeb2006}, we have
\begin{equation}
\iota(B)^* \contr \omega = - \kappa(\iota(B), \difLog J) = 
- \dualPair{B}{\rho \circ \difLog J} = 
- \dualPair{B}{\difLog (\rho \circ J)},
\end{equation}
which proves the proposition.
\end{proof}

\section{Global Analysis of Symplectic Fiber Bundles}
\label{sec::globalAnalysisSymplecticFibreBundles}

\subsection[Symplectic form]{Symplectic form on the space of sections}
\label{sec::symplecticFormSectionSpace}

The prime example of an infinite-dimensional manifold is the space of 
sections of a fiber bundle. So, for the study of infinite-dimensional 
symplectic geometry, it is suggestive to search for a way to construct a 
symplectic form on the space of sections using some geometric structure on 
the fiber bundle. As we will see, a symplectic form along the fibers 
is what we are looking for.

Let \( \pi: P \to M \) be a finite-dimensional right principal \( G \)-bundle.
Given a left \( G \)-manifold \(\FibreBundleModel{F}\), where the action is denoted by \( \FibreBundleModel{l}: G \times
\FibreBundleModel{F} \to \FibreBundleModel{F} \), form the associated fiber bundle \( F = P \times_G \FibreBundleModel{F} \to M \) (see \cref{sec::conventions} for definitions and conventions).
If, in addition, \( \FibreBundleModel{F} \) carries a \( G \)-invariant symplectic form \( \FibreBundleModel{\omega} \), then there exists a unique symplectic structure \( \omega_m \) on each fiber \( F_m \) such that the map 
\begin{equation}
	\label{eq::associatedFibreBundle:fibreInjection}
	\iota_p: \FibreBundleModel{F} \to F_m, \qquad \FibreBundleModel{f} 
	\mapsto \equivClass{p, \FibreBundleModel{f}}
\end{equation}
is a symplectomorphism for all \( p \in P_m \).
The fiber bundle \( F \) with the induced fiberwise symplectic structure \( \omega \) is called a \emphDef{symplectic fiber bundle} and we denote it by \( (F, \omega) \). 
\textit{We emphasize that \( \omega \) is not a \( 2 \)-form on \( F \) but is merely a notation for the fiberwise symplectic structure.}

A \( 2 \)-form \( \hat{\omega} \) on \( F \) is called an \emphDef{extension} of the fiberwise symplectic structure if its pull-back to each fiber \( F_m \) coincides with \( \omega_m \). 
Such extensions always exist (for example, choose a connection in \( F \) and construct the extension \( \hat{\omega} \) by letting it vanish on horizontal vectors).
In the theory of symplectic fiber bundles one usually tries to construct extensions that are again symplectic forms \parencite{Weinstein1980}.
We emphasize that we do not require the extension \( \hat{\omega} \) to be non-degenerate. 
An extension \( \hat{\omega} \) is called \emphDef{compatible with a given connection} on \( P \) if the vertical and horizontal subbundles of the associated connection on \( F \) are orthogonal with respect to \( \hat{\omega} \).

In what follows, we assume that the base manifold \( M \) is closed 
(\ie, compact without boundary) so that the space of sections 
\( \SectionSpaceAbb{F} = \sSectionSpace(F) \) is a Fr\'echet manifold.
The tangent space to \( \SectionSpaceAbb{F} \) at a section \( \phi \) 
consists of vertical vector fields along \( \phi \), \ie, sections of 
\( \phi^* \VBundle F \), where \( \VBundle F \subset \TBundle F \) is 
the vector subbundle of vertical vectors. Fix a volume form \( \mu \) 
on \( M \) and choose an extension \( \hat{\omega} \) of the fiberwise 
symplectic structure \( \omega \) on \( F \). Using the evaluation map 
\( \ev: M \times \SectionSpaceAbb{F} \to F \) and the projection 
\( \pr_M: M \times \SectionSpaceAbb{F} \to M \), we define a 
\( 2 \)-form \( \Omega \) on \( \SectionSpaceAbb{F} \) by
\begin{equation}
\Omega = \mu \hatProduct \omega \equiv \intFibre_M \pr_M^* \mu \wedge \ev^* 
\hat{\omega},
\end{equation}
which uses the hat product introduced by \textcite{Vizman2011}, here 
generalized to sections of fiber bundles. Details about the resulting hat 
calculus are collected in \cref{sec::hatProductFibreBundles}.
The form \( \Omega \) depends only on the fiber symplectic form 
\( \omega \) and not on the extension \( \hat{\omega} \), since only 
the vertical part of the tangent vector plays a role; see 
\cref{rem::hatProduct:WithVolumeForm} (which also justifies the notation 
\( \mu \hatProduct \omega \) instead of \( \mu\hatProduct\hat{\omega} \)).
In fact, by~\eqref{eq::hatProduct:VolumeForm}, we have the explicit description
\begin{equation}
	\label{eq:symplecticFormSectionSpace:symplecticForm:explicit}
\Omega_\phi(Y_1, Y_2) = \int_M \omega_{\phi}(Y_1, Y_2) \, \mu, 
\qquad Y_1, Y_2 \in \sSectionSpace(\phi^* \VBundle F).
\end{equation}
This identity also shows that \( \Omega \) is non-degenerate, because 
\( \omega \) is fiberwise non-degenerate, by assumption. 
The exterior differential of \( \Omega \) is given by 
\( \dif \Omega = (-1)^{\dim M} \mu\hatProduct \dif_\pi \omega  \), 
where \( \dif_\pi \) denotes the 
vertical differential along the fibers; 
see~\eqref{eq::hatProduct:DifferentialWithVolumeForm}.
Thus \( \Omega \) is closed.
Let us record this observation.
\begin{prop}
\label{prop::symplecticStructureOnSectionsOfSymplecticFibreBundle}
Let \( M \) be a closed manifold endowed with volume form \( \mu \) and 
let \( (F, \omega) \) be a symplectic fiber bundle over \( M \).
Then \( \Omega = \mu \hatProduct \omega \) is a symplectic structure on 
the space of smooth sections \( \SectionSpaceAbb{F} \) of \( F \).
\end{prop}
Note that~\eqref{eq:symplecticFormSectionSpace:symplecticForm:explicit} 
shows that \( \Omega \) coincides with the symplectic form considered in a 
slightly more restricted setting by \textcite[Equation~7]{Donaldson2003}.

\subsection[Action of the automorphism group]{Action of $ \AutGroup(P) $}
\label{sec:momentum_map_for_automorphism_group}

Every bundle automorphism \( \psi \) of \( F \) induces a diffeomorphism 
\( \check{\psi} \) on the base \( M \).
Consider the natural smooth left action of \( \AutGroup(F) \) on the section space \( \SectionSpaceAbb{F} \) defined by
\begin{equation}
\Upsilon: \AutGroup(F) \times \SectionSpaceAbb{F} \to \SectionSpaceAbb{F}, 
\qquad \psi \cdot \phi = \psi \circ \phi \circ \check{\psi}^{-1}.
\end{equation}
Let 
\begin{equation}
	\AutGroup_{\mu, \omega}(F) = \set{\psi \in \AutGroup(F) \given \check{\psi} \in \DiffGroup_\mu(M), \; \psi^* \omega = \omega}
\end{equation}
be the subgroup of automorphisms of \( F \) preserving the fiber 
symplectic structure and inducing volume-preserving diffeomorphisms 
on the base.
\Cref{eq::hatProduct:TransfUnderAutomorphisms} implies that the induced 
action of \( \AutGroup_{\mu, \omega}(F) \) on \( \SectionSpaceAbb{F} \)  leaves the symplectic form \( \Omega = \mu \hatProduct \omega \) invariant. 

Although it is thus perfectly natural to seek a momentum map for the \( \AutGroup_{\mu, \omega}(F) \) action on \( \SectionSpaceAbb{F} \), one is confronted with serious difficulties in this quest.
We will further comment on these matters below in \cref{rem::extendMomentumMapOfGaugeGroupToSymplectomorphisms} and for now concentrate on bundle automorphisms of \( F = P \times_G \FibreBundleModel{F} \) coming from the principal bundle.
Every principal bundle automorphism \( \varphi \) of \( P \) defines a bundle automorphism \( \varphi_F \) of \( F \) by \( \varphi_F = \varphi \times_G \id_{\FibreBundleModel{F}} \). 
Hence, by definition, \( \varphi_F \circ \iota_p = \iota_{\varphi(p)} \) for every \( p \in P \), where \( \iota_p: \FibreBundleModel{F} \to F \) was introduced in~\eqref{eq::associatedFibreBundle:fibreInjection}.
This relation implies \( \iota^*_p \varphi_F^* \omega_{\check{\varphi}(m)} = \iota^*_{\varphi(p)} \omega_{\check{\varphi}(m)} = \FibreBundleModel{\omega} \), or equivalently \( \varphi_F^* \omega_{\check{\varphi}(m)} = \omega_m \) for \( m = \pi(p) \).
Thus, \( \varphi_F \) leaves the fiber symplectic structure invariant and we get a group homomorphism \( \AutGroup_\mu(P) \to \AutGroup_{\mu, \omega}(F) \), where
\begin{equation}
	\AutGroup_{\mu}(P) = \set{\varphi \in \AutGroup(P) \given \check{\varphi} \in \DiffGroup_\mu(M)}
\end{equation}
is the subgroup of automorphisms of \( P \) that induce volume-preserving diffeomorphisms on the base.
Therefore, we seek a momentum map for the \( \AutGroup_\mu(P) \) action on \( \SectionSpaceAbb{F} \).
We will divide this problem into its vertical and horizontal parts using \cref{prop::groupMomentumMap:forExtension}.

Assigning to a principal bundle automorphism \( \psi \) of \( P \) the 
induced diffeomorphism \( \check{\psi} \) on the base \( M \), yields a 
group homomorphism \( \AutGroup_\mu(P) \to \DiffGroup_\mu(M) \), whose 
kernel is the gauge group \( \GauGroup(P) \). Thus we obtain a short exact 
sequence of Lie groups 
\begin{equation}
\label{eq::hatProduct:AutGroupExactSequence}
\id\to\GauGroup(P) \to \AutGroup_\mu(P) \to \DiffGroup_{\mu, P}(M) \to \id,
\end{equation}
where \( \DiffGroup_{\mu, P}(M) \) denotes the subgroup of 
volume-preserving diffeomorphisms which is the range of the projection 
given by the third arrow in~\eqref{eq::hatProduct:AutGroupExactSequence}.
Note that \( \DiffGroup_{\mu, P}(M) \) consists of those diffeomorphisms 
that preserve the equivalence class of \( P \) under pull-back, and thus 
it is an open subgroup of \( \DiffGroup_{\mu}(M) \).

So we find ourselves in the setting discussed in 
\cref{sec::groupMomentumMap:forExtension}. 
\Cref{prop::groupMomentumMap:forExtension} shows that the momentum map for 
\( \AutGroup_\mu(P) \) is obtained from the momentum map for the gauge 
group as well as the `partial' momentum map for \(\DiffGroup_{\mu,P}(M)\). 
Before we determine these two ingredients, let us record the common 
starting point. By~\eqref{eq::hatProduct:InfActionSectionSpace}, the infinitesimal generator vector field on 
\( \SectionSpaceAbb{F} \) induced by \( Y \in \AutAlgebra_\mu(P) \) is given by
\begin{equation}
	Y^*_\phi = Y_F \circ \phi - \tangent \phi (Y_M) \in \sSectionSpace(\phi^* \VBundle F),
\end{equation}
where \( Y_M = \tangent \pi (Y) \) denotes the induced vector field on 
\( M \) and where \( Y_F \) is the vector field on \( F \) induced by the natural map \( \AutAlgebra_\mu(P) \to \AutAlgebra_{\mu, \omega}(F) \).
Using~\eqref{eq::hatProduct:ContractionWithKillingVectorField}, we obtain
\begin{equation}
\label{eq::autAction:contractionWithKillingVectorField}
Y^* \contr \Omega = (Y_M \contr \mu) \hatProduct \hat{\omega} 
+ (-1)^{\dim M} \mu \hatProduct (Y_F \contr \hat{\omega}), 
\qquad Y \in \AutAlgebra_\mu(P).
\end{equation}
Hence, \( Y^* \contr \Omega \) decomposes into  vertical and horizontal terms, so that we can determine the momentum map for the gauge group separately from the momentum map for the diffeomorphism group.

\subsubsection{Momentum map for the gauge group}
\label{sec:momentum_map_for_the_gauge_group}
Since for every \( Y \in \GauAlgebra(P) \) the induced vector field 
\( Y_M \) on the base vanishes,
\eqref{eq::autAction:contractionWithKillingVectorField} simplifies to
\begin{equation}
Y^*\contr\Omega + (-1)^{\dim M+1}\mu \hatProduct (Y_F \contr \hat{\omega}) 
= 0.
\end{equation}
To study the second term \( (-1)^{\dim M} \mu \hatProduct (Y_F \contr \hat{\omega}) \in \DiffFormSpace^1(\SectionSpaceAbb{F}) \), we evaluate it on \( Z \in \TBundle_\phi \SectionSpaceAbb{F} \) and, using~\eqref{eq::fibreIntegration:defWithPartialPullback}, obtain
\begin{equation}
\label{eq::diffAction:momentumMapCalculationSecondTerm}
(-1)^{\dim M} \bigl(\mu \hatProduct (Y_F \contr \hat{\omega})\bigr)_\phi (Z) = \int_M \hat{\omega}_{\phi(m)}\bigl(Y_F \circ \phi(m), Z(m)\bigr) \, \mu.
\end{equation}
In order to shift the calculation from section spaces to the fiber 
model \( \FibreBundleModel{F} \), we introduce the smooth bundle map\footnote{Morally, 
\( l = \id_P \times_G \FibreBundleModel{l} \).} 
over \( M \)
\begin{equation}
	l: \ConjBundle P \times_M F \to F, \qquad (\equivClass{p, g}, \equivClass{p, \FibreBundleModel{f}})
		\mapsto \equivClass{p, g \cdot \FibreBundleModel{f}},
\end{equation} 
where $\ConjBundle P = P \times_G G \to M$ is the conjugation bundle associated to the principal \( G \)-bundle $P \to M$ (see \cref{sec::conventions}).
The fiber derivative of \( l \) is given by 
\begin{equation}
	\label{eq:symplecticFibreBundle:fibreDerivativeAction}
	\difFibre_1 l: \AdBundle P \times_M F \to \VBundle F = P \times_G \TBundle \FibreBundleModel{F},
	\qquad 
	\left(\equivClass{p, A}, \equivClass{p, \FibreBundleModel{f}}\right) \mapsto \equivClass{p, A^*_{\FibreBundleModel{f}}},
\end{equation}
where $\AdBundle P \to M$ is the adjoint bundle of the principal \( G \)-bundle $P \to M$.
Let \( \VectorFieldSpace^v(F) \) denote the set of vertical vector fields on the associated fiber bundle $F$.
Note that \( l \) induces a natural map \( l_*: \GauGroup(P) \to \GauGroup(F) \) on sections and that the map \( \GauAlgebra(P) \to \VectorFieldSpace^v(F) \) sending a vertical vector field \( Y \) on \( P \) to its partner \( Y_F \) on \( F \) is identified with \( (\difFibre_1 l)_*: \sSectionSpace(\AdBundle P) \to \VectorFieldSpace^v(F) \). 
In other words, \( Y_F = (\difFibre_1 l)_* Y \).

To calculate the contraction \( Y_F \contr \omega = (\difFibre_1 l)_* Y \contr \omega \) of the 
vertical vector field \( Y_F \) with the fiber symplectic form \( \omega \) 
we need a \( G \)-equivariant momentum map \( \FibreBundleModel{J}: 
\FibreBundleModel{F} \to \LieA{g}^* \) for 
the symplectic \( G \)-action on \( (\FibreBundleModel{F}, 
\FibreBundleModel{\omega}) \).
By \( G \)-equivariance, \( \FibreBundleModel{J} \) yields a bundle map \( J: F \to \CoAdBundle P \), where \( \CoAdBundle P = P \times_G \LieA{g}^* \) is the coadjoint bundle. 
Taking the derivative in the fiber direction yields a bundle map \( \difFibre J: \VBundle F \to \CoAdBundle P \).
Note that the derivative of \( J_*: \SectionSpaceAbb{F} \to \sSectionSpace(\CoAdBundle P) \) at \( \phi \in \SectionSpaceAbb{F} \) equals \( (\difFibre J)_* \circ \phi \).
Let \( l'A \) be the infinitesimal generator vector field for the \( G \)-action \( l:G \times \FibreBundleModel{F} \to \FibreBundleModel{F} \) associated to \( A \in \LieA{g} \).
Unravelling the definitions shows that the momentum map relation \( (\FibreBundleModel{l}' A) \contr \FibreBundleModel{\omega} + \dualPair{A}{\dif \FibreBundleModel{J}} = 0 \) on \( \FibreBundleModel{F} \) translates in the bundle picture to 
\begin{equation}\label{eq:symplecticFibreBundle:momentumMapBundle}
 	\bigl((\difFibre_1 l)_* Y\bigr) \contr \omega + \dualPair{Y}{\difFibre J} = 0
\end{equation} 
on vertical vectors.
Indeed, let \( \equivClass{p, \FibreBundleModel{u}} \in P \times_G 
\TBundle \FibreBundleModel{F} \) be a vertical vector 
at the point \( \equivClass{p, \FibreBundleModel{f}} \in F \) over 
\( m \in M \) and represent \( Y(m) \in \AdBundle P \) by the 
pair \( \equivClass{p, A} \) with \( A \in \LieA{g} \).
Then, using~\eqref{eq:symplecticFibreBundle:fibreDerivativeAction}, we obtain
\begin{equation}\begin{split}
	\omega_{m} \bigl(\difFibre_1 l \,(Y (m)), \equivClass{p, \FibreBundleModel{u}}\bigr)
		&= \omega_{m} \bigl(\equivClass{p, A^*_{\FibreBundleModel{f}}}, \equivClass{p, \FibreBundleModel{u}}\bigr)
		\\
		&= \FibreBundleModel{\omega}_{\FibreBundleModel{f}} 
		(A^*_{\FibreBundleModel{f}}, \FibreBundleModel{u})
		\\
		&= -\dualPair[\big]{A}{\tangent_{\FibreBundleModel{f}} \FibreBundleModel{J} (\FibreBundleModel{u})}.
\end{split}\end{equation}
On the other hand, we have
\begin{equation}\begin{split}
	\dualPair[\big]{Y(m)}{\difFibre J (\equivClass{p, \FibreBundleModel{u}})}
		= \dualPair[\big]{\equivClass{p, A}}{\equivClass{p, \tangent_{\FibreBundleModel{f}} \FibreBundleModel{J} (\FibreBundleModel{u})}}
		= \dualPair[\big]{A}{\tangent_{\FibreBundleModel{f}} \FibreBundleModel{J} (\FibreBundleModel{u})},
\end{split}\end{equation}
which verifies~\eqref{eq:symplecticFibreBundle:momentumMapBundle}.
The additional structure on the bundle \( F \) derived from the fiberwise Hamiltonian system on \( \FibreBundleModel{F} \) is summarized in \cref{table_2}.
\begin{table}[tbp]
	\centering
	\begin{tabular}{l l l}
		\toprule
			&
			Fiber model \( \FibreBundleModel{F} \) &
			Bundle \( F \) \\
		\midrule
			Symplectic form &
			\( \FibreBundleModel{\omega} \) &
			\( \omega \) \\

			Action &
			\( \FibreBundleModel{l}: G \times \FibreBundleModel{F} \to \FibreBundleModel{F} \) &
			\( l: \ConjBundle P \times_M F \to F \) \\

			Infinitesimal action &
			\( \FibreBundleModel{l}': \LieA{g} \times \FibreBundleModel{F} \to \TBundle \FibreBundleModel{F} \) &
			\( \difFibre_1 l: \AdBundle P \times_M F \to \VBundle F \) \\

			Momentum map &
		\( \FibreBundleModel{J}: \FibreBundleModel{F} \to \LieA{g}^* \) &
		\( J: F \to \CoAdBundle P \) \\
		\bottomrule
	\end{tabular}
\caption{Overview of the objects of the fiberwise Hamiltonian system and their companions in the bundle picture.}
\label{table_2}
\end{table}

Returning to the original problem of computing the momentum map of the 
gauge group, we can now 
rewrite~\eqref{eq::diffAction:momentumMapCalculationSecondTerm} as
\begin{equation}\begin{split}
	(-1)^{\dim M} \bigl(\mu \hatProduct (Y_F \contr \hat{\omega})\bigr)_\phi (Z) 
		&= \int_M \hat{\omega}_{\phi(m)}\bigl(Y_F \circ \phi(m), Z(m)\bigr) \, \mu \\
		&= \int_M \hat{\omega}_{\phi(m)}\bigl((\difFibre_1 l)_* Y \circ \phi(m), Z(m)\bigr) \, \mu \\
		&= - \int_M \dualPair[\big]{Y(m)}{(\difFibre J)_{\phi(m)} \bigl(Z(m)\bigr)} \, \mu \\
		&= - \scalarProd[\big]{Y}{(\difFibre J)_\phi (Z)}_{\AdAction},
\end{split}\end{equation}
where in the last step we used the natural pairing \( \scalarProd{\cdot}
{\cdot}_{\AdAction}: \sSectionSpace(\AdBundle P) \times 
\sSectionSpace(\CoAdBundle P) \to \R \) discussed in 
\cref{ex::dualPairLieGroups:gaugeGroup}. Summarizing, we get
\begin{equation}
Y^* \contr \Omega + \scalarProd{Y}{(\difFibre J)_*}_{\AdAction} = 0.
\end{equation}
Keeping in mind that \( (\difFibre J)_* \circ \phi \) is the derivative 
of \( J_*: \SectionSpaceAbb{F} \to \CoGauAlgebra(P) \) at 
\( \phi \in \SectionSpaceAbb{F} \), we read off the momentum map for 
the group of gauge transformation.

\begin{prop}\label{prop::momentumMap:gaugeGroup}
Let \( P \to M \) be a principal \( G \)-bundle over the closed manifold 
\( M \) endowed with a volume form \( \mu \).
Let \( F = P \times_G \FibreBundleModel{F} \) be a symplectic fiber bundle 
modeled on the Hamiltonian \( G \)-manifold \( (\FibreBundleModel{F}, 
\FibreBundleModel{\omega},\FibreBundleModel{J}) \).  
The momentum map for the action of the group \( \GauGroup(P) \) of gauge 
transformations on the space \( \SectionSpaceAbb{F} \) of sections 
of \( F \) is given by
	\begin{equation}
		\SectionSpaceAbb{J}_{\GauGroup}: \SectionSpaceAbb{F} \to \CoGauAlgebra(P), 
		\qquad \phi \mapsto J_* \phi.
	\end{equation}	
Here, \( \CoGauAlgebra(P) = \sSectionSpace(\CoAdBundle P) \) is dual 
to \( \GauAlgebra(P) \) by the integration pairing 
\( \scalarProdDot_{\AdAction} \) introduced in \cref{ex::dualPairLieGroups:gaugeGroup} and, moreover, 
\( J = \id_P \times_G \FibreBundleModel{J}: F \to \CoAdBundle P \) 
is the bundle map induced by \( \FibreBundleModel{J} \). That is,
	\begin{equation}
\scalarProd[\big]{Y}{\SectionSpaceAbb{J}_{\GauGroup}(\phi)}_{\AdAction}
= \int_M \dualPair[\big]{Y(m)}{(\id_P \times_G \FibreBundleModel{J}) 
\circ \phi(m)} \, \mu(m)
\end{equation}
for all \( Y \in \GauAlgebra(P) \).
\end{prop} 
In special cases, similar expressions for the momentum map of the gauge 
group were found in \parencite[Remark~3.1]{CieliebakGaioEtAl1999} (on 
Riemann surfaces) and in \parencite[4.2.2]{MundetiRiera2000} (on K\"ahler 
manifolds), both in the context of Yang--Mills equations coupled to 
matter fields.  

\subsubsection{Momentum map for the diffeomorphism group}
Let us now determine the momentum map for the action of the 
diffeomorphism group of the base. 
In the computations below we use a dual pair 
\( \kappa(\DiffGroup_{\mu, P}(M), \SectionSpaceAbb{H}) \) but 
we do not specify \(\SectionSpaceAbb{H}\) for the general case. 
The situations interesting to us are treated explicitly later.
In light of 
\cref{prop::groupMomentumMap:forExtension}, it is enough to split 
the infinitesimal sequence corresponding to~\eqref{eq::hatProduct:AutGroupExactSequence} 
and subsequently determine the `partial' momentum map. In the 
present setting, the associated sequence is 
\begin{equation}
\label{sequence_momentum_diffeos}
0 \to \GauAlgebra(P) \to \AutAlgebra_\mu(P) \to \VectorFieldSpace_{\mu}(M) 
\to 0,
\end{equation}
which naturally splits (in the category of locally convex vector spaces, 
not as Lie algebras) using a principal connection in \( P \).
Given a principal connection \( \Gamma \) in \( P \), let \( \Gamma: 
\VectorFieldSpace_{\mu}(M) \to \AutAlgebra_\mu(P) \) be the 
vector space homomorphism that assigns to a vector field 
\( X \in \VectorFieldSpace_{\mu}(M) \) its horizontal lift 
\( X^h \in \AutAlgebra_\mu(P) \).
Recall that the curvature of \( \Gamma \) is the obstruction for the 
lift \( \Gamma: \VectorFieldSpace_{\mu}(M) \to \AutAlgebra_\mu(P) \) 
to be a Lie algebra homomorphism.
By \cref{prop::groupMomentumMap:forExtension} applied 
to~\eqref{sequence_momentum_diffeos} with \( \Gamma: 
\VectorFieldSpace_{\mu}(M) \to \AutAlgebra_\mu(P) \) as the splitting, 
the `partial' momentum map \( \SectionMapAbb{J}_\Gamma: 
\SectionSpaceAbb{F} \to \SectionSpaceAbb{H} \) has to 
satisfy
\begin{equation}
\label{eq::diffGroup:partialMomentumMap}
(X^h)^* \contr \Omega + \kappa(X, \difLog \SectionMapAbb{J}_\Gamma) = 0, 
\qquad X \in \VectorFieldSpace_\mu(M)
\end{equation}
for a suitable dual pair \(\kappa(\DiffGroup_{\mu,P}(M),
\SectionSpaceAbb{H})\).
If we choose an extension \( \hat{\omega} \) compatible with the 
connection \( \Gamma \), then the second term 
in~\eqref{eq::autAction:contractionWithKillingVectorField} vanishes and 
thus~\eqref{eq::diffGroup:partialMomentumMap} simplifies to
\begin{equation}
\label{eq::diffGroup:partialMomentumMapCompatibleConnection}
(X \contr \mu) \hatProduct \hat{\omega} + \kappa(X, \difLog \SectionMapAbb{J}_\Gamma) = 0, 
\qquad X \in \VectorFieldSpace_\mu(M).
\end{equation} 
In general, it is hard to determine the `partial' momentum map for the 
diffeomorphism group, mainly because one has to describe the dual group 
of \( \DiffGroup_{\mu, P}(M) \) rather concretely. For this reason, we 
will only study the momentum map for the most important cases of 
volume-preserving and symplectic diffeomorphisms.

\subsubsection{Momentum map for volume-preserving diffeomorphisms}
\label{sec:momentum_map_for_volume_preserving_diffeomorphisms}

Recall the exact sequence~\eqref{eq::hatProduct:AutGroupExactSequence} 
and assume, for simplicity, that \( \DiffGroup_{\mu, P}(M) = \DiffGroup_{\mu}(M) \). In 
\cref{ex::dualPairLieGroups:volumePreservingDiffeos} we showed that the 
dual group to \( \DiffGroup_{\mu}(M) \) is the group of Cheeger--Simons 
differential characters \( \csCohomology^2(M, \UGroup(1)) \), which 
parametrizes principal \( \UGroup(1) \)-bundles over \( M \). 
Let \( \kappa(\DiffGroup_{\mu}(M), \csCohomology^2(M, \UGroup(1))) \) be 
the resulting dual pair of Lie groups.

We continue to work in the general setting where \( F = P \times_G \FibreBundleModel{F} \) is an associated fiber bundle endowed with a 
fiber symplectic structure \( \omega \) induced by a symplectic form 
\( \FibreBundleModel{\omega} \) on \( \FibreBundleModel{F} \).
Assume that \( (\FibreBundleModel{F}, \FibreBundleModel{\omega}) \) 
has a \( G \)-equivariant prequantization $\UGroup(1)$-bundle 
\( \FibreBundleModel{L} \to  \FibreBundleModel{F} \) with connection 
\( \FibreBundleModel{\vartheta} \).
That is, \( G \) acts on \( \FibreBundleModel{L} \) by principal bundle 
automorphisms covering the action on \( \FibreBundleModel{F} \), and the 
connection \( \FibreBundleModel{\vartheta} \) is \( G \)-invariant. 
Since \( G \) acts by bundle automorphisms, the actions of \( \UGroup(1) \) 
and \( G \) commute.
Consider the associated $\UGroup(1)$-bundle \( L \defeq P \times_G  \FibreBundleModel{L} \to F= P \times _G \FibreBundleModel{F} \).
Since $L \to M$ is a fiber bundle with typical fiber $\FibreBundleModel{L}$, the vertical subbundle $\VBundle L = P \times_G \TBundle \FibreBundleModel{L}$ naturally projects onto $\VBundle F = P \times_G \TBundle \FibreBundleModel{F}$.
The \( G \)-equivariant \( \UGroup(1) \)-connection \( \FibreBundleModel{\vartheta} \) in \( \FibreBundleModel{L} \to \FibreBundleModel{F} \) yields an \( G \)-equivariant lift $\TBundle \FibreBundleModel{F} \to \TBundle \FibreBundleModel{L}$ and, thereby, induces a lift \( \vartheta: \VBundle F \to \VBundle L \).
The principal \( G \)-connection \( \Gamma \) on \( P \rightarrow  M \) completes \( \vartheta \) to a \( \UGroup(1) \)-connection \( \vartheta^\Gamma \) in \( L \to F \).
Indeed, the associated connection \( 1 \)-form on \( L \) is, under the identification \( \TBundle L = 
\TBundle P \times_{\TBundle G} \TBundle \FibreBundleModel{L} \), given by
\begin{equation}
\label{eq:connectionOnL}
\vartheta^\Gamma_{\equivClass{p, \FibreBundleModel{l}}}\bigl( \equivClass[\big]{Z^\Gamma_p + 
B^*_p, u_{\FibreBundleModel{l}}} \bigr) \defeq 
\FibreBundleModel{\vartheta}_{\FibreBundleModel{l}}
( B^*_{\FibreBundleModel{l}} + u_{\FibreBundleModel{l}} ),
\end{equation}
where \( Z^\Gamma_p \) is a \(\Gamma \)-horizontal vector at the 
point \( p \in  P \), \( B \in \LieA{g} \), and 
\( u_{\FibreBundleModel{l}} \in \TBundle_{\FibreBundleModel{l}} 
\FibreBundleModel{L} \).
By construction, \( \vartheta^\Gamma \) coincides with \( \vartheta \) on the vertical part and hence the curvature \( F_{\vartheta^\Gamma} \) of \( \vartheta^\Gamma \), seen as a \( 2 \)-form on \( F \), is an extension of the fiber symplectic structure \( \omega \).
In other words, \( (L, \vartheta^\Gamma) \) yields a prequantization of the fiber symplectic bundle \( (F, \omega) \).
We will set \( \hat{\omega} = F_{\vartheta^\Gamma} \) in the general theory discussed above.
Moreover, as discussed in \cref{sec::differentialCharacters:motivationCircleBundles}, the holonomy of \( (L, \vartheta^\Gamma) \) yields a differential character \( h_\Gamma \in \csCohomology^2(F, \UGroup(1)) \).

The pull-back map
\begin{equation}
	\pb_h: \SectionSpaceAbb{F} \to \csCohomology^2(M, \UGroup(1)), 
	\qquad \phi \mapsto \phi^* h_\Gamma
\end{equation}
has logarithmic derivative \( (\difLog\, \pb_h)_\phi Y = \phi^* \bigl((Y \contr 
F_{\vartheta^\Gamma}) \circ \phi\bigr) = \phi^* \bigl((Y \contr \hat{\omega}) \circ 
\phi\bigr) \) according to  \cref{prop::differentialCharacter:pullback}.
On the other hand, 
using~\eqref{eq::fibreIntegration:defWithPartialPullback}, we evaluate 
\( (X \contr \mu) \hatProduct \hat{\omega} \) on \( Y \in \TBundle_\phi 
\SectionSpaceAbb{F} \) to obtain
\begin{equation}
\bigl((X \contr \mu) \hatProduct \hat{\omega}\bigr)_\phi(Y) = 
(-1)^{\dim M -1} \int_M (X \contr \mu) \wedge \phi^* \bigl((Y \contr \hat{\omega}) 
\circ \phi\bigr).
\end{equation}
Thus we have
\begin{equation}
(X \contr \mu) \hatProduct \hat{\omega} = \kappa(X, \difLog \, \pb_h),
\end{equation}
so that \( \SectionMapAbb{J}_\Gamma \defeq - \pb_h: \SectionSpaceAbb{F} \to 
\csCohomology^2(M, \UGroup(1)) \) satisfies the `partial' momentum map relation~\eqref{eq::diffGroup:partialMomentumMapCompatibleConnection}. 

We finally gathered all the necessary ingredients to apply 
\cref{prop::groupMomentumMap:forExtension,prop::groupMomentumMap:forLift}.

\begin{thm}[Momentum map in the volume-preserving case]
\label{thm::symplecticFibreBundle:momentumMapVolume}
Let \( \pi: P \to M \) be a finite-dimensional principal \( G \)-bundle 
over the closed volume manifold \( (M, \mu) \) and assume that 
\( \DiffGroup_{\mu, P}(M) = \DiffGroup_{\mu}(M) \).
Then the following holds:
\begin{enumerate}
\item For every Hamiltonian \( G \)-manifold \( (\FibreBundleModel{F}, 
\FibreBundleModel{\omega}, \FibreBundleModel{J}) \), the space 
\( \SectionSpaceAbb{F} \) of sections of the associated symplectic fiber bundle 
\( (F = P \times_G \FibreBundleModel{F}, \omega) \) is a symplectic 
Fr\'echet manifold with weak symplectic form \( \Omega = \mu \hatProduct 
\omega \).

\item 	Choose a principal connection \( \Gamma \) in \( P \). Assume that \( (\FibreBundleModel{F}, \FibreBundleModel{\omega}) \) 
has a \( G \)-equivariant prequantization $\UGroup(1)$-bundle 
\( \FibreBundleModel{L} \) and denote the resulting prequantization differential character
of \( (F, \omega) \) by \( h_\Gamma \in 
\csCohomology^2(F, \UGroup(1))\). 
The group \( \AutGroup_\mu(P) \) of  
bundle automorphisms, whose induced base diffeomorphisms
are volume-preserving, acts symplectically on \( (\SectionSpaceAbb{F}, 
\Omega) \) and has a group-valued  momentum map
\begin{equation}
\SectionSpaceAbb{J}_{\AutGroup}: \SectionSpaceAbb{F} \to \CoGauAlgebra(P) \times 
\csCohomology^2(M, \UGroup(1)), \qquad \phi \mapsto 
(J_* \phi, - \phi^* h_\Gamma).
\end{equation}
This momentum map is equivariant relative to the natural actions of \( \AutGroup_\mu(P) \) (see \cref{ex:dualPairLieGroups:volumePreservingDiffeos:coconjugation}).

\item 	A lift \( \chi: \DiffGroup_\mu(M) \to \AutGroup_\mu(P) \) of the group of volume-preserving diffeomorphisms to principal bundle automorphisms (this is a section of the projection in~\eqref{eq::hatProduct:AutGroupExactSequence}) yields a symplectic group action of \( \DiffGroup_\mu(M) \) on \( \SectionSpaceAbb{F} \) with \( \csCohomology^2(M, \UGroup(1)) \)-valued momentum map
\begin{equation}
\SectionSpaceAbb{J}_{\DiffGroup}: \SectionSpaceAbb{F} \to 
\csCohomology^2(M, \UGroup(1)), \qquad \phi \mapsto - \phi^* h_\Gamma + 
\iota \circ \tau_{\Gamma, \chi}^* \circ J_* (\phi).
	 \end{equation}
Here, \( \iota: \csAlgebra^2(M, \UGroup(1)) \to \csCohomology^2(M, \UGroup(1)) \) denotes the natural inclusion of topological trivial bundles (see \cref{sec::differentialCharacters:definition}). Moreover, \( \tau_{\Gamma, \chi}^*: \CoGauAlgebra(P)\to \csAlgebra^2(M, \UGroup(1))\) denotes the adjoint of \( \tau_{\Gamma, \chi} \defeq \chi' - \Gamma: \VectorFieldSpace_\mu(M) \to \GauAlgebra(P) \) with respect to the dual pairs \( \kappa(\VectorFieldSpace_\mu(M), \csAlgebra^2(M, \UGroup(1))) \) and \( \scalarProd{\GauAlgebra(P)}{\CoGauAlgebra(P)}_{\AdAction} \).
Furthermore, \( \SectionSpaceAbb{J}_{\DiffGroup} \) does not depend on the connection \( \Gamma \) that was used in the construction.
\qedhere
\end{enumerate}
\end{thm}

\begin{remark}\label{rem:symplecticFibreBundle:tauMapVolume}
In order to illustrate the meaning of the map \( \tau_{\Gamma, \chi} \) and its adjoint, we compute it for the $\SLGroup(n, \mathbb{R})$-frame bundle \( \pi: \FrameBundle M \to M \). Every volume-preserving diffeomorphism \( \phi \) lifts in a natural way to a diffeomorphism
\begin{equation}
	\chi(\phi) = \hat{\phi}: \FrameBundle M \to \FrameBundle M, 
	\quad
	u_m \mapsto \tangent_m \phi \circ u_m,
\end{equation}
where we view a frame \( u_m \in \FrameBundle_m M \) as a linear isomorphism \( u_m: \R^n \to \TBundle_m M \).
Thus the frame bundle comes equipped with a natural lift \( \chi: \DiffGroup_\mu(M) \to \AutGroup_\mu(\FrameBundle M) \).
Of course, we also get an induced lifting map for vector fields, which we write as \( \chi': \VectorFieldSpace_\mu(M) \ni X \mapsto \hat{X} \in \AutAlgebra_\mu(\FrameBundle M) \).

Recall that there exists a natural \( 1 \)-form \( \theta \in \DiffFormSpace^1(\FrameBundle M, \R^n) \) defined by (see \parencite[Section~III.2]{KobayashiNomizu1963})
\begin{equation}
\theta (Z_u) = u^{-1} \circ \tangent_u \pi (Z_u), \quad 
Z_u \in \TBundle_u (\FrameBundle M). 
\end{equation}
The form \( \theta \) is invariant under naturally lifted vector fields, \ie, \( \difLie_{\hat{X}} \theta = 0 \).

Let $\alpha \in \DiffFormSpace^1(\FrameBundle M, \SLAlgebra(n, \R))$ be a principal connection \( 1 \)-form on the bundle $\FrameBundle M \to M$ and denote by $\Gamma$ the associated horizontal lift.
The associated covariant derivative on the tangent bundle $\TBundle M \to M$ is defined by (see \parencite[Lemma on p.~133]{KobayashiNomizu1963})
\begin{equation}
\nabla_Y X \defeq u \circ \hat{Y}^\Gamma_u (\theta(\hat{X})), 
\quad u \in \FrameBundle M,
\end{equation}
where \( \hat{Y}^\Gamma \) denotes the \( \Gamma \)-horizontal lift of \( Y \in \VectorFieldSpace(M) \). The torsion \( t_\Gamma \in 
\DiffFormSpace^2(\FrameBundle M, \R^n) \) of \( \Gamma \) is given as the covariant derivative of \( \theta \).
The structure equation (see \parencite[Theorem~III.2.4]{KobayashiNomizu1963}) allows us to write
\begin{equation}
t_\Gamma (Z_1, Z_2) = \dif \theta (Z_1, Z_2) + 
\alpha(Z_1) \cdot \theta (Z_2) - \alpha(Z_2) \cdot \theta (Z_1),
\end{equation}
where the $ \cdot $ in the last two terms denotes the natural action of \( \SLAlgebra(n, \R) \) on \(\R^n \).
We evaluate this expression for the case where \( Z_1 = \hat{X} \in \VectorFieldSpace(\FrameBundle M)\) is the natural lift of \( X \in \VectorFieldSpace(M) \) and \( Z_2 = \hat{Y}^\Gamma \in \VectorFieldSpace(\FrameBundle M)\) is the horizontal lift of \( Y \in \VectorFieldSpace(M) \):
\begin{equation}\begin{split}
t_\Gamma \bigl(\hat{X}, \hat{Y}^\Gamma\bigr) 
&= \dif \theta \bigl(\hat{X}, \hat{Y}^\Gamma\bigr) + \alpha (\hat{X}) \cdot \theta \bigl(\hat{Y}^\Gamma\bigr) \\
&= - \hat{Y}^\Gamma \contr \dif (\hat{X} \contr \theta) + 
\alpha (\hat{X}) \cdot u^{-1}(Y) \\
&= - u^{-1} (\nabla_Y X) + \alpha (\hat{X}) \cdot u^{-1}(Y),
\end{split}\end{equation}
where we used \( \difLie_{\hat{X}} \theta = 0 \). 
On the other hand, for \( \tau_{\Gamma, \chi} = \chi' - \Gamma \), 
we have
\begin{equation}
\tau_{\Gamma, \chi} (X)
			= \alpha \bigl(\tau_{\Gamma, \chi} (X)\bigr)
			= \alpha \bigl(\hat{X} - \hat{X}^\Gamma\bigr) 
			= \alpha (\hat{X}).
\end{equation}
Combining both equations, we get
\begin{equation}\begin{split}
	u \bigl(\tau_{\Gamma, \chi} (X) \cdot u^{-1}(Y)\bigr) 
		&= u \circ t_\Gamma \bigl(\hat{X}, \hat{Y}^\Gamma\bigr) + \nabla_Y X \\
		&= \nabla_X Y - \commutator{X}{Y}.
\end{split}\end{equation}
Thus, for torsion-free connections, we identify \( \tau_{\Gamma, \chi} \) with the covariant derivative
\begin{equation}
\tau_{\Gamma, \chi}: \VectorFieldSpace_\mu (M) \to 
\sSectionSpace(\CotBundle M \otimes \TBundle M), 
\quad X \mapsto \nabla X.
\end{equation}
For a section \(s\) of \(\TBundle M \otimes \CotBundle M\rightarrow M\), the covariant derivative \( \nabla s \) is a section of \((\TBundle M \otimes \CotBundle M) \otimes \CotBundle M\rightarrow M\).
We obtain a \( 1 \)-form \( c (\nabla s) \) by contracting the first with the last component (\ie, the vector component with the \enquote{derivative} component; in Penrose's abstract index notation \( c(\nabla s)_i = \nabla_j s^j{}_i \)). The identity
\begin{equation}
\divergence (X \contr s) = \dualPair{\nabla X}{s} + X \contr c (\nabla s)
\end{equation}
yields the following expression for the adjoint operator of 
\( \tau_{\Gamma, \chi} \):
\begin{equation}
\tau_{\Gamma, \chi}^*: \sSectionSpace(\TBundle M \otimes \CotBundle M) 
\to \DiffFormSpace^1(M) \slash \dif \DiffFormSpace^0(M), \quad
s \mapsto - \equivClass{c (\nabla s)}.
		\qedhere
\end{equation}
\end{remark}

\begin{remark}[Momentum map in the classical sense]
\label{rem::momentumMap:obstructionForClassicalExistence}
If we compose the momentum map \( \SectionSpaceAbb{J}_{\DiffGroup}: 
\SectionSpaceAbb{F} \to \csCohomology^2(M, \UGroup(1)) \) with the 
curvature map \(\csCohomology^2(M,\UGroup(1)) \to \DiffFormSpace^2(M,\R)\), 
then  we recover Donaldson's expression for the "momentum map" 
\parencite[Equation~8]{Donaldson2003}\footnote{Actually, Donaldson's 
expression for the momentum map differs by a sign, which comes from the 
fact that in \parencite{Donaldson2003} the diffeomorphism group acts 
via a right action on the section space.}
\begin{equation}
	\label{eq:momentumMap:classical}
\SectionSpaceAbb{F} \to \DiffFormSpace^2(M, \R), \qquad \phi \mapsto - \phi^* \hat{\omega} + \dif \bigl(\tau_{\Gamma, \chi}^* \circ J_* (\phi)\bigr).
\end{equation}
Note, however, that \( \DiffFormSpace^2(M, \R) \) is not the dual of the 
Lie algebra \( \VectorFieldSpace_\mu(M) \). In fact, it is only the 
dual of the Lie algebra of \emph{exact} volume-preserving vector fields.
Thus,~\eqref{eq:momentumMap:classical} does not yield a bona-fide 
momentum map for the action of \( \DiffGroup_\mu(M) \).

Although we were able to determine a group-valued momentum map for the diffeomorphism action under rather weak assumptions on the fiber model \( (\FibreBundleModel{F}, \FibreBundleModel{\omega}) \), \emph{a momentum map in the classical sense exists only under certain topological conditions on the fiber bundle \( F \)}. Indeed, a necessary and sufficient condition for the existence of a momentum map \( \SectionMapAbb{J}: \SectionSpaceAbb{F} \to (\VectorFieldSpace_\mu(M))^* \) is that the obstruction map \( \VectorFieldSpace_\mu(M) \slash \commutator{\VectorFieldSpace_\mu(M)}{\VectorFieldSpace_\mu(M)} \to \sCohomology^1(\SectionSpaceAbb{F}, \R) \) sending the class of  
\( X \in \VectorFieldSpace_\mu(M) \) to the cohomology class of 
\( X^* \contr \Omega \) is identically zero (see, for example, \parencite[Proposition~4.5.17]{OrtegaRatiu2003}). 

Suppose we have a lift \( \chi: \DiffGroup_\mu(M) \to \AutGroup_\mu(P) \).
Then there is an induced lifting map for vector fields, which we write as 
\( \VectorFieldSpace_\mu(M) \ni X \mapsto \hat{X}\in\AutAlgebra_\mu(P) \).
Let us assume\sideNote{These assumptions are satisfied, for example, for 
the trivial bundle \( F = M \times \FibreBundleModel{F} \to M \).} that 
\(\hat{X}\) is horizontal with respect to a principal connection on 
$P \to M$ compatible with the extension $\hat{\omega}$ and that 
\( \DiffGroup_{\mu, P}(M) = \DiffGroup_{\mu}(M) \). In this setting, 
the obstruction to the existence of a classical momentum map takes the 
form
\begin{equation}
\label{eq::obstructionUsualMomentumMap}
\sCohomology^{\dim M-1}(M, \R) \to 
\sCohomology^1(\SectionSpaceAbb{F}, \R), \qquad 
\equivClass{\alpha} \mapsto \equivClass{\alpha} \hatProduct 
\equivClass{\hat{\omega}},
\end{equation}
where we used the isomorphism
\( \VectorFieldSpace_\mu(M) \slash [\VectorFieldSpace_\mu(M), 
\VectorFieldSpace_\mu(M)] \ni \equivClass{X} \mapsto 
\equivClass{X \contr \mu} \in \sCohomology^{\dim M-1} (M, \R)\), 
see \parencite[Theorem~1.4.3]{Banyaga1997}.
We do not know what are the minimal conditions guaranteeing that the 
above obstruction~\eqref{eq::obstructionUsualMomentumMap} vanishes.
Clearly, this is the case if the corresponding cohomology class of 
\( M \) vanishes. However, this is a rather strong topological assumption. 
\end{remark}

\begin{remark}[Topological data contained in the group-valued momentum map]
	\label{rem:momentumMapDiffVol:topologicalData}
	The \( \csCohomology^2(M, \UGroup(1)) \)-valued momentum map contains two types of information.
	First, the curvature as a geometric datum essentially corresponds to the classical momentum map as discussed in \cref{rem::momentumMap:obstructionForClassicalExistence} above.

	Second, the group-valued momentum map also yields topological classes that are not present in the classical momentum map.
	The Chern class 
	\begin{equation}
		c(\SectionSpaceAbb{J}_{\DiffGroup}(\phi)) = - \phi^* c(h_\Gamma) 
	\end{equation}
	can be viewed as an obstruction to the existence of a classical momentum map. Indeed, if a classical momentum map with values in \( \DiffFormSpace^1(M) \slash \dif \DiffFormSpace^0(M) \) exists, then its composition with the inclusion \( \iota: \DiffFormSpace^1(M) \slash \dif \DiffFormSpace^0(M) \to  \csCohomology^2(M, \UGroup(1)) \) (see~\eqref{eq::differentialCharacter:definingDiagram}) is a group-valued momentum map whose Chern class vanishes, because the image of \( \iota \) lies in the kernel of the Chern class map \( c \). 

If the Chern class is trivial, then a secondary topological class related to the equivariance of the momentum map appears.
This secondary class has its origin in the short exact sequence from~\eqref{eq::differentialCharacter:definingDiagram}
\begin{equationcd}
\!\!\!\!
0 \to[r]
	& \DiffFormSpace^1(M) \slash \clZDiffFormSpace^1(M) \to[r, "\iota"]
	& \csCohomology^2(M, \UGroup(1)) \to[r, "c"] 
	& \sCohomology^2(M, \Z) \to[r]
	& 0.
\end{equationcd}
Exactness implies that the kernel of the characteristic class map \( c \) is isomorphic to \( \DiffFormSpace^1(M) \slash \clZDiffFormSpace^1(M) \).
However, this isomorphism is not canonical and involves a choice.
For example, we may choose a basis of the singular homology 
\( \sHomology_1(M, \Z) \) given by piece-wise smooth loops \( \gamma_i \). 
The curvature of a differential character \( h \in \csCohomology^2(M, 
\UGroup(1)) \) with trivial Chern class is an exact \( 2 \)-form, 
say \( \curv h = \dif \alpha_h \).
The primitive \( \alpha_h \) is uniquely determined, modulo 
\( \clZDiffFormSpace^1(M) \), by the requirement that \( \exp\left(2\pi 
\I \int_{\gamma_i} \alpha_h\right) = h(\gamma_i) \) for all \( i \). By 
construction, the map \( \iota_{\gamma}^{-1}: h \mapsto 
\equivClass{\alpha_h} \) is inverse to \( \iota \) and thus yields the 
claimed isomorphism \( \ker c \to \DiffFormSpace^1(M) \slash 
\clZDiffFormSpace^1(M) \) of Abelian Lie groups.
	
Given an \( \csCohomology^2(M, \UGroup(1)) \)-valued momentum map 
\( \SectionMapAbb{J} \), whose image is contained in \( \ker c \), the 
composition \( \SectionMapAbb{J}_{\gamma} = \iota_{\gamma}^{-1} \circ 
\SectionMapAbb{J} \) is a \( \DiffFormSpace^1(M) \slash 
\clZDiffFormSpace^1(M) \)-valued momentum map, \ie, almost a classical 
momentum map.
Note that, even though \( \SectionMapAbb{J} \) is equivariant with 
respect to pull-backs by diffeomorphisms, the reduced momentum map 
\( \SectionMapAbb{J}_\gamma \) is no longer equivariant, because the 
loops \( \gamma_i \) are also affected by the diffeomorphism.
The obstruction for the equivariance is the class in \( \sCohomology^1(M, 
\UGroup(1)) \) given by evaluating \(h\) on the generators \( \gamma_i \).
We note that the same construction goes through in the slightly more 
general case, where the Chern class does not vanish but where it is a 
torsion class\footnote{An element \( c \in \sCohomology^k(M, \Z) \) is 
called a torsion class if it belongs to the subgroup 
\( \ExtAb(\sHomology_{k-1}(M, \Z), \Z) \subseteq \sCohomology^k(M, \Z) \).} 
(because then the curvature is an exact form as well).
	
We return to this secondary topological class in \cref{sec::applications}, 
where we recover, as an example, the Liouville class of a Lagrangian
embedding. 

In summary, if an equivariant \( \csCohomology^2(M, \UGroup(1)) \)-valued momentum map exists, then the existence of a classical equivariant momentum map is obstructed by two topological classes: the Chern class in \( \sCohomology^2(M, \Z) \) and a class in \( \sCohomology^1(M, \UGroup(1)) \). 
\end{remark}

\subsubsection{Momentum map for symplectomorphisms}%
\label{sec:globalAnalysisSymplecticFibreBundles:momentumMapForSymplectomorphisms}
Let \( M \) be a closed \( 2n \)-dimensional manifold endowed with a symplectic form \( \sigma \). The Liouville volume form is denoted by \( \mu_\sigma = \sigma^n \slash n! \). Instead of the exact sequence~\eqref{eq::hatProduct:AutGroupExactSequence}, we consider the exact sequence
\begin{equation}
\label{eq::hatProduct:AutGroupExactSequenceSymplectic}
\id \to \GauGroup(P) \to \AutGroup_\sigma(P) \to 
\DiffGroup_{\sigma, P}(M) \to \id,
\end{equation}
where \( \AutGroup_\sigma(P) \) is the group of automorphisms of the 
bundle \( P \to M \) whose induced diffeomorphisms on the base p
reserve \( \sigma \). 
We assume that \( \DiffGroup_{\sigma, P}(M) = \DiffGroup_{\sigma}(M) \). 
The momentum map is determined similarly to the volume-preserving case. Indeed, we only need to replace the dual pair \( \kappa(\DiffGroup_{\mu}(M), \csCohomology^2(M, \UGroup(1))) \) by the dual pair \( \kappa(\DiffGroup_{\sigma}(M), \csCohomology^{2n}(M, \UGroup(1))) \) appearing in \cref{ex::dualPairLieGroups:symplectomorphisms}.

Assume that \( \sigma \) is prequantizable and choose a prequantization circle bundle with connection \( h_\sigma \in \csCohomology^2(M, \UGroup(1)) \). 
As above, we assume that \( (\FibreBundleModel{F}, \FibreBundleModel{\omega}) \) has a \( G \)-equivariant prequantization bundle \( \FibreBundleModel{L} \).
A principal connection \( \Gamma \) in \( P \to M \) yields a prequantization \( h_\Gamma \in \csCohomology^2(F, \UGroup(1)) \) of the symplectic fiber bundle \( (F, \omega) \).
We claim that the `partial' momentum map \( \SectionMapAbb{J}_\Gamma: \SectionSpaceAbb{F} \to \csCohomology^{2n}(M, \UGroup(1)) \) assigns to a section \( \phi \) the degree \( 2n \)-bundle \( - \phi^* h_\Gamma \star h_\sigma^{n-1} \), where the star product extends the wedge product to differential characters and the power \( n-1 \) should also be understood in this sense, see~\eqref{Appendix_D_def_fat_star}.
In order to verify that the `partial' momentum map relation, \cf~\eqref{eq::diffGroup:partialMomentumMapCompatibleConnection},
\begin{equation}
	\label{eq:momentumMapForSymplectomorphisms:partialMomentumMap}
(X \contr \mu_\sigma) \hatProduct \hat{\omega} + \kappa(X, \difLog 
\SectionMapAbb{J}_\Gamma) = 0, \qquad X \in \VectorFieldSpace_\sigma(M)
\end{equation}
holds, we first observe that \( \SectionMapAbb{J}_\Gamma \) is the composition of the pull-back map \( \pb_{h_\Gamma} \) with the Lie group homomorphism
\begin{equation}
\csCohomology^{2}(M, \UGroup(1)) \to \csCohomology^{2n}(M, \UGroup(1)), 
\quad h \mapsto h \star h_\sigma^{n-1}.
\end{equation}
The corresponding Lie algebra homomorphism is given by \( \alpha \mapsto \alpha \wedge \sigma^{n-1} \), because, using the properties listed in \cref{sec::differentialCharacter:starProduct}, we have 
\begin{equation}
\iota(\alpha) \star h_\sigma^{n-1} = \iota (\alpha \wedge \curv 
h_\sigma^{n-1}).
\end{equation}
Using the distribution law of the logarithmic derivative over Lie group homomorphisms \parencite[Proposition~II.4.1.1]{Neeb2006}, we thus get
\begin{equation}\begin{split}
	\kappa\bigl(X, \difLog_\phi \SectionMapAbb{J}_\Gamma (Y)\bigr)
		&= - \, \frac{(-1)^{\dim M - 1}}{(n-1)!} \int_M (X \contr \sigma) \wedge \phi^* \bigl((Y \contr \hat{\omega}) \circ \phi\bigr) \wedge \sigma^{n-1}  
		\\
		&= - \, (-1)^{\dim M - 1} \int_M (X \contr \sigma) \wedge \frac{\sigma^{n-1}}{(n-1)!} \wedge \phi^* \bigl((Y \contr \hat{\omega}) \circ \phi\bigr)
		\\
		&= - \, (-1)^{\dim M - 1} \int_M (X \contr \mu_\sigma) \wedge \phi^* \bigl((Y \contr \hat{\omega}) \circ \phi\bigr),
		\\
		&= - \, \bigl((X \contr \mu_\sigma) \hatProduct \hat{\omega}\bigr)_\phi (Y),
\end{split}\end{equation}
where the last equality follows from~\eqref{eq::fibreIntegration:defWithPartialPullback}.
This completes the proof that \( \SectionMapAbb{J}_\Gamma \) satisfies 
the `partial' momentum map relation~\eqref{eq:momentumMapForSymplectomorphisms:partialMomentumMap}.
Thus, \cref{prop::groupMomentumMap:forExtension,prop::groupMomentumMap:forLift} yield the following. 

\begin{thm}[Momentum map in the symplectic case]
\label{thm::symplecticFibreBundle:momentumMapSymplectic}
Let \( \pi: P \to M \) be a finite-dimensional principal \( G \)-bundle over the closed $2n$-dimensional symplectic manifold \( (M, \sigma) \) and assume that \( \DiffGroup_{\sigma, P}(M) = \DiffGroup_{\sigma}(M)\) in the exact sequence~\eqref{eq::hatProduct:AutGroupExactSequenceSymplectic}.
Let $\mu_\sigma \defeq \sigma^n/n!$ be the Liouville volume form on \(M\).
Then the following holds:
\begin{enumerate}
\item For every Hamiltonian \( G \)-manifold 
\((\FibreBundleModel{F},\FibreBundleModel{\omega},\FibreBundleModel{J})\), 
the space of sections \( \SectionSpaceAbb{F} \) of the associated symplectic 
fiber bundle \( (F = P \times_G \FibreBundleModel{F},\omega)\) is a 
symplectic Fr\'echet manifold with symplectic form \( \Omega = \mu_\sigma 
\hatProduct \omega \).

\item 	
Choose a principal connection \( \Gamma \) in \(P\rightarrow  M\). 
Assume that \( (\FibreBundleModel{F}, \FibreBundleModel{\omega}) \) has a \( G \)-equivariant prequantization $\UGroup(1)$-bundle \( \FibreBundleModel{L} \) and denote the resulting prequantization differential character of \( (F, \omega) \) by \( h_\Gamma \in \csCohomology^2(F, \UGroup(1)) \).
Assume that \( \sigma \) has a prequantization \( h_\sigma \in \csCohomology^2(M, \UGroup(1)) \).
The group \( \AutGroup_\sigma(P) \) of bundle automorphisms whose induced 
base diffeomorphisms preserve the symplectic form \( \sigma \) acts 
symplectically on \( (\SectionSpaceAbb{F}, \Omega) \) and has a group-valued momentum map
\begin{equation}
\SectionSpaceAbb{J}_{\AutGroup}: \SectionSpaceAbb{F} \to \CoGauAlgebra(P) \times 
\csCohomology^{2n}(M, \UGroup(1)), \qquad \phi \mapsto 
(J_* \phi, - \phi^* h_\Gamma \star h_\sigma^{n-1}).
\end{equation}
This momentum map is equivariant relative to the natural actions of \( \AutGroup_\sigma(P) \) (compare \cref{ex:dualPairLieGroups:volumePreservingDiffeos:coconjugation}).

\item 	A lift \( \chi: \DiffGroup_\sigma(M) \to \AutGroup_\sigma(P) \) 
of the group of symplectomorphisms to principal bundle automorphisms (this is a section of the projection 
in~\eqref{eq::hatProduct:AutGroupExactSequenceSymplectic}) yields a symplectic group action of \( \DiffGroup_\sigma(M) \) on 
\(\SectionSpaceAbb{F}\) with \(\csCohomology^{2n}(M, \UGroup(1))\)-valued momentum map
\begin{equation}
\SectionSpaceAbb{J}_{\DiffGroup}: \SectionSpaceAbb{F} \to 
\csCohomology^{2n}(M, \UGroup(1)), \qquad \phi \mapsto - 
\phi^* h_\Gamma \star h_\sigma^{n-1} + 
\iota \circ \tau_{\Gamma, \chi}^* \circ J_* (\phi).
\end{equation}

Here, \( \iota: \csAlgebra^{2n}(M, \UGroup(1)) \to \csCohomology^{2n}(M, \UGroup(1)) \) denotes the natural inclusion of topological trivial bundles (see \cref{sec::differentialCharacters:definition}).
Moreover, \(\tau_{\Gamma,\chi}^*: \CoGauAlgebra(P)\to \csAlgebra^{2n}(M,\UGroup(1))\) denotes the adjoint of \( \tau_{\Gamma, \chi} = \chi' - \Gamma: \VectorFieldSpace_{\sigma}(M) \to \GauAlgebra(P) \) with respect to the dual pairs \( \kappa(\VectorFieldSpace_{\sigma}(M), 
\csAlgebra^{2n}(M, \UGroup(1))) \) and \( \scalarProd{\GauAlgebra(P)}{\CoGauAlgebra(P)}_{\AdAction} \).
Furthermore, \( \SectionSpaceAbb{J}_{\DiffGroup} \) does not depend on the connection \( \Gamma \) that was used in the construction.
\qedhere
\end{enumerate}
\end{thm}

\begin{remark}[Momentum map in the classical sense]
Consider the case of a trivial bundle 
\( F = M \times \FibreBundleModel{F} \to M \). 
In this setting the term involving \( \tau_{\Gamma, \chi} \) vanishes.
After composing the momentum map \( \SectionSpaceAbb{J}_{\DiffGroup}: \SectionSpaceAbb{F} \to \csCohomology^{2n}(M, \UGroup(1)) \) with the curvature \( \csCohomology^{2n}(M, \UGroup(1)) \to \DiffFormSpace^{2n}(M, \R) \) we get the map
\begin{equation}
\SectionSpaceAbb{F} \to \DiffFormSpace^{2n}(M, \R), \quad \phi \mapsto - \phi^* \hat{\omega} \wedge \sigma^{n-1}.
\end{equation}
\Textcite[Proposition~3]{Donaldson1999} found that this map is the 
momentum map for the action of the symplectomorphism group 
\emph{under certain topological conditions on \( M \)}.
However, we remark again that \( \DiffFormSpace^{2n}(M, \R) \) is not the 
dual of the Lie algebra of symplectic vector fields 
\( \VectorFieldSpace_\sigma(M) \) but that it is only the dual of 
the space of \emph{Hamiltonian} vector fields.

The existence of a momentum map in the classical sense is again 
obstructed by a map on the Abelianization of 
\(\VectorFieldSpace_\sigma(M) \).
The map  \( \VectorFieldSpace_\sigma(M) \slash [\VectorFieldSpace_\sigma(M), \VectorFieldSpace_\sigma(M)] \ni \equivClass{X} \mapsto \equivClass{X \contr \sigma} \in \sCohomology^1(M, \R) \) is an isomorphism; see, \parencite[Theorem~1.4.3]{Banyaga1997}.
In particular, if \( \sCohomology^1(M, \R) \) is trivial, then a classical momentum map exists.
\end{remark}

\section{Applications}%
\label{sec::applications}

\subsection{Trivial bundle case}%
\label{sec:applications:trivialBundle}

Let \( M \) be a closed manifold endowed with a volume form \( \mu \). We are interested in the space of smooth maps \( \SectionSpaceAbb{F} = \sFunctionSpace(M, F) \) from \( M \) to a given Hamiltonian \( G \)-manifold \( (F, \omega, J) \). Of course, this set-up fits in the general theory developed in \cref{sec::groupMomentumMap,sec::globalAnalysisSymplecticFibreBundles} by thinking of \(\SectionSpaceAbb{F}\) as the space of sections of the trivial bundle \( M \times F \). 
The symplectic form on $\mathcal{F}$ is expressed in this case as
\begin{equation}
 	\Omega_\phi (Y_1, Y_2) = \int_M \omega_\phi (Y_1, Y_2) \mu.
 \end{equation}
We consider the two natural actions
\begin{equation}
\begin{split}
 \sFunctionSpace(M, G) \times \SectionSpaceAbb{F} &\to \SectionSpaceAbb{F}, 
 \qquad \psi \cdot \phi = \psi \cdot_G \phi, \\
\DiffGroup_\mu(M) \times \SectionSpaceAbb{F} &\to \SectionSpaceAbb{F}, 
\qquad \varphi \cdot \phi = \phi \circ \varphi^{-1},
\end{split}
\end{equation}
where $\cdot_G$ denotes the \( G \)-action on \( F \).
If \( (F, \omega) \) has a \( G \)-equivariant prequantum bundle 
\( (L, \vartheta) \), then, by \cref{thm::symplecticFibreBundle:momentumMapVolume}, the above actions 
have group-valued momentum maps 
\begin{equationcd}
&	\SectionSpaceAbb{F} \to[swap]{ld}{\SectionSpaceAbb{J}_G}\to{rd}
{\SectionSpaceAbb{J}_{\DiffGroup_\mu}}	& \\
\sFunctionSpace(M, \LieA{g}^*)	&		& \csCohomology^2(M, \UGroup(1))
\end{equationcd}
given by
\begin{equation}\label{eq:sec_4_1_momentum}
\quad\SectionSpaceAbb{J}_G(\phi) = J \circ \phi,
\qquad\qquad
\SectionSpaceAbb{J}_{\DiffGroup_\mu}(\phi) = - \phi^* \equivClass{L, \vartheta}.
\end{equation}
Moreover, if the symplectic form is exact, say \( \omega = 
\dif \vartheta \), then the \( \DiffGroup_\mu \)-action admits 
an ordinary momentum map
\begin{equation}\label{eq:trivialBundle:momentumMapExact}
\SectionSpaceAbb{J}_{\DiffGroup_\mu}: \SectionSpaceAbb{F} \to \DiffFormSpace^1(M) \slash \dif \DiffFormSpace^0(M), \qquad 
\phi \mapsto - \equivClass{\phi^* \vartheta}.
\end{equation}

\begin{remark}[Extension to the action of \( \DiffGroup_\omega(F) \)]%
\label{rem::extendMomentumMapOfGaugeGroupToSymplectomorphisms}
The group \( \DiffGroup_\omega(F) \) of symplectomorphisms leaves the symplectic form on \( \SectionSpaceAbb{F} \) invariant and thus it is natural to seek a momentum map for it. This corresponds to replacing \( \AutGroup_\mu(P) \) by \( \AutGroup_{\mu, \omega}(F) \) in the general theory outlined in the previous chapter. However, we are confronted with serious problems if we try to determine the momentum map in this generality: Such a momentum map would take values in the \emph{distributional} dual.
As shown in \parencite{Gay-BalmazVizman2012}, the action of the subgroup 
of Hamiltonian diffeomorphisms of \( F \) admits a momentum map 
\(\SectionMapAbb{J}:\SectionSpaceAbb{F}\to\HamVectorFields(F, \omega)^* \) 
satisfying \( \dualPair{X_h}{\SectionMapAbb{J}(\phi)} = 
\int_M (\phi^* h)\, \mu \), after choosing a basis in 
\( \sCohomology^0(F, \R) \) to identify Hamiltonian vector fields with 
smooth functions. However, the functional
\begin{equation}
T_\phi:\sFunctionSpace(F,\R) \to \R,\qquad h\mapsto \int_M (\phi^* h)\,\mu
\end{equation}
is usually not regular, \ie not represented by integration against a 
smooth function on \( F \).
For example, if \( M \) is the one-point space \( \set{x} \), then 
\( T_\phi \) is the delta distribution at the point \( \phi(x) \).
The distributional character of the momentum map for the 
\( \DiffGroup_\omega(F) \)-action is not surprising.
Conceptually, the momentum map is the push-forward of \( \mu \), and the 
push-forward of a differential form into a higher-dimensional manifold results in a distributional functional.

In order to apply the same strategy as in the general theory above, we would need a Lie group whose Lie algebra consists of distributions; or, even better, we would need \enquote{distributional line bundles}, \ie, singular bundles whose curvature is a distribution.
\end{remark}

\subsubsection{Hydrodynamics}%
\label{sec:applications:hydrodynamics}
The Lie group-valued momentum map can be used to construct a generalized 
Clebsch representation of a vector field that cannot be expressed in 
terms of classical Clebsch variables.

Let us start by recalling the usual setting.
Let \( M \) be a compact \( 3 \)-dimensional orientable boundaryless 
Riemannian manifold, oriented by the Riemannian volume form \( \mu \).
For simplicity, we also assume that \( M \) is connected.
Euler's equation for an incompressible perfect fluid on \( M \) is
\begin{equation}
	\label{eq:applications:hydrodynamics:euler}
	\difp_t v + \nabla_v v = - \grad p
\end{equation}
for \( v \in \VectorFieldSpace(M) \) subject to the incompressibility 
condition \( \difLie_v \mu = 0 \) (\ie, \( \operatorname{div}_\mu v = 0\)). 
By taking the Riemannian dual and using the identity 
\begin{equation}
\difLie_v (v^\flat ) = (\nabla_v v)^\flat + \frac{1}{2} \dif \norm{v}^2
\end{equation}
we can express Euler's equation as
\begin{equation}
\begin{cases}
\difp_t \alpha + \difLie_v \alpha = \dif \tilde{p}, \\
v^\flat = \alpha,
\end{cases}
\end{equation}
where \( \tilde{p} \defeq -p + \frac{1}{2} \norm{v}^2 \).
The Hodge decomposition \( \DiffFormSpace^1(M) = \ker d^* \oplus \img d \) 
(which is \(L^2\)-orthogonal) shows that the projection 
\( \pr_{\ker \dif^*} \) onto the kernel of the codifferential induces 
an isomorphism between divergence free vector fields and 
\( \DiffFormSpace^1(M) \slash \dif \DiffFormSpace^0(M) \).
The resulting isomorphism 
\begin{equation}
A: \VectorFieldSpace_\mu(M)\ni v \mapsto \equivClass[\big]{v^\flat} \in \DiffFormSpace^1(M) \slash \dif \DiffFormSpace^0(M) 
\end{equation}
is called the \emphDef{inertia operator}.
This allows us to view Euler's equation as a dynamical system on 
\( \DiffFormSpace^1(M) \slash \dif \DiffFormSpace^0(M) \):
\begin{equation}
\label{eq:applications:hydrodynamics:eulerModExact}
\begin{cases}
\difp_t \equivClass{\alpha} + \difLie_v \equivClass{\alpha} = 0 \\
A(v) = \equivClass{\alpha}.
\end{cases}
\end{equation}
It turns out that these equations are a Lie-Poisson system for the Hamiltonian \( \SectionMapAbb{H}(\equivClass{\alpha}) = \frac{1}{2} \int_M \norm[\big]{A^{-1} \equivClass{\alpha}}^2 \mu \), see \parencite{MarsdenWeinstein1983}.

Clebsch variables give a parametrization of \( \equivClass{\alpha} \in \DiffFormSpace^1(M) \slash \dif \DiffFormSpace^0(M) \) in such a way that Euler's equation is cast into a classical (\ie, symplectic) Hamiltonian form.
We refer the reader to \parencite{MarsdenWeinstein1983,Gay-BalmazVizman2012} for more details on the following geometric 
construction of the classical Clebsch representation.
Let \( (F, \omega = \dif \vartheta) \) be an exact symplectic manifold 
of dimension \( 2 \) and let \( \SectionSpaceAbb{F} \) denote the space 
of smooth functions from \( M \) to \( F \).
As we have seen in~\eqref{eq:trivialBundle:momentumMapExact}, the map
\begin{equation}
\SectionMapAbb{J}: \SectionSpaceAbb{F} \to \DiffFormSpace^1(M) \slash 
\dif \DiffFormSpace^0(M),
	\qquad
	\phi \mapsto - \equivClass{\phi^* \vartheta}
\end{equation}
is an equivariant momentum map for the action of \( \DiffGroup_\mu (M) \) on \( \SectionSpaceAbb{F} \). Since an equivariant momentum map is Poisson, \( \SectionMapAbb{J} \) maps solutions of the Hamiltonian system \((\SectionSpaceAbb{F},\Omega,\SectionMapAbb{H} \circ \SectionMapAbb{J})\) to solutions of Euler's equation. For the velocity field, \( \SectionMapAbb{J} \) gives the following representation in terms of Clebsch variables:
\begin{equation}
v^\flat = -\pr_{\ker \dif^*}\phi^* \vartheta = -\phi^* \vartheta + \dif f,
\end{equation}
where \( f \in \sFunctionSpace(M) \) ensures that the resulting vector field is divergence free (such a function always exists by the Hodge decomposition \( \DiffFormSpace^1(M) = \ker d^* \oplus \img d \)).
For example, in the case \( F = \R^2 \) with \( \vartheta = - x \dif y \) we obtain
\begin{equation}
	v^\flat = \phi_1 \dif \phi_2 + \dif f,
\end{equation}
which is the classical Clebsch representation for the velocity field \parencite{Clebsch1859,Clebsch1857}.

If a vector field admits a representation in terms of Clebsch variables, then its helicity has to vanish. Indeed, since $\dim F = 2$ and $\partial M = \varnothing$, we have 
\begin{equation}\label{eq:applications:hydrodynamics:helicityVanishesClebsch}
	\Hel(v) \defeq \int_M v^\flat \wedge \dif v^\flat
		= \int_M \phi^* (\vartheta \wedge \dif \vartheta)
		= 0.
\end{equation}
Helicity measures the linkage and/or knottedness of vortex lines.
Hence the helicity constraint implies that topological non-trivial vector fields do not admit a classical Clebsch representation.

We show in the rest of this subsection how the group-valued momentum map gives a generalized Clebsch representation for vector fields with integral helicity.
We continue to work in the general setting where \( (M, \mu) \) is a compact \( 3 \)-dimensional manifold endowed with a volume form \( \mu \) and where \( (F, \omega) \) is a \( 2 \)-dimensional symplectic manifold.
However, instead of requiring that \( \omega \) is exact, we now only assume that it has integral periods.
Thus, there exists a prequantum \( \UGroup(1) \)-bundle \( L \to F \) with connection \( \vartheta \) and curvature \( \omega \).
By~\eqref{eq:sec_4_1_momentum} in \cref{sec:applications:trivialBundle}, the group-valued momentum map for the \( \DiffGroup_\mu(M) \)-action is given by
\begin{equation}
\SectionMapAbb{J}: \SectionSpaceAbb{F} \to \csCohomology^2(M, \UGroup(1))
	\qquad
	\phi \mapsto - \phi^* \equivClass{L, \vartheta}.
\end{equation}
Suppose there exists \( \phi_0 \in \SectionSpaceAbb{F} \) such that the pull-back bundle \( \phi_0^* L \) over \( M \) is trivial.
Then also for every \( \phi \in \SectionSpaceAbb{F} \) in the connected component \( \SectionSpaceAbb{F}_{\phi_0} \) of \( \phi_0 \), the bundle \( \phi^* L \) is trivial, because the composition
\begin{equationcd}
	\SectionSpaceAbb{F}
		\to[r, "\SectionMapAbb{J}"]
	& \csCohomology^2(M, \UGroup(1))
		\to[r, "c"]
	& \sCohomology^2(M, \Z)
\end{equationcd}
is a continuous map.
Hence the restriction 
\(\restr{\SectionMapAbb{J}}{\SectionSpaceAbb{F}_{\phi_0}} \) of 
\( \SectionMapAbb{J} \) to \( \SectionSpaceAbb{F}_{\phi_0} \) takes 
values in \( c^{-1}(0) \isomorph \DiffFormSpace^1(M) \slash 
\clZDiffFormSpace^1(M) \).
By duality, we get the following \emph{generalized Clebsch representation} of a vector field
\begin{equation}
	\label{eq:applications:hydrodynamics:generalizedClebsch}
	v^\flat = - \phi^* \vartheta + \nu,
\end{equation} 
where \( \nu \in \clZDiffFormSpace^1(M) \) and, with a slight abuse of notation, \( \phi^* \vartheta \) denotes the \( 1 \)-form on \( M \) that is induced by the pull-back connection \( \phi^* \vartheta \) in the trivial bundle \( \phi^* L \).

We note that Euler's equation descend to \( \DiffFormSpace^1(M) \slash \clZDiffFormSpace^1(M) \).
In fact, the Lie derivative of a \( 1 \)-form \( \beta \) with integral periods is exact since
\begin{equation}
\int_\gamma \difLie_X \beta 
= \difFracAt{}{\varepsilon}{\varepsilon = 0} 
\underbrace{\int_{\flow^X_\varepsilon(\gamma)} \beta}_{\in \Z} = 0
\end{equation}
for every closed curve \( \gamma \) in \( M \).
Thus Euler's equation~\eqref{eq:applications:hydrodynamics:eulerModExact} implies that the class \( \equivClass{\beta} \in \clZDiffFormSpace^1(M) \slash \dif \DiffFormSpace^0(M) \isomorph \sCohomology^1(M, \Z) \) is constant in time, and that thus the interesting part of the motion takes place in \( \DiffFormSpace^1(M) \slash \clZDiffFormSpace^1(M) \).
We note, however, that it is only possible to define an inertia operator 
\( \DiffFormSpace^1(M) \slash \clZDiffFormSpace^1(M) \to 
\VectorFieldSpace_\mu(M) \) for small times since the principal bundle 
\( \sCohomology^1(M, \R) \to \sCohomology^1(M, \R) \slash 
\sCohomology^1(M, \Z) \) is only locally trivial.

In order to show that the helicity of a vector field admitting a generalized Clebsch representation has to be an integer, we start 
with the following general remarks about the helicity.
On abstract grounds, the helicity functional is defined by
\begin{equation}
\label{eq:applications:hydrodynamics:defHelicity}
\Hel: \DiffFormSpace^1(M) \ni \alpha \mapsto 
\int_M \alpha \wedge \dif \alpha \in \R.
\end{equation}
Note that \( \Hel \) vanishes on closed forms. If one forgets the hydrodynamical context for a moment, then one would mistake such an 
expression as the Abelian Chern--Simons functional of a trivial bundle.
In fact, the similarity between helicity and Chern--Simons theory was already pointed out previously in \parencite{LiuRicca2012}.
The Chern--Simons functional admits a well-known generalization to non-trivial circle bundles; see, \eg, \parencite[Eq.~1.28]{FreedMooreEtAl2006}.
In fact, for a differential character \( h \in \csCohomology^2(M, \UGroup(1)) \), the square \( h\star h\in\csCohomology^4(M, \UGroup(1)) \) is a group homomorphism \( \sCycles_3(M, \Z) \to \UGroup(1) \).
We can evaluate it, in particular, on the \( 3 \)-cycle \( M \).
Writing the evaluation as integration in order to keep the similarity with~\eqref{eq:applications:hydrodynamics:defHelicity}, we obtain the \emphDef{generalized helicity} (or generalized Chern--Simons functional)
\begin{equation}
\Hel: \csCohomology^2(M, \UGroup(1)) \ni h \mapsto \int_M h \star h \in 
\UGroup(1).
\end{equation}
The properties of the star product (see \cref{sec::differentialCharacter:starProduct}) yield the following commutative diagram connecting the ordinary helicity with its generalized version:
\begin{equationcd}[label=eq:applications:hydrodynamics:helicityRelation]
	\DiffFormSpace^1(M) \slash \clZDiffFormSpace^1(M)
		\to[r, "\Hel"]
		\to[d, "\iota"]
	& \R
		\to[d, "\exp"]
	\\
	\csCohomology^2(M, \UGroup(1))
		\to[r, "\Hel"]
	& \UGroup(1).
\end{equationcd}

We finally observe that the generalized helicity of elements in the image of the group-valued momentum map \( \SectionMapAbb{J} \) vanishes, due to a dimension count similar to the one leading to the identity~\eqref{eq:applications:hydrodynamics:helicityVanishesClebsch}, namely,
\begin{equation}\begin{split}
	(\Hel \circ \, \SectionMapAbb{J}) (\phi) 
		&= \Hel (- \phi^* \equivClass{L, \vartheta}) 
		= \int_M (\phi^* \equivClass{L, \vartheta}) \star 
		(\phi^* \equivClass{L, \vartheta}) 
		\\
		&= \int_M \phi^* (\equivClass{L, \vartheta} \star 
		\equivClass{L, \vartheta}) 
		= 1.
\end{split}\end{equation}
Hence, the commutative 
diagram~\eqref{eq:applications:hydrodynamics:helicityRelation} implies 
that the ordinary helicity of \( \phi^* \vartheta \) has to be an integer.
In other words, a vector field admitting a generalized Clebsch 
representation of the 
form~\eqref{eq:applications:hydrodynamics:generalizedClebsch} must 
have integral helicity.

We thus arrive at the following.
\begin{thm}
	Let \( (M, \mu) \) be a compact \( 3 \)-dimensional manifold endowed with a volume form \( \mu \).
	A solution \( v(t) \in \VectorFieldSpace_\mu(M) \) of Euler's equation~\eqref{eq:applications:hydrodynamics:euler} for an incompressible perfect fluid on \( M \) projects via \( \VectorFieldSpace_\mu(M)\ni v \mapsto \equivClass[\big]{v^\flat} \in \DiffFormSpace^1(M) \slash \clZDiffFormSpace^1(M) \) to a solution of the system
	\begin{equation}
		\label{eq:applications:hydrodynamics:eulerModClosedIntegral}
		\begin{cases}
			\difp_t \equivClass{\alpha} + \difLie_v \equivClass{\alpha} = 0, \\
			\equivClass{\alpha} = \equivClass[\big]{v^\flat}.
		\end{cases}
	\end{equation}
	on \( \DiffFormSpace^1(M) \slash \clZDiffFormSpace^1(M) \).

	Let \( (F, \omega) \) be a \( 2 \)-dimensional symplectic manifold that has a prequantum \( \UGroup(1) \)-bundle \( L \to F \) with connection \( \vartheta \).
	The natural action of \( \DiffGroup_\mu(M) \) on \( \SectionSpaceAbb{F} = \sFunctionSpace(M, F) \) preserves the symplectic form \( \mu \hatProduct \omega \) and has a group-valued momentum map
	\begin{equation}
		\SectionSpaceAbb{J}: \SectionSpaceAbb{F} \to \csCohomology^2(M, \UGroup(1)), \qquad \phi \mapsto - \phi^* \equivClass{L, \vartheta}.
	\end{equation}
	Moreover, assume there exists \( \phi_0 \in \SectionSpaceAbb{F} \) such that the pull-back bundle \( \phi_0^* L \) over \( M \) is trivial.
	Then the restriction of \( \SectionSpaceAbb{J} \) to the connected component \( \SectionSpaceAbb{F}_{\phi_0} \) of \( \phi_0 \) takes values in \( \DiffFormSpace^1(M) \slash \clZDiffFormSpace^1(M) \) and yields generalized Clebsch variables for the projected Euler equation~\eqref{eq:applications:hydrodynamics:eulerModClosedIntegral} in the sense of~\eqref{eq:applications:hydrodynamics:generalizedClebsch}.
	Every vector field admitting such a generalized Clebsch representation has integral helicity.
\end{thm}

\begin{example}
For concreteness, let \( M = S^3 \), and consider \(F = S^2 \) endowed 
with its usual normalized volume \( \omega \). We take the Hopf bundle 
\( (L, \vartheta) \) as the prequantization of \( (S^2, \omega) \).
Since \( S^3 \) has trivial cohomology in degree \( 2 \), the pull-back of \( L \) by every map \( \phi: S^3 \to S^2 \) is trivial.
Many interesting and topological non-trivial knots on \( S^3 \) can be constructed from the pull-back of \( \vartheta \) by a special choice of \( \phi \). For example, the Hopf vector field \parencite[Example~III.1.9]{ArnoldKhesin1998} corresponds to the classical Hopf fibration
\begin{equation}
\phi: \C^2 \supseteq S^3 \ni (z_1, z_2) \mapsto \equivClass{z_1 : z_2} 
\in \CP^1 \isomorph S^2.
\end{equation}
Other knots that can be realized in this way include the figure of eight knot, linked rings, and the trefoil knot, among others; see 
\parencite{KediaFosterEtAl2016} for details.
The helicity of such vector fields coincides with the Hopf invariant of 
\( \phi \).
Hence, these knots have a non-vanishing but integral helicity.
Hence they \emph{do not} admit a classical Clebsch representation, but have, by construction, a generalized Clebsch representation.
\end{example}

The following example suggests that not every vector field with integral helicity admits a generalized Clebsch representation and that one might need to replace the symplectic manifold \( F \) by a singular symplectic stratified space to ensure such a representation for every vector field with integral helicity.
\begin{example}
Consider the vector field \( v = A \sin z \difp_x + A \cos z \difp_y \) on \( \R^3 \).
Since \( v \) is \( 2 \pi \)-periodic, it naturally lives on the torus \( M = T^3 \).
Note that \( v \) is a special case of an Arnold--Beltrami--Childress (ABC) flow with \( B = 0 = C \) in the usual notation.
In particular, \( v \) provides an example for a simple steady-state solution of Euler's equation.
The vorticity of \( v \) is given by \( \omega = A \cos z \dif z \wedge \dif x + A \sin z \dif y \wedge \dif z \), so that its helicity is
\begin{equation}
\Hel(v) = \int_{T^3} v^\flat \wedge \omega = 
\int_{T^3} A^2 \vol = A^2 \vol(T^3).
	\end{equation}
 	Seen as a vector field on \( \R^3 \), we can represent \( v \) as
 	\begin{equation}
 		\label{eq:applications:hydrodynamics:ABC:representation}
 		v^\flat = f \dif g + \dif h, \qquad
 		\begin{cases}
 			f = y \sin z - x \cos z, \\
 			g = z,\\
 			h = x \sin z + y \cos z.
 		\end{cases}
 	\end{equation}
The pair \( (f, g) \), seen as a map \( \R^3 \to \R^2 \), is 
\( \Z^3 \)-equivariant with respect to the natural action on \( \R^3 \) and the twisted action on \( \R^2 \), where the group element \((n_1, n_2, n_3) \in \Z^3\) acts by
\begin{equation}
 		\R^2 \ni (u, v) \mapsto (2 \pi n_2 \sin v - 2 \pi n_1 \cos v, v + 2 \pi n_3) \in \R^2.
 	\end{equation}
Let \( F = \R^2 \slash \Z^3 \) be the orbit space. The 
representation~\eqref{eq:applications:hydrodynamics:ABC:representation} 
suggests that the induced map \( (f, g): \T^3 \to F \) yields generalized 
Clebsch variables for \( v \). Note, however, that the \( \Z^3 \)-action 
on \( \R^2 \) is not free and that the orbit space \( F \) thus has singularities.
\end{example}
\begin{remark}
	The classical Clebsch representation admits a natural generalization (sometimes called a Weber--Clebsch representation)
	\begin{equation}
		v^\flat = \sum_i \phi_i \dif \psi_i + \dif f,
	\end{equation}
	where one allows for multiple potentials, see \eg \parencite{GrahamHenyey2000}.
	In our setting, this corresponds to considering higher-dimensional symplectic manifolds \( F \).
	We claim that every vector field admits a generalized Clebsch representation if we allow higher-dimensional targets \( F \).
	To see this, let \( v \) be a vector field on \( M \) and let \( v^\flat \in \DiffFormSpace^1(M) \) be the associated \( 1 \)-form.
	We may view \( v^\flat \) as a connection \( 1 \)-form on the trivial bundle \( L = M \times \UGroup(1) \to M \).
	The theory of universal bundles \parencite[Theorem~1]{NarasimhanRamanan1961} entails that there exists a smooth map \( \phi: M \to \CP^N \) for sufficient large \( N \) such that the bundle \( (L, v^\flat) \) with connection is isomorphic to the pull-back of the canonical bundle over \( \CP^N \) endowed with the canonical connection \( \vartheta_{\textrm{can}} \).
	Hence, there exists a gauge transformation \( f \in \sFunctionSpace(M, \UGroup(1)) \) of \( L \) such that
	\begin{equation}
		v^\flat = \phi^* \vartheta_{\textrm{can}} + \difLog f ,
	\end{equation}
	where we again identified the connection \( \phi^* \vartheta_{\textrm{can}} \) in the trivial bundle \( L \) with a \( 1 \)-form on \( M \).
\end{remark}

\subsubsection[Lagrangian embeddings]{Isotropic and Lagrangian embeddings}
By a slight modification of the previous set-up,  the moduli space of Lagrangian embeddings can be realized as a symplectic quotient.  
Let \( (S, \mu) \) be a closed manifold with volume form, and let \( (M, \omega) \) be a symplectic manifold. 
Since the space of smooth embeddings of \( S \) into \( M \) is an open subset of \( \sFunctionSpace(S, M) \) 
\parencite[Theorem~44.1.]{KrieglMichor1997}, the symplectic form \( \Omega = \omega \hatProduct \mu \) on \( \sFunctionSpace(S, M) \) restricts to a symplectic form on \( \EmbeddingSpace(S,M) \).
We assume that \( (M, \omega) \) has a prequantization bundle \( (L, \vartheta) \). By~\eqref{eq:sec_4_1_momentum}, the momentum map for the action of \( \DiffGroup_\mu(S) \) is given by
\begin{equation}
\SectionSpaceAbb{J}_{\DiffGroup_\mu}: \EmbeddingSpace(S, M) \to 
\csCohomology^2(S, \UGroup(1)), \qquad \iota \mapsto - 
\iota^* \equivClass{L, \vartheta}.
\end{equation}
Thus, the inverse image under \( \SectionSpaceAbb{J}_{\DiffGroup_\mu} \) 
of the normal subgroup of all flat circle bundles on \( S \) consists precisely of isotropic embeddings.
For \( \dim(S) = \dim M \slash 2 \), the set \( \SectionSpaceAbb{J}_{\DiffGroup_\mu}^{-1}(\text{flat}) \) 
corresponds to Lagrangian embeddings.
More generally, the curvature \( -\iota^* \omega \) of 
\( \SectionSpaceAbb{J}_{\DiffGroup_\mu}(\iota) \) vanishes if and only 
if \( \iota \) is an isotropic embedding. Nonetheless, the Chern class 
of \( \iota^* \equivClass{L, \vartheta} \) may still be non-trivial in 
this case. The exact 
sequence~\eqref{eq::differentialCharacter:definingDiagram} implies that 
\( c(\iota^* \equivClass{L, \vartheta}) \) has to be completely torsion. 
We are not aware that this torsion class in \( \sCohomology^2(S, \Z) \) 
is discussed anywhere in the literature.

If the symplectic form is exact, say \( \omega = \dif \vartheta \), then 
the Chern class of \( \iota^* \equivClass{L, \vartheta} \) vanishes.
Nonetheless, the group-valued momentum map contains secondary topological 
data which corresponds to the well-known Liouville class as we will see 
now. Indeed, then the line bundle \( (L, \vartheta) \) is trivial so that 
the momentum map takes values in the subgroup of topological trivial
bundles. Hence \( \iota^* \equivClass{L, \vartheta} \) is completely 
characterized by its curvature \( \iota^* \omega \) and a class in the 
torus \( \sCohomology^1(S, \R) \slash \sCohomology^1(S, \Z) \); 
see~\eqref{eq::differentialCharacter:definingDiagram}.
For an isotropic embedding \( \iota \), the curvature \( \iota^* \omega \) 
vanishes and the topological class equals 
\(\equivClass{\iota^*\vartheta} \) and hence it coincides with the 
\emph{Liouville class} of the embedding \( \iota \) 
\parencite[p.~30]{BatesWeinstein1997}.

Recall that an isotropic embedding is called \emphDef{prequantizable} if 
its Liouville class vanishes; see 
\parencite[Definition~4.4]{BatesWeinstein1997}.
More generally, we say that an isotropic embedding is 
\emphDef{torsion-free} if the torsion class \( c(\iota^* \equivClass{L, 
\vartheta}) \) in \( \sCohomology^2(S, \Z) \) discussed above vanishes.
Summarizing, we arrive at the following correspondence:
\begin{center}
\small
\begin{tabular}{l l}
	\toprule
	Moduli space of 					&	Symplectic reduction \( \SectionSpaceAbb{J}_{\DiffGroup_\mu}^{-1}(\cdot) \slash \DiffGroup_\mu(S) \) at \\
	\midrule
	Lagrangian embeddings				&	flat bundles \\
	Torsion-free Lagrangian embeddings		&	flat and trivial bundles \\
	Prequantizable Lagrangian embeddings	&	0 \\
	\bottomrule
\end{tabular}
\end{center}
Here, we are a bit sloppy with the notion \enquote{moduli space}.
As we only take the quotient by the group \( \DiffGroup_\mu(S) \subseteq 
\DiffGroup(S) \), the resulting reduced space still remembers the volume
form. Hence, being more precise, the symplectic quotient 
\( \SectionSpaceAbb{J}_{\DiffGroup_\mu}^{-1}(\cdot) \slash 
\DiffGroup_\mu(S) \) yields the moduli space of Lagrangian (or isotropic) 
weighted submanifolds of type \( (S, \mu) \).

In the case of an exact symplectic form, a non-equivariant momentum map for the action of the diffeomorphism group \( \DiffGroup_\mu(S) \) is discussed in \parencite[Example~5.32]{McDuffSalamon1999}.
Recently, \textcite{Gay-BalmazVizman2016} showed that the quotient of isotropic embeddings modulo volume-preserving diffeomorphisms is indeed a symplectic manifold and that its connected components correspond to coadjoint orbits of \( \HamDiffGroup(M, \omega) \).
These results are obtained under the additional assumptions that the 
first cohomology group of \( S \) vanishes.
It would be of interest to study the symplectic quotient without this topological assumption and search for relations with coadjoint 
orbits of the whole symplectomorphism group.
In light of our refined results, one would like to conclude:
\begin{conjecture}
The moduli space of weighted prequantizable Lagrangian embeddings is a symplectic manifold whose connected components are related to the coadjoint orbits of the symplectomorphism group.
\end{conjecture}
Moreover, it is an interesting question whether the moduli space of (torsion-free) Lagrangian embeddings is also symplectic.
A systematic study along the same lines as \parencite{Gay-BalmazVizman2016} requires to develop the dual pair picture for group-valued momentum maps.

\subsection{Reduction of structure group}
\label{sec::reductionOfStructureGroup}

Geometric objects usually correspond to the reduction of a principal 
\( G \)-bundle to a Lie subgroup \( H \subseteq G \).
In this setting, the fiber model \( \FibreBundleModel{F} \) is the homogeneous space \( G \slash H \). A section \( \phi \) of the associated bundle \( F = P \times_G (G \slash H) \) results in a reduction of \( \pi: P \to M \) to the principal \( H \)-bundle 
\begin{equation}
\label{eq::reducedBundleDefinition}
Q_\phi \defeq \set{p\in P \given \phi(\pi(p)) = \equivClass{p, e H}} \to M.
\end{equation}
Conversely, every reduction of \( P \) to a \( H \)-subbundle defines a section of \( F \).

Let \( \mu \) be a volume form on the base manifold \( M \) of \( P \).
If \( \FibreBundleModel{F} = G \slash H \) is endowed with a \( G \)-invariant symplectic form \( \FibreBundleModel{\omega} \) then, according to \cref{prop::symplecticStructureOnSectionsOfSymplecticFibreBundle}, the space of \( H \)-reductions of \( P \) carries a natural symplectic structure \( \Omega = \mu \hatProduct \omega \) with \( \omega \) being 
the induced fiber symplectic structure on \( F \).
In this subsection, we discuss the content of \cref{thm::symplecticFibreBundle:momentumMapVolume} for this particular setting.

\subsubsection{Pull-back of prequantum bundles as associated bundles}
\label{subsubsec_pull_back_prequantum_bundles}

As in the general case, assume that \( (\FibreBundleModel{F} = 
G \slash H, \FibreBundleModel{\omega}, \FibreBundleModel{J}) \) 
is a Hamiltonian \( G \)-space.
Let \( \FibreBundleModel{L} \to \FibreBundleModel{F} \) be a 
\( G \)-equivariant prequantum bundle, and let 
\( L = P \times_G \FibreBundleModel{L} \).
As a preparation, we now discuss how the pull-back bundle \( \phi^* L \) 
can be identified with an associated bundle to \( Q_\phi \) for every 
section \( \phi \) of \( F \). Evaluating the momentum map 
\( \FibreBundleModel{J}: \FibreBundleModel{F} \to \LieA{g}^* \) 
at the point \( eH \) yields a Lie algebra homomorphism
\begin{equation}
\restr{\FibreBundleModel{J}(eH)}{\LieA{h}}: \LieA{h} \to \R.
\end{equation}
It is a well-known fact in prequantization theory 
\parencite[Theorem~5.7.1]{Kostant1970} that \( G \)-equivariant 
prequantizations over \( G\slash H \) are in one-to-one correspondence 
with Lie group homomorphisms \( \rho: H \to \UGroup(1) \) integrating 
\( \restr{\FibreBundleModel{J}(eH)}{\LieA{h}} \).
Indeed, given an equivariant prequantization bundle 
\( \FibreBundleModel{L} \), such a homomorphism can be extracted from 
the circle action in the fiber over \( e H \), that is
\begin{equation}
	h \cdot \FibreBundleModel{l}  = \FibreBundleModel{l} \cdot \rho(h), 
	\quad \FibreBundleModel{l} \in \FibreBundleModel{L}_{eH}.
\end{equation}
Conversely, a Lie group homomorphism \( \rho: H \to \UGroup(1) \) yields a \( G \)-equivariant prequantization bundle with connection as the 
associated bundle \( \FibreBundleModel{L} = G \times_\rho \UGroup(1) \) to the canonical principal \( H \)-bundle \( G \to G \slash H \).
The \( G \)-action on \( \FibreBundleModel{L} \) is given by left translation in the \( G \)-factor, \ie, \( g \cdot \equivClass{a, z} = \equivClass{ga,z} \), and the connection \( 1 \)-form \( \FibreBundleModel{\vartheta} \) on \( \FibreBundleModel{L} \) is defined by
\begin{equation}
\FibreBundleModel{\vartheta}_{\equivClass{a,z}} \bigl(\equivClass{a \ldot A, v}\bigr) 
= \dualPair{\FibreBundleModel{J}(eH)}{A} + v,
\end{equation}
where \( A \in \LieA{g} \) and \( v \in \TBundle_z \UGroup(1) \isomorph \R \).

Now we establish a natural isomorphism \( \phi^* L \isomorph Q_\phi \times_\rho \UGroup(1) \).
Since the left \( G \)- and the right \( H \)-actions on \( G \) commute, we can identify the bundle \( L = P \times_G \FibreBundleModel{L} \to F \) 
with the quotient of \( P \times G \times \UGroup(1) \) by the simultaneous actions of \( G \) and \( H \), namely,
\begin{equation}
(p, a, z) \cdot (g, h) \defeq \bigl(p \cdot g, g^{-1} a h, \rho(h^{-1}) z\bigr).
\end{equation}
From this viewpoint, the projection \(L \to F\) maps \(\equivClass{p,a,z}\) to \( \equivClass{p, aH} \).
Thus, for a section \( \phi \) of \( F \to M \), the pull-back bundle is identified with
\begin{equation}
	\phi^* L 
		= \set{(m, \equivClass{p, a, z}) \in M \times L \given \phi(m) = \equivClass{p, aH}}.
\end{equation}
Recall that \( Q_\phi \) denotes the reduced \( H \)-bundle defined in~\eqref{eq::reducedBundleDefinition}.
\begin{lemma}
\label{prop::applications:reductionStructureGroup:indenificationPullbackBundle}
The map
\begin{equation}\
\label{eq::prequantization:pullbackBundleIsomorphismsAssociatedReduced}
	\phi^* L \to Q_\phi \times_\rho \UGroup(1), \qquad 
	(m,\equivClass{p,a,z}) \mapsto \equivClass{p \cdot a, z}
\end{equation}
is a principal bundle isomorphism over \( M \).
\end{lemma}
\begin{proof}
First, note that \( p \cdot a \in Q_\phi \), because
\begin{equation}
	\phi(\pi(p\cdot a))
		=\phi(m)=\equivClass{p,aH}
		= \equivClass{p \cdot a, eH}.
\end{equation} 
Second, the definition of the map in~\eqref{eq::prequantization:pullbackBundleIsomorphismsAssociatedReduced} 
does not depend on the chosen representative.
Indeed, the point \((m, \equivClass{p\cdot g, g^{-1}ah, \rho(h^{-1})z})\) is sent under~\eqref{eq::prequantization:pullbackBundleIsomorphismsAssociatedReduced} to \(\equivClass{p\cdot g g^{-1}ah,\rho(h^{-1})z}=\equivClass{p\cdot a,z}\).
Third, the vector bundle map
\begin{equation}
	Q_\phi \times_\rho \UGroup(1) \to \phi^* L, \qquad \equivClass{q, z} 
	\mapsto (\pi (q), \equivClass{q, e, z})
\end{equation}
is clearly inverse to~\eqref{eq::prequantization:pullbackBundleIsomorphismsAssociatedReduced}, 
so that we obtain a principal bundle isomorphism \( \phi^* L \isomorph Q_\phi \times_\rho \UGroup(1) \), as claimed.
\end{proof}
The next lemma shows that this isomorphism also identifies the natural 
connections on both bundles. Recall from~\eqref{eq:connectionOnL} the 
connection \( \vartheta^\Gamma \) on \( L \) constructed from a 
connection \( \Gamma \) in \( P \) and the connection 
\( \FibreBundleModel{\vartheta} \) on \( \FibreBundleModel{L} \).

\begin{lemma}%
\label{prop::applications:reductionStructureGroup:indenificationPullbackPrequantum}
Let \( \Gamma \) be a principal connection on \( P \) and assume that it reduces to a connection on \( Q_\phi \). Then the pull-back of the connection \( \vartheta^\Gamma \) from \( L \) to \( \phi^* L \) coincides with the induced connection \( \rho_* \Gamma \) on the associated bundle \( Q_\phi \times_\rho \UGroup(1) \) under the identification \( \phi^* L \isomorph Q_\phi \times_\rho \UGroup(1) \) of \cref{prop::applications:reductionStructureGroup:indenificationPullbackBundle}. 
\end{lemma}
\begin{proof}
We have seen that \( L \) can be realized as the quotient of \( P \times G \times \UGroup(1) \) by the simultaneous actions of \( G \) and \( H \).
From this point of view, the connection \( \vartheta^\Gamma \) on \( L \)  takes the form
\begin{equation}\begin{split}
\vartheta^\Gamma_{\equivClass{p, a, z}}\bigl(\equivClass{Z^\Gamma_p + 
p \ldot B, a \ldot A, v}\bigr) 
&= \FibreBundleModel{\vartheta}_{\equivClass{a,z}}\bigl(B \ldot\equivClass{a,z}  
+ \equivClass{a \ldot A, v}\bigr)
		\\
&= \FibreBundleModel{\vartheta}_{\equivClass{a, z}} \bigl(\equivClass{a 
		\ldot \AdAction^{-1}_a B + a \ldot A, v}\bigr)
		\\
&= \dualPair{\FibreBundleModel{J}(eH)}{\AdAction^{-1}_a B + A} + v,
\end{split}
\end{equation}
where \( p \in P \), \( a \in G \), \( z \in \UGroup(1) \), 
\( Z_p^\Gamma \in \TBundle_p P \) is horizontal, \( B \in \LieA{g} \) 
and \( z \in \TBundle_z \UGroup(1) \isomorph \R \).
Note that the following diagram of fiber bundle morphisms is commutative
\begin{equationcd}
	Q_\phi \times_\rho \UGroup(1) \to[r] 
		& \phi^* L \to[r]
		& L
		\\[-5ex]
	\equivClass{q, z} \mapsto{r} \mapsto{d}
		& (\pi(q), \equivClass{q, e, z}) \mapsto{r} \mapsto{d}
		& \equivClass{q, e, z} \mapsto{d}
		\\
	\pi(q) \mapsto{r}
		& \pi(q) \mapsto{r}
		& \equivClass{q, eH}
		\\[-5ex]
	M \to[r, "\id"]
		& M \to[r, "\phi"]
		& F
\end{equationcd}
Thus the pull-back of \( \vartheta^\Gamma \) to \(Q_\phi\times_\rho\UGroup(1) \) is given by
\begin{equation}
(\phi^* \vartheta^\Gamma)_{\equivClass{q, z}}\bigl(\equivClass{Z_q^\Gamma + q \ldot B, v}\bigr) = 
\vartheta^\Gamma_{\equivClass{q, e, z}}\bigl(\equivClass{Z^\Gamma_q + 
q \ldot B, 0, v}\bigr)
= \dualPair{\FibreBundleModel{J}(eH)}{B} + v.
\end{equation}
On the other hand, recall that \( \restr{\FibreBundleModel{J}(eH)}{\LieA{h}} \) is, by definition, the derivative of the Lie group homomorphism \( \rho \).
Hence, the connection \( \rho_* \Gamma \) on \( Q_\phi \times_\rho \UGroup(1) \) induced by \( \Gamma \) satisfies
\begin{equation}
	(\rho_* \Gamma)_{\equivClass{q, z}}\bigl(\equivClass{Z_q^\Gamma + q \ldot B, v}\bigr)
		= \tangent_e \rho (B) + v = \dualPair{\FibreBundleModel{J}(eH)}{B} + v.
\end{equation}
Thus, the connections \( \phi^* \vartheta^\Gamma \) and \( \rho_* \Gamma \) coincide.
\end{proof}

Thus, we get the following as a special case of \cref{thm::symplecticFibreBundle:momentumMapVolume}.
\begin{thm}[Momentum map for reductions of structure group]
\label{prop::applications:reductionStructureGroup:momentumMap}
Let \( \pi: P \to M \) be a finite-dimensional principal \( G \)-bundle over the closed volume manifold \( (M, \mu) \) and assume that \( \DiffGroup_{\mu, P}(M) = \DiffGroup_{\mu}(M) \).
Let \( H \) be a Lie subgroup of \( G \) and let \( \FibreBundleModel{\omega} \) be a \( G \)-invariant symplectic form on \( G \slash H \) with equivariant momentum map \( \FibreBundleModel{J}: G \slash H \to \LieA{g}^* \).
Then the following holds:
\begin{enumerate}
\item
The space \( \SectionSpaceAbb{F} \) of reductions of \( P \) to a principal \( H \)-subbundle is a symplectic Fr\'echet manifold with weak symplectic form \( \Omega = \mu \hatProduct \omega \), where \( \omega \) is the fiber symplectic form on \( P \times_G (G \slash H) \) induced by \( \FibreBundleModel{\omega} \).

\item
The group \( \AutGroup_\mu(P) \) of  bundle automorphisms, whose induced base diffeomorphisms are volume-preserving, acts symplectically on \( (\SectionSpaceAbb{F}, \Omega) \).
Moreover, for a principal connection \( \Gamma \) in \( P \) that is reducible to each \( Q_\phi \) and for a Lie group homomorphism \( \rho: H \to \UGroup(1) \) integrating \( \restr{\FibreBundleModel{J}(eH)}{\LieA{h}} \), the \( \AutGroup_\mu(P) \)-action on \( \SectionSpaceAbb{F} \) has a group-valued momentum map \( \SectionSpaceAbb{J}_{\AutGroup}: \SectionSpaceAbb{F} \to \CoGauAlgebra(P) \times 
\csCohomology^2(M, \UGroup(1)) \) given by
\begin{equation}
\SectionSpaceAbb{J}_{\AutGroup}(\phi)= \Bigl(J_* \phi, \equivClass*{Q_\phi \times_{\rho^{-1}} \UGroup(1), (\rho^{-1})_* \Gamma}\Bigr),
\end{equation}
where \( J: P \times_G (G \slash H) \to \CoAdBundle P \) is the bundle map induced by \( \FibreBundleModel{J} \).
The momentum map \( \SectionSpaceAbb{J}_{\AutGroup} \) is equivariant relative to the natural actions of \( \AutGroup_\mu(P) \) (see \cref{ex:dualPairLieGroups:volumePreservingDiffeos:coconjugation}).

\item
A lift \( \chi: \DiffGroup_\mu(M) \to \AutGroup_\mu(P) \) of the group of volume-preserving diffeomorphisms to principal bundle automorphisms yields a symplectic group action of \( \DiffGroup_\mu(M) \) on \( \SectionSpaceAbb{F} \) with group-valued momentum map \( \SectionSpaceAbb{J}_{\DiffGroup}: \SectionSpaceAbb{F} \to \csCohomology^2(M, \UGroup(1)) \) given by
\begin{equation}
\SectionSpaceAbb{J}_{\DiffGroup}(\phi) = \equivClass*{Q_\phi \times_{\rho^{-1}} \UGroup(1), (\rho^{-1})_* \Gamma} + 
\iota \circ \tau_{\Gamma, \chi}^* \circ J_* (\phi),
	 \end{equation}
where \( \iota: \csAlgebra^2(M, \UGroup(1)) \to \csCohomology^2(M, \UGroup(1)) \) and \( \tau_{\Gamma, \chi}^*: \CoGauAlgebra(P)\to \csAlgebra^2(M, \UGroup(1))\) are the same maps as in \cref{thm::symplecticFibreBundle:momentumMapVolume}.
\qedhere
\end{enumerate}
\end{thm}
Clearly there is also a version of this theorem for symplectic base manifolds similar to \cref{thm::symplecticFibreBundle:momentumMapSymplectic}.
We leave it to the interested reader to spell out the details.

In the following, we are interested in the special case where \( P \) is the frame bundle of \( M \) so that there is a canonical lift of diffeomorphisms of \( M \) to bundle automorphisms of \( P = \FrameBundle M \).
Thus, this setting falls in the realm of the last point in \cref{prop::applications:reductionStructureGroup:momentumMap}.

\subsubsection{K\"ahler geometry}
Let \( (M, \sigma) \) be a symplectic $2n$-dimensional manifold. 
We are interested in the space \( \SectionSpaceAbb{I} \) of almost complex structures compatible with \( \sigma \), \ie, almost complex structures 
\( J \) for which \( \sigma(\cdot, J \cdot) \) is a Riemannian metric 
on \( M \).
Each compatible almost complex structure on \( M \) gives rise to an almost K\"ahler manifold structure on \( M \) and thus reduces the symplectic frame bundle \( \FrameBundle M \to M \) to an \( \UGroup(n) \)-bundle. Hence \( \SectionSpaceAbb{I} \) is identified 
with the space of sections of \( \FrameBundle M \times_{\SpGroup(2n, \R)} \FibreBundleModel{F} \) with fiber model \( \FibreBundleModel{F} = \SpGroup(2n, \R) \slash \UGroup(n) \).
This homogeneous space \( \FibreBundleModel{F} \) can be identified with the Siegel upper half space and thus is a symplectic manifold.

To get some additional insight in the problem, we will follow \parencite{Ohsawa2015,SkerrittVizman2018} and realize the Siegel upper 
half space as a 
symplectic quotient.
Consider the vector space \( \MatrixSpace(2n \times 2n, \R) \) of 
real \( 2n \times 2n \) matrices endowed with the constant symplectic 
form
\begin{equation}
 	\omega(X, Y) = \tr (X^\T \symplMatrix Y),
\end{equation} 
where \( \symplMatrix \) denotes the standard symplectic matrix (whose right upper \( n \times n \)-corner is the identity).
Left multiplication by \( \SpGroup(2n, \R) \) leaves the symplectic form invariant.
Similarly, \( \OGroup(2n) \) acts by right multiplication by the inverse (so it is also a left action), and this action is symplectic as well.
These actions possess the momentum maps 
\begin{equationcd}
 	\SpAlgebra(2n, \R)
& \MatrixSpace(2n\times 2n,\R)\to[r,"J_{\OGroup}"]\to[l,"J_{\SpGroup}", swap]
 		& \OAlgebra(2n) \\[-5ex]
 	- X X^\T \symplMatrix
 		& X \mapsto{l}\mapsto{r}
 		& X^\T \symplMatrix X,
\end{equationcd} 
where we identified the duals with the Lie algebras itself using the bilinear form \( \frac{1}{2} \tr(X Y^\T) \).
Note that \( \UGroup(n) \) is the stabilizer of \( \symplMatrix \) under the (co)adjoint action of either \( \SpGroup(2n, \R) \) 
or \( \OGroup(2n) \). Hence, we identify the symplectic quotients at 
\( \symplMatrix \) with the spaces
\begin{equation}\label{eq:applications:kaehler:reducedSpacesModel}
	J_{\SpGroup}^{-1}(\symplMatrix) \slash \SpGroup(2n, \R)_\symplMatrix 
		= \OGroup(2n) \slash \UGroup(n), 
	\qquad 
	J_{\OGroup}^{-1}(\symplMatrix) \slash \OGroup(2n)_\symplMatrix
		= \SpGroup(2n, \R) \slash \UGroup(n).
\end{equation}
It is intriguing that this duality on the level of momentum maps and 
symplectic quotients corresponds to the two ways to get an almost 
K\"ahler manifold: start either from a symplectic manifold or a 
Riemannian manifold, and then choose a compatible almost complex structure.

Using the momentum maps, we can determine the prequantizations. 
For this, recall from \cref{subsubsec_pull_back_prequantum_bundles} that equivariant prequantization bundles over the homogeneous space \( G \slash H \) are in bijective correspondence with Lie group homomorphisms \( \rho: H \to \UGroup(1) \) integrating the Lie algebra homomorphism \( \restr{J(eH)}{\LieA{h}}: \LieA{h} \to \R \). In the present setting, we have \( J_{\SpGroup}(e \UGroup(n)) = - \symplMatrix = - J_{\OGroup}(e \UGroup(n)) \). 
The restriction of the corresponding functionals to \( \UAlgebra(n) \) yields the Lie algebra homomorphism \( \mp \frac{1}{2} \tr(\symplMatrix \, \cdot): \UAlgebra(n) \to \R \), which equals \( \pm \tr_\C: \UAlgebra(n) \to \I \R \) when \( \UAlgebra(n) \) is viewed as a subalgebra of complex matrices. Thus the integrating group homomorphism  is given by
\begin{equation}
\label{eq::Kaehler:characterPrequantization}
\rho_\pm: \UGroup(n) \to \UGroup(1), \qquad A \mapsto 
(\det\nolimits_\C A)^{\pm 1}.
\end{equation}
In our sign convention, \( \rho_+ \) corresponds to \( J_{\OGroup} \)
while \( \rho_- \) integrates \( J_{\SpGroup} \).

Let us now return to the non-linear setting given by a symplectic 
manifold \( (M, \sigma) \) with associated symplectic frame bundle 
\( \FrameBundle M \to M \). 
For every almost complex structure \( I \) on \( M \), let 
\( \FrameBundle_I M \) be the corresponding \( \UGroup(n) \)-reduction 
of \( \FrameBundle M \) and let \( \Gamma_I \) be the Levi-Civita connection, viewed as a connection on \( \FrameBundle_I M \).
With a slight abuse of notation, we also denote the extended connection on \( \FrameBundle M \) by \( \Gamma_I \).
The set \( \SectionSpaceAbb{I} \) of almost complex structures compatible with \( \sigma \) is identified with the space of sections of \( F = \FrameBundle M \times_{\SpGroup(2n, \R)} \FibreBundleModel{F}\) with 
\( \FibreBundleModel{F} = \SpGroup(2n, \R) \slash \UGroup(n) \).
As discussed above, the fiber model \( \FibreBundleModel{F} \) is a symplectic manifold and hence our general theory implies that 
\( \SectionSpaceAbb{I} \) is a symplectic manifold and that the group 
of symplectomorphisms acts symplectically on it.
According to \cref{prop::applications:reductionStructureGroup:momentumMap}, 
the momentum map for the action of the group of symplectomorphisms involves 
the associated bundle \( (\FrameBundle_I M \times_{\rho_{-}^{-1}} 
\UGroup(1), (\rho_{-}^{-1})_* \Gamma_I) \).
The group homomorphism \( \rho_-^{-1}: \UGroup(n) \to \UGroup(1) \) 
is the determinant and thus the associated bundle 
\( \FrameBundle_I M \times_{\rho_{-}^{-1}} \UGroup(1) \) coincides with 
the canonical circle bundle\footnotemark{} \( \KBundle_I M \) endowed 
with its natural Chern connection determined by the almost complex structure \( I \).
\footnotetext{Given an almost complex structure \( I \) on \( M \), the 
complex line bundle \( \ExtBundle^{n,0} M \) of 
holomorphic forms is called the \emphDef{canonical line bundle} and its 
dual \( \ExtBundle^{0,n} M \) the \emphDef{anti-canonical line bundle}.
The volume form \( \frac{\sigma^n}{n!} \) is a non-vanishing section 
of \( \ExtBundle^{n,n} M = \ExtBundle^{n,0} M \otimes \ExtBundle^{0,n} M \) and 
thus can be viewed as a Hermitian metric on the (anti-)canonical bundle.
The associated Hermitian frame bundles \( \KBundle_I M \) and 
\(\KBundle^{-1}_I M\) are principal \(\UGroup(1)\)-bundles 
and are called the \emphDef{canonical}, respectively, 
\emphDef{anti-canonical circle bundles}. }
Moreover, note that the adjoint of \( \tau_{\Gamma_I, \chi}: 
\VectorFieldSpace_{\sigma}(M) \to \GauAlgebra(\FrameBundle M) \) appears 
in the construction of the momentum map.
As shown in \cref{rem:symplecticFibreBundle:tauMapVolume}, \( \tau_{\Gamma_I, \chi}^*(s) \) involves the covariant derivative of \( s = J_*(I) \) and thus this term vanishes, because \( I \) is parallel with respect to \( \Gamma_I \).
To summarize, we obtain an important special case of \cref{prop::applications:reductionStructureGroup:momentumMap} (in its symplectic version analogous to \cref{thm::symplecticFibreBundle:momentumMapSymplectic}).
\begin{thm}
Let \( (M, \sigma) \) be a compact symplectic manifold of dimension 
\( 2n \).
Then the set \( \SectionSpaceAbb{I} \) of almost complex structures compatible with \( \sigma \) is naturally endowed with the structure of a symplectic Fréchet manifold.
If \( (M, \sigma) \) has a prequantum bundle \( h_\sigma \in 	
\csCohomology^2(M, \UGroup(1)) \), then the symplectic action of 
\( \DiffGroup_\sigma(M) \) on \( \SectionSpaceAbb{I} \) has a group-valued momentum map
\begin{equation}
\SectionSpaceAbb{J}_{\DiffGroup}:
\SectionSpaceAbb{I} \to \csCohomology^{2n}(M, \UGroup(1)),
			\qquad 
I \mapsto \KBundle_I M \star h_\sigma^{n-1},
\end{equation}
where \( \KBundle_I M \) is the canonical bundle, viewed as an element of \( \csCohomology^2(M, \UGroup(1)) \).
\end{thm}

By definition, the Chern--Ricci form \( \Ric_I \) is the curvature of 
the Chern connection on the anti-canonical bundle.
If \( I \) is integrable, then \( \Ric_I \) coincides with the Ricci form of the associated K\"ahler metric. We hence have
\begin{equation}\label{eq:applications:Kaehler:curvatureMomentumMap}
	\curv \circ \, \SectionSpaceAbb{J}_{\DiffGroup}:
		\SectionSpaceAbb{I} \to \DiffFormSpace^{2n}(M, \UGroup(1)), 
		\qquad 
	I \mapsto - \Ric_I \wedge \sigma^{n-1} = - \frac{S_I}{n} \sigma^n,
\end{equation}
where \( S_I = \Tr_\sigma \Ric_I \) is the Chern scalar curvature.
It has been shown by \textcite[Proposition~9]{Donaldson1997} that, under 
rather restrictive topological assumptions, the map \( I \mapsto S_I \) 
yields a momentum map for the \( \DiffGroup_\sigma(M) \)-action on 
\( \SectionSpaceAbb{I} \).
As in the examples above, the curvature momentum map is, in general, 
not the momentum map for the full symplectomorphism group but only for 
the subgroup of \emph{Hamiltonian symplectomorphisms}.
Moreover, the group-valued momentum map contains additional topological 
information: the Chern class of the bundle \( \KBundle^{-1}_I M \star 
h_\sigma^{n-1} \), which is given as the cup-product \( c_I(M) \cup 
c(h_\sigma)^{n-1} \) of the Chern class \( c(h_\sigma) \) with the first 
Chern class \( c_I(M) \defeq c(\KBundle_I M) \in \sCohomology^2(M, \Z) \) 
of the almost complex manifold \( (M, I) \).
In real cohomology, this class simplifies to \( c_I(M) \cup \equivClass{\sigma}^{n-1} \). 

\Cref{eq:applications:Kaehler:curvatureMomentumMap} shows that the curvature of the momentum map is essentially the scalar curvature.
Hence, taking different symplectic quotients yield various moduli spaces
 of almost K\"ahler metrics with prescribed scalar curvature.
The following cases are of special interest.
\begin{enumerate}
	\item For \( c \in \csCohomology^{2n}(M, \UGroup(1)) \), the level set \( \SectionSpaceAbb{J}_{\DiffGroup}^{-1}(c) \) consists of all almost complex structures \( I \) whose canonical bundle satisfies \( \KBundle_I M \star h_\sigma^{n-1} = c \).
	\item Recall from the exact sequence~\eqref{eq::differentialCharacter:definingDiagram} that the set \( c + \sCohomology^{2n-1}(M, \UGroup(1)) \) comprises all differential characters with the same curvature as \( c \). Thus the inverse image \( \SectionSpaceAbb{J}_{\DiffGroup}^{-1}\bigl(c + \sCohomology^{2n-1}(M, \UGroup(1))\bigr) \) is the set of all almost complex structures \( I \) with prescribed scalar curvature \( S_I = -n \frac{\curv c}{\sigma^n} \).
	\item Similarly, the set \( c + \JacTorus^{n-1}(M) \) parametrizes differential characters having the same curvature and topological class as \( c \). Thus the inverse image \( \SectionSpaceAbb{J}_{\DiffGroup}^{-1}\bigl(c + \JacTorus^{n-1}(M))\bigr) \) is identified with the set of all almost complex structures whose scalar curvature is prescribed by the above relation and, in addition, for which \( c_I(M) \cup c(h_\sigma)^{n-1} \) coincides with the Chern class of \( c \). 
\end{enumerate}
Ignoring the particularities of the infinite-dimensional setting, one 
would thus conjecture that all the corresponding moduli spaces are symplectic manifolds or at least symplectic stratified spaces.
\begin{remark}[K\"ahler--Einstein metrics]
If the curvature of \( c \) is a constant multiple of \( \sigma^n \), then the inverse image \( \SectionSpaceAbb{J}_{\DiffGroup}^{-1}(c + \JacTorus^{n-1}(M)) \) is equal to the set of K\"ahler--Einstein metrics. Indeed, an (almost) complex structure \( I \) whose Ricci form satisfies the Einstein relation \( \Ric_I = \lambda \sigma \) has constant scalar curvature \( S_I = n \lambda \).
Conversely, every \emph{integrable} complex structure \( I \) with constant scalar curvature and whose Chern class satisfies 
\( c_I(M) = \lambda \equivClass{\sigma} \) in real cohomology is automatically K\"ahler--Einstein. Indeed, in this case, \( \Ric_I \) and \( \lambda \sigma \) are two harmonic forms representing the same cohomology class and thus are equal. The issue of existence of K\"ahler-Einstein metrics has been an important subject of study for decades and was recently settled with the resolution of the Yau--Tian--Donaldson conjecture (we refer the reader to the recent review \parencite{Donaldson2015} and references therein for more details).
The infinite-dimensional symplectic picture we have just developed suggests that K\"ahler metrics with constant scalar curvature are 
actually a more natural object to study, especially since their moduli spaces are symplectic quotients. However, the existence problem is harder than for K\"ahler--Einstein metrics, and thus it is not clear how the analytic question concerning the infinite-dimensional symplectic reduction pan out. 
\end{remark}
\begin{remark}[Relation to GIT quotient]
It is a recurrent theme in K\"ahler geometry that the existence of certain special metrics (K\"ahler-Einstein, positive Ricci curvature, \ldots) has an algebraic stability obstruction.
Donaldson \parencite{Donaldson2001,Donaldson1997} related this geometric-algebraic equivalence to the interplay of the symplectic 
momentum map geometry with geometric invariant theory.
The fundamental idea is to detect zeros of the momentum map using an algebraic condition phrased in terms of the prequantum bundle.
 However, this approach relies on the 
existence of a classical momentum map and thus requires certain topological classes to vanish. In light of our refined results, it would be 
interesting to see how the group-valued momentum map interacts with geometric invariant theory. In this direction, 
\parencite{DiezJanssensNeebVizmann2019} shows how the hat product of differential characters can be used to construct a prequantization 
bundle over certain infinite-dimensional function spaces.
However, it is not yet clear what the existence of a group-valued momentum map means in terms of the prequantum geometry.
\end{remark}

\begin{remark}
The anti-canonical bundle often plays the role of the algebraic counterpart to a geometric property. For example, a compact complex 
manifold admits a K\"ahler metric of positive Ricci curvature only if it has an ample anti-canonical bundle (by Yau's proof of the Calabi conjecture combined with Kodaira's theorem). Our general group-valued momentum map picture explains, at least in principle, the importance of the (anti-)canonical bundle in many important questions concerning the geometry of (almost) K\"ahler manifolds: \( \KBundle_I M \) is the \enquote{conserved quantity} of the group of symplectomorphisms.
\end{remark}

Recall that we realized the Siegel upper half space \( \SpGroup(2n, \R) \slash \UGroup(n) \) as a symplectic quotient. This additional structure transfers to the space of sections, and puts the moduli space of constant scalar curvature metrics in the framework of reduction by stages.
For the moment, let us return to the following general set-up.
Let \( (\FibreBundleModel{F}, \FibreBundleModel{\omega}) \) be a finite-dimensional symplectic manifold. Consider two Lie groups \( G \) and \( K \) that act on \( \FibreBundleModel{F} \) from the left and preserve the symplectic form.
In the K\"ahler setting, we have \(\FibreBundleModel{F}=\GLGroup(2n,\R)\), \( G = \OGroup(2n) \), and \( K = \SpGroup(2n, \R) \).
Assume that the two actions are free, proper, commute, and have momentum maps
\begin{equationcd}
 	\LieA{g}^*	&	\FibreBundleModel{F}\to[swap]{l}
	{\FibreBundleModel{J}_G}\to{r}{\FibreBundleModel{J}_K}	
	&	\LieA{k}^*.
\end{equationcd}
Moreover, we assume that \( \FibreBundleModel{J}_K: \FibreBundleModel{F} \to \LieA{k}^* \) is invariant under \( G \) and that, vice-versa, \( \FibreBundleModel{J}_G \) is invariant under \( K \).
Hence we are in the setting of commuting reduction by stages \parencite[Section~15.5]{MarsdenMisiolekEtAl2007}.
Let \( \mu \in \LieA{g}^* \). The general theory \parencite[Lemma~4.1.2]{MarsdenMisiolekEtAl2007} tells us that the symplectic reduction by \( G \) yields the first stage reduced Hamiltonian \( K \)-space \( (\FibreBundleModel{F}_\mu, \FibreBundleModel{\omega}_\mu, \FibreBundleModel{J}_\mu) \).
The momentum map of the reduced space is induced by \( \FibreBundleModel{J}_K: \FibreBundleModel{F} \to \LieA{k}^* \) in the obvious way. Reducing once more by the remaining \(K\)-action gives a symplectic quotient that is symplectically diffeomorphic to the corresponding symplectic quotient of the product group \(G \times K\).

Let us now discuss the corresponding story in the bundle picture and 
for the spaces of sections.
For this, let \( \pi: P \to M \) be a principal \( K \)-bundle.
The spaces occurring in the symplectic reduction diagram
\begin{equationcd}
	\FibreBundleModel{J}^{-1}_G(\mu) \toInject{r}\toSurject{d}{\slash \, G}
		&	\FibreBundleModel{F} \\
	\FibreBundleModel{F}_\mu
		&
\end{equationcd}
all carry a natural \( K \)-action such that the inclusion and projection maps are equivariant.
Thus we obtain a similar diagram for the section spaces of the associated bundles
\begin{equationcd}
\sSectionSpace\bigl(P \times_K \FibreBundleModel{J}^{-1}_G (\mu)\bigr) \toInject{r}
\toSurject{d}{\slash \, \sFunctionSpace(M, G)}
		&	\SectionSpaceAbb{F} = \sSectionSpace\bigl(P \times_K \FibreBundleModel{F}\bigr) \\
	\sSectionSpace\bigl(P \times_K \FibreBundleModel{F}_\mu\bigr).
		&
\end{equationcd}
The reduced space \( \sSectionSpace(P \times_K \FibreBundleModel{F}_\mu) \) is symplectic by our general theory, since it is the section space of a symplectic fiber bundle. Hence it is natural to suppose that \( \sSectionSpace(P \times_K \FibreBundleModel{F}_\mu) \) is the symplectic reduction starting from \( \SectionSpaceAbb{F} \).
This is indeed the case as we will see now.
Consider the action of \( \sFunctionSpace(M, G) \) on \( \SectionSpaceAbb{F} = \sSectionSpace(P \times_K \FibreBundleModel{F}) \) 
induced by the \( G \)-action on \( \FibreBundleModel{F} \).
This action is well-defined since the \( G \)- and \( K \)-actions on \( \FibreBundleModel{F} \) commute.
Although this setting is not completely the same as the action of a gauge group on the space of sections of an associated bundle --- 
\( \sFunctionSpace(M, G) \) is the gauge group of the trivial bundle not of \( P \) ---, the general strategy of \cref{prop::momentumMap:gaugeGroup} carries through and yields the momentum map
\begin{equation}
\SectionSpaceAbb{J}_{G}: \SectionSpaceAbb{F} \to 
\sFunctionSpace(M, \LieA{g}^*), \qquad \phi \mapsto J_G \circ \phi,
\end{equation}
where \( J_G: P \times_K \FibreBundleModel{F} \ni \equivClass{p, \FibreBundleModel{f}} \mapsto \bigl(\pi(p), \FibreBundleModel{J}_G(\FibreBundleModel{f})\bigr) \in M \times \LieA{g}^* \) is the bundle map induced by the \( K \)-invariant map 
\( \FibreBundleModel{J}_G: \FibreBundleModel{F} \to \LieA{g}^* \).
From this expression, it is easy to see that \( \sSectionSpace(P \times_K \FibreBundleModel{F}_\mu) \) is the symplectic reduced space of \( \SectionSpaceAbb{F} \) by \( \sFunctionSpace(M, G) \) at the constant map \( \mu \in \sFunctionSpace(M, \LieA{g}^*) \).

In the K\"ahler setting, we have \(\FibreBundleModel{F}=\GLGroup(2n,\R)\), \( G = \OGroup(2n) \), and \( K = \SpGroup(2n, \R) \) so that the \( G \)-reduced space \( \FibreBundleModel{F}_\mu \) is the upper Siegel half space for \( \mu = \symplMatrix \).
The space \( \SectionSpaceAbb{F} \) of sections is identified with the space \( \TensorFieldSpace^2(M) \) of invertible \( 2 \)-tensors, and the first reduced stage \( \sSectionSpace(P \times_K \FibreBundleModel{F}_\mu) \) coincides with the space \( \SectionSpaceAbb{I} \) of almost complex structures compatible with the symplectic form \( \sigma \). As discussed above, the second reduction by \( \DiffGroup_\sigma(M) \) at the subset \( c + \csCohomology^{2n-1}(M, \UGroup(1)) \) yields the moduli 
space \( \SectionSpaceAbb{SC} \) of almost complex structures \( I \) with prescribed scalar curvature.
Thus we get the following reduction by stages diagram:
\begin{equation}\begin{tikzcd}[column sep=5.9em, row sep=0.02em]
	\TensorFieldSpace^2(M)
\ar[r, twoheadrightarrow, "{\sslash_{\symplMatrix} \sFunctionSpace(M, \OGroup(2n))}"]
\ar[rr, twoheadrightarrow, "{\sslash_{\symplMatrix, c} \sFunctionSpace(M, \OGroup(2n)) \times 
\DiffGroup_\sigma(M)}", swap, bend right]
& \SectionSpaceAbb{I}
\ar[r,twoheadrightarrow,"{\sslash_c \DiffGroup_\sigma(M)}"]
& \SectionSpaceAbb{SC}.
	\\
& ?
&
\end{tikzcd}\end{equation}

We believe that this diagram commutes, \ie, the two-stage reduced space \( \SectionSpaceAbb{SC} \) can be obtained by the one-shot reduction from \( \TensorFieldSpace^2(M) \) by the product group \( \sFunctionSpace(M, \OGroup(2n)) \times \DiffGroup_\sigma(M) \).
In the finite-dimensional setting, such questions about commuting reductions are well-understood \parencite{MarsdenMisiolekEtAl2007}.
However, the framework of reduction by stages is completely missing in infinite dimensions. Once the general theory of reduction by stages in infinite dimensions is formulated, one can state that the curved arrow is a one-shot reduction. This general theory is the subject of a future paper.

\subsubsection{Teichm\"uller space and weighted Lagrangian subbundles}
Coadjoint orbits are the canonical examples of symplectic homogeneous spaces. Let \( G \) be a finite-dimensional Lie group and 
\( \mu \in \LieA{g}^* \). The coadjoint orbit \( \CoAdOrbit_\mu \isomorph G \slash G_\mu \) through \( \mu \) is endowed with the 
Kostant--Kirillov--Souriau symplectic form
\begin{equation}\label{orbit_symplectic_form}
	\omega_{\nu} (\CoadAction_A \nu, \CoadAction_B \nu) 
		= \dualPair{\nu}{\commutator{A}{B}}, \quad \nu \in \CoAdOrbit_\mu 
		\subset\mathfrak{g}^\ast, 
		\quad A, B\in \mathfrak{g}.
\end{equation}
Let \( M \) be a finite-dimensional closed manifold carrying a 
\( G \)-structure, \ie, the structure group of the frame bundle 
is reduced to \( G \).
A section of \( \FrameBundle M \times_G (G \slash G_\mu) \) corresponds to a further reduction to \( G_\mu \).
Thus, our general theory implies that the space of all such 
\( G_\mu \)-reductions carries a symplectic structure, which is 
naturally induced by the coadjoint orbit symplectic form~\eqref{orbit_symplectic_form}. The momentum map for the action 
of the group of diffeomorphisms of \( M \) on the space of all such
 \( G_\mu \)- structures can be calculated with the help of \cref{prop::applications:reductionStructureGroup:momentumMap}.
Recall that the momentum map involved the prequantization of the fiber model. In the present setting, prequantizations of \( G \slash G_\mu \) bijectively correspond to Lie group homomorphisms \( \rho_\mu: G_\mu \to \UGroup(1) \) that integrate the Lie algebra homomorphism (see, \cref{subsubsec_pull_back_prequantum_bundles})
\begin{equation}
\check{\rho}_\mu = \restr{\FibreBundleModel{J}(e G_\mu)}{\LieA{g}_\mu}: 
\LieA{g}_\mu \to \R,
\end{equation}
where \( \FibreBundleModel{J}: G \slash G_\mu \to \LieA{g}^* \) is the momentum map for the \( G \)-action on the coadjoint orbit.
It is a straightforward calculation to verify that the momentum map is just the canonical embedding of the coadjoint orbit in \(\LieA{g}^*\).
Thus we are looking for group homomorphisms \( \rho_\mu \) that integrate the pairing with \( \mu \), that is,
\begin{equation}
\check{\rho}_\mu: \LieA{g}_\mu \to \R, \quad A \mapsto \dualPair{\mu}{A}.
\end{equation}
In summary, starting from a \( G \)-structure we can derive different geometric moduli spaces that arise from an infinite-dimensional symplectic reduction by performing the following steps:
\begin{enumerate}
	\item
		Classify all coadjoint orbits \( \CoAdOrbit_\mu \) of \( G \).
	\item
		For each \( \mu \in \LieA{g}^* \) in this classification, determine the stabilizer \( G_\mu \) and find an identification of \( G \slash G_\mu \) in terms of a more geometric object.
		In practice, this comes down to finding a transitive \( G \)-action on some space having stabilizer \( G_\mu \).
	\item
		Calculate \( \check{\rho}_\mu: \LieA{g}_\mu \to \R \) and integrate it to a Lie group homomorphism \( \rho_\mu: G_\mu \to \UGroup(1) \).
		This integration is usually possible only for certain values of \( \lambda \).
		These are the so-called prequantizable coadjoint orbits.
	\item
		Give a geometric interpretation of the associated bundle \( \FrameBundle_\mu M \times_{\rho_\mu} \UGroup(1) \), where \( \FrameBundle_\mu M \) is a given \( G_\mu \)-reduction of the \( G \)-frame bundle \( \FrameBundle M \). 
\end{enumerate}

In this subsection, we carry out this program for the coadjoint orbits of \( \SLGroup(2, \R) \). The low dimensionality of the spaces involved allows to explicitly determine the coadjoint orbits and the stabilizer subgroups \( G_\mu \). Even in this relatively simple example, the geometric structures involved turn out to be of special interest: we recover the Teichm\"uller moduli space as a symplectic reduction; another coadjoint orbit yields a hyperbolic cousin of the classical Teichm\"uller space and yet another orbit is related to Lagrangian distributions endowed with a density. Throughout this subsection, let \( (M, \sigma) \) be closed surface endowed with a symplectic form \( \sigma \) (which can, equivalently, be viewed as a volume form).
\begin{figure}
	\centering
	\begin{tikzpicture}[scale=1]
		\draw plot[domain=-2.35:2.35] ({-\x},{\x});
		\draw plot[domain=-2.35:2.35] ({\x},{\x});
         \begin{scope}
			\clip (-2.35,-2.35) rectangle (2.35,2.35);
	\foreach \lsqrt in {0.3,0.5,0.8,1.2,1.5,1.9}{
				\draw[lightgray] plot[domain=-1.5:1.5] ({\lsqrt*cosh(\x)},{\lsqrt*sinh(\x)});
				\draw[lightgray] plot[domain=-1.5:1.5] ({-\lsqrt*cosh(\x)},{\lsqrt*sinh(\x)});

				\draw[lightgray] plot[domain=-1.5:1.5] ({\lsqrt*sinh(\x)}, {\lsqrt*cosh(\x)});
				\draw[lightgray] plot[domain=-1.5:1.5] ({\lsqrt*sinh(\x)}, {-\lsqrt*cosh(\x)});
			}

			\pgfmathsetmacro{\lsqrt}{1.2}

			\draw plot[domain=-1.5:1.5] ({\lsqrt*cosh(\x)},{\lsqrt*sinh(\x)});
			\draw plot[domain=-1.5:1.5] ({-\lsqrt*cosh(\x)},{\lsqrt*sinh(\x)});

			\draw plot[domain=-1.5:1.5] ({\lsqrt*sinh(\x)}, {\lsqrt*cosh(\x)});
			\draw plot[domain=-1.5:1.5] ({\lsqrt*sinh(\x)}, {-\lsqrt*cosh(\x)});

			\node[draw, circle, fill, inner sep=1pt] at (0,0) {};
			\node[draw, circle, fill, inner sep=1pt, label=above:$N_e^1$] at (0,\lsqrt) {};
			\node[draw, circle, fill, inner sep=1pt, label=right:$N_h^1$] at (\lsqrt,0) {};
			\node[draw, circle, fill, inner sep=1pt, label=right:$N_p^+$] at (0.5,0.5) {};
			\node[draw, circle, fill, inner sep=1pt, label={[xshift=-0.2cm, yshift=0.1]below:$N_p^-$}] at (0.5,-0.5) {};
		\end{scope}
	\end{tikzpicture}
	\caption{Coadjoint orbits of \( \SLGroup(2, \R) \) projected onto the \( e_- \)- \( h \) plane.}
\end{figure}
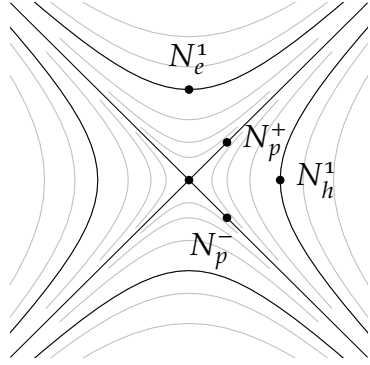

We begin with a systematic discussion of the coadjoint orbits of \( \SLGroup(2, R) \).
The Lie algebra \( \SLAlgebra(2, \R) \) consists of real \( 2 \times 2 \) matrices with vanishing trace.
We choose the following generators
\begin{equation}
    e_+ = \Matrix{1 & 0 \\ 0 & -1} \quad
    e_- = \Matrix{0 & 1 \\ 1 & 0} \quad
    h = \Matrix{0 & 1 \\ -1 & 0}
\end{equation}
subject to the commutation relations \( \commutator{h}{e_\pm} = \mp \, 2 e_\mp \) and \( \commutator{e_+}{e_-} = 2h \).
The Killing form \( \kappa(A, B) = \sfrac{1}{2} \tr(AB) \) is non-degenerate and thus yields an identification of \( \LieA{g} \) 
with \( \LieA{g}^* \). The generators form an orthonormal basis of \( \LieA{g} \) (and thus also of \( \LieA{g}^* \)) with signature \( (++-) \). The Killing form is \( \AdAction \)-invariant and so identifies the coadjoint action with the adjoint action. Every non-vanishing element \( A \in \SLAlgebra(2, \R) \) can be brought into one of the following four normal forms under the adjoint action:
\begin{equation}
    N_e^\lambda = \lambda \Matrix{0 & 1 \\ -1 & 0}, \quad
    N_h^\lambda = \lambda \Matrix{0 & 1 \\ 1 & 0}, \quad
    N_p^+ = \Matrix{0 & 1 \\ 0 & 0}, \quad
    N_p^- = \Matrix{0 & 0 \\ 1 & 0}.
\end{equation}

The orbit through \( N_e^\lambda \) is an elliptic hyperboloid (or two-sheeted hyperboloid). Positive \( \lambda \) parametrize the family of the upper sheets and negative \( \lambda \) correspond to the lower sheets.
The stabilizer of \( N_e^\lambda \) is independent of \( \lambda \) and is isomorphic to \( \SOGroup(2) \). The coadjoint orbit for 
\( \lambda = 1 \) is symplectomorphic to the reduced space \( J^{-1}_{\OGroup}(\symplMatrix) \slash \OGroup(2)_{\symplMatrix} = 
\SpGroup(2, \R) \slash \UGroup(1) \) discussed in the previous section (see~\eqref{eq:applications:kaehler:reducedSpacesModel} with \( n = 1 \)).
Sections of \( \FrameBundle M \times_{\SLGroup(2, \R)} (\SLGroup(2, \R) \slash \SOGroup(2))\) correspond to reductions of the 
\(\SLGroup(2, \R)\)-frame bundle to \( \SOGroup(2) \), \ie, Riemannian metrics \( g \) compatible with the prescribed volume form \( \sigma \).
To make this interpretation in terms of Riemannian metrics more explicit, we identify the coadjoint orbit \( \mathcal{O}_{N_e^\lambda} \isomorph \SLGroup(2, \R) \slash \SOGroup(2) \) with the space \( \posSSymSpace(2, \R) \) of symmetric positive definite matrices having determinant \( 1 \) via
\begin{equation}
	\mathcal{O}_{N_e^\lambda} \ni \AdAction_g N_e^\lambda 
		\mapsto g g^\T \in \posSSymSpace(2, \R), 
		\quad g \in \SLGroup(2, \R).
\end{equation}
Using this diffeomorphism, the coadjoint orbit symplectic form can be pushed forward to yield the following symplectic form on 
\( \posSSymSpace(2, \R) \):
\begin{equation}
\omega_B (C_1, C_2) = - \frac{1}{4} \tr(B^{-1} C_1 \, B^{-1}
 N_e^\lambda \, B^{-1} C_2),
\end{equation}
where \( B\in \posSSymSpace(2, \R) \), and \( C_1, C_2 \) are traceless symmetric matrices.

The orbit through \( N_h^\lambda \) is a hyperbolic hyperboloid (or one-sheeted hyperboloid). The stabilizer of \( N_h^\lambda \) is again independent of \( \lambda \), consists of elements of the form
\begin{equation}
	S = \begin{pmatrix} \pm \cosh r & \sinh r \\ \sinh r & \pm \cosh r 
	\end{pmatrix},
		\quad r \in \R,
\end{equation}
and hence is isomorphic to \( \SOGroup(1,1) \).
Sections of \( \FrameBundle M \times_{\SLGroup(2, \R)} (\SLGroup(2, \R) \slash \SOGroup(1,1)) \) are again reductions of the 
\( \SLGroup(2, \R) \)-frame bundle, this time, to \( \SOGroup(1,1) \).
Geometrically, these are semi- (or pseudo-, in another terminology) Riemannian metrics \( \eta \). Hence the space of semi-Riemannian metrics inducing the given volume form \( \sigma \) is an infinite-dimensional symplectic manifold.

The parabolic orbits, that is, the two light-cones, are orbits through \( N_p^+ \) and \( N_p^- \). As the notation suggests, \( N_p^+ \) corresponds to the upper cone. Its stabilizer \( P^+ \) consists of elements of the form
\begin{equation}
	S = \begin{pmatrix} \pm 1 & r \\ 0 & \pm 1 \end{pmatrix}
\end{equation}
for \( r \in \R \).
Similarly, the lower cone is the orbit through \( N_p^- \) and the elements of its stabilizer \( P_- \) are of the form
\begin{equation}
	S = \begin{pmatrix} \pm 1 & 0 \\ r & \pm 1 \end{pmatrix}
\end{equation}
for \( r \in \R \).

In order to illuminate the geometric structure that leads to a reduction to the stabilizer subgroups \( P^\pm \), it is actually helpful to consider the higher dimensional case.
Thus, let \( (M, \sigma) \) be, for the moment, a \( 2n \)-dimensional symplectic manifold. An \( n \)-dimensional distribution \( D \) on \( M \) yields a reduction of the symplectic frame bundle to the group consisting of elements of the form
\begin{equation}
	\Matrix{A & B \\ 0 & C},
\end{equation}
where \( A,C \in \GLGroup(n, \mathbb{R}) \) and \( B \) is an arbitrary \( n \times n \) matrix. Indeed, the reduced bundle consists only of frames \( u_m: \R^{2n} \to \TBundle_m M \) that map the last \( n \) standard vectors to vectors spanning \( D_m \).
If we would, instead, require that the first \( n \) vectors span \( D_m \), we obtain a reduction to elements of the form
\begin{equation}
	\Matrix{A & 0 \\ B & C}.
\end{equation}
If the distribution carries an additional structure, then this is reflected in further conditions on the matrix \( C \).
For example, a distribution endowed with a density \( \abs{\nu} \) leads to the condition \( \abs{\det C} = 1 \). We then say that the distribution is weighted. 

Specializing again to the case \( n = 1 \) at hand, we see that a reduction of the symplectic frame bundle to \( P^\pm \) corresponds to a weighted Lagrangian subbundle \( (L, \abs{\nu}) \) of \( \TBundle M \).
Thus the space of all weighted Lagrangian distributions carries two natural symplectic structures (depending on the point \( N_p^\pm \) chosen for the identification).

\begin{table}[tbp]
	\centering
	\footnotesize
	\newcolumntype{C}[1]{%
		>{\vrule height 2ex depth 1.5ex width 0pt\centering}%
			p{#1}%
		<{}}
	\begin{tabular}{C{2.35cm} C{3.4cm} C{3.6cm} C{3.4cm}}
		\toprule
			\( \mu \)
			& \( N_e^\lambda \)
			& \( N_h^\lambda \)
			& \( N_p^\pm \)
			\tabularnewline
		\midrule
			Type
			& elliptic
			& hyperbolic
			& parabolic
			\tabularnewline
			\( G_\mu \)
			& \( \SOGroup(2) \)
			& \( \SOGroup(1,1) \)
			& \( P^\pm \)
			\tabularnewline
			Quantizable
			& \( \lambda \in \Z \)
			& always
			& always
			\tabularnewline
			\( \rho_\mu \)
			& \( \rho_{N_e^\lambda} \)
			& \( \rho_{N_h^\lambda}, \sqrt{\rho_{N_h^{2\lambda}}} \)
			& \( 1, \rho_{N_p^\pm} \)
			\tabularnewline
		\midrule
			\( \SectionSpaceAbb{F}_\mu \)
			& Riemannian metric \( g \)
			& semi-Riemannian metric \( \eta \)
			& weighted Lagrangian distribution \( (L, \abs{\nu}) \)
			\tabularnewline
			\( P_\mu \)
			& \( \KBundle^{-1}_g M \)
			& \( \KBundle_\eta M, \KBundle^{\frac{1}{2}}_\eta M \)
			& \( \OrientationBundle L \)
			\tabularnewline
			Characteristic class
			& \( \chernClass_1(M) \)
			& ?
			& \( \stiefelWhitneyClass_1(L) \)
			\tabularnewline
			Moduli space
			& Teichm\"uller space
			& hyperbolic Teichm\"uller space
			& ?
			\tabularnewline
		\bottomrule
	\end{tabular}
\caption{\footnotesize Overview over the properties of the coadjoint orbits.
The information concerning the geometric structures induced by the coadjoint orbits will become clear after the statement of \cref{prop:applications:coadjontOrbitsSL:momentumMapsInfinite} and the comments following it.
Here, \( \SectionSpaceAbb{F}_\mu \) denotes the space of sections of \( \FrameBundle M \times_{\SLGroup(2, \R)}(\SLGroup(2, \R)\slash G_\mu)\).
Furthermore, \( P_\mu \) is the associated circle bundle of the reduced \( G_\mu \)-frame bundle using the group homomorphism \( \rho_\mu \).
}%
	\label{tab:applications:coadSL:overview}
\end{table}
In order to determine the prequantizations of the orbits, we have to calculate
\begin{equation}
\check{\rho}_\mu: \LieA{g}_\mu \to \R, \quad A \mapsto \dualPair{\mu}{A}
\end{equation}
for each orbit. Note that we always have \( \LieA{g}_\mu = \R\mu \) since the stabilizers are \( 1 \)-dimensional.
We obtain
\begin{align}
    \check{\rho}_{N_e^\lambda} (a \cdot N_e^\lambda) &= - a \lambda, \\
    \check{\rho}_{N_h^\lambda} (a \cdot N_h^\lambda) &= a \lambda, \\
    \check{\rho}_{N_p^\pm} (a \cdot N_p^\pm) &= 0.
\end{align}
If \( \lambda \) satisfies the prequantization condition \( \lambda \in \Z \), then \( \check{\rho}_{N_e^\lambda} \) integrates to the Lie group homomorphism
\begin{equation}
	\rho_{N_e^\lambda}: \SOGroup(2) \to \UGroup(1),
	\quad 
	\Matrix{\cos \vartheta & \sin \vartheta \\ - \sin \vartheta & 
	\cos \vartheta} \mapsto e^{-\vartheta \lambda \I},
\end{equation}
Note that for \( \lambda = 1 \), this is just the inverse of the canonical identification of \( \SOGroup(2) \) with \( \UGroup(1) \).
Thus, the associated bundle \( \FrameBundle_{\SOGroup(2)} M \times_{\rho_{N_e^1}} \UGroup(1) \) is the anti-canonical bundle \( \KBundle^{-1}_g M \) constructed from the Riemannian metric \( g \).
Other values of \( \lambda \in \Z \) correspond to higher powers \( \KBundle^{-\lambda}_g M \).

In contrast, \( \SOGroup(1,1) \) is not connected and \( \check{\rho}_{N_h^\lambda} \) admits two types of integrating group homomorphisms for all \( \lambda \in \R \)
\begin{align}
	\rho_{N_h^\lambda}: \SOGroup(1,1) \to \UGroup(1), 
		&\quad \Matrix{\pm \cosh r & \sinh r \\ \sinh r & 
		\pm \cosh r} \mapsto e^{\pm r \lambda \I},
	\\
	\sqrt{\rho_{N_h^{2\lambda}}}: \SOGroup(1,1) \to \UGroup(1), 
		&\quad \Matrix{\pm \cosh r & \sinh r \\ \sinh r & \pm \cosh r} 
\mapsto 
		\begin{cases}
			e^{r \lambda \I} & \text{for } + \\
			- e^{-r \lambda \I} & \text{for } -
		\end{cases}
\end{align}
As the notation suggests, we have \( \sqrt{\rho_{N_h^{2\lambda}}}^2 = \rho_{N_h^{2\lambda}} \). But, of course, also \( \rho_{N_h^{\lambda}}^2 = \rho_{N_h^{2\lambda}} \) holds, so that \( \sqrt{\rho_{N_h^{2\lambda}}} \) is the \enquote{non-canonical square root}. In analogy to the previous case, the associated bundle \( \FrameBundle_{\SOGroup(1,1)} M \times_{\rho_{N_h^2}} \UGroup(1) \) is the canonical bundle \( \KBundle_\eta M \) induced by the semi-Riemannian metric \( \eta \).
The bundles \( \rho_{N_h^1} \) and \( \sqrt{\rho_{N_h^2}} \) yield the 
two inequivalent square roots \( \KBundle^{+\frac{1}{2}}_\eta M \) and 
\( \KBundle^{-\frac{1}{2}}_\eta M \) of \( \KBundle_\eta M \). Thus, in contrast to the Riemannian case, also the square roots of the canonical bundle are of particular importance based on the coadjoint orbit geometry.
Analogously, for \( \lambda \in \R \), we use the notations 
\( \KBundle^{+\frac{\lambda}{2}}_\eta M \) and 
\( \KBundle^{-\frac{\lambda}{2}}_\eta M \) for the associated bundles 
\( \FrameBundle_{\SOGroup(1,1)} M \times_{\rho_{N_h^\lambda}} \UGroup(1) \) and \( \FrameBundle_{\SOGroup(1,1)} M \times_{\sqrt{\rho_{N_h^{2\lambda}}}} \UGroup(1) \), respectively.

Since \( \check{\rho}_{N_p^\pm} = 0 \), the integrating group 
homomorphism \( \rho_{N_p^\pm}: P^\pm \to \UGroup(1) \) may be chosen 
to be the trivial homomorphism. However, \( P^\pm \) is also not connected 
and thus there is a second non-trivial integrating homomorphisms 
\( \rho_{N_p^\pm} \), which sends \( \smallMatrix{-1 & r \\ 0 & - 1} \), 
respectively \( \smallMatrix{-1 & 0 \\ r & - 1} \), to 
\( -1 \in \UGroup(1) \).
Note that the latter homomorphism factors through a homomorphism with values in \( \Z_2 \).

In order to illuminate the significance of the associated \( \Z_2 \)-bundle \( \FrameBundle_{P^\pm}M\times_{\rho_{N_p^\pm}} \Z_2 \), it is again helpful to leave the safe \( 2 \)-dimensional harbor and move to higher dimensional terrain.
Thus, let \( (M, \sigma) \) be, temporarily, a \( 2n \)-dimensional symplectic manifold. Recall that a weighted Lagrangian distribution on \( M \) yields a reduction of the symplectic frame bundle to the structure group consisting of symplectic block matrices \( \smallMatrix{A & B \\ 0 & C} \) with \( C \in \SLGroup^\pm(n, \R) \defeq \set{C \in \SLGroup(n, \R) \given \det C = \pm 1} \).
The higher dimensional analogue of \( \rho_{N_p^+} \) is the group homomorphism
\begin{equation}
	\rho_{N_p^+}: \Matrix{A & B \\ 0 & C} \mapsto \det C \in \Z_2.
\end{equation}
Thus a global section of the associated flat bundle \( \FrameBundle M \times_{\rho_{N_p^+}} \Z_2 \) corresponds to a consistent choice of orientation of the Lagrangian subspaces. We hence refer to this associated \( \Z_2 \)-bundle as the \emphDef{orientation bundle} \( \OrientationBundle L \) of the Lagrangian distribution \( L \).
The orientation bundle is trivial if and only if the first Stiefel--Whitney class \( \stiefelWhitneyClass_1(L) \) of the Lagrangian distribution vanishes. 

We have thus assembled all ingredients to apply \cref{prop::applications:reductionStructureGroup:momentumMap} (see also \cref{tab:applications:coadSL:overview} for a summary).
\begin{thm}
\label{prop:applications:coadjontOrbitsSL:momentumMapsInfinite}
Let \(M\) be a closed surface endowed with a symplectic form \( \sigma \).
Then the following hold:
\begin{enumerate}
\item
The space \(\MetricSpace_\sigma(M)\) of Riemannian metrics on \(M\) compatible with \( \sigma \) is endowed  with the symplectic form
\begin{equation}\label{eq:applications:coadSL:symplecticFormOnMetrics}
\Omega_g (h_1, h_2) = - \frac{\lambda}{4}\int_M\tr(g^{-1} h_1\cdot 
g^{-1}\sigma\cdot g^{-1}h_2) \, \sigma.
\end{equation}
If \( \lambda \in \Z \), then the action of \( \DiffGroup_\sigma (M) \) has a group-valued momentum map given by
\begin{equation}\label{eq:applications:coadSL:momentumMapForMetrics}
\MetricSpace_\sigma (M) \to \csCohomology^2(M, \UGroup(1)), 
\quad g \mapsto \KBundle_g M.
\end{equation}
\item
The space \( \MetricSpace^{+-}_\sigma (M) \) of semi-Riemannian metrics with signature \((+-)\) and compatible with \( \sigma \)  carries a symplectic form defined as in the Riemannian case by~\eqref{eq:applications:coadSL:symplecticFormOnMetrics}.
For all \( \lambda \in \R \), the action of \( \DiffGroup_\sigma (M) \) 
has two group-valued momentum maps given by
\begin{align}
\MetricSpace^{+-}_\sigma (M) &\to \csCohomology^2(M, \UGroup(1)), 
\quad \eta \mapsto \KBundle^{+\frac{\lambda}{2}}_\eta M
\intertext{and}
\MetricSpace^{+-}_\sigma (M) &\to \csCohomology^2(M, \UGroup(1)),
\quad \eta \mapsto \KBundle^{-\frac{\lambda}{2}}_\eta M.
\end{align}
\item
The space of weighted Lagrangian distributions is a symplectic manifold and the group-valued momentum map for the action of 
\( \DiffGroup_\sigma(M) \) is the assignment of the orientation bundle \( \OrientationBundle L \) to each Lagrangian distribution \( L \).
			\qedhere
\end{enumerate}
\end{thm}

The curvature of the canonical bundle \( \KBundle_g M \) is given by \( - S_I \sigma \), where \( S_I \) denotes the scalar curvature. 
Hence, symplectic reduction at the subset \( \curv^{-1}(\sigma) \) of all bundles with constant curvature \( \sigma \) yields the Riemann moduli space (for \( \lambda = 1 \)):
\begin{equation}
\SectionSpaceAbb{J}_{\DiffGroup}^{-1} (\curv^{-1}(\sigma)) \slash 
\DiffGroup_\sigma(M) =
\set{I \in \SectionSpaceAbb{I} \given S_I = -1} \slash \DiffGroup_\sigma(M).
\end{equation}
Instead of taking the quotient with respect to \( \DiffGroup_\sigma(M) \), we could restrict attention to the connected component of the identity \( \DiffGroup_\sigma(M)^\circ \). The action of \( \DiffGroup_\sigma(M)^\circ \) is free and its momentum map is given by the same formula~\eqref{eq:applications:coadSL:momentumMapForMetrics}.
Thus, the reduced space coincides with \( \set{I \in \SectionSpaceAbb{I} \given S_I = -1} \slash \DiffGroup_\sigma(M)^\circ \) and hence yields the Teichm\"uller space.
Formula~\eqref{eq:applications:coadSL:symplecticFormOnMetrics} for the symplectic form on the space of Riemannian metrics shows that the reduced symplectic form is proportional to the Weil--Petersson symplectic form on the Teichm\"uller space. However, in contrast to classical symplectic reduction, we take the inverse image of a set and not just of a point.
We show now that \( \DiffGroup_\sigma(M) \) acts (infinitesimally) transitively on \( \curv^{-1}(\sigma) \), so that the reduction is a \emph{symplectic orbit reduction} \parencite[Section~6.3]{OrtegaRatiu2003}.
Suppose that \( \sigma \) is normalized to have total volume \( 1 \).
Then there exists a differential character \( h_\sigma \in \csCohomology^2(M, \UGroup(1)) \) with curvature \( \sigma \).
Consider the flux homomorphism\footnote{\( \flux \) is actually only a cocycle and not necessarily a group homomorphism on \(\DiffGroup_\sigma(M) \). Only when restricted to \( \DiffGroup_\sigma(M) \intersect \DiffGroup(M)^\circ \) is the flux a bona-fide group homomorphism.} relative to \( h_\sigma \)
\begin{equation}
\flux_{h_\sigma}: \DiffGroup_\sigma(M) \to \sCohomology^1(M, \UGroup(1)), 
		\quad \phi \mapsto (\phi^{-1})^*h_\sigma - h_\sigma.
\end{equation}
By \cref{prop::differentialCharacter:pullback}, the corresponding Lie algebra homomorphism is 
\begin{equation}
\flux_\sigma: \VectorFieldSpace_\sigma(M) \to \sCohomology^1(M, \R), 
		\quad X \mapsto \equivClass{X \contr \sigma},
\end{equation}
which is clearly surjective. Thus \( \DiffGroup_\sigma(M) \) acts infinitesimally transitive on \( \curv^{-1}(\sigma) \).
In summary, we obtain the following.
\begin{thm}
The Riemann moduli space
\begin{equation}
\set{I \in \SectionSpaceAbb{I} \given S_I = -1} \slash \DiffGroup_\sigma(M)
\end{equation}
and the Teichm\"uller space 
\begin{equation}
\set{I \in \SectionSpaceAbb{I} \given S_I = -1} \slash 
\DiffGroup_\sigma(M)^\circ  
\end{equation}
are symplectic orbit reduced spaces.
\end{thm}

An analogous discussion for the coadjoint orbit \( \SLGroup(2, \R) \slash \SOGroup(1,1) \) yields a hyperbolic version of the Teichmüller space.
Instead of Riemannian metrics, the moduli space then consists of semi-Riemannian metrics with prescribed curvature modulo symplectomorphisms.
It also carries a symplectic form analogous to the one induced by the Weil--Petersson metric, and it is a symplectic orbit reduced space.  
Formally, we just have to replace \enquote*{elliptic} by \enquote*{hyperbolic}, \enquote*{Riemannian} by \enquote*{semi-Riemannian} 
and \enquote*{complex structure} by \enquote*{para-complex structure}.
The analysis of hyperbolic operators is, of course, more delicate than elliptic operators and thus it is not clear that the hyperbolic 
Teichm\"uller space is indeed a smooth manifold.
From a physics point of view, elements of the hyperbolic Teichm\"uller space parametrize equivalence classes of solutions of Einstein's equation in two dimensions.
In the physics literature \parencite{Strobl2000} one finds arguments that such moduli spaces of semi-Riemannian metrics with prescribed curvature are in fact finite dimensional, similar to the Riemannian case.
We are not aware of any rigorous mathematical proof of this statement.

\subsection{Action of the quantomorphism group}
Let \( (M, \omega) \) be a compact $2n$-dimensional symplectic manifold and let \( \mu_\omega = \omega^n \slash n! \) be the Liouville volume form.
Assume that \( \omega \) has integral periods so that there exists a prequantum circle bundle \( \pi: P \to M \) with connection \( \Gamma \) having curvature \( \omega \).
As the fiber model we take \( \FibreBundleModel{F} = \C \) with the usual symplectic structure \( \omega_\C \) endowed with \( k \)-times the natural \( \UGroup(1) \)-action for some \( k \in \N \).
The associated line bundle \( F^k = P \times_{k \cdot \UGroup(1)} \C \) is the \( k \)-th tensor product of the line bundle \( F^1 \).
As in \cref{sec:momentum_map_for_automorphism_group}, the connection \( \Gamma \) in \( P \) and the symplectic structure \( \omega_\C \) on \( \C \) combine to an extension \( \hat{\omega}^\Gamma_{\C} \in \DiffFormSpace^2(F^k) \) of the induced fiber symplectic structure on \( F^k \). 

\Cref{thm::symplecticFibreBundle:momentumMapSymplectic} determines the momentum map of the action on \( \sSectionSpace(F^k) \) of the group \( \AutGroup_\omega(P) \) of automorphisms of \( P \) that preserve the symplectic structure on the base.
In this section, we put the focus on the subgroup \( \AutGroup_\Gamma(P) \subseteq \AutGroup_\omega(P) \) consisting of automorphisms that even preserve the connection \( \Gamma \).
Due to its importance in the theory of geometric quantization, \( \AutGroup_\Gamma(P) \) is often called the \emphDef{quantomorphism group} of \( (P, \Gamma) \).
Recall that the Lie algebra \( \AutAlgebra_\Gamma(P) \) of \( \AutGroup_\Gamma(P) \) is identified with the Poisson algebra \( \sFunctionSpace(M) \) using the Koszul prescription
\begin{equation}
	\sFunctionSpace(M) \ni f \mapsto X^\Gamma_f + \pi^* f \difp_\vartheta 
	\in \AutAlgebra_\Gamma(P),
\end{equation}
where \( \difp_\vartheta \) is the canonical vector field along the fiber and \( X^\Gamma_f \) is the horizontal lift of the Hamiltonian vector field \( X_f \).
Accordingly, it is natural to choose \( \AutAlgebra_\Gamma(P)^* = \DiffFormSpace^{2n}(M) \) as the dual of \( \AutAlgebra_\Gamma(P) \).
On the level of vector spaces, the Lie algebra \( \AutAlgebra_\omega(P) \) of \( \AutGroup_\omega(P) \) is identified with \( \sFunctionSpace(M) \times \VectorFieldSpace_\omega(M) \).
Hence, a natural choice for the dual group of \( \AutGroup_\omega(P) \) is the Abelian Lie group \( \AutGroup_\omega(P)^* = \sFunctionSpace(M) \times \csCohomology^{2n}(M, \UGroup(1)) \), whose Lie algebra is \( \sFunctionSpace(M) \times \DiffFormSpace^{2n - 1}(M) \slash 
	\dif \DiffFormSpace^{2n-2}(M) \).
Under the above identification, the inclusion of \( \AutAlgebra_\Gamma(P) \) into \( \AutAlgebra_\omega(P) \) is given by
\begin{equation}
	\sFunctionSpace(M) \ni f \mapsto (f, X_f) \in \sFunctionSpace(M) \times 
	\VectorFieldSpace_\omega(M).
\end{equation}
The dual projection, with respect to the pairings~\eqref{eq::diffAction:dualPairOSymplPresVectorFields} and~\eqref{eq::diffAction:dualPairOfGaugeGroup}, has the form
\begin{equation}
	\sFunctionSpace(M) \times \DiffFormSpace^{2n - 1}(M) \slash 
	\dif \DiffFormSpace^{2n-2}(M) \ni (g, \equivClass{\alpha}) 
	\mapsto g \mu_\omega - \frac{1}{(n-1)!} \dif \alpha \in \DiffFormSpace^{2n}(M)
\end{equation}
and lifts to the dual group \( \AutGroup_\omega(P)^* \) as follows:
\begin{equation}
	\sFunctionSpace(M) \times \csCohomology^{2n}(M, \UGroup(1)) \ni (g, h) 
	\mapsto g \mu_\omega - \frac{1}{(n-1)!} \curv h \in \DiffFormSpace^{2n}(M).
\end{equation}
Clearly, the momentum map \( \SectionMapAbb{J}_{\AutGroup_\Gamma} \) for the quantomorphism group is the composition of this \enquote{projection} with the momentum map for the automorphism group.
Using \cref{thm::symplecticFibreBundle:momentumMapSymplectic}, we thus obtain
\begin{equation}
 	\SectionMapAbb{J}_{\AutGroup_\Gamma}: \sSectionSpace(F^k) \to   
	\DiffFormSpace^{2n}(M), \quad \phi \mapsto - k \norm{\phi}^2 \mu_\omega + 
	\frac{1}{(n-1)!} \phi^* \hat{\omega}^\Gamma_\C \wedge \omega^{n-1}.
\end{equation}
Using the identity \( \phi^* \hat{\omega}^\Gamma_\C = \omega_\C (\nabla^\Gamma \phi, \nabla^\Gamma \phi) + (J_* \phi) \, \curv \Gamma \) of \parencite[Lemma~13]{Donaldson2003}, we can rewrite the momentum map as follows:
\begin{equation}\begin{split}
	\SectionMapAbb{J}_{\AutGroup_\Gamma}(\phi) 
		&= - k \norm{\phi}^2 \mu_\omega + \frac{1}{(n-1)!} 
		\Bigl(\omega_\C (\nabla^\Gamma \phi, \nabla^\Gamma \phi) - 
		k \norm{\phi}^2 \omega\Bigr) \wedge \omega^{n-1}
		\\
		&= - k (n + 1 ) \norm{\phi}^2 \mu_\omega + 
		\frac{\I}{2} \nabla^\Gamma \phi \wedge 
		\nabla^\Gamma \bar{\phi} \wedge \frac{\omega^{n-1}}{(n-1)!} \, .
\end{split}\end{equation}
Thus, in summary, we obtain the following.
\begin{thm}
	Let \( (M, \omega) \) be a compact $2n$-dimensional symplectic manifold.
	Assume that \( \omega \) has integral periods so that there exists a prequantum circle bundle \( \pi: P \to M \) with connection \( \Gamma \) having curvature \( \omega \).
	Then the space \( \sSectionSpace(F^k) \) of smooth sections of the associated line bundle \( F^k = P \times_{k \cdot \UGroup(1)} \C \) carries a natural symplectic structure \( \Omega = \mu_\omega \hatProduct \omega_\C \).
	Moreover, the quantomorphism group \( \AutGroup_\Gamma(P) \) of \( P \) acts symplectically on \( \sSectionSpace(F^k) \) and has a classical momentum map \( \SectionMapAbb{J}_{\AutGroup_\Gamma}: \sSectionSpace(F^k) \to \DiffFormSpace^{2n}(M) \) given by
	\begin{equation}
		\SectionMapAbb{J}_{\AutGroup_\Gamma}(\phi)
			= - k (n + 1 ) \norm{\phi}^2 \mu_\omega + 
		\frac{\I}{2} \nabla^\Gamma \phi \wedge 
		\nabla^\Gamma \bar{\phi} \wedge \frac{\omega^{n-1}}{(n-1)!} \, .
			\qedhere
	\end{equation}
\end{thm}

A similar expression for the momentum map has been found by \Textcite[eq.~11]{Donaldson2001}, but with a few typos in the relative prefactors (in the derivation, \parencite[eq.~13]{Donaldson2001} misses a factor of \( n \)).
We emphasize that \( \SectionMapAbb{J}_{\AutGroup_\Gamma} \) is a classical momentum map with values in the dual vector space of \( \AutAlgebra_\Gamma(P) \).
The existence of a classical momentum map should not come as a surprise, because the quantomorphism group plays the role of the group of exact symplectomorphisms and as such is not sensitive to the same topological properties of \( M \) that the group of all symplectomorphisms sees.

\subsection{Gauge theory}

\Textcite{AtiyahBott1983} showed that the curvature furnishes a momentum map for the action of the group of gauge transformations on the space of connections over a surface.
In this section, we outline how this setting fits in our general framework.
Moreover, we also discuss the action of the whole group of bundle automorphisms, whose existence is tightly connected to the triviality of a certain class in the third cohomology of the base.

As we have already seen in \cref{sec:momentum_map_for_the_gauge_group}, the group of gauge transformations admits an ordinary momentum map, and thus the concept of a group-valued momentum map is not needed for this action.
Nonetheless, even in this case, our general framework provides the interesting insight that the momentum map geometry of the infinite-dimensional space of connection is induced by a finite-dimensional symplectic system. 

Originally, \citeauthor{AtiyahBott1983} studied the Yang--Mills equation over a Riemannian surface.
We consider the more general situation where the base is a compact \( 2n \)-dimensional symplectic manifold \( (M, \sigma) \); see also \parencite{Donaldson1987a} for an almost identical setting.

Let \( \pi: P \to M \) be a principal \( G \)-bundle, and let \( \AdBundle P \to M \) be the adjoint bundle of \( P \).
Recall that a connection \( \Gamma \) on \( P \) is a splitting of the exact sequence of vector bundles over \( M \)
\begin{equationcd}[label=eq:application:gaugeTheory:atiyahSeq]
	0 \to[r]	
		& \AdBundle P \to[r] 
		& (\TBundle P) \slash G \to[r, "\tangent \pi"] 
		& \TBundle M \to[r] 
		&0, 
\end{equationcd}
where \( \TBundle \pi \) also denotes the natural map on \( (\TBundle P) \slash G \) induced by \( \TBundle \pi: \TBundle P \to  \TBundle M \).
The bundle \( \ConnBundle P \to M \) of principal connections on \( P \) is, by definition, the subbundle of \( \LinMapBundle(\TBundle M, (\TBundle P) \slash G) \to  M \) whose fiber \( \ConnBundle_m P \) at \( m \in M \)  consists of all linear maps \( \Gamma_m:\TBundle_m M \to (\TBundle P \slash G)_m \) satisfying \( \tangent \pi (\Gamma_m(X_m)) = X_m \) for all  \( X_m \in \TBundle_m M \).
A smooth section \( m \mapsto \Gamma_m \) of \( \ConnBundle P \) determines a principal \( G \)-connection on \( P \) as follows.
The image of \( \Gamma_m \) yields a consistent choice of a subspace of \(\TBundle_p P \) which is complementary to the vertical subspace for any \( p\in \pi^{-1}(m) \).
This induces a complementary subbundle of \( \TBundle P \) to the vertical subbundle that is invariant under the \( G \)-action, \ie, a horizontal \( G \)-equivariant distribution, and hence a principal connection.
Conversely, the horizontal lift of a principal connection on \( P \) yields a splitting of~\eqref{eq:application:gaugeTheory:atiyahSeq} and hence a section of \( \ConnBundle P \). 
Thus, the space \( \ConnSpace (P) \) of principal connections on \( P \) is identified with the space of sections of \( \ConnBundle P \).
Note that \( \ConnBundle P \to M \) is an affine bundle modeled over the vector bundle \( \LinMapBundle(\TBundle M, \AdBundle P)\to M \).

In order to apply our general framework, we need to realize \( \ConnBundle P \) as an associated fiber bundle with a symplectic fiber.
This will require the notion of the principal jet prolongation \( \pJetBundle P \) of \( P \). We now recall this construction (mostly 
suppressing the jet-theoretic background) and refer the reader to \parencite[Section~15.6 and~17.5]{KolarSlovakEtAl1993} for more details.
\begin{itemize}
\item
The principal jet prolongation of the Lie group \( G \) is by definition the semidirect product
\begin{equation}
\pJetGroup G = (\SpGroup(2n, \R) \times G) \rSemiProduct 
\LinMapSpace(\R^{2n}, \LieA{g})
\end{equation}
with multiplication
\begin{equation}
(a, g, S) \cdot (b, h, T) 
= (a \cdot b, g \cdot h, \AdAction_{h^{-1}} S \circ b + T)
\end{equation}
and inversion 
\begin{equation}
(a, g, S)^{-1} = (a^{-1}, g^{-1}, - \AdAction_g S \circ a^{-1}).
\end{equation}
The Lie algebra of \( \pJetGroup G \) is the semidirect product
\begin{equation}
	\pJetAlgebra \LieA{g} = (\SpAlgebra(2n, \R) \times \LieA{g}) \rSemiProduct 
\LinMapSpace(\R^{2n}, \LieA{g}).
\end{equation}
\item
The principal jet prolongation \( \pJetBundle P \) is the product fiber bundle \( \FrameBundle M \times_M \JetBundle^1 P \) over \( M \), 
where \( \FrameBundle M \) denotes the symplectic frame bundle.
Equivalently, the fiber \( \pJetBundle_m P \) over \(m \in M\) consists of pairs \((u_m, v_p)\) of a linear symplectomorphism \( u_m: \R^{2n} \to \TBundle_m M \) of \( \R^{2n} \) endowed with its standard symplectic form \( \omega_0 \) and \( (\TBundle_m M, \sigma_m) \), and a linear map \( v_p: \TBundle_m M \to \TBundle_p P \) satisfying \( \tangent_p \pi \circ v_p = \id_{\TBundle_m M} \).
The right action of \( \pJetGroup G \) on \( \pJetBundle P \) defined by
\begin{equation}\label{eq:application:gaugeTheory:actionJetGroupBundle}
(u_m, v_p) \cdot (a, g, S) \defeq \bigl({u_m} \circ a, v_p \ldot g 
+ (p \cdot g) \ldot (S \circ a^{-1} \circ u_m^{-1})\bigr)
\end{equation}
endows \( \pJetBundle P \) with the structure of a principal \( \pJetGroup G \)-bundle.
Here, the lower dot notation stands for the natural action of \( G \) on \( \TBundle P \) and of \( \LieA{g} \)
 on \( P \).
\item
\(\pJetGroup G\) acts affinely on \( \LinMapSpace(\R^{2n}, \LieA{g}) \) by
\begin{equation}
(a, g, S) \cdot T = \AdAction_g (S + T) \circ a^{-1}.
\end{equation}
The associated bundle \( \pJetBundle P \times_{\pJetGroup G} \LinMapSpace(\R^{2n}, \LieA{g}) \) is the bundle \( \ConnBundle P \) of 
connections. The horizontal lift \( \Gamma_m \) corresponding to the equivalence class \( \equivClass{u_m, v_p, T\,} \in \pJetBundle P \times_{\pJetGroup G} \LinMapSpace(\R^{2n}, \LieA{g}) \) is given by
\begin{equation}
\label{eq:application:gaugeTheory:horizontalLiftAssociatedBundle}
\Gamma_m = v_p - p \ldot \bigl(T \circ u_m^{-1}\bigr): \TBundle_m M \to \TBundle_p P.
\end{equation}
\end{itemize}
Assume that \( \LieA{g} \) is endowed with an \( \AdAction \)-invariant scalar product. We denote by \( \tr(\cdot \wedge \cdot) \) the induced wedge product of linear maps with values in \( \LieA{g} \).
The linear symplectic form \( \FibreBundleModel{\omega} \) on \( \LinMapSpace(\R^{2n}, \LieA{g}) \) is defined by the identity
\begin{equation}
\FibreBundleModel{\omega}(T_1, T_2) \, \omega_0^n = \tr(T_1 \wedge T_2) \wedge \omega_0^{n-1},
\end{equation}
where \( \omega_0 \) denotes the standard symplectic form on \( \R^{2n} \).
Clearly, \( \FibreBundleModel{\omega} \) is invariant under the action of \( \pJetGroup G \), because \( a \in \SpGroup(2n, \R) \) leaves \( \omega_0 \) invariant and because the scalar product on \( \LieA{g} \) is \( \AdAction \)-invariant.
The associated momentum map \( \FibreBundleModel{J}: \LinMapSpace(\R^{2n}, \LieA{g}) \to \SpAlgebra(2n, \R)^* \times \LieA{g}^* \times \LinMapSpace(\R^{2n}, \LieA{g})^* \) is given by
\begin{equation}\label{eq:applications:gaugeTheoryFibreMomentumMap}
\dualPair{\FibreBundleModel{J}(T)}{(A, \xi, L)} \, \omega_0^n
= \Bigl(\frac{1}{2} \tr\bigl(\xi, \wedgeLie{T}{T\,}\bigr) - \tr\bigl(L \wedge T\,\bigr) + 
\frac{1}{2} \tr\bigl((T \circ A) \wedge T\,\bigr)\Bigr) \wedge 
\omega_0^{n-1}
\end{equation}
for \( A \in \SpAlgebra(2n, \R) \), \( \xi \in \LieA{g} \) and \( L \in \LinMapSpace(\R^{2n}, \LieA{g}) \).
Thus the typical fiber \( \LinMapSpace(\R^{2n}, \LieA{g}) \) of the connection bundle is a symplectic system with \( \pJetGroup G \)-symmetry.
However, this system does not quite fit into the general setting discussed in \cref{sec::globalAnalysisSymplecticFibreBundles}, because the momentum map \( \FibreBundleModel{J} \) is not equivariant.
In fact, a straightforward calculation shows that the non-equivariance \( 1 \)-cocycle \( \sigma\bigl(a, g, S\bigr) = \FibreBundleModel{J}\bigl((a, g, S) \cdot T \, \bigr) - \CoAdAction_{(a, g, S)^{-1}} \FibreBundleModel{J}(T) \) is given by
\begin{equation}
	\sigma\bigl(a, g, S\bigr) = \CoAdAction_{(a, g, 0)^{-1}} \FibreBundleModel{J}(S).
\end{equation}
The associated Lie algebra \( 2 \)-cocycle \( \Sigma: \pJetAlgebra \LieA{g} \times \pJetAlgebra \LieA{g} \to \R \) is
\begin{equation}
	\Sigma\bigl( (A_1, \xi_1, L_1), (A_1, \xi_1, L_1) \bigr) = \FibreBundleModel{\omega}(L_1, L_2).
\end{equation}
For the same reason, we need to pass to a central extension of \( \pJetGroup G \) to obtain an equivariant prequantization bundle.
Endow the trivial circle bundle \( \FibreBundleModel{L} = \LinMapSpace(\R^{2n}, \LieA{g}) \times \UGroup(1) \) with the connection \( 1 \)-form
\begin{equation}
	\FibreBundleModel{\theta}_{(T, z)}(U, u) = \frac{1}{2} \FibreBundleModel{\omega}(T, U) + u,
\end{equation}
where \( T, U \in \LinMapSpace(\R^{2n}, \LieA{g}) \), \( z \in \UGroup(1) \) and \( u \in \R \isomorph \UAlgebra(1) \).
The curvature of \( \FibreBundleModel{\theta} \) is \( \FibreBundleModel{\omega} \), so that the trivial bundle \( \FibreBundleModel{L} \) endowed with the connection \( \FibreBundleModel{\theta} \) is a prequantum bundle of the symplectic vector space \( (\LinMapSpace(\R^{2n}, \LieA{g}), \FibreBundleModel{\omega}) \).
\begin{lemma}
	The group of connection-preserving bundle automorphisms of the prequantum bundle \( (\FibreBundleModel{L}, \FibreBundleModel{\theta}) \) covering the action of \( \pJetGroup G \) on \( \LinMapSpace(\R^{2n}, \LieA{g}) \) is identified with \( \widehat{\pJetGroup G} = \pJetGroup G \times \UGroup(1) \) endowed with the multiplication
	\begin{equation}\label{eq:application:gaugeTheory:multJetHeisenberg}\begin{split}
		(a, g, &S, u) \star (b, h, T, v)
			\\
			&= \left(a \cdot b, g \cdot h, \AdAction_{h^{-1}} S \circ b + T,  \exp\Bigl(- \frac{\I}{2} \FibreBundleModel{\omega}(S, \AdAction_h T \circ b^{-1})\Bigr) u v\right),
	\end{split}\end{equation}
	where \( (a, g, S), (b, h, T) \in \pJetGroup G \) and \( u, v \in \UGroup(1) \).
	The group \( \widehat{\pJetGroup G} \) is a non-trivial central extension of \( \pJetGroup G \) by \( \UGroup(1) \) whose induced Lie algebra \( 2 \)-cocycle is \( \Sigma \).
\end{lemma}
Roughly speaking, one can view \( \widehat{\pJetGroup G} \) as the principal jet prolongated Heisenberg group of the symplectic vector space \( (\LinMapSpace(\R^{2n}, \LieA{g}), \FibreBundleModel{\omega}) \).
\begin{proof}
	Every bundle automorphism \( \phi \) of \( \LinMapSpace(\R^{2n}, \LieA{g}) \times \UGroup(1) \) covering the action of \( (a, g, S) \in \pJetGroup G \) is of the from
	\begin{equation}
		(T, z) \mapsto \bigl( (a, g, S) \cdot T, \phi_{(a, g, S)}(T) \, z\bigr)
	\end{equation}
	for some smooth map \( \phi_{(a, g, S)}: \LinMapSpace(\R^{2n}, \LieA{g}) \to \UGroup(1) \).
	The automorphism \( \phi \) preserves the connection \( \FibreBundleModel{\theta} \) if and only if
	\begin{equation}\begin{split}
		\difLog_T \, \phi_{(a, g, S)} (U)
			&= \frac{1}{2} \FibreBundleModel{\omega}(T, U) - \frac{1}{2} \FibreBundleModel{\omega}\bigl((a, g, S) \cdot T, (a, g, S) \ldot U\bigr)
			\\
			&= \frac{1}{2} \FibreBundleModel{\omega}(T, U) - \frac{1}{2} \FibreBundleModel{\omega}\bigl(\AdAction_g (S + T) \circ a^{-1}, \AdAction_g U \circ a^{-1}\bigr)
			\\
			&= - \frac{1}{2} \FibreBundleModel{\omega}\bigl(S, U\bigr).
	\end{split}\end{equation}
	Consequently, \( \phi_{(a, g, S)}(T) = \exp\bigl(- \frac{\I}{2} \FibreBundleModel{\omega}(S, T)\bigr) \, u \) for some \( u \in \UGroup(1) \).
	One easily verifies that the assignment of the bundle automorphism
	\begin{equation}
		\phi_{(a, g, S, u)}(T, z) = \left( (a, g, S) \cdot T, \exp\Bigl(- \frac{\I}{2} \FibreBundleModel{\omega}(S, T)\Bigr) u z\right)
	\end{equation}
	to \( (a, g, S, u) \in \widehat{\pJetGroup G} \) is a group isomorphism with respect to the product~\eqref{eq:application:gaugeTheory:multJetHeisenberg}.
\end{proof}
Finally, observe that the action of \( \widehat{\pJetGroup G} \) on \( \LinMapSpace(\R^{2n}, \LieA{g}) \) is transitive, and that the stabilizer of \( 0 \) is the direct product group \( \SpGroup(2n, \R) \times G \times \UGroup(1) \).
We can thus view \( \LinMapSpace(\R^{2n}, \LieA{g}) \) as the homogeneous space \( \widehat{\pJetGroup G} \slash \bigl( \SpGroup(2n, \R) \times G \times \UGroup(1) \bigr) \).
The Lie group homomorphism that characterizes the \( \widehat{\pJetGroup G} \)-equivariant prequantization bundle \( (\FibreBundleModel{L}, \FibreBundleModel{\theta}) \) according to \cref{sec::reductionOfStructureGroup} is given by
\begin{equation}
	\rho: \SpGroup(2n, \R) \times G \times \UGroup(1) \to \UGroup(1), \qquad (a, g, u) \mapsto u.
\end{equation}
In summary, one needs to pass to the central extension \( \widehat{\pJetGroup G} \) of \( \pJetGroup G \) to obtain a Hamiltonian system \( (\LinMapSpace(\R^{2n}, \LieA{g}), \FibreBundleModel{\omega}) \) with equivariant momentum map and equivariant prequantum bundle \( (\FibreBundleModel{L}, \FibreBundleModel{\theta}) \).

The issue of non-equivariance extends to the bundle picture as well.
For example, the action of the gauge group of \( \pJetBundle P \) on the space \( \ConnSpace(P) \) of principal connections on \( P \) has a non-equivariant momentum map and one needs to pass to an extension to get an equivariant momentum map.
Surprisingly perhaps, these problems disappear under the jet prolongation, and we obtain the following generalization of the result of \textcite[p.~587]{AtiyahBott1983} to higher dimensional base manifolds.
\begin{thm}
	\label{prop:application:gaugeTheory:momentumMapGaugeGroup}
	Let \( (M, \sigma) \) be a compact \( 2n \)-dimensional symplectic manifold, and let \( \pi: P \to M \) be a principal \( G \)-bundle with compact structure group \( G \).
	The \( 2 \)-form
	\begin{equation}
		\Omega (\alpha, \beta) 
		= \int_M \tr(\alpha \wedge \beta) \wedge \sigma^{n-1} ,
		\qquad \alpha, \beta \in \DiffFormSpace^1(M, \AdBundle P)
	\end{equation}
	defines a symplectic structure on the space \( \ConnSpace(P) \) of principal connections on \( P \).
	Moreover, the natural action of the group \( \GauGroup(P) \) of gauge transformations of \( P \) on \( \ConnSpace(P) \) leaves \( \Omega \) invariant, and possesses an equivariant momentum map
	\begin{equation}
		\SectionMapAbb{J}_{\GauGroup}: \ConnSpace(P) \to \DiffFormSpace^{2n}(M, \AdBundle P),
		\qquad
		\Gamma \mapsto \curv \Gamma \wedge \sigma^{n-1},
	\end{equation} 
	where \( \curv \Gamma \in \DiffFormSpace^2(M, \AdBundle P) \) is the curvature of \( \Gamma \).
\end{thm}
\begin{proof}
The claim follows from a simple direct calculation along the lines of \parencite[p.~587]{AtiyahBott1983}.
We now sketch how the same result can be obtained in our framework utilizing the fiberwise symplectic system.

\Cref{thm::symplecticFibreBundle:momentumMapSymplectic} shows that the momentum map for the action of the gauge group of \( \pJetBundle P \) is the push-forward of the fiber momentum map~\eqref{eq:applications:gaugeTheoryFibreMomentumMap}, that is,
\begin{equation}
	\dualPair{\SectionMapAbb{J}_{\GauGroup(\pJetBundle P)}(\Gamma)}{\eta} = \int_M \dualPair{J \circ T_\Gamma}{\eta} \, \sigma^n,
\end{equation}
where \( \eta \in \GauAlgebra(\pJetBundle P) \), and where \( T_\Gamma: \pJetBundle P \to \LinMapSpace(\R^{2n}, \LieA{g}) \) denotes the equivariant map corresponding to the connection \( \Gamma \); that is, \( \Gamma_p = v_p - p \ldot \bigl(T_\Gamma (u_m, v_p) \circ u_m^{-1} \bigr) \), see~\eqref{eq:application:gaugeTheory:horizontalLiftAssociatedBundle}.
The momentum map for \( \GauGroup(P) \) is obtained by composing with the projection \( \GauAlgebra(\pJetBundle P)^* \to \GauAlgebra(P)^* \) dual to the derivative of the natural prolongation map \( \pJetProlongation: \GauGroup(P) \to \GauGroup(\pJetBundle P) \), \cf \parencite[p.~151]{KolarSlovakEtAl1993}.
If \( \chi \in \GauGroup(P) \) is viewed as an equivariant map \( \chi: P \to G \), then we find for the prolongation \( \pJetProlongation \chi : \pJetBundle P \to \pJetBundle P \) 
\begin{equation}\begin{split}
\pJetProlongation \chi (u_m, v_p) 
&= \bigl(u_m, v_p \ldot \chi(p) + p \ldot (\tangent_p \chi \circ v_p)\bigr) 
			\\
&= (u_m, v_p) \cdot \bigl(\id, \chi(p), \chi(p)^{-1} \ldot (\tangent_p \chi \circ v_p \circ u_m)\bigr),
\end{split}\end{equation}
where we used~\eqref{eq:application:gaugeTheory:actionJetGroupBundle} for the second equality.
Thus, changing viewpoints again, we read off
\begin{equation}
	\pJetProlongation \chi: \pJetBundle P \to \pJetGroup G, \quad (u_m, v_p) 
	\mapsto \bigl(\id, \chi(p), \difLog^L_p \chi \circ v_p \circ u_m\bigr).
\end{equation}
On the Lie algebra level, \( \xi \in \GauAlgebra(P) \) seen as a map \( \xi: P \to \LieA{g} \) is prolongated to
\begin{equation}
\pJetProlongation \xi: \pJetBundle P \to \pJetAlgebra \LieA{g}, 
\quad (u_m, v_p) \mapsto \bigl(0, \xi(p), \tangent_p \xi \circ v_p \circ u_m\bigr).
\end{equation}
The momentum map \( \SectionMapAbb{J}_{\GauGroup(P)} \) for the action of \( \GauGroup(P) \) satisfies, by definition,
\begin{equation}
\dualPair{\SectionMapAbb{J}_{\GauGroup(P)}(\Gamma)}{\xi}
= \dualPair{\SectionMapAbb{J}_{\GauGroup(\pJetBundle P)}(\Gamma)}
{\pJetProlongation \xi}
= \int_M \dualPair{J \circ T_\Gamma}{\pJetProlongation \xi} \, \sigma^n.
\end{equation}
For the integrand we find using~\eqref{eq:applications:gaugeTheoryFibreMomentumMap}
\begin{equation}\begin{split}
&\dualPair{J \circ {T}_\Gamma}{\pJetProlongation \xi} (u_m, v_p) \, \sigma^n_m
		\\
&\quad = \left(\frac{1}{2} \tr\bigl(\xi(p), 
\wedgeLie{{T}_\Gamma(u_m, v_p)}{{T}_\Gamma(u_m, v_p)}\bigr) \right. 
		\\
&\qquad \left. \phantom{\frac{1}{2}} - \tr\bigl((\tangent_p \xi \circ v_p \circ u_m) \wedge {T}_\Gamma(u_m, v_p)\bigr) \right) \wedge 
\omega_0^{n-1}
		\\
&\quad = \left(\frac{1}{2} \tr\bigl(\xi(p), 
\wedgeLie{{T}_\Gamma(v_p)}{{T}_\Gamma(v_p)}\bigr) - 
\tr\bigl((\tangent_p \xi \circ v_p) \wedge {T}_\Gamma(v_p)\bigr)\right) \wedge \sigma^{n-1},
	\end{split}\end{equation}
	where we abbreviated \({T}_\Gamma(v_p) \equiv {T}_\Gamma (u_m, v_p) \circ u_m^{-1} \).
	After partial integration on the second term, we use the structure equation \( \curv \Gamma = \frac{1}{2} \wedgeLie{\Gamma}{\Gamma\,} + \dif \Gamma \) to obtain
	\begin{equation}
		\dualPair{\SectionMapAbb{J}_{\GauGroup(P)}(\Gamma)}{\xi}
			= \int_M \tr(\xi, \curv \Gamma) \wedge \sigma^{n-1}.
	\end{equation}
	Thus \( \Gamma \mapsto \curv \Gamma \wedge \sigma^{n-1} \) is the momentum map for the action of \( \GauGroup(P) \), as claimed.
\end{proof}

The discussion of the momentum map for the action of the full automorphism group of \( P \) is more involved and requires the following additional structure.
Recall the central extension
\begin{equationcd}[label=eq:application:gaugeTheory:extJetHeisenberg]
	\UGroup(1) \to[r] &\widehat{\pJetGroup G} \to[r] &\pJetGroup G.
\end{equationcd}
A lift of \( \pJetBundle P \) to \( \widehat{\pJetGroup G} \) is a principal \( \widehat{\pJetGroup G} \)-bundle \( \widehat{\pJetBundle P} \) together with a vertical bundle morphism \( \chi: \widehat{\pJetBundle P} \to \pJetBundle P \) intertwining the principal actions of \( \widehat{\pJetGroup G} \) and \( \pJetGroup G \).
The bundle morphism \( \chi \) induces an isomorphism of the bundles \( \widehat{\pJetBundle P} \slash \UGroup(1) \) and \( \pJetBundle P \).
We will usually suppress this isomorphism and directly view \( \widehat{\pJetBundle P} \) as a principal \( \UGroup(1) \)-bundle over \( \pJetBundle P \).
Note, however, that such a lift \( \widehat{\pJetBundle P} \) does not need to exist.
In fact, the existence and uniqueness question is governed by the following exact sequence in Čech cohomology
\begin{equationcd}[label=eq:application:gaugeTheory:seqLiftpJetBundle, tikz={row sep=small}]
	\cechCohomology^1(M, \UGroup(1))
		\to[r]
		\arrow[d, equal]
	& \cechCohomology^1(M, \widehat{\pJetGroup G})
		\to[r]
	& \cechCohomology^1(M, \pJetGroup G)
		\to[r]
	& \cechCohomology^2(M, \UGroup(1))
		\arrow[d, equal]
	\\
	\sCohomology^2(M, \Z)
	&
	&
	& \sCohomology^3(M, \Z)
\end{equationcd}
induced by the exact sequence~\eqref{eq:application:gaugeTheory:extJetHeisenberg}, see \parencite[Theorem~4.1.4]{Brylinski2007}.
The vertical isomorphisms are induced by the short exact sequence \( \Z \to \R \to \UGroup(1) \).
As the Čech cohomology \( \cechCohomology^1(M, K) \) parameterizes equivalence classes of \( K \)-bundles, this exact sequence shows that there exists a lift of \( \pJetBundle P \) to \( \widehat{\pJetGroup G} \) given that a certain class in \( \sCohomology^3(M, \Z) \) vanishes.
In fact, this class admits the following description, \cf \parencite[p.~161]{Brylinski2007}.
Choose an open covering \( U_i \) of \( M \), and let \( g_{ij}: U_i \intersect U_j \to G \) be the transition functions of \( P \).
By jet prolongation, we obtain the transition functions \( \pJetProlongation g_{ij}: U_i \intersect U_j \to \pJetGroup G \) of \( \pJetBundle P \).
Choose a lift \( \widehat{\pJetProlongation g_{ij}} \) of \( \pJetProlongation g_{ij} \) to \( \widehat{\pJetGroup G} \).
These are the candidate transition functions for \( \widehat{\pJetBundle P} \).
However, they satisfy the cocycle condition of triple overlaps \( U_i \intersect U_j \intersect U_k \) only in the form
\begin{equation}
	\widehat{\pJetProlongation g_{ij}} \, \widehat{\pJetProlongation g_{jk}} = u_{ijk} \, \widehat{\pJetProlongation g_{ik}}
\end{equation}
for some \( \UGroup(1) \)-valued map \( u_{ijk} \).
The class of \( u_{ijk} \) in \( \cechCohomology^2(M, \UGroup(1)) \isomorph \sCohomology^3(M, \Z) \) is the obstruction for the lift.
We emphasize that this class in the third cohomology can be constructed for each principal bundle over a symplectic manifold in a natural and canonical way.
Although being canonical, we are not aware that this class has been discussed before in the literature.
Recall that elements in \( \sCohomology^3(M, \Z) \) parameterize equivalence classes of bundle gerbes (in the same way \( \sCohomology^2(M, \Z) \) parameterizes principal circle bundles).
The bundle gerbe corresponding to \( \equivClass{u_{ijk}} \) is known as the so-called lifting gerbe associated to the lifting problem of \( \pJetBundle P \) to \( \widehat{\pJetGroup G} \), see \parencite[Section~4]{Murray1996}.
It comes equipped with a natural connection and curving \parencite{Gomi2003}.
In this context, the following open question is of interest.
\begin{problem}
	Find an explicit description of the lifting gerbe (and the class in \( \sCohomology^3(M, \Z) \)) that obstructs the lifting of \( \pJetBundle P \) to \( \widehat{\pJetGroup G} \).
	In particular, clarify the geometric meaning of this gerbe that is naturally associated to each principal bundle over a symplectic manifold.
\end{problem}
In the following, we assume that there exists a lift \( \widehat{\pJetBundle P} \) of \( \pJetBundle P \) to \( \widehat{\pJetGroup G} \).
According to the sequence~\eqref{eq:application:gaugeTheory:seqLiftpJetBundle}, such a lift is unique up to a principal circle bundle on \( M \).
The natural identifications
\begin{equation}\begin{split}
	\widehat{\pJetBundle P} \times_{\widehat{\pJetGroup G}} \bigl(\widehat{\pJetGroup G} \slash (\SpGroup(2n, \R) \times G \times \UGroup(1))\bigr)
	&\isomorph
	\widehat{\pJetBundle P} \times_{\widehat{\pJetGroup G}} \LinMapSpace(\R^{2n}, \LieA{g})
	\\
	&\isomorph
	\pJetBundle P \times_{\pJetGroup G} \LinMapSpace(\R^{2n}, \LieA{g})
	\\
	&\isomorph
	\ConnBundle P
\end{split}\end{equation}
allow us to view a connection \( \Gamma \) on \( P \) as a reduction \( Q_\Gamma \) of \( \widehat{\pJetBundle P} \) to the structure group \( \SpGroup(2n, \R) \times G \times \UGroup(1) \).
We are thus in the setting discussed in \cref{sec::reductionOfStructureGroup}.
Assume that there exists a principal connection on \( \widehat{\pJetBundle P} \) that is reducible to each reduction \( Q_\Gamma \).
Then, by \cref{prop::applications:reductionStructureGroup:momentumMap}, the group of bundle automorphism of \( \widehat{\pJetBundle P} \) has an equivariant group-valued momentum map, whose second component is given by the principal \( \UGroup(1) \)-bundle \( Q_\Gamma \times_{\rho^{-1}} \UGroup(1) \).
We would like to proceed as in the proof of \cref{prop:application:gaugeTheory:momentumMapGaugeGroup}, and pull back everything along the jet prolongation \( \pJetProlongation: \AutGroup(P) \to \AutGroup(\pJetBundle P) \).
However, we are faced with the problem that \( \pJetProlongation \) takes values in \( \AutGroup(\pJetBundle P) \) and not in \( \AutGroup(\widehat{\pJetBundle P}) \).
Clearly, we have the sequence
\begin{equationcd}[label=eq:application:gaugeTheory:seqAutGroupsJet]
	e \to[r] & \sFunctionSpace(M, \UGroup(1)) \to[r] & \AutGroup(\widehat{\pJetBundle P}) \to[r] & \AutGroup(\pJetBundle P).
\end{equationcd}
However, the last map in this sequence is in general not surjective.
If the \( \UGroup(1) \)-bundle \( \widehat{\pJetBundle P} \to \pJetBundle P \) has a \( \widehat{\pJetGroup G} \)-invariant connection, then we can use its horizontal lift to show exactness of the infinitesimal sequence:
\begin{equationcd}[label=eq:application:gaugeTheory:seqAutAlgebraJet]
	0 \to[r] & \sFunctionSpace(M, \R) \to[r] & \AutAlgebra(\widehat{\pJetBundle P}) \to[r] & \AutAlgebra(\pJetBundle P) \to[r] & 0.
\end{equationcd}
In this case, we can pull-back~\eqref{eq:application:gaugeTheory:seqAutAlgebraJet} along the jet prolongation \( \pJetProlongation: \AutAlgebra(P) \to \AutAlgebra(\pJetBundle P) \) to obtain an Abelian extension
\begin{equationcd}[label=eq:application:gaugeTheory:seqAutAlgebra]
	0 \to[r] & \sFunctionSpace(M, \R) \to[r] & \widehat{\AutAlgebra}(P) \to[r] & \AutAlgebra(P) \to[r] & 0.
\end{equationcd}
Under these assumptions, we conclude that the Lie algebra \( \widehat{\AutAlgebra}(P) \) acts on the space \( \ConnSpace(P) \) of connections and that it has an (infinitesimal) equivariant momentum map.
Comparing with the usual non-equivariance picture, we conjecture that the \( \AutGroup(P) \)-action has a non-equivariant momentum map and that the passage to the extension is only necessary to obtain an equivariant momentum map.
In this context, one would also expect that an extension by \( \R \) is sufficient instead of an extension by the much larger space \( \sFunctionSpace(M, \R) \).
Moreover, it is not clear whether the extension~\eqref{eq:application:gaugeTheory:seqAutAlgebra} splits.
As the general methods and techniques developed in \cref{sec::globalAnalysisSymplecticFibreBundles} are tailored to symplectic fiber models with \emph{equivariant} momentum maps, these questions go beyond the scope of this paper.

\appendix

\section*{Appendices}
\addcontentsline{toc}{section}{Appendices}
\renewcommand{\thesubsection}{\Alph{subsection}}
\numberwithin{equation}{subsection}	
\numberwithin{thm}{subsection}		
\numberwithin{prop}{subsection}		
\numberwithin{lemma}{subsection}		
\numberwithin{remark}{subsection}		
\numberwithin{defn}{subsection}		
\automark[subsection]{subsection}

\subsection{Notations and conventions}
\label{sec::conventions}

Because notations and, especially, sign and coefficient conventions are 
not uniform in the literature, we provide below those used in this article.
We also comment on some peculiarities of the infinite-dimensional setting.

By a manifold we understand a possibly infinite-dimensional smooth 
manifold without boundary modeled on a locally convex space as, for 
example discussed in \parencite{Neeb2006}. If $f:M \to N$ is a smooth 
map between manifolds, its derivative, or tangent map, at $m \in M$ is 
denoted by $\tangent_m f: \TBundle_m M \to \TBundle_{f(m)} N$.
The Lie algebra of vector fields on a smooth manifold \( M \) is denoted 
by $\VectorFieldSpace(M)$, the Jacobi-Lie bracket on $\VectorFieldSpace(M)$ 
by $[\cdot , \cdot ]$, the space of $k$-exterior differential forms by 
$\DiffFormSpace^k(M)$, the space of $k$-contravariant tensors by 
$\VectorFieldSpace^k(M)$, the exterior differential by 
$\dif: \DiffFormSpace^k(M) \to \DiffFormSpace^{k+1}(M)$, the Lie 
derivative along a vector field $Z\in \VectorFieldSpace(M)$ by 
$\difLie_Z$ (acting on any tensor field), the interior product 
(contraction on the first index) of a form $\alpha \in\DiffFormSpace^k(M)$ 
with $Z \in \VectorFieldSpace(M)$ by ${\rm i}_Z \alpha  = 
Z \contr\alpha \in \DiffFormSpace^{k-1}(M)$.
In infinite dimensions, the cotangent bundle is no longer a \emph{smooth} 
bundle \parencite[Remark~1.3.9]{Neeb2006} and hence smoothness of 
differential forms has to be specified explicitly.
A differential form is defined as a set-theoretic section of the exterior 
bundle such that all chart representations are smooth as maps 
\( U \times E^k \to \R \), where \( U \) is an open 
subset of \( M \) and \( E \) denotes the model space 
of \( M \). For the wedge product $\alpha \wedge \beta \in 
\DiffFormSpace^{k+l}(M)$ of the differential forms $\alpha \in 
\DiffFormSpace^k(M)$
and $\beta \in\DiffFormSpace^l(M)$ we use the Bourbaki convention
\[
(\alpha \wedge \beta)(X_1, \ldots X_{k+l}) = \sum_\sigma 
(\operatorname{sign} \sigma) \alpha\left(X_{\sigma(1)}, 
\ldots,X_{\sigma(k)}\right) \beta\left(X_{\sigma(k+1)}, 
\ldots,X_{\sigma(k+l)}\right),
\]
where $X_1, \ldots X_{k+l} \in \VectorFieldSpace(M)$ and the sum is
taken over all $(k,l)$-shuffles $\sigma$, \ie, permutations
of $\{1, 2, \ldots, k+l\}$ such that $\sigma(1)< \cdots \sigma(k)$
and $\sigma(k+1)< \cdots \sigma(k+l)$. Evaluation of a tensor
field $\xi$ on \( M \) at $m \in M$ is denoted by $\xi_m = \xi(m)$, 
yielding a tensor of the same type on $\TBundle_m M$. The vector 
space of skew-symmetric $k$-linear maps on $\CotBundle_m M$ and
$\TBundle_m M$ is denoted by $\ExtBundle^k(\CotBundle_m M)$
and $\ExtBundle^k(\TBundle_m M)$, respectively; thus if $\alpha \in 
\DiffFormSpace^k(M)$ and $m \in M$, then $\alpha_m \in 
\ExtBundle^k(\CotBundle_m M)$.

Given topological vector spaces $V$ and $V^*$ (not necessarily the
functional analytic dual), $\left\langle\cdot , \cdot \right\rangle:
V ^* \times V \to \mathbb{R}$ denotes a weakly non-degenerate
pairing. If $W$ is another topological vector space in weak duality
with some topological vector space $W^*$ and $T: V \to W$ is
a linear continuous map, then its dual $T^*: W^*\to V^*$ is 
defined by $\left\langle T^* \alpha, v \right\rangle \defeq
\left\langle \alpha, Tv \right\rangle$, for any $v \in V$
and $\alpha \in W^*$.
Weak non-degeneracy of the pairings guarantees only
uniqueness of the linear dual map $T^*$, not its existence, so in 
concrete situations, its existence needs to be checked.

Given a Lie group \( G \), left and right translations by 
$a \in G$ are denoted by $b \ni G \mapsto \LeftTrans_a b\defeq ab \in G$ 
and $b \ni G \mapsto \RightTrans_a b\defeq ba \in G$, respectively.
The tangent map $\tangent_b \LeftTrans_a: \TBundle_b G
\to \TBundle_{ab} G$ of $\LeftTrans_a$ at an arbitrary $b \in G$ 
is denoted, for simplicity, by the same symbol $\LeftTrans_a$; similar 
conventions are used for right translation.
To shorten notation even further, we will abbreviate \( \tangent_b 
\LeftTrans_a (A_b) \) for \( A_b \in \TBundle_b G \) by 
\(a\ldot A_b\in\TBundle_{ab}G\). Similarly, we write 
\( A_b \ldot a \equiv \tangent_b \RightTrans_a (A_b) \) for the right 
multiplication. The Lie algebra of \( G \) is denoted by
\( \LieA{g} \) and the (left) adjoint representation by 
$\AdAction_a\defeq \tangent_e \left(\LeftTrans_a
\circ \RightTrans_{a^{-1}}\right): \LieA{g} \to 
\LieA{g}$; $\AdAction_a$ is a Lie algebra
isomorphism for any $a \in G$. The (left) coadjoint
representation of \( G \) on the dual $\LieA{g}^\ast$ of \( \LieA{g} \) 
is denoted by $\CoAdAction_{a^{-1}}: 
\LieA{g}^\ast\to \LieA{g}^\ast$ for any 
$a \in G$. The adjoint and coadjoint Lie algebra representations
of \( \LieA{g} \) on \( \LieA{g} \) and $\LieA{g}^\ast$,
respectively, are denoted by $\adAction_A B\defeq [A,B]$ and
$\CoadAction_A: \LieA{g}^\ast\to 
\LieA{g}^\ast$, respectively, for any $A, B \in \LieA{g}$.
We have $\tangent_e (\AdAction_{\cdot} B) (A) = 
[A, B] =\adAction_A B$ for any $A, B \in \LieA{g}$.

The space of \( \LieA{g} \)-valued $k$-forms on a manifold \( M \)
is denoted by $\DiffFormSpace^k(M, \LieA{g})$. The wedge product
associated to the Lie bracket $[ \cdot , \cdot ]$ on \( \LieA{g} \),
denoted by $[\alpha\wedge \beta] \in 
\DiffFormSpace^{k+l}(M, \LieA{g})$ for any $\alpha \in
\DiffFormSpace^k(M, \LieA{g})$ and 
$\beta\in \DiffFormSpace^l(M, \LieA{g})$, 
is defined with the same coefficient conventions as in the scalar 
case, \ie,
\[
[\alpha \wedge \beta](X_1, \ldots X_{k+l}) = \sum_\sigma 
(\operatorname{sign} \sigma) \commutator*{\alpha\left(X_{\sigma(1)}, 
\ldots,X_{\sigma(k)}\right)}{\beta\left(X_{\sigma(k+1)}, 
\ldots,X_{\sigma(k+l)}\right)},
\]
the sum being over all $(k,l)$-shuffles $\sigma$ and 
$X_1, \ldots X_{k+l} \in \VectorFieldSpace(M)$ arbitrary. 

If $\Upsilon:G \times M \to M$ is a left Lie group action and
$A \in \LieA{g}$, the infinitesimal generator, or fundamental,
vector field defined by $A$ is denoted by 
\[
A^\ast_m \defeq \difFracAt{}{t}{t=0}\Upsilon(\exp(tA), m)
\in \TBundle_m M
\]
for any $m\in M$, where $\exp: \LieA{g} \to G$ is the
exponential map of \( G \). We have $[A,B]^* = - [A^*, B^*]$ for all
$A, B \in \LieA{g}$. If the \( G \)-action on \( M \) is on the right, 
then $[A,B]^* = [A^*, B^*]$. If the action is clear from the context, 
we also use the \enquote{dot notation} for the action and its derivatives.
So, for \( m \in M \) and \( g \in G \), we abbreviate 
\( g \cdot m = \Upsilon(g, m) \).
On the infinitesimal level, for \( A \in \LieA{g} \), we will write the 
fundamental vector field as \( A^*_m = A \ldot m \).
Similarly, \( g \ldot X_m = \tangent_m \Upsilon_g (X_m) \).

If $f:M \to G$ is a smooth function, the left
logarithmic derivative $\difLog_m f: \TBundle_m M \to \LieA{g}$
of $f$ at $m \in M$ is defined by $\difLog_m f \defeq
\LeftTrans_{f(m)}^{-1} \circ \tangent_m f$. Therefore $\difLog f \in 
\DiffFormSpace^1(M, \LieA{g})$.

A Poisson structure on a smooth finite-dimensional manifold \( M \) is a 
Lie algebra structure $\{ \cdot , \cdot \}$ on the vector space 
$\sFunctionSpace(M)$ of smooth functions satisfying the Leibniz 
identity $\{fg, h\} = f\{g, h\} + \{f, h\}g$ for any 
$f, g, h \in \sFunctionSpace(M)$. Thus,
${\{h, \cdot \}}: \sFunctionSpace(M) \to \sFunctionSpace(M)$ is
a derivation and hence defines a vector field 
$X_h \in \VectorFieldSpace(M)$ on \( M \), \ie, $\left\langle \dif f, X_h 
\right\rangle \defeq {\{h, f\}}$ for any $f \in \sFunctionSpace(M)$, 
called the Hamiltonian vector field of $h$. Hamilton's equations
$\difFrac{}{t} m(t) = X_h(m(t))$ can be written in Poisson bracket
form as $\difFrac{}{t} f(m(t)) = {\{h, f \}}(m(t))$ for any 
$f \in \sFunctionSpace(M)$. The Poisson tensor $\varpi \in 
\VectorFieldSpace^2(M)$
is defined by $\varpi(\dif f, \dif g)\defeq \{f, g\}$ for
any $f, g \in \sFunctionSpace(M)$ and the Schouten bracket of 
\( \varpi \) with itself vanishes.
Conversely, any $\varpi \in \VectorFieldSpace^2(M)$ with this property 
defines a Poisson bracket by the formula above. Denote
by $\varpi^\sharp : \CotBundle M \to \TBundle M$ the
vector bundle map covering the identity defined by 
$\left\langle \alpha, \varpi^\sharp \beta \right\rangle \defeq 
\varpi(\alpha, \beta)$ for any $\alpha, \beta \in \DiffFormSpace^1(M)$.
If the map $\varpi^\sharp$ is invertible, its inverse, 
$\omega^\flat: \TBundle M \to \CotBundle M$ defines
a closed two form $\omega \in \DiffFormSpace^2(M)$ by $\omega(X_1, X_2) 
\defeq\left\langle \omega^\flat(X_1), X_2 \right\rangle = 
\varpi\left(\bigl(\varpi^\sharp \bigr)^{-1}X_1, 
\bigl(\varpi^\sharp \bigr)^{-1}X_2 \right)$ for
any $X_1, X_2 \in \VectorFieldSpace(M)$. In this case, $(M, \omega)$
is a symplectic manifold. The Hamiltonian vector field $X_h$ of
$h \in \sFunctionSpace(M)$ is given then, alternatively, by
the identity $X_h \contr \omega = -\dif h$. Any symplectic
manifold $(M, \omega)$ is Poisson by setting $\{f, h\}\defeq 
\omega(X_f, X_h)$.

 If $Q$ is finite dimensional, the cotangent bundle is naturally an exact
symplectic manifold $(\CotBundle Q, \omega_\textrm{can})$, where
$\omega_\textrm{can} = \dif  \theta_\textrm{can}$, where 
$\theta_\textrm{can} \in \DiffFormSpace^1(\CotBundle Q)$ is given by 
$\left\langle\theta_\textrm{can}(\alpha_q), V_{\alpha_q} \right\rangle = 
\left\langle\alpha_q, \tangent_q \pi \bigl(V_{\alpha_q}\bigr)
 \right\rangle$, for all $\alpha_q \in \CotBundle Q$, 
$V_{\alpha_q}\in \TBundle_{\alpha_q} (\CotBundle Q)$,
and $\pi:\CotBundle Q \ni \alpha_q \mapsto q \in Q$
is the cotangent bundle projection.
In canonical cotangent bundle coordinates $(q^1, \ldots, q^n, 
p_1, \ldots, p_n)$, we have $\theta_\textrm{can} = p_i \dif q^i$ and
$\omega_\textrm{can} =\dif p_i \wedge \dif q ^i$.
Thus, the Hamiltonian vector field is given by \( X_f = 
\difpFrac{f}{p_i} \difp_{q^i} - \difpFrac{f}{q^i} \difp_{p_i} \) and 
the Poisson bracket takes the form \( \poisson{f}{h} = \difpFrac{f}{p_i} 
\difpFrac{h}{q^i} - \difpFrac{f}{q^i} \difpFrac{h}{p_i} \).

Let \( (M, \omega) \) be a symplectic \( G \)-manifold. A momentum map 
for the action, if it exists, is a smooth map \( J: M \to \LieA{g}^* \) 
satisfying
\begin{equation*}
	A^* \contr \omega + \dif J_A = 0, \quad A \in \LieA{g},
\end{equation*}
where \( J_A \defeq \dualPair{A}{J}: M \to \R \).
 
A principal \( G \)-bundle is a fiber bundle \( P \to M \) together 
with a smooth right action of \( G \) on \( P \), which is 
fiber-preserving, free, and transitive on each fiber. In particular, $M= P/G$.
Given a left \( G \)-manifold \( \FibreBundleModel{F} \), we form the 
associated bundle \( F = P \times_G \FibreBundleModel{F} \defeq 
(P \times \FibreBundleModel{F})/G\) as the quotient of 
\( P \times \FibreBundleModel{F} \) by the natural free right action of 
\( G \), \ie, $(p,\FibreBundleModel{f})\cdot g \defeq (p\cdot g, g ^{-1} 
\cdot \FibreBundleModel{f})$, for any $p\in P$, 
$\FibreBundleModel{f}\in \FibreBundleModel{F}$, and $g\in G$.  Denoting by 
$[p,\FibreBundleModel{f}]$ points in $P \times_G \FibreBundleModel{F}$,
we have the identity 
\begin{equation*}
\equivClass{p \cdot g, \FibreBundleModel{f}} = 
\equivClass{p, g \cdot \FibreBundleModel{f}},
\end{equation*}
often used in the text.
The smooth projection map 
$P \times_G \FibreBundleModel{F}\ni [p,\FibreBundleModel{f}] \mapsto
[p] \in M=P/G$ defines a locally trivial fiber bundle with total space
$P \times_G \FibreBundleModel{F}$, base \( M \), and typical fiber 
$\FibreBundleModel{F}$.

An important particular case of this construction occurs by taking
$\FibreBundleModel{F} =\LieA{g}$, where \( \LieA{g} \) 
is the Lie algebra of \( G \), and the \( G \)-action on \( \LieA{g} \)
is the adjoint action. In 
this case, the associated bundle is called the
adjoint bundle and is denoted by $\AdBundle P$; it is a vector bundle 
all of whose fibers are Lie algebras and the bracket operation depends 
smoothly on the base point, \ie, $\AdBundle P$ is a Lie algebra 
bundle. Taking $\FibreBundleModel{F} = \LieA{g}^*$ and the 
coadjoint action on $\LieA{g}^\ast$ yields as associated bundle
the coadjoint bundle $\CoAdBundle P \to M$. Each fiber
of this vector bundle carries the Lie-Poisson bracket which depends 
smoothly on the base, \ie, $\CoAdBundle P$ is a Lie-Poisson
vector bundle. A third useful example is the conjugation bundle 
$\operatorname{Conj}P \to M$ which is the associated bundle
obtained by taking $\FibreBundleModel{F} = G$ with the 
\( G \)-action being the conjugation on \( G \). The fibers
of this bundle carry a natural group structure, depending smoothly
on the base, so this is a group fiber bundle.

\subsection{Fiber integration}%
\label{sec::fibreIntegration}
In this appendix, we review the method of fiber integration in the 
special case of trivial bundles, because it appears often in formulas 
and computations in the text. All the proofs of the following statements 
can be found in \parencite[Section~VII.5]{GreubHalperinEtAl1972}, where 
also the more general situation of locally trivial fiber bundles is 
discussed.

Let \( \pi: M \times F \to M \) be a finite-dimensional trivial fiber 
bundle. Assume that the fiber \( F \) is oriented and closed (\ie, $F$ 
is compact and without boundary). Integrating a differential form 
\( \alpha \in \DiffFormSpace^k(M \times F) \) over the fiber \( F \) 
yields a \( (k - \dim F) \)-form \( \intFibre_F \alpha \) on the base 
\( M \) by defining
\begin{equation}
\left(\,\,\intFibre_F \alpha\right)_m \defeq \int_F \alpha_m \in 
\ExtBundle^{k-\dim F} (\TBundle_m^* M),
\end{equation}
where, for \( m \in M \), the \emphDef{retrenchment} \( \alpha_m \in 
\DiffFormSpace^{\dim F}(F, \ExtBundle^{k-\dim F} (\TBundle_m^* M)) \), 
a $(\dim F)$-form on $F$ with values in the vector space 
$\ExtBundle^{k-\dim F} (\TBundle_m^* M)$, is defined by
\begin{equation}\begin{split}
&\dualPair*{(\alpha_m)_f(Z_1, \dotsc, Z_{\dim F})}{X_1 \wedge \dotsb 
\wedge X_{k-\dim F}}\\
& \qquad
= \alpha_{(m, f)} \bigl(X_1, \dotsc, X_{k-\dim F}, Z_1, 
\dotsc, Z_{\dim F}\bigr)
\end{split}\end{equation}
for \( m \in M\), \(f \in F \), \( X_i \in \TBundle_m M\), and 
\(Z_j \in \TBundle_f F \). This defines a linear map
\begin{equation}
\intFibre_F: \DiffFormSpace^k(M \times F) \to \DiffFormSpace^{k-\dim F}(M).
\end{equation}

\begin{remark}
\label{rem::fibreIntegration:inInfiniteDimensions}
Note that, on a formal level, the discussion above also applies to the 
case of an infinite-dimensional manifold \( M \), because in the end, 
we are just integrating over the finite-dimensional fiber \( F \). 
However, a direct and carefree translation of the previous approach 
to an infinite-dimensional manifold \( M \) is impeded by the fact 
that the cotangent bundle \( \CotBundle M \) is not a \emph{smooth}
bundle in this case \parencite[Remark~I.3.9]{Neeb2006}. Thus, it is 
then not clear in which sense the retrenchment depends smoothly on 
\( m \in M \). Moreover, to integrate a form \( \beta \in 
\DiffFormSpace^{\dim F}(F, V)\) on $F$ with values in a locally convex 
space \( V \), one would need at least Mackey completeness of \( V \) 
(see \parencite[Remark~I.4.4]{Neeb2006}). To circumvent these problems 
for infinite-dimensional base manifolds, 
we directly define the fiber integral by
\begin{equation}
\label{eq::fibreIntegration:expliciteDefinition}
\left(\;\intFibre_F \alpha \right)_m (X_1, \dotsc, X_{k-\dim F}) 
= \int_F \alpha_{(m, \cdot)}(X_1, \dotsc, X_{k-\dim F}, \cdot).
\end{equation}
Here, the integrand is a real valued, top-degree form on \( F \). We then 
have to check that \( \intFibre \alpha \)
is indeed a \emph{smooth} form on \( M \). But this follows from the local 
expression as in the finite-dimensional case; compare 
\parencite[Proof of Proposition~VII.5.7.11.VII]{GreubHalperinEtAl1972}.
\end{remark}

Fiber integration enjoys the following properties:
\begin{enumerate}
	\item 
		Equivariance with respect to maps on the base:
		\begin{equation}
			\label{eq::fibreIntegration:underTransfOnBase}
	\phi^*\intFibre_F \alpha = \intFibre_F (\phi \times \id_F)^* \alpha
		\end{equation}
		holds for all smooth maps \( \phi: M \to M \), see \parencite[Proposition~VII.5.7.12.VIII]{GreubHalperinEtAl1972}.
	\item 
Invariance under fiber transformations: for all orientation-preserving 
diffeomorphisms \( \psi \) of \( F \),
		\begin{equation}
			\label{eq::fibreIntegration:underTransfOnFibre}
			\intFibre_F(\id_M \times \psi)^* \alpha = \intFibre_F \alpha,
		\end{equation}
		see \parencite[Proposition~VII.5.7.12.VIII]{GreubHalperinEtAl1972}.
	\item 
		Contraction with vector fields:
		\begin{equation}
			\label{eq::fibreIntegration:contractionWithFieldsOnBase}
  X \contr \intFibre_F \alpha = \intFibre_F (X \times 0_F) \contr \alpha
		\end{equation}
holds for all vector fields \( X \in \VectorFieldSpace(M) \). Furthermore, 
for \( Z \in \VectorFieldSpace(F) \), we have
		\begin{equation}
			\label{eq::fibreIntegration:contractionWithFieldsOnFibre}
			\intFibre_F (0_M \times Z) \contr \alpha = 0,
		\end{equation}
		see \parencite[Proposition~VII.5.7.13.X]{GreubHalperinEtAl1972}.
	\item
Fiber integration commutes with the exterior differential on \( M \) 
and $M \times F$: 
		\begin{equation}
			\label{eq::fibreIntegration:commuteDifferentiating}
			\dif \intFibre_F \alpha = \intFibre_F \dif \alpha,
		\end{equation}
see \parencite[Proposition~VII.5.7.13.X]{GreubHalperinEtAl1972}. In this 
formula, the fact that $F$ is boundaryless is essential. If $F$ has a 
non-empty boundary, the formula above has contributions from the 
boundary (see \parencite[Problem~VII.4]{GreubHalperinEtAl1972}):
\begin{equation}
\label{eq::fibreIntegration:commuteDifferentiatingBoundary}
\dif \intFibre_F \alpha = \intFibre_F \dif \alpha + 
(-1)^{\deg \alpha + \dim F} \intFibre_{\boundary F} \alpha.
\end{equation}
	\item
		Up-down formula: for any \( \beta \in \DiffFormSpace^k(M) \) and 
		\( \alpha \in \DiffFormSpace^l (M \times F) \), we have
		\begin{equation}
			\label{eq:fibreIntegration:upDownFormula}
\intFibre_F \pi^* \beta \wedge \alpha = \beta \wedge \intFibre_F \alpha,
		\end{equation}
		see \parencite[Proposition~VII.5.7.13.IX]{GreubHalperinEtAl1972}.
\end{enumerate}

\subsection{Hat product for fiber bundles}
\label{sec::hatProductFibreBundles}
\Textcite{Vizman2011} introduced the hat product of differential forms 
as a method to construct differential forms on the mapping space 
\( \sFunctionSpace(M, F) \) starting from differential forms on \( M \) 
and \( F \). We generalize these ideas to induce differential forms on 
the space of smooth sections of a fiber bundle. 

Let \( M \) be a compact, oriented, finite-dimensional manifold without 
boundary and let \( \pi: F \to M \) be a finite-dimensional fiber bundle 
over \( M \). Then the space \( \SectionSpaceAbb{F} \) of smooth sections 
of \( F \) is a Fr\'echet manifold. The evaluation map \( \ev: M \times 
\SectionSpaceAbb{F} \to F \) assigns to a section \( \phi \in 
\SectionSpaceAbb{F} \) and a point \(m \in M\) the value \(\phi(m) \in F\). 
Furthermore, let \( \pr_M: M \times \SectionSpaceAbb{F} \to M \) denote 
the projection on the first factor. We define the \emphDef{hat product} 
as the bilinear map\footnote{We use a different sign convention than the 
one used in \parencite{Vizman2011}.}
\begin{equation}
	\hatProduct: \DiffFormSpace^k(M) \times \DiffFormSpace^l(F) \to 
	\DiffFormSpace^{k+l - \dim M}(\SectionSpaceAbb{F}), \qquad \alpha 
	\hatProduct \omega = \intFibre_M \pr_M^* \alpha \wedge \ev^* \omega,
\end{equation}
where the fiber integration over the trivial bundle \( M \times 
\SectionSpaceAbb{F} \to \SectionSpaceAbb{F} \) is defined as explained in 
\cref{rem::fibreIntegration:inInfiniteDimensions}. 
Unravelling the definitions 
using~\eqref{eq::fibreIntegration:expliciteDefinition} gives the explicit 
formula
\begin{equation}
\label{eq::fibreIntegration:defWithPartialPullback}
(\alpha \hatProduct \omega)_\phi(Y_1, \dotsc Y_r) = 
(-1)^{k r} \int_M \alpha \wedge 
\phi^* \bigl(Y_r \contr \dotsb Y_1 \contr (\omega \circ \phi)\bigr),
\end{equation}
where \( r = k + l - \dim M \) denotes the degree of the hat product, 
\( \alpha \in \DiffFormSpace^k(M) \), \( \omega \in 
\DiffFormSpace^l(F) \), \( Y_i \in \sSectionSpace(\phi^* \VBundle F) \), 
and $\VBundle F \subset  \TBundle F$ is the vertical subbundle. Moreover, 
we used the abbreviation \( \phi^*(Y_r \contr \dotsb Y_1 \contr (\omega 
\circ \phi)) \in \DiffFormSpace^{l-r}(M) \) for the partial pull-back
\begin{equation}\label{eq::diffForm:defnPartialPullback}
\begin{split}
&\phi^*\bigl(Y_r \contr \dotsb Y_1 \contr (\omega \circ \phi)\bigr)_m 
(X_1, \dotsc X_{l-r}) \\
&\qquad = \omega_{\phi(m)}\bigl(Y_1(m), \dotsc, Y_r(m), 
\tangent_m\phi(X_1), \dotsc, \tangent_m \phi(X_{l-r})\bigr) \\
&\qquad = (\ev^* \omega)_{m,\phi} (Y_1, \dotsc, Y_r, X_1, \dotsc, X_{l-r}).
\end{split}
\end{equation}

\begin{remark}%
\label{rem::hatProduct:WithVolumeForm}
In the special case where \( \alpha \) is a volume form \( \mu \) on 
\( M \), we get the following \( l \)-form on \( \SectionSpaceAbb{F} \) 
	\begin{equation}
	\label{eq::hatProduct:VolumeForm}
(\mu \hatProduct \omega)_\phi (Y_1, \dotsc, Y_l) = (-1)^{l \dim M} 
\ \int_M \phi^*(Y_l \contr \dotsb Y_1 \contr (\omega \circ \phi)) \, \mu.
	\end{equation}
Since \( Y_i(m) \) is a vertical tangent vector at \( \phi(m) \), we see 
that \( \mu \hatProduct \omega \) only depends on the vertical component 
of the form \( \omega \). Hence, in this case, we can also consider the 
hat product of \( \mu \) with a differential form \( \omega \) along 
the fibers of \( F \) by using~\eqref{eq::hatProduct:VolumeForm} as the 
definition.
\end{remark}

Let us now explore the resulting calculus.
Since the exterior differential commutes with fiber integration by 
\eqref{eq::fibreIntegration:commuteDifferentiating}, we have
\begin{equation}
\label{eq::hatProduct:exteriorDifferential}
\begin{split}
\dif (\alpha \hatProduct \omega) 
	&= \intFibre_M \dif (\pr_M^* \alpha \wedge \ev^* \omega) \\
	&= \intFibre_M \pr_M^* \dif \alpha \wedge \ev^* \omega + (-1)^k 
	\intFibre_M \pr_M^* \alpha \wedge \ev^* \dif \omega \\
	&= \dif\alpha\hatProduct \omega + (-1)^k \alpha\hatProduct\dif\omega.
\end{split}
\end{equation}
In particular, if \( \alpha = \mu \) is a volume form on \( M \), then, 
by \cref{rem::hatProduct:WithVolumeForm}, 
\begin{equation}
\label{eq::hatProduct:DifferentialWithVolumeForm}
\dif (\mu\hatProduct\omega) = (-1)^{\dim M} \mu\hatProduct\dif_\pi\omega,
\end{equation}
where \( \dif_\pi \) denotes the vertical differential along the fibers. 

For every bundle automorphism \( \psi \) of \( F \), we denote the 
induced diffeomorphism on the base \( M \) by \( \check{\psi} \). With 
this notation, the canonical smooth left action of \( \AutGroup(F) \) 
on the section space \( \SectionSpaceAbb{F} \) takes the form
\begin{equation}
\Upsilon: \AutGroup(F) \times \SectionSpaceAbb{F} \to \SectionSpaceAbb{F}, 
\qquad \psi \cdot \phi = \psi \circ \phi \circ \check{\psi}^{-1}.
\end{equation}
On the infinitesimal level, we obtain
\begin{equation}
\label{eq::hatProduct:InfActionSectionSpace}
Y^*_\phi = \tangent_{\id} \Upsilon_\phi (Y) = \difFracAt{}{\varepsilon}{0} 
\psi_\varepsilon \circ \phi \circ \check{\psi}^{-1}_\varepsilon = 
Y \circ \phi - \tangent \phi (\check{Y}) \in 
\sSectionSpace(\phi^* \VBundle F),
\end{equation}
where, for \( Y \in \AutAlgebra(F) \), the induced vector field on 
\( M \) is denoted by \( \check{Y} \) and the curve 
\( \varepsilon \mapsto \psi_\varepsilon \in \AutGroup(F) \)
is tangent to \( Y \) at \(\varepsilon=0\), \ie, 
\( \psi_0 = \id \) and \( \difFracAt{}{\varepsilon}{\varepsilon=0}
\psi_\varepsilon= Y\). 
The next proposition shows how the action of \( \AutGroup(F) \) 
on \( \SectionSpaceAbb{F} \) interacts with the hat product.
\begin{prop}
Let \( \alpha \in \DiffFormSpace^k(M) \) and \( \omega \in 
\DiffFormSpace^l(F) \). For every \( \psi \in \AutGroup(F) \), whose 
induced diffeomorphism \( \check{\psi} \) on \( M \) is 
orientation-preserving, and \( Y \in \AutAlgebra(F) \) we have
\begin{align}
\label{eq::hatProduct:TransfUnderAutomorphisms}
\Upsilon_\psi^* (\alpha \hatProduct \omega) &= 
(\check{\psi}^* \alpha) \hatProduct (\psi^* \omega), \\
\label{eq::hatProduct:TransfUnderAutomorphismsInf}
\difLie_{Y^*} (\alpha \hatProduct \omega) &= (\difLie_{\check{Y}} \alpha) 
\hatProduct \omega + \alpha \hatProduct (\difLie_Y \omega).
\end{align}
Moreover,
\begin{equation}
\label{eq::hatProduct:ContractionWithKillingVectorField}
Y^* \contr (\alpha \hatProduct \omega) = (\check{Y} \contr \alpha) 
\hatProduct \omega + (-1)^k \alpha \hatProduct (Y \contr \omega).
	\qedhere
	\end{equation}
\end{prop}
\begin{proof}
The following identities clearly hold:
\begin{align}
\pr_M \circ (\id_M \times \Upsilon_\psi) &= \pr_M, \\
\label{eq::hatProduct:evAction}
\ev \circ (\id_M \times \Upsilon_\psi) &= \psi \circ \ev \circ 
(\check{\psi}^{-1} \times \id_{\SectionSpaceAbb{F}}), \\
\pr_M \circ (\check{\psi} \times \id_{\SectionSpaceAbb{F}}) 
&= \check{\psi} \circ \pr_M.
\end{align}
Using these relations, we calculate
\begin{equation}\begin{split}
\Upsilon_\psi^* (\alpha \hatProduct \omega)
&\oversetSmall{\eqref{eq::fibreIntegration:underTransfOnBase}}{=} 
\intFibre_M (\id_M \times \Upsilon_\psi)^* (\pr_M^* \alpha \wedge 
\ev^* \omega) \\
&= \intFibre_M \pr_M^* \alpha \wedge \left(\psi \circ \ev \circ 
(\check{\psi}^{-1} \times \id_{\SectionSpaceAbb{F}})\right)^* \omega \\
&= \intFibre_M (\check{\psi}^{-1} \times \id_{\SectionSpaceAbb{F}})^* 
\left((\check{\psi} \times \id_{\SectionSpaceAbb{F}})^* \pr_M^* \alpha 
\wedge \ev^* (\psi^* \omega)\right) \\
&\oversetSmall{\eqref{eq::fibreIntegration:underTransfOnFibre}}{=} 
\intFibre_M \pr_M^* (\check{\psi}^* \alpha) \wedge \ev^* (\psi^* \omega) \\
&= (\check{\psi}^* \alpha) \hatProduct (\psi^* \omega).
\end{split}\end{equation}

The infinitesimal 
version~\eqref{eq::hatProduct:TransfUnderAutomorphismsInf} 
of~\eqref{eq::hatProduct:TransfUnderAutomorphisms} follows 
by differentiating using the product law. 
For the contraction with the Killing vector field, we have
\begin{equation}
\begin{split}
Y^* \contr (\alpha \hatProduct \omega) 
&\oversetSmall{\eqref{eq::fibreIntegration:contractionWithFieldsOnBase}}{=} 
\intFibre_M (0_M \times Y^*) \contr (\pr_M^* \alpha \wedge \ev^* \omega) \\
&= (-1)^k \intFibre_M \pr_M^* \alpha \wedge \bigl( (0_M \times Y^*) \contr 
\ev^* \omega \bigr) \\
&= (-1)^k \intFibre_M \pr_M^* \alpha \wedge \ev^* (Y \contr \omega)
\\
&\qquad + (-1)^k \intFibre_M \pr_M^* \alpha \wedge \left((-\check{Y} \times 
0_{\SectionSpaceAbb{F}}) \contr \ev^* \omega \right) \\
&= (-1)^k \alpha \hatProduct (Y \contr \omega) - \intFibre_M (\check{Y} 
\times 0_{\SectionSpaceAbb{F}}) \contr (\pr_M^* \alpha \wedge \ev^* \omega)
\\ 
& \qquad + \intFibre_M \left((\check{Y} \times 0_{\SectionSpaceAbb{F}}) 
\contr \pr_M^* \alpha\right) \wedge \ev^* \omega \\
&\oversetSmall{\eqref{eq::fibreIntegration:contractionWithFieldsOnFibre}}
{=} (-1)^k \alpha \hatProduct (Y \contr \omega) + 0 + \intFibre_M \pr_M^* 
(\check{Y} \contr \alpha) \wedge \ev^* \omega \\
&= (\check{Y} \contr \alpha) \hatProduct \omega + (-1)^k \alpha \hatProduct 
(Y \contr \omega).
\end{split}
\end{equation}
Here, the general identity \( Y \contr f^* \beta = f^* ({\rm T} Y \contr 
\beta \circ f) \) for the partial pull-back, 
see~\eqref{eq::diffForm:defnPartialPullback}, was utilized in the form
\begin{align}
(0_M \times Y^*) \contr \pr_M^* \alpha &= \pr_M^* (\tangent \pr_M (0_M 
\times Y^*) \contr \alpha \circ \pr) = 0, \\
\begin{split}
(0_M \times Y^*) \contr \ev^* \omega &= 
\ev^* (\tangent \ev (0_M \times Y^*) 
\contr \omega \circ \ev) \\
&= \ev^* (Y \contr \omega) - \ev^*(\tangent \ev (\check{Y} \times 
0_{\SectionSpaceAbb{F}}) \contr \omega \circ \ev) \\
&= \ev^* (Y \contr \omega) - (\check{Y} \times 0_{\SectionSpaceAbb{F}}) 
\contr \ev^* \omega,
\end{split}\\
(\check{Y} \times 0_{\SectionSpaceAbb{F}}) \contr \pr_M^* \alpha 
&= \pr_M^* (\tangent \pr_M (\check{Y} \times 0_{\SectionSpaceAbb{F}}) 
\contr \alpha) = \pr_M^* (\check{Y} \contr \alpha),
\end{align}
where, for the second identity, we used the infinitesimal version 
of~\eqref{eq::hatProduct:evAction}.
\end{proof}

\subsection{Cheeger-Simons differential characters}%
\label{sec::differentialCharacter}

Cheeger--Simons differential characters model connections on higher 
circle \( n \)-bundles. As such, they generalize circle bundles to 
\enquote{bundles} whose curvature is not a \(  2 \)-form but a 
\(  n \)-form. 
Differential characters were introduced by \textcite{CheegerSimons1985}. 
The idea is to adopt and generalize the representation of principal 
circle bundles with connection via their holonomy map. First
we summarize the main constructions and operations for differential 
characters and then we introduce the hat product of differential 
characters. All details and proofs can be found in 
\parencite{BaerBecker2013} if no other references are given.
The only original work in this section are 
\cref{prop::differentialCharacter:pullback} and 
\cref{sec::hatProduct:differentialCharacters} where we calculate the 
derivative of the pull-back map and where we generalize the hat 
product to differential characters, respectively.

\subsubsection{Motivation: circle bundles}
\label{sec::differentialCharacters:motivationCircleBundles}
Let us first recall some basic facts about ordinary principal 
\( \UGroup(1) \)-bundles in a language which will motivate their 
generalization to higher bundles. 

We will denote the set of equivalence classes of principal 
\( \UGroup(1) \)-bundles with connections over \( M \) by 
\( \csCohomology^2(M, \UGroup(1)) \). The reason for this notation 
will become clear from the considerations below.
We have four canonical maps:
\begin{description}[wide=0em]
	\item[Curvature]
The curvature \( \curv \Gamma \) of the connection \( \Gamma \) is a 
\( 2 \)-form on \( M \). By the Bianchi identity, \( \curv \Gamma \) 
is closed. Furthermore, the periods of the curvature are integral 
(after dividing by \( 2 \pi \)). Thus, in summary, we get a curvature 
map to the space of closed \( 2 \)-forms with integral periods
\begin{equation}
\curv: \csCohomology^2(M, \UGroup(1)) \to \clZDiffFormSpace^2(M, \R).
\end{equation}
\item[Flat bundle]
The holonomy of a flat connection yields a homomorphism \( \Hol: \\ 
\pi_1(M) \to \UGroup(1) \) from the fundamental group to the circle group. 
Since the target is an Abelian group, the commutator subgroup 
\( \commutator{\pi_1(M)}{\pi_1(M)} \) lies in the kernel of \( \Hol \). 
On the other hand, the Abelianization 
\[
\pi_1(M) \slash \commutator{\pi_1(M)}{\pi_1(M)}
\quad \text{is isomorphic to} \quad \sHomology_1(M, \Z) 
\] 
by the Hurewicz theorem. The universal coefficient theorem now gives 
the isomorphism \( \Hom(\sHomology_1(M, \Z), \UGroup(1)) = 
\sCohomology^1(M, \UGroup(1)) \) since \( \UGroup(1) \) is divisible. 
Thus, flat line bundles are represented by \(\sCohomology^1(M,\UGroup(1))\) 
and we get a natural injection
\begin{equation}
\jmath: \sCohomology^1(M, \UGroup(1)) \isomorph \Hom(\pi_1(M) \slash 
\commutator{\pi_1(M)}{\pi_1(M)}, \UGroup(1)) \to \csCohomology^2(M, 
\UGroup(1)).
\end{equation}

\item[Characteristic class]
Assigning the Chern class to a bundle furnishes a map
\begin{equation}
	c: \csCohomology^2(M, \UGroup(1)) \to \sCohomology^2(M, \Z).
\end{equation}
\item[Trivial bundle]
Connections on the trivial bundle \( M \times \UGroup(1) \) are 
canonically identified with real valued \( 1 \)-forms.
So, we can view elements of \( \DiffFormSpace^1(M, \R) \) as 
trivial bundles with connection, and thus get a map 
\( \iota: \DiffFormSpace^1(M, \R) \to \csCohomology^2(M, \UGroup(1)) \).
The kernel of \( \iota \) consists of all \( \alpha \in 
\DiffFormSpace^1(M, \R) \) which are gauge-equivalent to \( 0 \), \ie, 
there exists a gauge transformation \( \phi: M \to \UGroup(1) \) such 
that \( \alpha = \difLog \phi \). For connected \( M \), this is 
exactly the case if \( \alpha \) is closed and integral.
Thus we get an injection
\begin{equation}
\iota: \DiffFormSpace^1(M, \R) \slash \clZDiffFormSpace^1(M, \R) \to 
\csCohomology^2(M, \UGroup(1)).
\end{equation}
\end{description}

\subsubsection{Differential characters}
\label{sec::differentialCharacters:definition}
Let \( M \) be smooth manifold. 
In the following, we assume that all singular simplices are smooth, \ie, 
we work with smooth singular (co-)homology. Let \( \sCycles_{k-1}(M,\Z) \) 
denote the group of smooth integral-valued singular cycles of degree 
\(k-1\).

\begin{defn}[{\parencite[Section~5.1]{BaerBecker2013}}]
A group homomorphisms \( h: \sCycles_{k-1}(M, \Z) \to \UGroup(1) \) is 
called a \emphDef{differential character} if there exists a differential 
form \( \curv h \in \DiffFormSpace^k(M, \R) \) satisfying
	\begin{equation}
		h(\boundary c) = \exp\left( 2 \pi i \int_c \curv h \right)
	\end{equation}
for every smooth singular \( k \)-chain \( c \). The \emphDef{\( k \)-th 
Cheeger--Simons differential cohomology group} 
\( \csCohomology^k(M, \UGroup(1)) \) is the subgroup of 
\( \Hom(\sCycles_{k-1}(M, \Z), \UGroup(1)) \) consisting of differential 
characters.
\end{defn}

Cheeger--Simons differential characters provide a model for differential 
cohomology theory \parencite[Definition~5.9]{BaerBecker2013} in the sense 
that we get a functor \( \csCohomology^\ast(\cdot, \UGroup(1)) \) from the 
category of smooth manifolds to the category of \( \Z \)-graded Abelian 
groups, together with four natural transformations
\begin{itemize}
	\item curvature \( \curv: \csCohomology^\ast(\cdot, \UGroup(1)) \to 
	\clZDiffFormSpace^\ast(\cdot, \R) \)
	\item inclusion of flat classes \( \jmath: \sCohomology^{\ast-1}(\cdot, 
	\UGroup(1)) \to \csCohomology^\ast(\cdot, \UGroup(1)) \)
	\item characteristic class \( c: \csCohomology^\ast(\cdot, \UGroup(1)) 
	\to \sHomology^\ast(\cdot, \Z) \)
	\item topological trivialization \( \iota: \DiffFormSpace^{\ast-1}
	(\cdot, \R) \slash \clZDiffFormSpace^{\ast-1}(\cdot, \R) \to 
	\csCohomology^\ast(\cdot, \UGroup(1)) \)
\end{itemize}
These transformations generalize the maps discussed in 
\cref{sec::differentialCharacters:motivationCircleBundles}.
Moreover, for every smooth manifold \( M \) and \( k \in \Z \) the 
following diagram commutes and all rows and columns are exact:
\begingroup
\small
\begin{equationcd}[label=eq::differentialCharacter:definingDiagram, tikz={column sep=small}]
	&
	0 \to[d] &
	0 \to[d] &
	0 \to[d] &
	\\
	0 \to[r] &
	\JacTorus^{k-1}(M) \to[d]\to[r] &
	\frac{\DiffFormSpace^{k-1}(M, \R)}{\clZDiffFormSpace^{k-1}(M, \R)} 
	\to[d, thick, "\iota"] \to[r] &
	\dif \DiffFormSpace^{k-1}(M, \R) \to[d]\to[r] & 
	0
	\\
	0 \to[r] &
	\sCohomology^{k-1}(M, \UGroup(1)) \to[r, thick, "\jmath"]\to[d] &
	\csCohomology^k(M, \UGroup(1)) \to[d, thick, "c"] 
	\to[r, thick, "\curv"] &
	\clZDiffFormSpace^k(M, \R) \to[d]\to[r] &
	0
	\\
	0 \to[r]	 &
	\ExtAb(\sHomology_{k-1}(M, \Z), \Z) \to[d]\to[r] &
	\sCohomology^k(M, \Z) \to[d]\to[r] &
	\Hom(\sHomology_k(M, \Z), \Z) \to[d]\to[r] &
	0
	\\
	&
	0 &
	0 &
	0 &
\end{equationcd}
\endgroup

Here \( \JacTorus^l(M) \) denotes the Jacobian torus 
\( \sCohomology^l(M, \R) \slash \sCohomology^l(M, \Z) \), and the Abelian 
group \( \ExtAb(\sHomology_{k-1}(M, \Z), \Z) \) of equivalence classes 
of \( \Z \)-extensions of \( \sHomology_{k-1}(M, \Z) \) is canonically 
identified with the torsion subgroup of \( \sCohomology^k(M, \Z) \).
The horizontal sequence in the middle captures information mainly 
about the connection. For ordinary principal \( \UGroup(1) \)-bundle it 
states that every closed integer \( 2 \)-form can be realized as the 
curvature of a connection and that flat bundles are parametrized by their 
holonomy on generators of \( \pi_1(M) \). The topological information is 
encoded in the vertical sequences. The surjective map 
\( c: \csCohomology^k(M, \UGroup(1)) \to \sCohomology^k(M, \Z) \) shows 
that differential cohomology refines integer-valued cohomology.
\medskip

Let us finally comment on the group structure for the case of ordinary 
principal \( \UGroup(1) \)-bundles.
As far as we know, the group structure was first observed by 
\textcite{Kobayashi1956}.
The addition of two principal \( \UGroup(1) \)-bundles \( P \) and 
\( \tilde{P} \) over \( M \) is defined as follows. First, consider 
the fiber product \( P \times_M \tilde{P} \) and then identify the 
points which differ by an \( \UGroup(1) \)-action, \ie, we set
\begin{equation}
	P + \tilde{P} \defeq P \times_M \tilde{P} \quad \slash \quad 
	(p, \tilde{p}) \sim (p \cdot z, \tilde{p} \cdot z^{-1})
\end{equation}
The so-defined object \( P + \tilde{P} \) is indeed a smooth principal 
\( \UGroup(1) \)-bundle, where the \( \UGroup(1) \)-action is the 
translation in the first factor. The trivial bundle constitutes the 
identity element, \ie, \( P + (M \times \UGroup(1)) \) is isomorphic to 
\( P \). For a given principal bundle \( P \), we let \( -P \) denote 
the \( \UGroup(1) \) bundle that has the same underlying bundle 
structure as \( P \) but carries the modified \( \UGroup(1) \)-action
\begin{equation}
 	p \ast z \defeq p \cdot z^{-1},
\end{equation}
where on the right side \( \UGroup(1) \)-acts as in \( P \). Then 
\( P + (-P) \) is isomorphic to the trivial bundle. We leave it to the 
reader to verify that the addition of bundles is commutative and 
associative.
Connections \( A \) and \( \tilde{A} \) on the bundles combine to a 
connection \(\pr_1^* A + \pr_2^* \tilde{A}\) on \(P \times_M \tilde{P}\), 
which descends to a connection \( A + \tilde{A} \) on \( P + \tilde{P} \). 
The curvature of \( A + \tilde{A} \) is the sum of the corresponding 
curvatures. 
\medskip

Coming back to the general case, \( \csCohomology^k(M, \UGroup(1)) \) is 
not just a group but even a Lie group.
\begin{prop}[Lie group structure 
{\parencite[Section~A]{BeckerSchenkelEtAl2014}}]%
\label{prop::differentialCharacter:lieGroup}
Let \( M \) be a compact manifold. For every \( k \geq 1 \), the group 
\( \csCohomology^k(M, \UGroup(1)) \) is an infinite-dimensional 
Fr\'echet Lie group diffeomorphic to 
\begin{equation}
\sCohomology^k(M,\Z) \times \left( \DiffFormSpace^{k-1}(M,\R)\slash 
\clZDiffFormSpace^{k-1}(M, \R) \right) .
	\end{equation}
Its Lie algebra \( \csAlgebra^k(M, \UGroup(1)) \) is isomorphic to 
the Abelian Fr\'echet Lie algebra 
\begin{equation}
 	 \DiffFormSpace^{k-1}(M, \R) \slash \dif \DiffFormSpace^{k-2}(M, \R). 
\end{equation}
With respect to this differentiable structure, all natural maps 
\( \curv, \jmath, c, \iota \) are smooth and hence the defining 
diagram~\eqref{eq::differentialCharacter:definingDiagram} is a 
commutative diagram of Lie groups.
\end{prop}

Let us now describe the different operations that constitute the calculus 
of differential cohomology.
\paragraph{Pull-back}
A smooth map \( f: M \to N \) induces a pull-back operation 
\begin{equation}
	f^*: \csCohomology^k(N, \UGroup(1)) \to \csCohomology^k(M,\UGroup(1))
\end{equation}
on the level of differential characters. Here \( f^* h \) is defined 
by \( f^* h = h \circ f_* \), \ie, as the dual of the induced map 
\( f_*: \sCycles_{k-1}(M, \Z) \to \sCycles_{k-1}(N, \Z) \) on cycles. 
Pull-back commutes with the curvature and the characteristic map.

\begin{prop}[Smoothness of the pull-back operation]%
\label{prop::differentialCharacter:pullback}
Let \( h \in \csCohomology^k(N, \UGroup(1)) \) be a differential 
character on the finite-dimensional manifold \( N \). For every 
closed manifold \( M \), the pull-back map
\begin{equation}
\pb_h: \sFunctionSpace(M, N) \to \csCohomology^k(M, \UGroup(1)),
		\qquad \phi \mapsto \phi^* h
\end{equation}
is smooth with logarithmic derivative
\begin{equation}
\label{eq::differentialCharacter:pullbackDerivative}
\difLog_\phi \pb_h (X) = \equivClass[\big]{\phi^* (X \contr \curv h \circ 
\phi)} \in \DiffFormSpace^{k-1}(M, \R) \slash \dif 
\DiffFormSpace^{k-2}(M, \R)
\end{equation}
for all \( X \in \TBundle_\phi \sFunctionSpace(M, N) \).
\end{prop}
\begin{proof}
Let \( \phi_\varepsilon: M \to N \) be a smooth family of maps whose derivative at $ \varepsilon =0$ is the 
tangent vector \( X \in \TBundle_\phi \sFunctionSpace(M, N) \).
The homotopy formula \parencite[Example~7.19]{BaerBecker2013}
\begin{equation}
\phi_\varepsilon^* h - \phi^* h = (-1)^{k-1}\ \iota 
\left(\ \ \ \intFibre_{[0,1]} \tilde{\phi}_\varepsilon^* \curv h \right)
\end{equation} 
reduces the problem to the case of differential forms.
Here, \( \tilde{\phi}_\varepsilon: M \times [0,1] \to N \) denotes the 
smooth homotopy between \( \phi \) and \( \phi_\varepsilon \) (for a 
fixed \( \varepsilon \)).
The following calculation using \cref{prop::pullbackDiffForm} below, 
proves the identity \eqref{eq::differentialCharacter:pullbackDerivative}:
\begin{equation}\begin{split}
\difFracAt{}{\varepsilon}{0}\ \ \intFibre_{[0,1]} 
\tilde{\phi}_\varepsilon^* \curv h &=\ 
\intFibre_{[0,1]}\difFracAt{}{\varepsilon}{0}\tilde{\phi}_\varepsilon^* 
			\curv h\\
&\oversetSmall{\eqref{eq::pullbackDiffFormDerivative}}{=}
\ \intFibre_{[0,1]} \dif\ \tilde{\phi}_0^* \left( 
\difFracAt{}{\varepsilon}{0} \tilde{\phi}_\varepsilon \contr \curv h \circ 
\tilde{\phi}_0 \right) \\ &=\ 
\intFibre_{[0,1]} \dif \left( \pr_M^* \bigl(\phi^*(X \contr \curv h \circ 
			\phi)\bigr) \wedge \pr_{[0,1]} \right) \\
&\oversetSmall{\eqref{eq::fibreIntegration:commuteDifferentiatingBoundary}}
{=}\ \dif \intFibre_{[0,1]} \pr_M^* \bigl(\phi^*(X \contr 
\curv h \circ \phi)\bigr) \wedge \pr_{[0,1]} \\
& \quad + (-1)^{k-1} \intFibre_{\boundary[0,1]} \pr_M^* \bigl(\phi^*(X 
\contr \curv h \circ \phi)\bigr) \wedge \pr_{[0,1]} \\
&\oversetSmall{\eqref{eq:fibreIntegration:upDownFormula}}{=}\ 0 +
(-1)^{k-1} \phi^*(X \contr \curv h \circ \phi) \wedge 
\intFibre_{\boundary[0,1]} \pr_{[0,1]} \\
&=\ (-1)^{k-1} \phi^*(X \contr \curv h \circ \phi),
\end{split}\end{equation}
where we used the notation of the partial 
pull-back~\eqref{eq::diffForm:defnPartialPullback}.
\end{proof}
\begin{proof}[Alternative proof]
In the case \( k = 2 \) (where \( \csCohomology^2(N, \UGroup(1)) \) 
parametrizes principal circle bundles with connection), we can give 
a more direct and geometric proof. Let \(P \to N\) be principal 
\( \UGroup(1) \)-bundle with connection \( \Gamma \). Let 
\( \eta_\varepsilon \) be a family of smooth closed curves in \( N \) 
representing the tangent vector \( Y \in \TBundle_\eta 
\sFunctionSpace(S^1, N) \). The derivative of the holonomy map 
\( \Hol: \sFunctionSpace(S^1, N) \to \UGroup(1) \) in the direction 
of \( Y \) can be calculated, heuristically, by 
\begin{equation}
\difFracAt{}{\varepsilon}{0} \Hol(\eta_\varepsilon) = \difFracAt{}
{\varepsilon}{0} \exp \left(\ \int_{\eta_\varepsilon} \Gamma \right) = 
\Hol(\eta) + \int_{S^1} (\curv \Gamma)_{\eta(\tau)} (Y(\tau), \dot{\eta}
(\tau)) \dif \tau.
\end{equation}
This calculation can be made rigorous and even generalizes to non-Abelian 
bundles; see \parencite{Diez2017c} for details. 
Now let \( \gamma \) be a closed curve in \( M \) and consider a family 
of smooth maps \( \phi_\varepsilon: M \to N \) representing \( X \in 
\TBundle_\phi \sFunctionSpace(M, N) \). The above calculation implies
\begin{equation}\begin{split}
\difFracAt{}{\varepsilon}{0} \pb_h (\phi_\varepsilon)(\gamma) 
&= \difFracAt{}{\varepsilon}{0} h (\phi_\varepsilon \circ \gamma) \\
&= h(\phi \circ \gamma) 
+ \int_{S^1} (\curv h)_{\phi \circ \gamma(\tau)} (X_{\gamma(\tau)}, 
\tangent \phi (\dot{\gamma}(\tau))) \dif \tau.
\end{split}\end{equation}
Thus, the derivative of the pull-back map is
\begin{equation}
\difLog_\phi \pb_h (X) (\gamma) = 
\int_\gamma \phi^* (X \contr \curv h \circ \phi).
		\qedhere
\end{equation}   
\end{proof}

\begin{lemma}\label{prop::pullbackDiffForm}
Let \( \alpha \) be a differential \( k \)-form on the finite-dimensional
manifold \( N \). For every closed finite-dimensional manifold \( M \), 
the pull-back map
\begin{equation}
\pb_\alpha: \sFunctionSpace(M, N) \to \DiffFormSpace^k(M), 
\qquad \phi \mapsto \phi^* \alpha
\end{equation}
is smooth and has derivative
\begin{equation}
\label{eq::pullbackDiffFormDerivative}
\tangent_\phi \pb_\alpha (X) = \difLie^\phi_X \alpha = 
\dif \bigl(\phi^* (X \contr \alpha \circ \phi)\bigr) + 
\phi^* \bigl(X \contr (\dif \alpha \circ \phi)\bigr)
\end{equation}
for \( X \in \TBundle_{\phi} \sFunctionSpace(M, N) \).
Here \( \difLie^\phi_X \alpha \defeq \difFracAt{}{\varepsilon}{0} 
\phi_\varepsilon^* \alpha \) denotes the Lie derivative along the 
map \( \phi \), so that the second equality 
in~\eqref{eq::pullbackDiffFormDerivative} generalizes Cartan's formula. 
Moreover, we used the notation of the partial pull-
back~\eqref{eq::diffForm:defnPartialPullback}.
\end{lemma}
\begin{proof}
Let \( \phi_\varepsilon: M \to N \) be a family of maps representing the 
tangent vector \( X \in \TBundle_\phi \sFunctionSpace(M, N) \). We have 
\( \tangent_\phi \pb_\alpha (X) = \difFracAt{}{\varepsilon}{0} 
\phi_\varepsilon^* \alpha = \difLie^\phi_X \alpha \) by definition. For the 
generalized Cartan formula, we need to check the identity only for 
functions and for exact \( 1 \)-forms, because these generate the 
exterior algebra. 
On functions \( \alpha = f \in \sFunctionSpace(N) \), we get
\begin{equation}
(\difLie^\phi_X f)_m =  
\difFracAt{}{\varepsilon}{0} (\phi_\varepsilon^* f)(m) = 
\tangent_{\phi(m)}f (X_m) = 
\bigl(\phi^* (X \contr \dif f \circ \phi)\bigr)_m.
\end{equation}
Similarly, for exact one forms \( \alpha = \dif f \), we find
\begin{equation}
\difLie^\phi_X (\dif f) = \difFracAt{}{\varepsilon}{0} \phi_\varepsilon^* 
\dif f = \dif \difFracAt{}{\varepsilon}{0} \phi_\varepsilon^* f = 
\dif \bigl(\phi^* (X \contr \dif f \circ \phi)\bigr).
\end{equation}
This completes the proof.
\end{proof}

\paragraph{Product}
\label{sec::differentialCharacter:starProduct}

There exists an associative and \( \Z \)-bilinear product of differential 
characters denoted by
\begin{equation}
\label{Appendix_D_def_fat_star}
\star: \csCohomology^k(M, \UGroup(1)) \times \csCohomology^l(M, \UGroup(1)) 
\to \csCohomology^{k+l}(M, \UGroup(1)).
\end{equation}
It satisfies the following properties 
\parencite[Definition~6.1]{BaerBecker2013}:
\begin{enumerate}
\item (Ring structure) The product \( \star \) is associative and bilinear.
\item (Graded commutativity) \( h \star g  = (-1)^{k l} g \star h \).
\item (Naturality under pull-back) For a smooth map \( f: N \to M \), we 
have \( f^* (h \star g) = (f^* h) \star (f^* g) \).
\item (Compatibility with curvature) \( \curv (h \star g) = 
\curv h \wedge \curv g \).
\item (Compatibility with characteristic class) \( c(h \star g) = 
c(h) \cup c(g) \).
\item (Compatibility with trivializations) \( \iota(\alpha) \star h = 
\iota(\alpha \wedge \curv h) \).
\end{enumerate}
As shown in \parencite{BaerBecker2013}, these properties completely 
characterize the product~\eqref{Appendix_D_def_fat_star}; we refer to 
this book for the lengthy proof of this statement.
In this paper, we use only these properties without appealing to the 
concrete construction of \( \star \). 

\paragraph{Fiber integration}
Let \( \pi: F \to M \) be a smooth fiber bundle whose typical fiber 
\( \FibreBundleModel{F} \) is a closed, finite-dimensional oriented 
manifold. Then fiber integration \parencite[Definition~7.1]{BaerBecker2013} 
yields a group homomorphism
\begin{equation}
	\intFibre: \csCohomology^k(F, \UGroup(1)) 
\to \csCohomology^{k - \dim \FibreBundleModel{F}}(M, \UGroup(1))
\end{equation}
such that the curvature map intertwines fiber integration of differential 
characters with ordinary fiber integration of differential forms
\( \intFibre: \DiffFormSpace^k(F) \to 
\DiffFormSpace^{k - \dim \FibreBundleModel{F}}(M) \), \ie, 
the following diagram commutes:
\begin{equationcd}
	\csCohomology^k(F, \UGroup(1))
		\to[r, "\intFibre"]
		\to[d, "\curv"]
	& \csCohomology^{k - \dim \FibreBundleModel{F}}(M, \UGroup(1))
		\to[d, "\curv"]
	\\
	\DiffFormSpace^k(F)
		\to[r, "\intFibre"]
	& \DiffFormSpace^{k - \dim \FibreBundleModel{F}}(M).
\end{equationcd}

\subsubsection{Hat product of differential characters}%
\label{sec::hatProduct:differentialCharacters}

\Textcite{Vizman2011} introduced the hat product of differential forms 
as a means to construct differential forms on the mapping space 
\( \sFunctionSpace(M, F) \) from differential forms on the 
finite-dimensional manifolds \( M \) and \( F \). 
We   now generalize these ideas to the level of differential characters.
A similar construction has been used also in 
\parencite{DiezJanssensNeebVizmann2019}.
The hat product of differential characters yields a method that combines 
higher bundles on \( M \) and \( F \) to bundles on the mapping space. 
In the process, we also generalize the procedure of higher dimensional 
transgression \parencite[Chapter~9.2]{BaerBecker2013}.

As in \cref{sec::hatProductFibreBundles}, we consider the more general 
setup of sections of a fiber bundle.
Let \( M \) be a compact, oriented, finite-dimensional manifold without 
boundary, and let \( \pi: F \to M \) be a finite-dimensional fiber 
bundle over \( M \). Recall from \parencite{Vizman2011} that the hat 
product of differential forms used the evaluation and projection map to 
construct differential forms on \( M \times \SectionSpaceAbb{F} \) and then 
integrate out the \( M \)-factor. We use the same strategy to define the 
hat product of differential characters. The evaluation map 
\( \ev: M \times \SectionSpaceAbb{F} \to F \) yields a pull-back map of 
differential characters \( \ev^*: \csCohomology^l(F, \UGroup(1)) \to 
\csCohomology^l(M \times \SectionSpaceAbb{F}, \UGroup(1)) \). Similarly, 
the projection on the first factor induces a homomorphism \( \pr_M^*: 
\csCohomology^k(M, \UGroup(1)) \to \csCohomology^k(M \times 
\SectionSpaceAbb{F}) \). Taking the product of these differential forms 
yields an element of \(\csCohomology^{k+l}(M\times\SectionSpaceAbb{F})\). 
Finally, we integrate over \( M \) by viewing 
\( M \times \SectionSpaceAbb{F} \) as a trivial bundle over 
\( \SectionSpaceAbb{F} \). Summarizing, we have constructed a map
\begin{equation}\begin{split}
\hatProduct: \csCohomology^k(M, \UGroup(1)) \,\times\, 
&\csCohomology^l(F, \UGroup(1)) 
\to \csCohomology^{k+l - \dim M}(\SectionSpaceAbb{F}, \UGroup(1)), 
\\
h \hatProduct g &\defeq \intFibre_M \pr_M^* h \star \ev^* g.
\end{split}\end{equation}
Using the compatibility with the curvature map, we calculate
\begin{equation}
\label{eq::differentialCharacter:hatProductCurvature}
\curv (h \hatProduct g) = \intFibre_M \pr_M^* \curv h \wedge \ev^* \curv g 
= \curv h \hatProduct \curv g,
\end{equation}
that is, the curvature of the hat product of differential characters 
equals the hat product of the corresponding curvatures. 
Let us now discuss a few special cases.
\begin{description}[wide=0em]	
\item[Prequantization]
Fix a differential character \( \mu \) on \( M \) whose degree 
coincides with the dimension of \( M \).
Then we define a \( \mu \)-dependent map by
\begin{equation}
\csCohomology^l(F, \UGroup(1)) \to 
\csCohomology^l(\SectionSpaceAbb{F}, \UGroup(1)), 
\qquad h \mapsto \mu \hatProduct h.
\end{equation}
In particular, for \( l=2 \), principal \( \UGroup(1) \)-bundles on 
\( F \) yield principal \( \UGroup(1) \)-bundles on the section 
space \( \SectionSpaceAbb{F} \).
By~\eqref{eq::differentialCharacter:hatProductCurvature}, the 
differential character \( \mu \hatProduct h \) is a prequantization 
of \( (\SectionSpaceAbb{F}, \curv \mu \hatProduct \curv g) \).
\item[Transgression]
The \emphDef{transgression map} \parencite[Definition~9.1]{BaerBecker2013} 
uses the hat product with the identity element 
\( e \in \csCohomology^0(M, \UGroup(1)) \):
\begin{equation}
\csCohomology^l(F, \UGroup(1)) \to \csCohomology^{l-\dim M}
(\SectionSpaceAbb{F}, \UGroup(1)), \qquad h \mapsto e \hatProduct h = 
\intFibre_M \ev^* h.
\end{equation}
For the special case of a trivial fiber bundle 
\( F = \FibreBundleModel{F} \times M \) over \( M = S^1 \), the space 
of smooth sections is identified with the space 
\( \sLoopSpace(\FibreBundleModel{F}) \) of smooth loops in 
\( \FibreBundleModel{F} \), and transgression yields a map 
\( \csCohomology^l(\FibreBundleModel{F}, \UGroup(1)) \to 
\csCohomology^{l-1}(\sLoopSpace(\FibreBundleModel{F}), \UGroup(1)) \).
In degree \( l=2 \), it assigns to a principal \( \UGroup(1) \)-bundle 
over \( \FibreBundleModel{F} \) its holonomy map 
\( \sLoopSpace(\FibreBundleModel{F}) \to \UGroup(1) \).
\end{description}

\begin{refcontext}[sorting=nyt]{}
	\printbibliography
\end{refcontext}

\end{document}